\documentclass[11pt]{amsart}
\usepackage{amsmath}
\usepackage[psamsfonts]{amssymb}
\usepackage[all]{xy}
\usepackage{graphicx}
\usepackage[T1]{fontenc}
\usepackage[bitstream-charter]{mathdesign}
\usepackage{hyperref}
\usepackage{amscd}
\usepackage{verbatim}
\usepackage[english]{babel}
\usepackage{xcolor}
\usepackage{tikz}
\usetikzlibrary{shapes.geometric, arrows}
\tikzstyle{box} = [rectangle, minimum width=3cm, text centered, text width=3cm, draw=black]
\tikzstyle{arrow} = [thick,->,>=stealth]
\let\oldtocsection=\tocsection

\let\oldtocsubsection=\tocsubsection

\renewcommand{\tocsection}[2]{\hspace{0em}\oldtocsection{#1}{#2}}
\renewcommand{\tocsubsection}[2]{\hspace{1em}\oldtocsubsection{#1}{#2}}

\usepackage{mathtools}
\usepackage{mathdots}
\theoremstyle{definition}
\newtheorem{theorem}{Theorem}[section]
\newtheorem{prop}[theorem]{Proposition}
\newtheorem{lemma}[theorem]{Lemma}
\newtheorem{cor}[theorem]{Corollary}

\newtheorem{notation}[theorem]{Notation}

\newtheorem{ex}[theorem]{Example}
\theoremstyle{remark}
\newtheorem{dfn}[theorem]{Definition}
\newtheorem{remark}[theorem]{Remark}

\numberwithin{equation}{section}

\def\co{\colon\thinspace}
\def\delbar{\bar{\partial}}

\def\ep{\epsilon}

\def\R{\mathbb{R}}
\def\Z{\mathbb{Z}}
\def\N{\mathbb{N}}
\def\C{\mathbb{C}}

\DeclareMathOperator{\Img}{Im}
\begin{document}
\title{Abstract interlevel persistence for Morse-Novikov and Floer theory}
\author{Michael Usher}
\address{Department of Mathematics, University of Georgia, 
Athens, GA 30602}
\email{usher@uga.edu}
\begin{abstract}
    We develop a general algebraic framework involving ``Poincar\'e--Novikov structures'' and ``filtered matched pairs'' to provide an abstract approach to the barcodes associated to the homologies of interlevel sets of $\R$- or $S^1$-valued Morse functions, which can then be applied to Floer-theoretic situations where no readily apparent analogue of an interlevel set is available.   The resulting barcodes satisfy abstract versions of stability and duality theorems, and in the case of Morse or Novikov theory they coincide with the standard barcodes coming from interlevel persistence.  In the case of Hamiltonian Floer theory, the lengths of the bars yield multiple quantities that are reminiscent of the spectral norm of a Hamiltonian diffeomorphism.
\end{abstract}
\maketitle
\tableofcontents

\section{Introduction}

If $f\co X\to \R$ is a suitably tame function on a compact topological space $X$, the \emph{sublevel persistence module} of $f$ (say with coefficients in a field $\kappa$) consists of the data of the homologies of the sublevel sets $H_{*}(X^{\leq t};\kappa)$ where $X^{\leq t}=f^{-1}((-\infty,t])$, together with the inclusion-induced maps $H_{*}(X^{\leq s};\kappa)\to H_*(X^{\leq t};\kappa)$.  Aspects of this have been studied at least since \cite{Mor}; in contemporary language, as seen in \cite{Bar} and \cite{ZC}, the sublevel persistence module can be regarded as arising from the homologies of the terms in a filtered chain complex, and decomposes as a sum  of interval modules $\kappa_{[s,t)}$ where $-\infty<s<t\leq \infty$.  Each summand $\kappa_{[s,t)}$ corresponds to a one-dimensional subspace that first appears in $H_{*}(X^{\leq s};\kappa)$ and, if $t=\infty$,  maps injectively to $H_{*}(X;\kappa)$, while if $t<\infty$ the subspace maps to zero in the homologies of the sublevels $H_{*}(X^{\leq t'};\kappa)$ iff $t'\geq t$.  The collection of intervals $[s,t)$ in this decomposition is called the sublevel barcode of $f$.  

Ideas related to (what is now called) sublevel persistence have for some time been influential in symplectic topology; key early works in this direction include \cite{Vit},\cite{FH},\cite{Oh},\cite{Sc00}.  While the first of these references used the (relative) singular homologies of sublevel sets of a function on an auxiliary topological space, work since that time has more often used some version of filtered Floer homology.  The latter involves a filtered chain complex that is constructed by analogy with the Morse complex familiar from Morse theory on finite-dimensional manifolds; however, as the filtered Floer complex is based on a function $\mathcal{A}$ on an infinite-dimensional manifold (such as the free loop space $\mathcal{L}M$ of a symplectic manifold $(M,\omega)$), with critical points of infinite Morse index, the associated filtered Floer groups do not actually represent the homologies of the sublevel sets of $\mathcal{A}$.  Nonetheless, from an algebraic standpoint the filtered Floer persistence module is fairly well-behaved, and various quantities that can be extracted from it convey interesting geometric information about, \emph{e.g.}, Hamiltonian diffeomorphisms and Lagrangian submanifolds.  The language of persistent homology and in particular barcodes was first brought to Floer theory in \cite{PS} and additional relevant theory was developed in \cite{UZ}; this framework has found use in a variety of  recent symplectic applications such as \cite{She},\cite{CGG}.

An \emph{interlevel set} of a function $f\co X\to \R$ is by definition the preimage $X_{[a,b]}:=f^{-1}([a,b])$ of some closed interval $[a,b]\subset \R$; for example level sets arise as the special case that $a=b$.  One then obtains a persistence module parametrized by the poset of closed intervals, based on the inclusion-induced maps $H_k(X_{[a,b]};\kappa)\to H_k(X_{[c,d]};\kappa)$ for $[a,b]\subset [c,d]$.  This persistence module, under suitable hypotheses, also satisfies a decomposition theorem which allows it to be classified by a collection of intervals which we will refer to as the interlevel barcode of $f$; this can be understood either in terms of the quiver-theoretic classification of zigzag diagrams such as \[ \cdots \leftarrow H_k(X_{[s,s]};\kappa)\to H_k(X_{[s,t]};\kappa)\leftarrow H_k(X_{[t,t]};\kappa)\to H_k(X_{[t,u]};\kappa)\leftarrow \cdots \] as in \cite{CDM}, or in terms of the fact that the Mayer-Vietoris sequence imposes special structure on the interlevel persistence module within the class of two-dimensional peristence modules, as in \cite{CO},\cite{BGO}.\footnote{As the Mayer-Vietoris sequence works better with open sets than closed ones, \cite{CO},\cite{BGO} use preimages of open intervals, $X_{(a,b)}=f^{-1}((a,b))$ in the role of their interlevel sets.  In the motivating cases for this paper, $f$ will be a Morse function, which implies that $X_{[a,b]}$ is a deformation retract of $X_{(a-\epsilon,b+\epsilon)}$ for sufficiently small $\epsilon>0$, so working with closed intervals will not cause additional difficulty.  See \cite{CDKM} and \cite{Bubook} for considerations related to (closed) interlevel persistence for less well-behaved functions $f$.}  The interlevel barcode comprises, in general, finite-length intervals of all four types $[a,b),(a,b],(a,b),[a,b]$; roughly speaking, the presence of an interval $I$ in the interlevel barcode corresponds to a one-dimensional summand in each level-set homology $H_k(X_{[t,t]};\kappa)$ for $t\in I$, such that if $[s,t]\subset I$ the respective summands of $H_k(X_{[s,s]};\kappa)$ and $H_k(X_{[t,t]};\kappa)$ have common, nontrivial image in $H_k(X_{[s,t]};\kappa)$.

The interlevel persistence barcode of a suitably well-behaved function $f\co X\to \R$ contains more information than the sublevel persistence barcode.  From \cite[Section 3]{CDM},\cite{ext} one can see that the (finite-length) bars of form $[a,b)$ are the same for both the sublevel and interlevel barcodes, while the bars of form $(a,b]$ in the interlevel barcode of $f$ correspond (modulo grading adjustment) to the finite-length bars $[-b,-a)$ in the sublevel barcode of $-f$; in the case that $f$ is a Morse function on a compact smooth $\kappa$-oriented manifold, Poincar\'e duality relates the latter to bars in the sublevel barcode of $f$.  Both varieties of half-open bars $[a,b),(a,b]$ detect homology classes in interlevel (or sublevel) sets of $\pm f$ that vanish upon inclusion into the whole space $X$; the closed or open bars $[a,b],(a,b)$, on the other hand, correspond to classes that are nontrivial in $H_*(X;\kappa)$.  Indeed, in the case of a Morse function on a compact smooth  $\kappa$-oriented manifold,  such bars come in pairs $\{[a,b],(a,b)\}$ with homological degrees adding to $n-1$, and each such pair corresponds to two infinite-length bars $[a,\infty),[b,\infty)$ in the sublevel barcode of $f$.  Thus, in this Morse case, the additional information provided by the interlevel barcode in comparison to the sublevel barcode amounts to a \emph{pairing} between the endpoints of the infinite-length bars of the sublevel barcode. See Figure \ref{genustwo} for a simple example of two functions on a genus-two surface having the same sublevel barcode but different interlevel barcodes.

\begin{center}
\begin{figure}
\includegraphics[width=4.5 in]{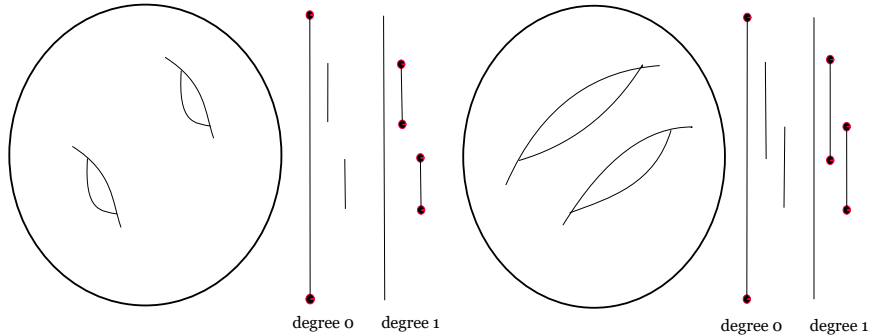}
\caption{The height functions on the above genus two surfaces are perfect Morse functions (\emph{i.e.} the Morse inequalities are equalities), so their interlevel barcodes have no half-open intervals, and their sublevel barcodes consist only of semi-infinite intervals beginning at the critical values, which are identical in the two cases.  However, the closed and open intervals in their interlevel barcodes differ, as shown. (Segments with solid endpoints represent closed intervals; those without them represent open intervals.)}
\label{genustwo}
\end{figure}
\end{center}

The present work grew out of a project to develop an abstract algebraic setup for interlevel persistence that would be general enough to provide interlevel-type barcodes in the context of Hamiltonian Floer theory and its variants in symplectic topology, just as the ``Floer-type complexes'' of \cite{UZ} allow one to construct symplectic-Floer-theoretic versions of sublevel barcodes.  Since these sublevel barcodes completely classify Floer-type complexes up to suitable isomorphism of filtered complexes by \cite[Theorems A and B]{UZ}, while the interlevel barcodes should contain more information than the sublevel barcodes, our treatment must take into account features separate from the filtered isomorphism type of the Floer complex.  Our resolution to this problem uses the abstract algebraic notion of a ``\textbf{Poincar\'e--Novikov structure},'' discussed briefly in Section \ref{intropn} and in more detail in Section \ref{pnsect}, which, when applied to Hamiltonian Floer theory in Section \ref{floersect}, incorporates information from the PSS isomorphism \cite{PSS} as well as the Poincar\'e intersection pairing in the underlying symplectic manifold. A general algebraic procedure associates to a Poincar\'e--Novikov structure a ``\textbf{filtered matched pair}'' in every grading, from which interlevel-barcode-type invariants can be extracted.   
While the constructions in this paper were originally motivated by symplectic topology, they are algebraic in nature and as such are adaptable enough to potentially be of broader interest.  Various classical-topological instances of our constructions are described in Section \ref{geomintro}.

There are two main issues that make the adaptation of interlevel persistence to Floer theory nontrivial; in isolation, either of these could perhaps be addressed by small modifications of standard methods, but in combination they seem to require new techniques such as the ones we develop here.  The first point is that, while the homologies of sublevel sets are naturally mimicked by the homologies of filtered subcomplexes of the Floer complex, there does not seem to be any established Floer-theoretic counterpart to the homology of an interlevel set.\footnote{The often-used ``action window Floer homology,'' denoted $S^{[a,b)}$ in \cite{FH}, can be considered a counterpart to the \emph{relative} homology $H_{*}(f^{-1}([a,b)),f^{-1}(\{a\}))$, but for interlevel persistence we would want absolute homology---in particular we would not want the homology to vanish if $[a,b)$ contains no critical values, as is the case for action-window homology.}  Thus the zigzag or two-parameter persistence modules that are usually used to develop interlevel persistence theory are not available to us, at least initially.  (We will eventually construct an object that can serve this purpose (see (\ref{hkdef})), but it should be regarded as a byproduct of our theory, not as an input.)  

A way around this difficulty is suggested by the isomorphism from \cite{CDM} between interlevel persistence and the \emph{extended persistence} of \cite{ext}.  The latter is based on constructions directly with sublevel and superlevel sets of a function rather than interlevel sets, and thus seems conceptually closer to Floer theory.  However, this runs into our second issue, namely that in many cases of interest the filtered Floer complex is an analogue not of the Morse complex of a smooth function on a compact manifold $X$ but rather of the Novikov complex \cite{Nov} of a function $\tilde{f}\co \tilde{X}\to \R$ on a covering space $\tilde{X}$ of $X$, such that $d\tilde{f}$ is the pullback of  closed one-form $\theta$ on $X$.  (If the group of periods $\Gamma=\{\int_{S^1}\gamma^*\theta|\gamma\co S^1\to X\}$ is nontrivial and discrete, $\theta$ can be considered as the derivative of a function $X\to S^1$.)   While versions of interlevel persistence have been developed in the Novikov-theoretic context  both for discrete (\cite{BD},\cite{BH17},\cite{B18}) and indiscrete (\cite{B1}\cite{B20b}) period groups, an analogous construction for extended persistence does not seem to be available in the literature in a form that lends itself to straightforward adaptations to Floer theory.  Some of the constructions in this paper can be regarded as providing such an extension of extended persistence, in a general algebraic framework.  See Section \ref{connext} for more about the connection to extended persistence, and Remark \ref{burghelea} for a comparison to Burghelea's work.

\subsection{Organization of the paper}  The following two sections can be regarded as an extended introduction: Section \ref{algintro} gives an overview of the key elements of our algebraic setup, and Section \ref{geomintro} explains more concretely how our structures arise in finite-dimensional Morse and Novikov theory and other contexts, and indicates relations to other constructions.  Sections \ref{algintro} and \ref{geomintro} both contain few proofs, and frequently refer forward in the paper for precise definitions.   We defer the discussion of the motivating case of Hamiltonian Floer theory to Section \ref{floersect}, after the relevant algebra has been explained in more detail.  Section \ref{basic} sets up some of the basic algebraic ingredients for the rest of the paper, including an abstraction of Poincar\'e duality in Section \ref{pdsec}.  Section \ref{fmpsec} introduces the central notion of a filtered matched pair and proves several key results about such objects, notably a stability theorem (Theorem \ref{stab}) and, in Section \ref{basisconstruct}, an existence theorem for the doubly-orthogonal bases that we use to construct the open and closed intervals in our version of the interlevel barcode when the period group $\Gamma$ is discrete. Section \ref{pnsect} introduces our notion of a Poincar\'e--Novikov structure, which mixes abstract versions of Poincar\'e duality and filtered Novikov theory, and connects Poincar\'e--Novikov structures to filtered matched pairs.  Having developed the algebra, we then apply it in Section \ref{floersect} to Hamiltonian Floer theory, yielding (again when $\Gamma$ is discrete, though some structure remains when it is not) an ``essential barcode'' associated to a Hamiltonian flow that is designed to play the role of the collection of intervals in an interlevel barcode that are either open or closed. (The remaining, half-open intervals in interlevel barcodes are analogous to the finite-length intervals in the barcodes from \cite{UZ}.)  Example \ref{smallfloer} shows that this analogy becomes a precise equivalence in the case that the Hamiltonian is small and time-independent, with suitable Floer-theoretic data equivalent to Morse-theoretic data.   At the end of Section \ref{floersect} we identify the endpoints of the intervals in the essential barcode (up to sign) with particular spectral invariants of the Hamiltonian flow and of its inverse; the lengths of these intervals---of which there are as many as the dimension of the homology of the manifold---could be regarded in somewhat the same spirit as the spectral norm of a Hamiltonian diffeomorphism.  Section \ref{floersect} and the related material on orientations in Section \ref{pssor} are the only parts of the rest of the paper that discuss symplectic topology, and can be skipped by those whose interests lie elsewhere.

The constructions through Section \ref{floersect} focus on persistence features that are ``homologically essential'' in that they arise from and can be computed in terms of the behavior of nontrivial global homology classes.  For a full picture one should also incorporate data---represented in interlevel barcodes by half-open intervals---that vanish in global homology.  Thus in Section \ref{clsec} we introduce chain-level versions of our constructions which make such information visible.  We synthesize in (\ref{hkdef}) a two-dimensional persistence module from such a chain-level structure that is designed to play the role of the two-dimensional persistence module given by homologies of interlevel sets. We also show that this persistence module can be computed in terms of the homology-level structures discussed in Section \ref{fmpsec} together with the sublevel-persistence-type barcodes from \cite{UZ}. Section \ref{concretemorse} reviews basic features of Morse and Novikov complexes, and explains in detail how to associate our chain-level structures to such complexes. Then, in Section \ref{isosect}, we use ideas from \cite{Pa} to prove (Theorem \ref{bigiso}) that the two-dimensional persistence module given by applying (\ref{hkdef}) to these chain-level structures in the case of the Novikov complex of an $\R$- or $S^1$-valued function is in fact isomorphic to the usual interlevel homology persistence module. Apart from its intrinsic interest, this justifies thinking of our Floer-theoretic barcodes as the correct analogues of interlevel barcodes, since Floer complexes and their associated Poincar\'e--Novikov structures are formally very similar to the corresponding structures in Novikov theory.  The paper concludes with an appendix concerning orientations in Morse and Floer theory, with a focus on explicitly describing the conventions used and working out their implications for the signs that arise in the respective Morse- and Floer-theoretic versions of Poincar\'e duality.

 \subsection*{Acknowledgements}
The seeds for this paper were planted at the 2018 Tel Aviv workshop ``Topological data analysis meets symplectic topology,'' in particular through discussions there with Dan Burghelea which left me curious as to whether there might be a more abstract version of his constructions.  Early stages of the work were supported by NSF grant DMS-1509213.  I am also grateful to the Geometry and Topology group at the University of Iowa and to IBS Center for Geometry and Physics for the opportunity to present this work at seminars in 2021 when it was still at a relatively early stage of development.

\section{The general setup} \label{algintro}
We now begin to introduce our algebraic framework, which is meant to include relevant features of the filtered Novikov complexes associated to functions $\tilde{f}\co\tilde{X}\to\R$ as discussed earlier, as well as their Floer-theoretic analogues.  In the background throughout the paper we fix a coefficient field $\kappa$, and a finitely-generated additive subgroup $\Gamma$ of $\R$ which in the case of the Novikov complex corresponds to the period group of a one-form on $X$ which pulls back to $d\tilde{f}$ under the covering map $\tilde{X}\to X$, and which is isomorphic to the deck transformation group of the cover.  $\Gamma$ might well be trivial, corresponding to the case of classical Morse theory.  Associated to $\kappa$ and $\Gamma$, as discussed in more detail in Section \ref{basic}, are the \textbf{group algebra} $\Lambda:=\kappa[\Gamma]$ as well as two versions of the \textbf{Novikov field} which we denote by $\Lambda_{\uparrow}$ and $\Lambda^{\downarrow}$, each of which contains $\Lambda$ as a subring.  In the case that $\Gamma$ is trivial, each of $\Lambda,\Lambda_{\uparrow},\Lambda^{\downarrow}$ degenerates to the original coefficient field $\kappa$.  If $\Gamma=\Z$, then $\Lambda$ is the Laurent polynomial ring $\kappa[T^{-1},T]$, while $\Lambda_{\uparrow}$ and $\Lambda^{\downarrow}$ are the two Laurent series fields $\kappa[T^{-1},T]]$ and $\kappa[[T^{-1},T]$.  The latter are of course isomorphic as fields, but we distinguish them because they are not isomorphic as $\Lambda$-algebras.

\subsection{Filtered matched pairs}
In the case, alluded to above, of the Novikov complex of a function $\tilde{f}\co\tilde{X}\to \R$ on a covering space $\tilde{X}\to X$ with $\Gamma$ playing the role of the period group, the deck transformation action of $\Gamma$ on $\tilde{X}$ makes the homology $H_{*}(\tilde{X};\kappa)$ into a graded $\Lambda$-module.  The usual Novikov chain complex of $\tilde{f}$ (reviewed in Section \ref{novsect} and written there as $\mathbf{CN}(f;\xi)_{\uparrow}$, with $\xi$ denoting the cohomology class of the form on $X$ of which $d\tilde{f}$ is the pullback) is a  chain complex of finite-dimensional vector-spaces over $\Lambda_{\uparrow}$, equipped with an ascending real-valued filtration.  The complex  $\mathbf{CN}(f;\xi)_{\uparrow}$ is constructed by considering negative gradient trajectories for $\tilde{f}$; one can equally well use positive gradient trajectories of $\tilde{f}$ to obtain a complex $\mathbf{CN}(\tilde{f};\xi)^{\downarrow}$ of vector spaces over $\Lambda^{\downarrow}$, with a descending real-valued filtration.\footnote{Of course, this other complex can also be understood as the usual Novikov complex of $-\tilde{f}$, after adjusting the coefficients and filtration.  Assuming orientability, it can also be interpreted as the dual of the chain complex $\mathbf{CN}(\tilde{f};\xi)_{\uparrow}$ (\emph{cf}. Proposition \ref{negdualnov}), but we grade both complexes $\mathbf{CN}(\tilde{f};\xi)_{\uparrow}$ and $\mathbf{CN}(\tilde{f};\xi)^{\downarrow}$ homologically.}

Our notion of a \textbf{filtered matched pair}, defined precisely in Definition \ref{fmpdfn}, models the relation between $H_{*}(\tilde{X};\kappa)$ and the homologies of the two versions of the filtered Novikov complex mentioned above.  A filtered matched pair $\mathcal{P}$ amounts to a diagram \begin{equation}\label{mpgen} \xymatrix{ & (V^{\downarrow},\rho^{\downarrow}) \\ H\ar[ru]^{\phi^{\downarrow}}\ar[rd]_{\phi_{\uparrow}} & \\ & (V_{\uparrow},\rho_{\uparrow}) } \end{equation} where $H$ is a finitely-generated $\Lambda$-module, $V_{\uparrow}$ (resp. $V^{\downarrow}$) is a $\Lambda_{\uparrow}$-vector space (resp. a $\Lambda^{\downarrow}$-vector space), the two maps are $\Lambda$-module homomorphisms that become isomorphisms after coefficient extension, and $\rho_{\uparrow},\rho^{\downarrow}$ are ``filtration functions'' on $V_{\uparrow},V^{\downarrow}$, interpreted as assigning to an element the first filtration level at which it appears.\footnote{The filtration on $V_{\uparrow}$ is ascending and that on $V^{\downarrow}$ is descending, so ``first'' means ``minimal'' in the case of $\rho_{\uparrow}$ and ``maximal'' in the case of $\rho^{\downarrow}$.  The precise conditions required of $\rho_{\uparrow}$ and $\rho^{\downarrow}$ are specified in Definition \ref{normdef}.}    

If $d$ is the maximal cardinality of a $\Lambda$-independent set in $H$, in Definition \ref{gapdfn} we associate to the filtered matched pair $\mathcal{P}$ as in (\ref{mpgen}) a sequence of real numbers $G_1(\mathcal{P})\geq \cdots\geq G_d(\mathcal{P})$.  In favorable situations---including, as we show in Theorem \ref{basistheorem}, all cases in which the period group $\Gamma$ is discrete---$\mathcal{P}$ will admit what we call a \textbf{doubly-orthogonal basis} $\{e_1,\ldots,e_d\}\subset H$.  In this case one has (modulo reordering) $G_i(\mathcal{P})
=\rho^{\downarrow}(\phi^{\downarrow}e_i)-\rho_{\uparrow}(\phi_{\uparrow}e_i)$, and the \textbf{basis spectrum} of $\mathcal{P}$ is defined as the multiset $\Sigma(\mathcal{P})$ of elements \[ \left(\rho_{\uparrow}(\phi_{\uparrow}e_i)\,\mathrm{mod}\Gamma,\,\,\rho^{\downarrow}(\phi^{\downarrow}e_i)-\rho_{\uparrow}(\phi_{\uparrow}e_i)\right) \] of $(\R/\Gamma)\times \R$.  By Proposition \ref{barinvt} this multiset is independent of the choice of doubly-orthogonal basis.  

In the examples that we have in mind, the modules $H,V_{\uparrow},V^{\downarrow}$ appearing in a filtered matched pair arise as, for some $k\in\Z$, the degree-$k$ homologies of chain complexes $\mathcal{C},\mathcal{C}_{\uparrow},\mathcal{C}^{\downarrow}$, with $\mathcal{C}_{\uparrow}$ and $\mathcal{C}^{\downarrow}$ admitting filtrations that make them into Floer-type complexes as in \cite{UZ}.  (See the notion of a ``chain-level filtered matched pair'' in Definition \ref{cfmpdfn}.) Elements of the basis spectrum are intended to correspond to intervals (canonical up to $\Gamma$-translation) in interlevel barcodes connecting endpoints $\rho^{\downarrow}(\phi^{\downarrow}e_i)$ and $\rho_{\uparrow}(\phi_{\uparrow}e_i)$; the relevant interval is closed and in homological degree $k$ when $\rho^{\downarrow}(\phi^{\downarrow}e_i)\geq \rho_{\uparrow}(\phi_{\uparrow}e_i)$ and open and in homological degree $k-1$ when $\rho^{\downarrow}(\phi^{\downarrow}e_i)< \rho_{\uparrow}(\phi_{\uparrow}e_i)$.    The absolute values of the ``gaps'' $G_i(\mathcal{P})$ then correspond to the lengths of these intervals (which do not need to be reduced modulo $\Gamma$).  By design, a filtered matched pair only carries information that survives to the homologies $H,V_{\uparrow},V^{\downarrow}$ of the full chain complexes $\mathcal{C},\mathcal{C}_{\uparrow},\mathcal{C}^{\downarrow}$; the remaining (half-open) bars in interlevel barcodes can be found by applying the methods of \cite{UZ} to the Floer-type complexes $\mathcal{C}_{\uparrow},\mathcal{C}^{\downarrow}$.

Note that the $G_i(\mathcal{P})$, unlike the basis spectrum, are defined regardless of whether $\mathcal{P}$ admits a doubly-orthogonal basis.  By Theorem \ref{basistheorem}, the only cases in which $\mathcal{P}$ might not admit a doubly-orthogonal basis occur when the subgroup $\Gamma$ of $\R$ is indiscrete and hence dense. In this scenario, the only quantity associated to an interval modulo $\Gamma$-translation that is robust to perturbation is its length, as represented by $G_i(\mathcal{P})$; thus in the cases where we do not have a well-defined basis spectrum the $G_i(\mathcal{P})$ could be said to carry all the geometrically meaningful information that such a spectrum might have conveyed.  Nonetheless, it would be interesting to know whether Theorem \ref{basistheorem} can be generalized to some cases\footnote{For such a generalization one would minimally need to assume that the $\Lambda$-module $H$ splits as a direct sum of a free module and a torsion module, which is not always true when $\Gamma$ is dense.  This hypothesis does hold in the motivating case of Hamiltonian Floer theory (in fact in that case $H$ is free).}  in which $\Gamma$ is dense, especially if this could be done in a way which (like our proof of Theorem \ref{basistheorem} in the discrete case) is constructive and hence would give a way of computing the $G_i(\mathcal{P})$, which seems difficult to do directly from the definition when $\Gamma$ is dense.

A crucial general property of persistence barcodes is \textbf{stability}: for instance, for interlevel persistence a small change in the function whose interlevel sets are being considered leads to a correspondingly small change in the barcode.  We now describe such a result in our algebraic framework of filtered matched pairs.  Since filtered matched pairs are only intended to model the contributions to interlevel persistence that correspond to global homology classes, the appropriate stability result should compare filtered matched pairs whose versions of the respective modules $H,V_{\uparrow},V^{\downarrow}$ are isomorphic, but with different versions of the filtration functions $\rho_{\uparrow},\rho^{\downarrow}$.  For ease of exposition in this introduction we assume the versions of the $H,V_{\uparrow},V^{\downarrow}$ are equal, not just isomorphic; we give a slightly more general formulation in Section \ref{fmpsec}.  Here is the fundamental stability result:

\begin{theorem}\label{introstab} Consider two filtered matched pairs
\[
\xymatrix{ & & (V^{\downarrow},\rho^{\downarrow}) & & & &  (V^{\downarrow},\hat{\rho}^{\downarrow})  \\ \mathcal{P}  \ar@{}[r]^(.35){}="a"^(.65){}="b" \ar@{=} "a";"b"    & H\ar[ru]^{\phi^{\downarrow}}\ar[rd]_{\phi_{\uparrow}} & & \mbox{and} & \hat{\mathcal{P}}  \ar@{}[r]^(.35){}="a"^(.65){}="b" \ar@{=} "a";"b"    & H\ar[ru]^{\phi^{\downarrow}}\ar[rd]_{\phi_{\uparrow}}   \\ & & (V_{\uparrow},\rho_{\uparrow}) & & & & (V_{\uparrow},\hat{\rho}_{\uparrow}) }
\] with the same data $H,V^{\downarrow},V_{\uparrow},\phi^{\downarrow},\phi_{\uparrow}$ but (possibly) different filtration functions $\rho^{\downarrow},\rho_{\uparrow},\hat{\rho}^{\downarrow},\hat{\rho}_{\uparrow}$, and let \[ t=\max\left\{\max_{V^{\downarrow}\setminus\{0\}}|\rho^{\downarrow}-\hat{\rho}^{\downarrow}|,\max_{V_{\uparrow}\setminus\{0\}}|\rho_{\uparrow}-\hat{\rho}_{\uparrow}|\right\}.\]  Then the gaps $G_i$ satisfy \[ |G_i(\mathcal{P})-G_i(\hat{\mathcal{P}})|\leq 2t.\]  Moreover, if $\Gamma$ is discrete, so that the basis spectra $\Sigma(\mathcal{P})$ and $\Sigma(\hat{\mathcal{P}})$ are well-defined, then there is a bijection $\sigma\co \Sigma(\mathcal{P})\to\Sigma(\hat{\mathcal{P}})$ such that for each $([a],\ell)\in\Sigma(\mathcal{P})$, the image $([\hat{a}],\hat{\ell})=\sigma([a],\ell)$ satisfies (for some choice of $\hat{a}$ within its $\Gamma$-coset) \[ |\hat{a}-a|\leq t\quad \mbox{and}\quad |(\hat{a}+\hat{\ell})-(a+\ell)|\leq t.\]
\end{theorem}

\begin{proof}
The statement about the $G_i$ follows from Proposition \ref{gapcts} (using the identity map in the role of the ``$t$-morphism'' of that proposition), and the statement about the basis spectra follows from Theorem \ref{stab}.
\end{proof}

The combination of Theorem \ref{introstab} (for the closed and open bars) with \cite[Theorem 8.17]{UZ} (for the half-open bars) can be regarded as an algebraic version of stability theorems for interlevel persistence such as the LZZ stability theorem of \cite{CDM} or \cite[Theorem 1.2]{BH17}.  As noted in \cite{BH17}, continuous deformations of a function can result in a closed interlevel barcode interval in degree $k$ shrinking and then transforming into an open interlevel barcode interval in degree $k-1$; in our context this corresponds to the parameter $\ell$ in an element $([a],\ell)$ of the basis spectrum (corresponding to the difference $\rho^{\downarrow}(\phi^{\downarrow}e_i)-\rho_{\uparrow}(\phi_{\uparrow}e_i)$) passing from positive to negative.

\subsection{Poincar\'e-Novikov structures} \label{intropn} A filtered matched pair involves two vector spaces $V_{\uparrow}$ and $V^{\downarrow}$ over (different) Novikov fields, one with an ascending filtration and the other with a descending filtration. In motivating examples, there is a filtered matched pair in every grading $k$, and (for some fixed integer $n$) a version of Poincar\'e duality relates the degree $k$ version of $V^{\downarrow}$ to the dual of the degree $n-k$ version of $V_{\uparrow}$.  This is abstracted in our definition of a $n$-\textbf{Poincar\'e-Novikov structure}, see Definition \ref{pnstr}.  Here one has, for each $k\in \Z$, a $\Lambda$-module $H_k$, a $\Lambda_{\uparrow}$-vector space $V_k$ equipped with an ascending filtration function $\rho_k$, and (as in the bottom half of the diagram (\ref{mpgen})) a module homomorphism $S_k\co H_k\to V_k$ that becomes an isomorphism after coefficient extension.  Moreover, we require an abstraction of the Poincar\'e intersection pairing between $H_k$ and $H_{n-k}$, taking the form of a suitable map $\mathcal{D}_k\co H_k\to {}^{\vee}\!H_{n-k}$ where ${}^{\vee}\!H_{n-k}$ is the ``conjugated dual'' (see Section \ref{dualsec}) of $H_{n-k}$.  (One obtains a ``strong'' or a ``weak'' $n$-Poincar\'e--Novikov structure depending on precisely what requirements are put on the $\mathcal{D}_k$.)  This information is sufficient to yield a filtered matched pair of the form  $\mathcal{P}(\mathcal{N})_k$ of the form \[\xymatrix{ & ({}^{\vee}(V_{n-k}),{}^{\vee}\rho_{n-k}) \\ H_k\ar[ru]^{\widetilde{S}_k} \ar[rd]_{S_k} & \\ & (V_k,\rho_k)} \] constructed in detail in Section \ref{pnsect}.  From the $\mathcal{P}(\mathcal{N})_{k}$ we can define the gaps $\mathcal{G}_{k,i}(\mathcal{N}):=G_i(\mathcal{P}(\mathcal{N})_k)$ and, when the $\mathcal{P}(\mathcal{N})_k$ admit doubly-orthogonal bases, the ``essential barcode'' of $\mathcal{N}$; the latter is defined in Definition \ref{pnbar} in terms of the basis spectra of the $\mathcal{P}(\mathcal{N})_k$ and consists, in each grading, of equivalence classes (written as $[a,b]^{\Gamma}$ or $(a,b)^{\Gamma}$) of closed or open intervals modulo the translation action of $\Gamma$.  The word ``essential'' alludes to the fact that such intervals correspond to nontrivial global homology classes in the case of interlevel persistence.   Theorem \ref{introstab} readily implies stability results for these invariants of Poincar\'e-Novikov structures, see Corollaries \ref{pngapstab} and \ref{pnstab}.

There is a natural notion of the dual of a filtered matched pair, developed in Section \ref{fmpdual}.  In the case of the filtered matched pairs $\mathcal{P}(\mathcal{N})_k$ associated to a $n$-Poincar\'e--Novikov structure, the dual of $\mathcal{P}(\mathcal{N})_k$ is closely related to $\mathcal{P}(\mathcal{N})_{n-k}$ (see Proposition \ref{dualpnchar} for a precise statement).  This leads to Corollary \ref{pndual}, asserting that if $\mathcal{N}$ is a strong $n$-Poincar\'e--Novikov structure such that the $\mathcal{P}(\mathcal{N})_k$ all admit doubly-orthogonal bases\footnote{For instance, if $\Gamma$ is discrete then the Poincar\'e--Novikov structures associated to filtered Novikov complexes and their Hamiltonian-Floer-theoretic analogues satisfy these properties, using Proposition \ref{cohtopdstr} and Theorem \ref{basistheorem}.} then for each $k$ there is a bijection between elements $[a,b]^{\Gamma}$ of the degree-$k$ essential barcode of $\mathcal{N}$ and elements $(a,b)^{\Gamma}$ of the degree-$(n-k-1)$ essential barcode of $\mathcal{N}$.  In the context of interlevel persistence for an $\mathbb{R}$- or $S^1$-valued Morse function $f$ on a smooth manifold, this, together with a similar statement for the half-open bars in the full barcode that can be inferred from \cite[Proposition 6.7]{UZ}\footnote{see the discussion in Section \ref{cpnstr}}, can be regarded as demonstrating Poincar\'e duality (at least in the classical sense of a symmetry of Betti numbers) for the regular fibers of $f$.  Thus our general formalism of Poincar\'e--Novikov structures is used to infer, entirely algebraically, a statement corresponding to Poincar\'e duality for the regular fibers of a map from hypotheses including Poincar\'e duality (in the form of the maps $\mathcal{D}_k\co H_k\to {}^{\vee}\!H_{n-k}$) for the full domain of the map.

\subsection{Chain-level notions}

While the modules $H,V_{\uparrow},V^{\downarrow}$ in a filtered matched pair as in (\ref{mpgen}) are in practice usually obtained as the homologies (in some degree) of chain complexes $\mathcal{C},\mathcal{C}_{\uparrow},\mathcal{C}^{\downarrow}$, the formalism in Section \ref{pnsect} does not require this, and the gaps and basis spectra of a filtered matched pair are defined only in terms of $H,V_{\uparrow},V^{\downarrow}$ and maps between them, without reference to any chain complex.  Likewise, the definition and properties of a $n$-Poincar\'e--Novikov structure do not appeal to any chain complexes having $H_k,V_k$ as their homologies.  However, in order to justify our interpretation of these as modeling  aspects of interlevel persistence, and also to capture the parts of interlevel persistence arising from homology classes of interlevel sets that vanish upon inclusion into the whole space, we should incorporate data from such chain complexes and not just their homologies.  This motivates the definitions of a ``chain-level filtered matched pair'' (Definition \ref{cfmpdfn}) and a ``chain-level $n$-Poincar\'e--Novikov structure'' (second paragraph of Section \ref{cpnstr}).  Roughly speaking, these chain-level definitions replace the relevant modules over $\Lambda,\Lambda_{\uparrow},\Lambda^{\downarrow}$ with chain complexes, and replace the conditions that various maps be isomorphisms after coefficient extension with the condition that they be chain homotopy equivalences after coefficient extension.

The data of a chain-level filtered matched pair include chain complexes $\mathcal{C}_{\uparrow},\mathcal{C}^{\downarrow}$ over $\Lambda_{\uparrow}$ and $\Lambda^{\downarrow}$, respectively, equipped with respective filtration functions $\ell_{\uparrow}$ and $\ell^{\downarrow}$ that make them Floer-type complexes in the the sense of \cite{UZ} (with a slight modification of the definition in the case of $\mathcal{C}^{\downarrow}$ so that the filtration will be descending rather than ascending).  Applying the homology functor to a chain-level filtered matched pair yields a filtered matched pair (in the original sense) in each degree $k$; the relevant (ascending) filtration function $\rho_{\uparrow}$ on $H_k(\mathcal{C}_{\uparrow})$ is given by \[ \rho_{\uparrow}(h)=\inf\{c\in \mathcal{C}_{\uparrow}|\partial_{\mathcal{C}_{\uparrow}}c=0,\,[c]=h\} \] and similarly the descending filtration function on $H_k(\mathcal{C}^{\downarrow})$ is given by \[ \rho^{\downarrow}(h)=\sup\{c\in \mathcal{C}^{\downarrow}|\partial_{\mathcal{C}^{\downarrow}}c=0,\,[c]=h\}.\]  (In the language of chain-level symplectic Floer theory, $\rho_{\uparrow}(h)$ or $\rho^{\downarrow}(h)$ would be called the \emph{spectral invariant} of the homology class $h$.)

To a chain-level filtered matched pair $\mathcal{CP}$ and a grading $k\in \Z$ we associate in (\ref{hkdef}) a two-parameter persistence module $\mathbb{H}_k(\mathcal{CP})$ in the category of vector spaces over $\kappa$; thus for all $(s,t)\in \R^2$, we have a $\kappa$-vector space $\mathbb{H}_k(\mathcal{CP})_{s,t}$, with suitably compatible maps $\varepsilon_{(s,t),(s',t')}\co \mathbb{H}_k(\mathcal{CP})_{s,t}\to \mathbb{H}_k(\mathcal{CP})_{s',t'}$ whenever $s\leq s'$ and $t\leq t'$.  It is this persistence module that most directly generalizes persistence modules built from interlevel sets.  To wit, suppose that $X$ is a smooth oriented $n$-dimensional manifold and $p\co \tilde{X}\to X$ is a regular covering space with deck transformation group $\tilde{\Gamma}$.  Suppose also that $\tilde{f}\co \tilde{X}\to \R$ is a Morse function such that for some isomorphism $\omega\co \tilde{\Gamma}\to\Gamma$ (where $\Gamma$ is the same subgroup of $\R$ as before), we have $\tilde{f}(\gamma\tilde{x})=\tilde{f}(x)-\omega(\gamma)$ for all $\tilde{x}\in\tilde{X}$.  One can then use the filtered Novikov complex of $\tilde{f}$, together with Poincar\'e duality for $\tilde{X}$ (as adapted to regular covering spaces  in, \emph{e.g.}, \cite[Section 4.5]{Ran}) to construct a chain-level $n$-Poincar\'e--Novikov structure $\mathcal{CN}(f;\xi)$; see Definition \ref{cndef}.  In this context we prove the following, which follows from Theorem \ref{bigiso}:
\begin{theorem}\label{introiso}  Assume that $\Gamma$ is discrete.  Then, for $s+t\geq 0$, there are isomorphisms \[ \sigma_{s,t}\co \mathbb{H}_k(\mathcal{CP}(\mathcal{CN}(\tilde{f};\xi)))_{s,t}\to H_k(\tilde{f}^{-1}([-s,t]);\kappa) \] such that, when $s\leq s'$ and $t\leq t'$, we have a commutative diagram \[ \xymatrix{ \mathbb{H}_k(\mathcal{CP}(\mathcal{CN}(\tilde{f};\xi)))_{s,t} \ar[r]^{\varepsilon_{(s,t),(s',t')}} \ar[d]_{\sigma_{s,t}} & \mathbb{H}_k(\mathcal{CP}(\mathcal{CN}(\tilde{f};\xi)))_{s',t'}  \ar[d]^{\sigma_{s',t'}} \\  H_k(\tilde{f}^{-1}([-s,t]);\kappa) \ar[r] & H_k(\tilde{f}^{-1}([-s',t']);\kappa) }  \]  where the bottom horizontal map is the map on singular homology induced by inclusion. 
\end{theorem} 

When $\Gamma$ is discrete, we show in Section \ref{decompsect} that, up to a notion of ``filtered matched homotopy equivalence'' (Definition \ref{fmhedef}) that preserves the isomorphism types of the associated persistence modules $\mathbb{H}_k(\mathcal{CP})$, any chain-level filtered matched pair  $\mathcal{CP}=(\mathcal{C},\mathcal{C}^{\downarrow},\mathcal{C}_{\uparrow},\phi^{\downarrow},\phi_{\uparrow})$ splits as a direct sum of five types of simple building blocks, leading to a corresponding decomposition of each $\mathbb{H}_k(\mathcal{CP})$.   Two of the five types of building block  correspond to homologically trivial summands in decompositions from \cite{UZ} of the Floer-type complexes $\mathcal{C}_{\uparrow},\mathcal{C}^{\downarrow}$.  Another two correspond to elements $([a],\ell)$ of the basis spectrum of the filtered matched pair $\mathcal{P}$ obtained by applying the homology functor to $\mathcal{CP}$ (with one variant corresponding to elements with $\ell\geq 0$ and the other to elements with $\ell<0$).  The remaining building block corresponds to the $\Lambda$-torsion part of the homology  of the chain complex of $\Lambda$-modules $\mathcal{C}$, and yields a contribution to each $\mathbb{H}_k(\mathcal{CP})_{s,t}$ that is independent of the  parameters $s$ and $t$ and of the filtrations on the Floer-type complexes $\mathcal{C}_{\uparrow},\mathcal{C}^{\downarrow}$.  (If $\Gamma=\{0\}$ then $\Lambda$ is the field $\kappa$ so the $\Lambda$-torsion is trivial; for nontrivial $\Gamma$ the $\Lambda$-torsion part of $H_k(\mathcal{CP})$ corresponds to what is denoted $V_{r}(\xi)$ at the start of \cite[Section 4]{BH17}, and thus to the ``Jordan cells'' of \cite{BD},\cite{BH17}.)  

The ``full barcode'' of $\mathcal{CP}$ (Definition \ref{fullbar}) comprises intervals-modulo-$\Gamma$-translation $[a,b]^{\Gamma}$ and $(a,b)^{\Gamma}$ obtained from the basis spectrum of $\mathcal{P}$, and $[a,b)^{\Gamma}$ and $(a,b]^{\Gamma}$ obtained from the concise barcodes (as in \cite{UZ}) of $\mathcal{C}_{\uparrow},\mathcal{C}^{\downarrow}$.  Then, continuing to assume that $\Gamma$ is discrete, $\mathbb{H}_k(\mathcal{CP})$ splits as a direct sum of $\Gamma$-periodic versions of the block modules of \cite{CO} associated to the intervals in the full barcode, together with a contribution from the torsion part of $H_k(\mathcal{C})$; see Theorem \ref{bigdecomp} for a precise statement.   
As a practical matter, this implies that the constructions of Section \ref{fmpsec} at the homological level, together with those of \cite{UZ} associated with the Floer-type complexes $\mathcal{C}_{\uparrow},\mathcal{C}^{\downarrow}$, are sufficient for the computation of all invariants associated to our algebraic version of interlevel persistence when $\Gamma$ is discrete; one does not need to engage directly with the somewhat more complicated constructions in Section \ref{clsec}, though such constructions have conceptual importance in that they justify our interpretation of the invariants via Theorem \ref{introiso}.  Note that both Theorem \ref{basistheorem} which allows one to find the intervals $[a,b]^{\Gamma}$ and $(a,b)^{\Gamma}$ in the full barcode, and \cite[Theorem 3.5]{UZ} which allows one to find the intervals $[a,b)^{\Gamma}$ and $(a,b]^{\Gamma}$, have proofs that are at least implicitly algorithmic.

If $\Gamma$ is not discrete, one cannot expect chain-level Poincar\'e--Novikov structures $\mathcal{CP}$ to be classified up to filtered matched homotopy equivalence in as straightforward a way as in the discrete case;  after all, any such classification would need to be at least as complicated as the classification of the possible $H_k(\mathcal{C})$, which can be arbitrary finitely-generated $\Lambda$-modules, and the indiscreteness of $\Gamma$ implies that $\Lambda$ is not a PID, so that its finitely-generated modules do not admit a simple classification.  From the persistent homology viewpoint, though, one would be interested mainly in the invariants of $\mathcal{CP}$ that depend nontrivially on the filtrations on $\mathcal{C}_{\uparrow}$ and $\mathcal{C}^{\downarrow}$ and are suitably stable under perturbation.  Examples of such invariants are the gaps $G_i$ of the filtered matched pair $\mathcal{P}$ obtained by applying the homology functor to $\mathcal{CP}$, and the lengths of the finite-length bars in the concise barcodes of $\mathcal{C}_{\uparrow}$ and $\mathcal{C}^{\downarrow}$.  It is not clear to me whether there are other such invariants going beyond these.

\section{Examples, connections, and interpretations}\label{geomintro}

\subsection{Morse functions and extended persistence} \label{connext}
We now discuss in a bit more detail how our algebraic setup arises in practice.  Let us begin with the special case of a Morse function $f\co X\to \R$ on a compact $n$-dimensional $\kappa$-oriented smooth manifold $X$.  (So in this case $\Gamma=\{0\}$.)  As we review in Section \ref{morseintro}, associated to $f$ is a $\Lambda_{\uparrow}$-Floer type complex $\mathbf{CM}(f)_{\uparrow}$; this notation encompasses  both the usual Morse chain complex which we denote on its own by $\mathbf{CM}_{*}(f)$ and the filtration function $\ell_{\uparrow}^{f}$ which assigns to a formal sum of critical points the maximal critical value of a critical point appearing in the sum.  There is a standard isomorphism $\phi_f\co H_*(X;\kappa)\to \mathbf{HM}_*(f)$ between the $\kappa$-coefficient singular homology of $X$ and the homology of $\mathbf{CM}_{*}(f)$ (see, \emph{e.g.}, \cite[Chapter 6]{Pa}).  Definition \ref{cmdef} associates to $f$ a chain-level $n$-Poincar\'e--Novikov structure; upon passing to homology (and using the isomorphism $\phi_{h_0}$ to identify the Morse homology of the auxiliary Morse function $h_0$ in Definition \ref{cmdef} with $H_*(X;\kappa)$) one obtains a $n$-Poincar\'e--Novikov structure consisting of the following data:
\begin{itemize} \item the isomorphism $\phi_f\co H_*(X;\kappa)\to \mathbf{HM}_*(f)$;
\item the filtration function $\rho_{\uparrow}^{f}\co \mathbf{HM}_{*}(f)\to\R\cup\{-\infty\}$ defined by $\rho_{\uparrow}^{f}(h)=\inf\{\ell_{\uparrow}^{f}(c)|c\in \mathbf{CM}_{*}(f),\,[c]=h\}$; and \item for each $k$, the map $\mathcal{D}_k\co H_k(X;\kappa)\to \mathrm{Hom}_{\kappa}(H_{n-k}(X;\kappa),\kappa)$ corresponding to the Poincar\'e intersection pairing: for $a\in H_k(X;\kappa)$ and $b\in H_{n-k}(X;\kappa)$, $(\mathcal{D}_ka)(b)=a\cap b$ is the signed count of intersections between suitably generic representatives of $a$ and $b$.
\end{itemize}

Denote this Poincar\'e--Novikov structure by $\mathcal{N}(f)$.  The general procedure of Section \ref{pnsect} then yields, for each $k\in \Z$, a filtered matched pair $\mathcal{P}(\mathcal{N}(f))_k$.  By using Proposition \ref{downup} (applied to the case that $\Gamma=\{0\}$) and then passing to homology,  one sees that $\mathcal{P}(\mathcal{N}(f))_k$ is isomorphic in the natural sense to the filtered matched pair \begin{equation}\label{introupdown} \xymatrix{ & (\mathbf{HM}_{k}(-f),-\rho_{\uparrow}^{-f}) \\ H_k(X;\kappa) \ar[ru]^{\phi_{-f}} \ar[rd]_{\phi_f} & \\ & (\mathbf{HM}_k(f),\rho_{\uparrow}^{f})  } \end{equation}  Let us write $\mu_{\uparrow}^{f}=\rho_{\uparrow}^{f}\circ \phi_f$ and $\mu^{\downarrow}_{f}=(-\rho^{\downarrow}_{-f})\circ\phi_{-f}$ for the 
pullbacks of the filtration functions appearing in the above diagram to functions on $H_k(X;\kappa)$.  These define, respectively, an ascending and a descending filtration on $H_k(X;\kappa)$; using, \emph{e.g.}, \cite[Theorem 3.8]{Qin} one can see that, for any $t\in \R$, \begin{align}\label{qiniso} \{h\in H_k(X;\kappa)|\mu_{\uparrow}^{f}(h)\leq t\}&=\mathrm{Im}\left(H_k(X^{\leq t};\kappa)\to H_k(X;\kappa)\right)  \\ 
\{h\in H_k(X;\kappa)|\mu^{\downarrow}_{f}(h)\geq t\} &=  \mathrm{Im}\left(H_k(X_{\geq t};\kappa)\to H_k(X;\kappa)\right) \nonumber\end{align} where the maps on homology are induced by inclusion and, as before, $X^{\leq t}=f^{-1}((-\infty,t]))$ and $X_{\geq t}=f^{-1}([t,\infty))$.  Thus $\mu_{\uparrow}^{f}$ associates to a homology class $h$ its ``minimax value''--the infimum, among all cycles $c$ representing $h$, of the maximal value of $f$ on $c$---while $\mu^{\downarrow}_{f}(h)$ is the similarly-defined maximin value.

A doubly-orthogonal basis for the filtered matched pair (\ref{introupdown})---as defined in general in Definition \ref{dodfn}---then amounts to a basis $\{e_1,\ldots,e_d\}$ for $H_k(X;\kappa)$ such that, for $c_1,\ldots,c_d\in \kappa$, one has both \[ \mu_{\uparrow}^{f}\left(\sum_{i=1}^{d}c_ie_i\right)=\max\left\{\left.\mu_{\uparrow}^{f}(e_i)\right| c_i\neq 0\right\}\quad\mbox{and}\quad \mu^{\downarrow}_{f}\left(\sum_{i=1}^{d}c_ie_i\right)=\min\left\{\left.\mu_{f}^{\downarrow}(e_i)\right| c_i\neq 0\right\} .\]  Our general prescription then associates a barcode interval $[\mu_{\uparrow}^{f}(e_i),\mu^{\downarrow}_{f}(e_i)]$ in degree $k$ to each $e_i$ with  $\mu_{\uparrow}^{f}(e_i)\leq \mu^{\downarrow}_{f}(e_i)$, and a barcode interval $\left(\mu^{\downarrow}_{f}(e_i),\mu_{\uparrow}^{f}(e_i)\right)$ in degree $k-1$ to each $e_i$ with $\mu^{\downarrow}_{f}(e_i)<\mu_{\uparrow}^{f}(e_i)$.

Recall from \cite{ext}, \cite[Section 3]{CDM} that the extended persistence of $f$ is constructed from the persistence module given by choosing
an increasing sequence $(s_0,\ldots,s_m)$ with one element in each connected component of the set of regular values of $f$ (so $s_0<\min f$ and $s_m>\max f$) and considering the diagram \begin{align} \label{essdiag} 0=& H_k(X^{\leq s_0};\kappa)\to H_k(X^{\leq s_1};\kappa)\to\cdots\to H_k(X^{\leq s_m};\kappa)=H_k(X;\kappa)=H_k(X,X_{\geq s_m};\kappa)\\ &\to H_k(X,X_{\geq s_{m-1}};\kappa)\to\cdots\to H_k(X,X_{\geq s_0};\kappa)=0\nonumber \end{align} where all maps are induced by inclusion. 
The standard persistence module decomposition theorem \cite{ZC} splits this persistence module into a direct sum of interval modules, \emph{i.e.} diagrams of form \begin{equation}\label{interval} 0\to\cdots 0\to \kappa\to\kappa\to\cdots\kappa\to 0\to\cdots\to 0\end{equation} where each map $\kappa\to\kappa$ is the identity. This yields a persistence diagram which is subdivided into ``ordinary,'' ``relative,'' and ``extended'' subdiagrams corresponding respectively to summands (\ref{interval}) in which the nonzero terms all appear in the first half of the diagram, all appear in the second half of the diagram, or bridge the two halves (and thus include a one-dimensional summand of $H_k(X;\kappa)$).  The ordinary subdiagram evidently corresponds to the finite-length part of the sublevel barcode of $f$, and by the EP Symmetry Corollary in \cite{CDM} the relative subdiagram corresponds to the finite-length part of the sublevel barcode of $-f$.  

We claim that there is a straightforward correspondence between the extended subdiagram and doubly-orthogonal bases for the filtered matched pair (\ref{introupdown}).  Indeed, for $i=1,\ldots,m$ let $c_i$ be the unique critical value for $f$ lying in the interval $(s_{i-1},s_i)$, and consider a summand (\ref{interval}) of (\ref{essdiag}) whose first nonzero term lies in $H_k(X^{\leq s_i};\kappa)$ and whose last nonzero term lies in $H_k(X,X_{\geq s_j};\kappa)$.  Let $h$ generate the $H_k(X;\kappa)$-term of this summand.  Then (since $s_i\in (a_i,a_{i+1})$ and the image of $H_k(X^{\leq s};\kappa)$ in $H_k(X;\kappa)$ is the same for all $s\in (a_i,a_{i+1})$) we see that $a_i$ is the infimal value of $s$ with the property that $h\in \mathrm{Im}(H_k(X^{\leq s};\kappa)\to H_k(X;\kappa))$, \emph{i.e.} that $\mu_{\uparrow}^{f}(h)=a_i$.  Similarly $a_{j}$ is the supremum of all values of $s$ such that $h\in \ker(H_k(X;\kappa)\to H_k(X,X_{\geq s};\kappa))$; by the exactness of the homology exact sequences of the pairs $(X,X_{\geq s})$ this is equivalent to the statement that $\mu^{\downarrow}_{f}(h)=a_j$.  

Now given a decomposition of (\ref{essdiag}) into interval modules, let $e_1,\ldots,e_d$ denote generators for the $H_k(X;\kappa)$-terms of those summands in the decomposition for which the $H_k(X;\kappa)$ term is nontrivial (\emph{i.e.}, of those summands that contribute to the extended subdiagram). Thus, for $\ell=1,\ldots,d$, the values $\mu_{\uparrow}^{f}(e_{\ell})$ and $\mu^{\downarrow}_{f}(e_{\ell})$ can be read off in the manner described above from the first and last nonzero terms of the summand corresponding to $e_{\ell}$.  More generally, if we apply the same reasoning to an arbitrary linear combination of the $e_{\ell}$, we see that  \[ \mu_{\uparrow}^{f}\left(\sum_{i=1}^{d}c_ie_i\right)=\max\left\{\left.\mu_{\uparrow}^{f}(e_i)\right| c_i\neq 0\right\}\quad\mbox{and}\quad \mu^{\downarrow}_{f}\left(\sum_{i=1}^{d}c_ie_i\right)=\min\left\{\left.\mu_{\uparrow}^{f}(e_i)\right| c_i\neq 0\right\} ,\] \emph{i.e.} that $\{e_1,\ldots,e_d\}$ is a doubly-orthogonal basis.  Thus we have a direct relationship between the extended subdiagram for (\ref{essdiag}) and the basis spectrum of (\ref{introupdown}), and hence of the essential barcode of $\mathcal{N}(f)$.  In view of this, basis spectra of arbitrary filtered matched pairs might be regarded as generalizing the extended subdiagram of extended persistence, and then Theorem \ref{introiso} could be seen as a Novikov-theoretic extension of the EP Equivalence Theorem (between extended and interlevel persistence) from \cite{CDM}.

\subsection{Novikov homology}

Let $X$ be a $\kappa$-oriented $n$-dimensional smooth manifold, let $\xi\in H^1(X;\R)$, and let $\pi\co\tilde{X}\to X$ be the smallest covering space such that $\pi^*\xi=0$.  The deck transformation group of $\tilde{X}$ is then naturally isomorphic to the subgroup $\Gamma=\{\langle\xi,a\rangle|a\in H_1(X;\Z)\}$ of $\R$, and we use this group $\Gamma$ (along with the field $\kappa$) to form the ring $\Lambda$ and the fields $\Lambda_{\uparrow},\Lambda^{\downarrow}$ as in Section \ref{basic}. For $\tilde{f}\co\tilde{X}\to\R$ a Morse function whose exterior derivative $d\tilde{f}$ is the pullback by $\pi$ of a de Rham representative of $\xi$, we recall in Section \ref{novsect} the $\Lambda_{\uparrow}$-Floer-type complex $\mathbf{CN}(\tilde{f};\xi)_{\uparrow}$.  The homology of this complex, $\mathbf{HN}_{*}(\tilde{f};\xi)$, comes equipped with a filtration function $\rho^{\tilde{f}}_{\uparrow}$ defined as usual by taking infima of filtration levels of representing cycles.  

The deck transformations of the covering space $\tilde{X}$ make its $\kappa$-coefficient singular homology into a graded module over $\Lambda=\kappa[\Gamma]$; we write this homology as $H_*(\tilde{X};\Lambda)$ to emphasize this module structure.  As originally shown in \cite{Lat}, \cite{Pa95},  there is a natural isomorphism $\Lambda_{\uparrow}\otimes_{\Lambda}H_*(\tilde{X};\Lambda)\cong\mathbf{HN}_{*}(f;\xi)$.  Let us write $\phi_{\tilde{f}}\co H_*(\tilde{X};\Lambda)\to\mathbf{HN}_{*}(f;\xi)$ for the composition of this isomorphism with the coefficient extension map  $H_*(\tilde{X};\Lambda)\to \Lambda_{\uparrow}\otimes_{\Lambda}H_*(\tilde{X};\Lambda)$.  (Then $\phi_f$ is the map induced on homology by the ``Latour map'' of Section \ref{novsect}.)  If $\Gamma\neq \{0\}$ (so that $\Lambda_{\uparrow}$ includes infinite sums that do not lie in $\Lambda$) this map will not be surjective unless its codomain vanishes; it also will typically not be injective, as its kernel is the $\Lambda$-torsion submodule of $H_*(\tilde{X};\Lambda)$.  

The Poincar\'e--Novikov structure $\mathcal{N}(\tilde{f};\xi)$ associated to $f$---obtained by passing to homology from the chain-level Poincar\'e--Novikov structure $\mathcal{CN}(f;\xi)$ from Section \ref{novsect}---is then given by:
\begin{itemize} \item the map $\phi_{f}\co H_*(\tilde{X};\Lambda)\to \mathbf{HN}_{*}(\tilde{f};\xi)$;
\item the aforementioned filtration function $\rho_{\uparrow}^{f}$ on $\mathbf{HN}_{*}(\tilde{f};\xi)$; and \item for each $k$, the map $\mathcal{D}_k\co H_k(\tilde{X};\Lambda)\to {}^{\vee}H_{n-k}(\tilde{X};\Lambda)$, where the notation ${}^{\vee}$ is defined in Section \ref{dualsec}, such that for $a\in H_k(\tilde{X};\Lambda)$ and $b\in H_{n-k}(\tilde{X};\Lambda)$, $\langle a,b\rangle:=(\mathcal{D}_ka)(b)\in \Lambda$ is the intersection pairing of $a$ with $b$ in the sense of \cite[Definition 4.66]{Ran}.
\end{itemize}
In other words, writing elements of $\Lambda=\kappa[\Gamma]$ as $\sum_{g\in \Gamma}c_gT^g$, we have $\langle a,b\rangle=\sum_{g\in\Gamma}\left((T^ga)\cap b\right)T^g$, with $\left((T^ga)\cap b\right)\in\kappa$ denoting the signed count of intersections of the image under the deck transformation corresponding to $g$ of a generic representative of $a$  with a generic representative of $b$.  The map \[ \langle\cdot,\cdot\rangle \co H_k(\tilde{X};\Lambda)\times H_{n-k}(\tilde{X};\Lambda)\to \Lambda \] is \emph{sesquilinear} in that it is biadditive and obeys, for $\lambda,\mu\in \Lambda$, \[ \langle\lambda a,\mu b\rangle=\bar{\lambda}\mu\langle a,b\rangle\] where the conjugation operation $\lambda\mapsto\bar{\lambda}$ is defined in Section \ref{conjsect}.  Note that the latter property illustrates that this pairing cannot be extended to a pairing on $\Lambda_{\uparrow}\otimes_{\Lambda}H_*(\tilde{X};\Lambda)$ (or on $\mathbf{HN}_{*}(\tilde{f};\xi)$ to which it is isomorphic): if $\lambda,\mu\in \Lambda_{\uparrow}\setminus \Lambda$, then while there is an element $\bar{\lambda}$ of the field $\Lambda^{\downarrow}$, there will usually be no way of making sense of the product $\bar{\lambda}\mu$.  The fact that the Poincar\'e pairing cannot be defined directly on Novikov homology is part of the motivation for the way that we have defined Poincar\'e--Novikov structures.

Following the procedure of Section \ref{pnsect} then yields, for each $k$, a filtered matched pair $\mathcal{P}(\mathcal{N}(\tilde{f};\xi))_{k}$ which Proposition \ref{downup} (after passing to homology) identifies with   \begin{equation}\label{novupdown} \xymatrix{ & (\overline{\mathbf{HN}_{k}(-\tilde{f};-\xi)},\rho^{\downarrow}_{\tilde{f}}) \\ H_k(\tilde{X};\Lambda) \ar[ru]^{\phi_{-\tilde{f}}} \ar[rd]_{\phi_{\tilde{f}}} & \\ & (\mathbf{HN}_k(\tilde{f};\xi),\rho_{\uparrow}^{\tilde{f}})  } \end{equation}  Here we write $\rho^{\downarrow}_{\tilde{f}}$ for $(-\rho_{\uparrow}^{-\tilde{f}})$, and the notation $\overline{\mathbf{HN}_{k}(-\tilde{f};-\xi)}$ refers to the conjugation functor defined in Section \ref{conjsect}, which converts $\Lambda_{\uparrow}$-vector spaces into $\Lambda^{\downarrow}$-vector spaces. The space $\overline{\mathbf{HN}_{k}(-\tilde{f};-\xi)}$ can be more directly interpreted in terms of $\tilde{f}$ as the $k$th homology of a $\Lambda^{\downarrow}$-Floer-type complex $\mathbf{CN}(f;\xi)^{\downarrow}$ with differential that counts positive gradient flowlines of $\tilde{f}$, just as $\mathbf{HN}_k(\tilde{f};\xi)$ is the $k$th homology of a $\Lambda_{\uparrow}$-Floer-type complex $\mathbf{CN}(f;\xi)_{\uparrow}$ whose differential counts negative gradient flowlines of $\tilde{f}$.  

If $\Gamma$ is discrete, Theorem \ref{basistheorem} shows that each filtered matched pair $\mathcal{P}(\mathcal{N}(\tilde{f};\xi))_k$ admits a doubly-orthogonal basis $\{e_1,\ldots,e_d\}\subset H_k(\tilde{X};\Lambda)$, yielding an essential barcode (Definition \ref{pnbar}) for $\mathcal{N}(\tilde{f};\xi)$ consisting of intervals modulo $\Gamma$-translation $\left[\rho_{\uparrow}^{\tilde{f}}(\phi_{\tilde{f}}e_i),\rho_{\tilde{f}}^{\downarrow}(\phi_{-\tilde{f}}e_i)\right]^{\Gamma}$ or $\left(\rho_{\tilde{f}}^{\downarrow}(\phi_{-\tilde{f}}e_i),\rho_{\uparrow}^{\tilde{f}}(\phi_{\tilde{f}}e_i)\right)^{\Gamma}$ (depending on which of $\rho_{\uparrow}^{\tilde{f}}(\phi_{\tilde{f}}e_i),\rho_{\tilde{f}}^{\downarrow}(\phi_{-\tilde{f}}e_i)$ is larger). In view of the discussion in Section \ref{connext} this essential barcode could be seen as providing a Novikov-theoretic analogue of the extended subdiagram of extended persistence.  The bars in the essential barcode of $\mathcal{N}(\tilde{f};\xi)$ represent contributions to the homologies of interlevel sets of $\tilde{f}$ according to Theorems \ref{bigdecomp} and \ref{introiso}.

Regardless of whether $\Gamma$ is discrete, the Poincar\'e--Novikov structure $\mathcal{N}(\tilde{f};\xi)$ has associated gaps $\mathcal{G}_{k,i}(\mathcal{N}(\tilde{f};\xi))$ (Definitions \ref{pngap} and \ref{gapdfn}), whose absolute values are equal to the lengths of the bars in the essential barcode in case the latter is defined.   Given two functions $\tilde{f}_0,\tilde{f}_1\co \tilde{X}\to \R$ of the type being considered in this section, a standard argument based on an estimate as in (\ref{filtchange}) for the behavior of the filtrations with respect to the usual continuation maps $\mathbf{CN}(\tilde{f}_0;\xi)_{\uparrow}\to \mathbf{CN}(\tilde{f}_1;\xi)_{\uparrow}$ and $\mathbf{CN}(\tilde{f}_1;\xi)_{\uparrow}\to \mathbf{CN}(\tilde{f}_0;\xi)_{\uparrow}$  implies that there is a $\|\tilde{f}_1-\tilde{f}_0\|_{C^0}$-morphism (in the sense of Definition \ref{pntmorph}) between $\mathcal{N}(\tilde{f}_0;\xi)$ and $\mathcal{N}(\tilde{f}_1;\xi)$.  Via  Corollaries \ref{pngapstab} and \ref{pnstab} of Theorem \ref{introstab}, this implies stability results for the gaps and (if $\Gamma$ is discrete) essential barcodes of $\mathcal{N}(\tilde{f}_0;\xi)$ and $\mathcal{N}(\tilde{f}_1;\xi)$ analogous to \cite[LZZ Stability Theorem]{CDM}, \cite[Theorem 1.2]{B18}, and \cite[Theorem 1.3]{B1}.

For $h\in H_*(\tilde{X};\Lambda)$, the values $\rho^{\downarrow}_{\tilde{f}}(\phi_{-\tilde{f}}h)$ and $\rho_{\uparrow}^{\tilde{f}}(\phi_{\tilde{f}}h)$ are, respectively, ``maximin'' and ``minimax'' quantities describing the levels of the function $\tilde{f}$ at which $h$ can be represented in the Novikov chain complexes of $-\tilde{f}$ and $\tilde{f}$.  As a sample of the kind of information that can be read from our techniques, we prove:

\begin{cor} With notation as above, suppose that $\Gamma=\Z$, so that $\tilde{f}\co \tilde{X}\to \R$ is a lift of a Morse function $f\co X\to \R/\Z$. Let $m\in \N$, and suppose that for each $k\in \Z$ there is a subset $S_k\subset H_k(\tilde{X};\Lambda)$, independent over $\Lambda$, such that for each $h\in S_k$ we have  $\rho^{\downarrow}_{\tilde{f}}(\phi_{-\tilde{f}}h)-\rho_{\uparrow}^{\tilde{f}}(\phi_{\tilde{f}}h)\geq m$.  Also let $\tau_k$ denote the dimension (over the field $\kappa$) of the $\Lambda$-torsion submodule $tH_k(\tilde{X};\Lambda)$ of $H_k(\tilde{X};\Lambda)$.  Then for each regular value $\theta$ of $f$ we have \[ \dim_{\kappa}H_k(f^{-1}(\{\theta\};\kappa)\geq \tau_{k}+m\left(\#S_k+\#S_{n-k-1}\right).\]
\end{cor}

\begin{proof}
If $\theta'\in\R$ projects to $\theta\in \R/\Z$ then $\theta'$ is a regular value of $\tilde{f}$ and the covering projection sends $\tilde{f}^{-1}(\{\theta'\})$ diffeomorphically to $f^{-1}(\{\theta\})$ so it suffices to prove the inequality with $f^{-1}(\{\theta\})$ replaced by $\tilde{f}^{-1}(\{\theta'\})$.  By Theorem \ref{introiso}, $H_k(\tilde{f}^{-1}(\{\theta'\});\kappa)$ is isomorphic to what is there denoted by $\mathbb{H}_k(\mathcal{CP}(\mathcal{CN}(f;\xi)))_{-\theta',\theta'}$.  By Theorem \ref{bigdecomp} (and Definition  \ref{fullbar}), $\dim_{\kappa}\mathbb{H}_k(\mathcal{CP}(\mathcal{CN}(f;\xi)))_{-\theta',\theta'}$ is equal to the sum of $\tau_k$ and, for each interval-mod-$\Z$-translation $I^{\Z}$ in the degree-$k$ full barcode of 
$\mathcal{CP}(\mathcal{CN}(f;\xi))$, the number of $g\in \Z$ such that $\theta'+g\in I$.  This degree-$k$ full barcode includes the degree-$k$ essential barcode of $\mathcal{N}(\tilde{f};\xi)$.  

Let us write $s_k=\#S_k$.  By hypothesis and Definition \ref{gapdfn}, for $i=1,\dots,s_k$ the gap $\mathcal{G}_{k,i}(\mathcal{N}(f;\xi))$ is greater than or equal to $m$.  It then follows from Proposition \ref{gapchar} that the degree-$k$ essential barcode of $\mathcal{N}(\tilde{f};\xi)$ includes at least $s_k$ elements $I^{\Gamma}$, counted with multilplicity, for which $I$ is a closed interval of length at least $m$.  Reversing the roles of $k$ and $n-k-1$, it follows from the duality result Corollary \ref{pndual} that the degree-$k$ essential barcode of $\mathcal{N}(\tilde{f};\xi)$ also includes at least $s_{n-k-1}$ elements $I^{\Gamma}$ for which $I$ is an open interval of length at least $m$.  Since the endpoints of any interval in the essential barcode are necessarily critical values of $\tilde{f}$,\footnote{This follows, even in the more subtle case where $\Gamma$ is dense, from \cite[Theorem 1.4]{U08}, since these endpoints are spectral numbers for Floer-type complexes all of whose nonzero elements have filtration equal to critical values of $\tilde{f}$.} regardless of whether such an interval $I$ is open or closed, if its length is at least $m$ then it must contain at least $m$ of the regular values $\theta'+g$ as $g$ ranges through $\Z$.  Thus from these (at least) $(s_k+s_{n-k-1})$-many intervals of length at least $m$ in the degree-$k$ full barcode we obtain a contribution $m(s_k+s_{n-k-1})$ to $\dim_{\kappa}(\tilde{f}^{-1}(\{\theta\});\kappa)$.  Combining this with the contribution $\tau_{k}$ mentioned in the previous paragraph implies the result.
\end{proof}

In the presence of more specific information about the bars in the essential barcode of $\mathcal{N}(\tilde{f};\xi)$ one could adapt the above proof to yield a sharper result, such as one that applies to some but not all of the regular level sets of $f$.

 \begin{remark}\label{burghelea} Let us  compare our approach to barcodes for Novikov homology to the one taken by Burghelea in works such as \cite{Bubook}, \cite{B18},\cite{B1}.  In both cases, the essential-barcode-type invariants can be viewed as being derived from the interaction between an ascending and a descending filtration by $\kappa$-subspaces of the graded $\Lambda$-module $H_*(\tilde{X};\Lambda)$.  For Burghelea these filtrations are given by the images of the inclusion-induced maps $H_*(\tilde{X}^{\leq t};\kappa)\to H_*(\tilde{X};\Lambda)$ (for the ascending filtration) and $H_*(\tilde{X}_{\geq t};\kappa)\to H_*(\tilde{X};\Lambda)$ (for the descending filtration), see, \emph{e.g.}, the notation at the start of \cite[Section 4]{B1}.  In the present work the filtrations are given by, respectively, sublevel and superlevel sets of the functions $\rho^{\tilde{f}}_{\uparrow}\circ\phi_{f}$ and $\rho_{\tilde{f}}^{\downarrow}\circ\phi_{-f}$ on $H_*(\tilde{X};\Lambda)$---said differently, by \emph{pre}images under $\phi_{\tilde{f}}$ and $\phi_{-\tilde{f}}$ of sub- and superlevel sets of the functions $\rho^{\tilde{f}}_{\uparrow}$ and $\rho_{\tilde{f}}^{\downarrow}$ defined on Novikov homology.  If $\Gamma=\{0\}$ then our filtrations are the same as Burghelea's in view of (\ref{qiniso}).  If $\Gamma\neq \{0\}$, the fact (and its analogue in Floer theory) that there is no natural inclusion-induced map to $H_*(\tilde{X};\Lambda)$ from the homologies of filtered subcomplexes of the Novikov complex---as the Novikov complex involves semi-infinite chains---motivated our different formulation.  When $\Gamma$ is discrete, one could perhaps use \cite[Proposition 5.2]{B18}, which connects Burghelea's filtrations to analogous ones on Borel--Moore homology, to obtain a relation to our approach.  

If $\Gamma$ is not discrete, the collection of our gaps $\mathcal{G}_{k,1}(\mathcal{N}(\tilde{f};\xi)),\ldots,\mathcal{G}_{k,d}(\mathcal{N}(\tilde{f};\xi))$, for $d$ equal to the maximal cardinality of a $\Lambda$-independent set in $H_k(\tilde{X};\Lambda)$, appears to be analogous to what is denoted $\mathbf{\delta}_{k}^{\omega}$ in \cite{B1}; in both cases these are multisets of $\R$ with cardinality $d$ that satisfy a stability theorem (Corollary \ref{pngapstab} for the $\mathcal{G}_{k,i}$, \cite[Theorem 1.3]{B1} for $\mathbf{\delta}_{k}^{\omega}$).  While one might speculate that   $\mathbf{\delta}_{k}^{\omega}$ consists precisely of our $\mathcal{G}_{k,i}$, any attempt to prove this lies beyond the scope of this paper.  Part of the difficulty is that, since in the indiscrete case the function $\tilde{f}$ will not be proper, the connection between homologies of filtered subcomplexes of the Novikov complex and homologies of sublevel sets of $\tilde{f}$ is less clear.  

In the indiscrete case, we also leave open the question of a duality result for the $\mathcal{G}_{k,i}(\mathcal{N}(\tilde{f};\xi))$, which would assert that (modulo reordering) the $\mathcal{G}_{k,i}$ are the negatives of the $\mathcal{G}_{n-k,i}$.  (If $\Gamma$ is discrete this follows from Corollary \ref{pndual}.  In general, Corollary \ref{weakdual} implies that the $\mathcal{G}_{k,i}$ are bounded above by the negatives of the $\mathcal{G}_{n-k,i}$ after reordering.)  There is a corresponding question of whether Burghelea's $\mathbf{\delta}_{k}^{\omega}$ is obtained by negating the elements of $\mathbf{\delta}_{n-k}^{\omega}$; by \cite[Theorem 1.4]{B20b} this is true if $\mathbf{\delta}_{n-k}^{\omega}$ is replaced by a Borel--Moore analogue ${}^{BM}\!\mathbf{\delta}_{n-k}^{\omega}$, but in the indiscrete case it is apparently unknown whether $\mathbf{\delta}_{n-k}^{\omega}={}^{BM}\!\mathbf{\delta}_{n-k}^{\omega}$.
\end{remark}

\subsection{Other filtered matched pairs}
The preceding two subsections have discussed constructions of Poincar\'e--Novikov structures $\mathcal{N}$, which give rise to essential barcodes via a general procedure that first associates to $\mathcal{N}$ filtered matched pairs $\mathcal{P}(\mathcal{N})_k$ and then associates basis spectra to the latter.  One can also consider filtered matched pairs that do not arise from Poincar\'e--Novikov structures; these still have gaps   and (if $\Gamma$ is discrete) basis spectra that satisfy the stability result Theorem \ref{introstab}, though they will not typically satisfy symmetry results such as Corollary \ref{pndual}.

For example, if one drops the hypothesis in the two preceding subsections that $X$ is $\kappa$-oriented, then due to sign issues one will no longer have the same type of Poincar\'e intersection pairing, so the construction of $\mathcal{N}(\tilde{f};\xi)$ fails, but one can still form the Novikov complexes $\mathbf{CN}(\pm \tilde{f};\pm \xi)_{\uparrow}$ and the filtered matched pairs (\ref{novupdown}). In Section \ref{isosect} we consider a persistence module $\mathbb{H}_k(\tilde{f})$ that is constructed directly from a chain-level version of (\ref{novupdown}) via (\ref{hkdef}) without using the orientation of $X$.  The isomorphism result Theorem \ref{introiso}, while stated for oriented manifolds, is in fact a consequence of Theorem \ref{bigiso} which, with no orientability assumption, connects the persistence module $\mathbb{H}_k(\tilde{f})$ to the homologies of interlevel sets. Thus, if $\Gamma$ is discrete, we will have the same relation between the basis spectrum of (\ref{novupdown}) and homologies of interlevel sets as in the previous subsection.

Another small variation on the preceding subsection would be to consider two different Morse functions $\tilde{f},\tilde{g}\co \tilde{X}\to \R$, both having derivatives equal to pullbacks of de Rham representatives of the same cohomology class $\xi$, and, for any grading $k$, form the filtered matched pair \[  \xymatrix{ & (\overline{\mathbf{HN}_{k}(-\tilde{g};-\xi)},\rho^{\downarrow}_{\tilde{g}}) \\ H_k(\tilde{X};\Lambda) \ar[ru]^{\phi_{-g}} \ar[rd]_{\phi_f} & \\ & (\mathbf{HN}_k(\tilde{f};\xi),\rho_{\uparrow}^{\tilde{f}})  } \]  The basis spectrum of this would (at least if $\Gamma$ is discrete) encode homological information about how the sublevel sets of $\tilde{f}$ interact with the superlevel sets of $\tilde{g}$.  As described in Remark \ref{twostab}, the stability theorem can be refined to give information about how the basis spectrum varies as $\tilde{f}$ and $\tilde{g}$ are varied independently of each other.

A setting that is closer to standard constructions in computational topology involves a polyhedron (\emph{i.e.}, geometric realization of a finite simplicial complex) $X$ together with a PL function $f\co X\to \R$.  The simplicial chain complex $\Delta_{*}(X;\kappa)$ carries an ascending filtration function $\ell_{\uparrow}$ and a descending filtration function $\ell^{\downarrow}$ defined by sending a simplicial chain to the maximal (resp. minimal) value of $f$ on the chain. Thus for any $t\in \R$ the subcomplex $\Delta_{*}^{\leq t}(X;\kappa)=\{a\in\Delta_*(X;\kappa)|\ell_{\uparrow}(a)\leq t\}$ is the simplicial chain complex of the subpolyhedron of $X$ consisting of simplices that are entirely contained in the sublevel set $X^{\leq t}$; by an argument on \cite[p. 135]{EH} this subpolyhedron is a deformation retract of $X^{\leq t}$.  An analogous remark applies to relate the filtration given by $\ell^{\downarrow}$ to the superlevel sets of $f$.    

This fits into our general theory of chain-level filtered matched pairs with $\Gamma=\{0\}$, taking all three chain complexes $\mathcal{C},\mathcal{C}_{\uparrow},\mathcal{C}^{\downarrow}$ in Definition \ref{cfmpdfn} equal to $\Delta_{*}(X;\kappa)$ (with the filtrations on $\mathcal{C}_{\uparrow}$ and $\mathcal{C}^{\downarrow}$ being given by $\ell_{\uparrow}$ and $\ell^{\downarrow}$ respectively), and taking the maps $\phi_{\uparrow},\phi^{\downarrow}$ both equal to the identity.  One can verify that the full barcode (Definition \ref{fullbar}) of this chain-level filtered matched pair coincides with the interlevel barcode of $f$, for instance by passing through an identification with extended persistence along the lines of Section \ref{connext}.

If we instead let $f\co X\to S^1$ be a PL function on a polyhedron (\emph{i.e.}, the restriction of $f$ to each simplex $\Delta\subset X$ should be the composition of an affine map $\Delta\to \R$ and the projection $\R\to S^1=\R/\Z$), then for a suitable infinite cyclic cover $\tilde{X}\to X$, $f$ lifts to a PL map $\tilde{f}\co \tilde{X}\to \R$.  The simplicial chain complex $\Delta_{*}(\tilde{X};\kappa)$ of the (infinite) simplicial complex $\tilde{X}$ is then a chain complex of free finite-rank modules over $\Lambda=\kappa[\Z]$.  We obtain a chain-level filtered matched pair (with $\Gamma=\Z$) by setting $\mathcal{C}=\Delta_{*}(\tilde{X};\kappa)$, $\mathcal{C}_{\uparrow}=\Lambda_{\uparrow}\otimes_{\Lambda}\mathcal{C}$, and $\mathcal{C}^{\downarrow}=\Lambda^{\downarrow}\otimes_{\Lambda}\mathcal{C}$, with $\phi_{\uparrow},\phi^{\downarrow}$ each given by coefficient extension.  As before, the filtration functions $\ell_{\uparrow},\ell^{\downarrow}$ are defined by taking, respectively, maxima and minima of $f$ on chains.  The resulting full barcode should agree with the one considered in \cite{BD},\cite{BH17}, though our method for computing it---based on the doubly-orthogonal basis constructed in the proof of Theorem \ref{basistheorem}---is rather different from the algorithm from \cite[Section 6]{BD}.

\section{Some algebraic preliminaries}\label{basic}

We now begin to give more precise explanations of the algebraic ingredients in this work.  
In the background throughout what follows---and often suppressed from the notation for brevity---are: 
\begin{itemize} \item a field $\kappa$ (which serves as a homology coefficient ring, though in some cases the appropriate ring will be an extension of $\kappa$) and
\item a finitely generated additive subgroup $\Gamma$ of $\mathbb{R}$.\end{itemize}

In motivating examples, the group $\Gamma$ is (isomorphic to) the deck transformation group of a covering space $p\co \tilde{X}\to X$, and we wish to understand the persistence theory of functions $\tilde{f}\co \tilde{X}\to \mathbb{R}$ that arise as lifts of functions $f\co X\to \R/\Gamma$.  The case that $\Gamma=\{0\}$ corresponds to classical Morse theory, or in the situation of Section \ref{floersect} to Hamiltonian Floer theory on weakly exact or positively or negatively monotone symplectic manifolds.  Readers who are (at least on first reading) inclined to confine themselves to this case might skim the present section with the understanding that, when $\Gamma=\{0\}$, the objects denoted below by $\Lambda,\Lambda_{\uparrow},\Lambda^{\downarrow},\mathcal{Q}(\Lambda),$ and $\Lambda_{\updownarrow}$ are all simply equal to the field $\kappa$, and the conjugation operation discussed in Section \ref{conjsect} is the identity, as a result of which many of the results of this section are vacuous in this special case.  

With ``$T$'' denoting a formal variable, we consider the following rings containing $\kappa$ (with addition and multiplication given by the obvious generalizations of the corresponding operations on polynomials or power series):

\begin{enumerate}\item \[ \Lambda = \left\{\left.\sum_{g\in \Gamma}a_gT^g\right|a_g\in \kappa,\, \#\{g|a_g\neq 0\}<\infty\right\} \]
\item \[  \Lambda_{\uparrow}=\left\{\left.\sum_{g\in \Gamma}a_gT^g\right|a_g\in \kappa,\, (\forall C\in \R)\left(\#\{g|a_g\neq 0\mbox{ and }g<C\}<\infty\right)\right\}  \]
\item \[ \Lambda^{\downarrow} = \left\{\left.\sum_{g\in \Gamma}a_gT^g\right|a_g\in \kappa,\, (\forall C\in \R)\left(\#\{g|a_g\neq 0\mbox{ and }g>-C\}<\infty\right)\right\}  \]\end{enumerate}

The ring $\Lambda$ is just the group algebra $\kappa[\Gamma]$ of $\Gamma$ over $\kappa$.  Since $\Gamma$ (being a finitely generated subgroup of $\mathbb{R}$) is isomorphic to $\mathbb{Z}^d$ for some $d\geq 0$, we see that $\Lambda$ is isomorphic as a $\kappa$-algebra to the multivariable Laurent polynomial algebra $\kappa[t_1,t_{1}^{-1},\ldots,t_d,t_{d}^{-1}]$.  In particular, $\Lambda$ is an integral domain.  

The rings $\Lambda_{\uparrow}$ and $\Lambda^{\downarrow}$ are each \emph{fields}, as can be seen for instance from \cite[Theorem 4.1]{HS}.  Since $\Lambda$ is a subring both of $\Lambda_{\uparrow}$ and of $\Lambda^{\downarrow}$, if $\mathcal{Q}(\Lambda)$ denotes the field of fractions of $\Lambda$, the inclusions of $\Lambda$ into $\Lambda_{\uparrow}$ and $\Lambda^{\downarrow}$ extend uniquely to field extensions $\mathcal{Q}(\Lambda)\hookrightarrow \Lambda_{\uparrow}$ and $\mathcal{Q}(\Lambda)\hookrightarrow \Lambda^{\downarrow}$.  Like any field of fractions of an integral domain, $\mathcal{Q}(\Lambda)$ is flat as a $\Lambda$-module (apply \cite[Theorem 4.80]{Rot} with the multiplicative set $S$ equal to the set of all nonzero elements of $\Lambda$).  So since field extensions are likewise flat as modules, it follows that the inclusions $i_{\uparrow}\co \Lambda\hookrightarrow \Lambda_{\uparrow}$ and $i^{\downarrow}\co \Lambda\hookrightarrow\Lambda^{\downarrow}$ make $\Lambda_{\uparrow}$ and $\Lambda^{\downarrow}$ into flat $\Lambda$-modules.  Related to this, we have:

\begin{prop}\label{kerext}
Let $M$ be any (left) $\Lambda$-module.  Then the coefficient extension maps $ i_{\uparrow}\otimes 1_M\co M\to \Lambda_{\uparrow}\otimes_{\Lambda}M$ and $ i^{\downarrow}\otimes 1_M\co M\to \Lambda^{\downarrow}\otimes_{\Lambda}M$ each have kernel equal to the $\Lambda$-torsion submodule \[ tM:=\{m\in M|(\exists \lambda\in \Lambda\setminus\{0\})(\lambda m=0)\}.\]  
\end{prop}

\begin{proof}
The two cases are identical so we just consider the case of $1_M\otimes i_{\uparrow}$.  This map can be written as a composition \[ M=\Lambda\otimes_{\Lambda}M\to \mathcal{Q}(\Lambda)\otimes_{\Lambda}M\to \Lambda_{\uparrow}\otimes_{\Lambda}M \] of maps obtained by tensoring the identity $1_M$ with injections $\Lambda\hookrightarrow \mathcal{Q}(\Lambda)\hookrightarrow \Lambda_{\uparrow}$.  The map $ M\to \mathcal{Q}(\Lambda)\otimes_{\Lambda}M$ has kernel equal to $tM$ by \cite[Proposition 4.78]{Rot}, while the map $\mathcal{Q}(\Lambda)\otimes_{\Lambda}M\to \Lambda_{\uparrow}\otimes_{\Lambda}M$ has trivial kernel since $\mathcal{Q}(\Lambda)\hookrightarrow \Lambda_{\uparrow}$ is a field extension.
\end{proof}

Now let us introduce, for use beginning in Section \ref{clsec}, \begin{equation}\label{updown} \Lambda_{\updownarrow}=\left\{\left.\sum_{g\in \Gamma}a_gT^g\right|a_g\in \kappa,\, (\forall C\in \R)\left(\#\{g|a_g\neq 0\mbox{ and }|g|<C\}<\infty\right)\right\}.  \end{equation}  We thus have inclusions $j_{\uparrow}\co \Lambda_{\uparrow}\to \Lambda_{\updownarrow}$ and $j^{\downarrow}\co \Lambda^{\downarrow}\to \Lambda_{\updownarrow}$. Since multiplication of elements of $\Lambda_{\uparrow}$ by elements of $\Lambda^{\downarrow}$ is typically ill-defined, their common superset $\Lambda_{\updownarrow}$ is not naturally a ring (except in the case $\Gamma=\{0\}$); however there is no difficulty in making sense of addition of elements of $\Lambda_{\updownarrow}$, or of multiplication of elements of $\Lambda_{\updownarrow}$ by elements of $\Lambda$, and so $\Lambda_{\updownarrow}$ is naturally a $\Lambda$-module.  Indeed we evidently have a short exact sequence of $\Lambda$-modules \begin{equation}\label{updownses} \xymatrix{0 \ar[r] & \Lambda\ar[r]^<<<<<{(i^{\downarrow},i_{\uparrow})} &  \Lambda^{\downarrow}\oplus \Lambda_{\uparrow}\ar[r]^<<<<<{\,\,-j^{\downarrow}+j_{\uparrow}} & \Lambda_{\updownarrow}\ar[r]&0}\end{equation} 

\begin{prop}\label{tor}  For any left $\Lambda$-module $M$ we have a natural isomorphism of $\Lambda$-modules $\mathrm{Tor}^{\Lambda}_{1}(\Lambda_{\updownarrow},M)\cong tM$ where $tM$ is the $\Lambda$-torsion submodule of $M$, and $\mathrm{Tor}^{\Lambda}_{i}(\Lambda_{\updownarrow},M)=\{0\}$ for all $i>1$.
\end{prop}

\begin{proof} We use the long exact sequence on $\mathrm{Tor}^{\Lambda}_{*}(\cdot,M)$ induced by (\ref{updownses}); see for instance \cite[Corollary 6.30]{Rot}.  Since $\Lambda_{\uparrow}$ and $\Lambda^{\downarrow}$ are flat $\Lambda$-modules we have $\mathrm{Tor}_{i}^{\Lambda}(\Lambda_{\uparrow}\oplus\Lambda^{\downarrow},M)=\{0\}$ for all $i\geq 1$.  So part of the long exact sequence reads \[ \xymatrix{0\ar[r]&  \mathrm{Tor}^{\Lambda}_{1}(\Lambda_{\updownarrow},M)\ar[r]^{\partial} & \Lambda\otimes_{\Lambda}M\ar[rr]^-{(i^{\downarrow}\otimes 1_M,i_{\uparrow}\otimes 1_M)}& & (\Lambda^{\downarrow}\otimes_{\Lambda}M)\oplus (\Lambda_{\uparrow}\otimes_{\Lambda}M)  }.\]  By Proposition \ref{kerext}, the last map above has kernel equal to $tM$, and so the connecting homomorphism $\partial$ maps $\mathrm{Tor}^{\Lambda}_{1}(\Lambda_{\updownarrow},M)$ isomorphically to $tM$.  If $i>1$ then the exact sequence and the flatness of $\Lambda_{\uparrow}$ and $\Lambda^{\downarrow}$ imply an isomorphism $\mathrm{Tor}^{\Lambda}_{i}(\Lambda_{\updownarrow},M)\cong \mathrm{Tor}_{i-1}^{\Lambda}(\Lambda,M)$, but of course the latter vanishes since $i-1\geq 1$.
\end{proof}

\subsection{Conjugation}\label{conjsect}

Given our field $\kappa$ and subgroup $\Gamma\subset \R$, with $\Lambda=\kappa[\Gamma]$ define the \emph{conjugation} automorphism $\mathfrak{c}\co\Lambda\to\Lambda$ by \[ \mathfrak{c}\left(\sum_g a_gT^g\right)=\sum_g a_gT^{-g} \]  In general, we will denote $\mathfrak{c}(\lambda)$ as $\bar{\lambda}$ for $\lambda\in \Lambda$. Note that $\mathfrak{c}$ extends, using the same formula as above, to a field isomorphism $\Lambda_{\uparrow}\to\Lambda^{\downarrow}$, or just as well to a field isomorphism $\Lambda^{\downarrow}\to\Lambda_{\uparrow}$, and we will continue to use the notation $\bar{\lambda}$ for the image of an element $\lambda$ of $\Lambda^{\downarrow}$ or of $\Lambda_{\uparrow}$ under these conjugation maps.  While conjugation defines an isomorphism of fields $\Lambda_{\uparrow}\cong\Lambda^{\downarrow}$, if $\Gamma\neq\{0\}$ this is of course not an isomorphism of $\Lambda$-algebras (if it were, all elements of $\Lambda$ would be fixed by $\mathfrak{c}$), and issues related to this are why we will maintain the distinction between $\Lambda_{\uparrow}$ and $\Lambda^{\downarrow}$ rather than using conjugation to work exclusively over one or the other of them.

\begin{remark}\label{conjrmk} For \emph{any} group $\pi$, abelian or not, the group algebra $\kappa[\pi]$ admits a conjugation map $\mathfrak{c}$ similar to that above, though if $\pi$ is not abelian (and hence $\kappa[\pi]$ is not commutative) this is not an algebra automorphism but rather an anti-automorphism in that $\mathfrak{c}(\lambda\mu)=\mathfrak{c}(\mu)\mathfrak{c}(\lambda)$.  This map $\mathfrak{c}$ arises in considerations of Poincar\'e duality for regular covering spaces with deck transformation group $\pi$, \emph{cf}. \cite[Section 4.5]{Ran}.  In this setting, despite the noncommutativity of $\kappa[\pi]$, a left $\kappa[\pi]$-module $M$ can be converted to a right $\kappa[\pi]$-module by defining $m\lambda=\mathfrak{c}(\lambda)m$ for $\lambda\in \kappa[\pi]$ and $m\in M$.  This conjugation operation on modules remains useful in cases such as ours where the group $\pi$ (our $\Gamma$) is abelian.  For this reason, contrary to the custom of making no distinction between the left and right modules of a commutative ring, we will often specify a module over $\Lambda$  as being a left module or a right module, with the understanding that if one wants to switch between these the appropriate way to do so is by conjugation, not by trivially 
renaming multiplication on the left as multiplication on the right.  An exception to this is that when one of the rings $\Lambda,\Lambda_{\uparrow},$ or $\Lambda^{\downarrow}$ is regarded as a module over its subring $\Lambda$, both the left and the right module structure should be interpreted as simply given by the ring multiplication, with no conjugation.  Also, in vector spaces over the fields $\Lambda_{\uparrow}$ or $\Lambda^{\downarrow}$ the scalar multiplication will consistently be written as acting on the left.  \end{remark}

Given a left (resp. right) $\Lambda$-module $M$, as just suggested we define its conjugate module $\bar{M}$ to be the right (resp. left) $\Lambda$-module whose underlying abelian group is the same as $M$, and with the scalar multiplication given by $m\lambda=\bar{\lambda}m$ (resp. $\lambda m=m\bar{\lambda}$).  Similarly, if $V$ is a vector space over $\Lambda_{\uparrow}$, it has an associated conjugate vector space $\bar{V}$ over $\Lambda^{\downarrow}$ with the same underlying abelian group but with the new scalar multiplication given by $\bar{\lambda}v=\lambda v$, and likewise with the roles of $\Lambda_{\uparrow}$ and $\Lambda^{\downarrow}$ reversed.  These operations are trivially functorial, in that if $\phi\co M\to N$ is a left $\Lambda$-module homomorphism (resp. $\Lambda_{\uparrow}$-vector space homomomrphism) then the same function $\phi$ can equally well be regarded as a right $\Lambda$-module homomorphism (resp. $\Lambda^{\downarrow}$-vector space homomorphism) $\bar{M}\to \bar{N}$.

The proof of the following is left as an exercise to readers who wish to check their understanding of the notation.
\begin{prop}\label{flipconj}
If $M$ is a right $\Lambda$-module there is an isomorphism of $\Lambda^{\downarrow}$-vector spaces \[ \alpha\co \overline{M\otimes_{\Lambda}\Lambda_{\uparrow}}\to \Lambda^{\downarrow}\otimes_{\Lambda}\bar{M} \] defined by, for $m_i\in M$ and $\lambda_i\in \Lambda_{\uparrow}$, \[ \alpha\left(\sum_i m_i\otimes\lambda_i\right) = \sum_i\bar{\lambda}_i\otimes m_i.\]
\end{prop}

\subsection{Non-archimedean norms and orthogonal bases}\label{nonarch}

Let us define functions \[ \nu_{\uparrow}\co \Lambda\to \mathbb{R}\cup\{\infty\}\qquad \nu^{\downarrow}\co \Lambda\to \R\cup\{-\infty\} \] by \begin{equation}\label{nuformula} \nu_{\uparrow}\left(\sum_ga_gT^g\right)=\min\{g|a_g\neq 0\},\qquad \nu^{\downarrow}\left(\sum_ga_gT^g\right)=\max\{g|a_g\neq 0\} \end{equation} (with the convention that the empty set has maximum $-\infty$ and minimum $\infty$).  Setting $\nu$ equal either to $\nu_{\uparrow}$ or to $-\nu^{\downarrow}$ gives a non-Archimedean valuation on $\Lambda$, which is to say a function $\nu\co \Lambda\to\R\cup\{\infty\}$ obeying: \begin{itemize} \item[(i)] $\nu(\lambda)=\infty$ iff $\lambda=0$; \item[(ii)] $\nu(\lambda\mu)=\nu(\lambda)+\nu(\mu)$ for all $\lambda,\mu\in\Lambda$; \item[(iii)] $\nu(\lambda+\mu)\geq\min\{\nu(\lambda),\nu(\mu)\}$ for all $\lambda,\mu\in \Lambda$. \end{itemize}

Thus either of the formulas $(\lambda,\mu)\mapsto e^{-\nu_{\uparrow}(\lambda-\mu)}$ or $(\lambda,\mu)\mapsto e^{\nu^{\downarrow}(\lambda-\mu)}$ defines a metric on $\Lambda$, and the fields $\Lambda_{\uparrow}$ and $\Lambda^{\downarrow}$ are, respectively, the completions of $\Lambda$ with respect to these metrics.   The functions $\nu_{\uparrow}$ and $\nu^{\downarrow}$ extend by the same formulas as in (\ref{nuformula}) to functions $\nu_{\uparrow}\co \Lambda_{\uparrow}\to \R\cup\{\infty\}$ and $\nu^{\downarrow}\co \Lambda^{\downarrow}\to \R\cup\{-\infty\}$ (note that while $\Lambda_{\uparrow}$ and $\Lambda^{\downarrow}$ contain certain infinite sums $\sum a_g T^g$,  the $\min$ in (\ref{nuformula}) will still be attained for a nonzero element of $\Lambda_{\uparrow}$, and the $\max$ in (\ref{nuformula}) will still be attained for a nonzero element of $\Lambda^{\downarrow}$).  Evidently we have \[ \nu_{\uparrow}(\lambda)=-\nu^{\downarrow}(\bar{\lambda}) \] for $\lambda\in \Lambda_{\uparrow}$, using the conjugation map $\Lambda_{\uparrow}\to\Lambda^{\downarrow}$ defined in Section \ref{conjsect}.

There is a standard notion of a non-Archimedean norm on a vector space over a non-Archimedean field, which the following adapts (\emph{cf}. \cite{MoSp},\cite[Definition 2.2]{UZ}):
\begin{dfn}\label{normdef} \begin{enumerate} \item A \textbf{normed $\Lambda_{\uparrow}$-space} is a pair $(V_{\uparrow},\rho_{\uparrow})$ where $V_{\uparrow}$ is a finite-dimensional vector space over $\Lambda_{\uparrow}$ and the function $\rho_{\uparrow}\co V_{\uparrow}\cup\{-\infty\}$ satisfies: 
		\begin{itemize} \item[(i)] $\rho_{\uparrow}(x)=-\infty$ if and only if $x=0$; \item[(ii)] $\rho_{\uparrow}(\lambda x)=\rho_{\uparrow}(x)-\nu_{\uparrow}(\lambda)$ for all $\lambda\in \Lambda_{\uparrow},x\in V_{\uparrow}$; and 
		\item[(iii)] $\rho_{\uparrow}(x+y)\leq \max\{\rho_{\uparrow}(x),\rho_{\uparrow}(y)\}$ for all $x,y\in V_{\uparrow}$.\end{itemize}
\item A \textbf{normed $\Lambda^{\downarrow}$-space} is a pair $(V^{\downarrow},\rho^{\downarrow})$ where $V^{\downarrow}$ is a finite-dimensional vector space over $\Lambda^{\downarrow}$ and the function $\rho^{\downarrow}\co V^{\downarrow}\to \R\cup\{\infty\}$ satisfies:		
\begin{itemize} \item[(i)] $\rho^{\downarrow}(x)=\infty$ if and only if $x=0$; \item[(ii)] $\rho^{\downarrow}(\lambda x)=\rho_{\downarrow}(x)-\nu^{\downarrow}(\lambda)$ for all $\lambda\in \Lambda^{\downarrow},x\in V^{\downarrow}$; and 
\item[(iii)] $\rho^{\downarrow}(x+y)\geq \min\{\rho^{\downarrow}(x),\rho^{\downarrow}(y)\}$ for all $x,y\in V^{\downarrow}$.\end{itemize} \end{enumerate}
\end{dfn}

(The actual ``norm'' on what we are calling a ``normed $\Lambda_{\uparrow}$-space'' $(V_{\uparrow},\rho_{\uparrow})$ or a ``normed $\Lambda^{\downarrow}$-space'' $(V^{\downarrow},\rho^{\downarrow})$ would be $e^{\rho_{\uparrow}}$ or $e^{-\rho^{\downarrow}}$.)

If $(V_{\uparrow},\rho_{\uparrow})$ is a normed $\Lambda_{\uparrow}$-space, then $\rho_{\uparrow}$ encodes an ascending  filtration on $V_{\uparrow}$, by $\kappa$-subspaces $V_{\uparrow}^{t}=\{x\in V_{\uparrow}|\rho_{\uparrow}(x)\leq t\}$ for $t\in \R$, satisfying $V_{\uparrow}^{s}\subset V_{\uparrow}^{t}$ for $s\leq t$.  Dually, if $(V^{\downarrow},\rho^{\downarrow})$ is a normed $\Lambda^{\downarrow}$-space then we have a descending filtration of $V^{\downarrow}$ by the $\kappa$-subspaces $\{x\in V^{\downarrow}|\rho^{\downarrow}(x)\geq t\}$ for $t\in \R$.

The usual notion of non-Archimedean orthogonality will be important for us:

\begin{dfn}\label{orthdef}
	Let $(V_{\uparrow},\rho_{\uparrow})$ be a normed $\Lambda_{\uparrow}$-space.  A subset $\{x_1,\ldots,x_k\}\subset  V_{\uparrow}$ is said to be $\rho_{\uparrow}$-\textbf{orthogonal} if one has, for all $\lambda_1,\ldots,\lambda_k\in \Lambda_{\uparrow}$, \[ \rho_{\uparrow}\left(\sum_{i=1}^{k}\lambda_ix_i\right)=\max_{1\leq i\leq k}\left(\ell_{\uparrow}(x_i)-\nu_{\uparrow}(\lambda_i)\right).\]  Similarly if $(V^{\downarrow},\rho^{\downarrow})$ is a normed $\Lambda^{\downarrow}$-space, a subset $\{x_1,\ldots,x_k\}\subset  V^{\downarrow}$ is said to be $\rho^{\downarrow}$-\textbf{orthogonal} if one has, for all $\lambda_1,\ldots,\lambda_k\in \Lambda^{\downarrow}$, \[ \rho^{\downarrow}\left(\sum_{i=1}^{k}\lambda_ix_i\right)=\min_{1\leq i\leq k}\left(\ell_{\uparrow}(x_i)-\nu_{\uparrow}(\lambda_i)\right).\]
 
 A normed $\Lambda_{\uparrow}$-space (resp. normed $\Lambda^{\downarrow}$-space) is called an \textbf{orthogonalizable $\Lambda_{\uparrow}$-space} (resp. orthogonalizable $\Lambda^{\downarrow}$-space) if it has a $\rho_{\uparrow}$-orthogonal (resp. $\rho^{\downarrow}$-orthogonal) basis.
\end{dfn}

\begin{remark}\label{equalcase}
	A standard exercise shows that, in any normed $\Lambda_{\uparrow}$-space, the non-Archimedean triangle inequality  $\ell_{\uparrow}(x+y)\leq \max\{\ell_{\uparrow}(x),\ell_{\uparrow}(y)\}$ is an equality whenever $\ell_{\uparrow}(x)\neq \ell_{\uparrow}(y)$.   From this it is easy to see that, if nonzero elements $x_1,\ldots,x_k$ of a normed $\Lambda_{\uparrow}$-space $(V_{\uparrow},\rho_{\uparrow})$ have the property that the real numbers $\rho_{\uparrow}(x_1),\ldots,\rho_{\uparrow}(x_k)$ all have distinct reductions modulo $\Gamma$, then $\{x_1,\ldots,x_k\}$ must be $\rho_{\uparrow}$-orthogonal.  (In particular, it must be linearly independent.) Since by our definition normed $\Lambda_{\uparrow}$-spaces are finite-dimensional, it follows that $\rho_{\uparrow}(V_{\uparrow}\setminus\{0\})$ is a union of at most $(\dim V_{\uparrow})$-many cosets of $\Gamma$ for any normed $\Lambda_{\uparrow}$-space $(V_{\uparrow},\rho_{\uparrow})$, orthogonalizable or not.  In particular, $\rho_{\uparrow}(V_{\uparrow}\setminus\{0\})$ is discrete if $\Gamma$ is discrete.  
	
	Furthermore, if $\rho_{\uparrow}(V_{\uparrow})\setminus\{0\})$ consists of exactly $(\dim V_{\uparrow})$-many cosets of $\Gamma$, then a general basis $\{x_1,\ldots,x_d\}$ for $V_{\uparrow}$ is $\rho_{\uparrow}$-orthogonal if and only if the reductions of $\rho_{\uparrow}(x_i)$ mod $\Gamma$ are distinct for $i=1,\ldots,d$.
	 
	Similar remarks of course apply to normed $\Lambda^{\downarrow}$-spaces.
\end{remark}

\begin{remark}\label{switch}
	If $(V_{\uparrow},\rho_{\uparrow})$ is an normed $\Lambda_{\uparrow}$-space then $(\overline{V_{\uparrow}},-\rho_{\uparrow})$ is an normed $\Lambda^{\downarrow}$-space, and similarly with the symbols $\uparrow$ and $\downarrow$ reversed.   Moreover $\{e_1,\ldots,e_m\}$ is a $\rho_{\uparrow}$-orthogonal basis for $V_{\uparrow}$ if and only if the same set $\{e_1,\ldots,e_d\}$ is a $(-\rho_{\uparrow})$-orthogonal basis for $\overline{V_{\uparrow}}$. This allows one to convert basic results about orthogonalizable $\Lambda_{\uparrow}$-spaces such as those appearing in \cite[Section 2]{UZ} to results about orthogonalizable $\Lambda^{\downarrow}$-spaces.  For instance, subspaces of such spaces are always orthogonalizable (\cite[Corollary 2.17]{UZ}).
\end{remark}



\begin{remark}
	Let us consider Definition \ref{orthdef} in the case that $\Gamma=\{0\}$ and hence $\Lambda_{\uparrow}$ and $\Lambda^{\downarrow}$ both coincide with the coefficient field $\kappa$.  Both $\nu_{\uparrow}$ and $\nu^{\downarrow}$ send all nonzero elements of $\kappa$ to $0$, and they send $0$ to $+\infty$ and $-\infty$, respectively.  So an orthogonalizable $\Lambda_{\uparrow}$-space $(V_{\uparrow},\rho_{\uparrow})$ has a basis $\{e_1,\ldots,e_m\}$ for which $\rho_{\uparrow}\left(\sum_i\lambda_ie_i\right)=\max\{\rho_{\uparrow}(e_i)|\lambda_i\neq 0\}$, while an orthogonalizable $\Lambda^{\downarrow}$-space has a basis $\{e_1,\ldots,e_m\}$ for which $\rho^{\downarrow}\left(\sum_i\lambda_ie_i\right)=\min\{\rho^{\downarrow}(e_i)|\lambda_i\neq 0\}$.\end{remark}

The following could be regarded as an instance of \cite[Th\'eor\`eme 1.1]{MoSp}, though part of the proof below is needed to confirm that the hypotheses of that theorem are satisfied.

\begin{prop}\label{discretenormed}
	Suppose that $\Gamma$ is discrete.  Then every normed $\Lambda_{\uparrow}$-space is orthgonalizable. 
\end{prop}
\begin{proof}
	Let $(V_{\uparrow},\rho_{\uparrow})$ be an orthogonalizable $\Lambda_{\uparrow}$-space.  By \cite[Lemma 2.15]{UZ}, if $\{e_1,\ldots,e_k\}$ is a $\rho_{\uparrow}$-orthogonal subset of $V_{\uparrow}$ and if $e_{k+1}\in V_{\uparrow}\setminus\mathrm{span}_{\Lambda_{\uparrow}}\{e_1,\ldots,e_k\}$, then $\{e_1,\ldots,e_k,e_{k+1}\}$ is $\rho_{\uparrow}$-orthogonal if and only if \[ \rho_{\uparrow}(e_{k+1})=\inf\left\{\rho_{\uparrow}(e_{k+1}-w)|w\in\mathrm{span}_{\Lambda_{\uparrow}}\{e_1,\ldots,e_k\}\right\}.\]
	So it suffices to show that for any orthogonalizable subspace $V\leq V_{\uparrow}$ and any $x\in V_{\uparrow}\setminus V$, there is $v\in V$ such that $\rho_{\uparrow}(x-v)=\inf\{\rho_{\uparrow}(x-w)|w\in V\}$, for then we may form an orthogonal basis for $V_{\uparrow}$ by choosing $e_1\in V_{\uparrow}\setminus\{0\}$ arbitrarily and inductively, having chosen an orthogonal set $\{e_1,\ldots,e_k\}$ with span $V_k$ and $k<\dim_{\Lambda_{\uparrow}}V_{\uparrow}$, setting $e_{k+1}=x-v$ where $x\in V_{\uparrow}\setminus V_k$ and $v\in V_k$ satisfies $\rho_{\uparrow}(x-v)=\inf\{\rho_{\uparrow}(x-w)|w\in V_k\}$.
	
	To see that optimization problems $\rho_{\uparrow}(x-v)=\inf\{\rho_{\uparrow}(x-w)|w\in V\}$ can indeed be solved for arbitrary subspaces $V<V_{\uparrow}$ and $x\notin V$, note that our assumption that $\Gamma$ is discrete implies via Remark \ref{equalcase} that the set $\{\rho_{\uparrow}(x-w)|w\in V\}$ is discrete, so its infimum will be attained by some $v\in V$ provided that it is bounded below.  If $\Gamma=\{0\}$ this is obvious since $\rho_{\uparrow}$ only takes finitely many values by Remark \ref{equalcase}. If $\Gamma$ is nontrivial, then \cite[I.2.3.Corollary 1]{Bou} applies to show that $V$ is closed with respect to the topology induced by the norm $e^{\rho_{\uparrow}}$ on $V_{\uparrow}$, so that for $x\notin V$ the set $\{\rho_{\uparrow}(x-w)|w\in V\}$ is again bounded below.
\end{proof}

On the other hand, if $\Gamma$ is not discrete then there are normed $\Lambda_{\uparrow}$-spaces $(V_{\uparrow},\rho_{\uparrow})$ over $\Lambda_{\uparrow}$ that are not orthogonalizable.  Indeed, if $V_{\uparrow}$ is orthogonalizable then by \cite[Corollary 2.18]{UZ}, for any $m$-dimensional subspace $V\leq V_{\uparrow}$ one can find an orthogonal basis $\{e_1,\ldots,e_d\}$ for $V_{\uparrow}$ such that $\{e_1,\ldots,e_m\}$ is a (necessarily orthogonal) basis for $V$.  This readily implies that for any $x\in V_{\uparrow}\setminus V$ the ``best approximation problem'' of finding $v\in V$ minimizing the value of $\rho_{\uparrow}(x-v)$ has a solution, as one can take $v$ equal to the projection of $x$ to $V$ given by the orthogonal basis.  But, as noted in \cite{U08}, if $\Gamma$ is not discrete one can adapt a construction of \cite{IH} to give a two-dimensional normed $\Lambda_{\uparrow}$ space with a subspace $V$ and a vector $x$ for which the best approximation problem has no solution.  

To describe such an example in detail, let $\Lambda_{\uparrow}^{\mathsf{w}}$ consist of generalized formal power series $\sum_{g\in \Gamma}a_g T^g$ where $a_g\in \kappa$ and, in place of the finiteness condition defining $\Lambda_{\uparrow}$, we impose the weaker condition that $\{g|a_g\neq 0\}$ is a well-ordered subset of $\R$. Then our non-Archimedean valuation $\nu_{\uparrow}\left(\sum_{g\in G}a_gT^g\right)=\min\{g|a_g\neq 0\}$ extends by the same formula to a function on $\Lambda_{\uparrow}^{\mathsf{w}}$ which in fact makes it into a non-Archimedean field extension of $\Lambda_{\uparrow}$, though we will merely view it as a $\Lambda_{\uparrow}$-vector space.  If we choose a strictly increasing but bounded sequence $\{\ep_i\}_{i=0}^{\infty}$ in $\Gamma$ (as is possible since $\Gamma$ is an indiscrete subgroup of $\R$) and set $v=\sum_{i=0}^{\infty}T^{\ep_i}$, then the $\Lambda_{\uparrow}$-subspace $V_{\uparrow}=\mathrm{span}_{\Lambda_{\uparrow}}\{1,v\}$ of $\Lambda_{\uparrow}^{\mathsf{w}}$ is a normed $\Lambda_{\uparrow}$-space with respect to the function $\ell=-\nu_{\uparrow}$, and the element $v\in V$ has no best approximation in the subspace $\mathrm{span}_{\Lambda_{\uparrow}}\{1\}$.

To provide some context, we now indicate some motivating examples of normed $\Lambda_{\uparrow}$- or $\Lambda^{\downarrow}$-spaces.
For an example with $\Gamma=\{0\}$, one can consider a continuous function $f\co X\to \R$ on a compact topological space $X$, let $V_{\uparrow}=H_k(X;\kappa)$ for some choice of degree $k$, and define $\rho_{\uparrow}^{f}\co V_{\uparrow}\to\R\cup\{-\infty\}$ by \begin{equation}\label{morsespec} \rho_{\uparrow}^f(a)=\inf\{t\in \R|a\in \Img(H_k(f^{-1}((-\infty,t]);\kappa)\to H_k(X;\kappa))\}\end{equation} (where the map is the obvious one induced by the inclusion $f^{-1}((-\infty,t])\to X$).  Thus $\rho_{\uparrow}^{f}(a)$ records the level at which the homology class $a$ first appears in the homologies of sublevel sets of $f$.  If $\{e_1,\ldots,e_m\}$ is a $\rho_{\uparrow}$-orthogonal basis for $V_{\uparrow}$ (as exists automatically by Proposition \ref{discretenormed}), then one can verify that the infinite intervals appearing in the degree-$k$ sublevel persistence barcode of $f$ are $[\rho_{\uparrow}^{f}(e_1),\infty),\ldots,[\rho_{\uparrow}^{f}(e_m),\infty)$.

Dually, an example of an orthogonalizable $\Lambda^{\downarrow}$-space with $\Gamma=\{0\}$ is given by again taking $V^{\downarrow}=H_k(X;\kappa)$ but now putting \[ \rho^{\downarrow}_f(a)=\sup\{t\in \R|a\in \Img(H_k(f^{-1}([t,\infty));\kappa)\to H_k(X;\kappa))\}.\]  Thus $\rho_{\uparrow}^f$ records how the images in $H_k(X;\kappa)$ of the homologies of sublevel sets grow as the level increases, while $\rho^{\downarrow}_f$ records how the images in $H_k(X;\kappa)$  of the homologies of superlevel sets grow as the level decreases.  (This is the motivation for the use of the arrows $\uparrow$ and $\downarrow$ in the notation.)

For general $\Gamma$, the constructions in \cite{UZ} (in particular, that paper's Proposition 6.6) show that, for a compact smooth manifold $X$, a cover $\tilde{X}\to X$ of with deck transformation group $\Gamma$ and a Morse function $\tilde{f}\co \tilde{X}\to \R$ such that $d\tilde{f}$ is the pullback of a closed one-form on $X$ representing the class $\xi\in H^1(X;\R)$, the Novikov homology $\mathbf{HN}_k(\tilde{f};\xi)$ (which is a vector space over $\Lambda_{\uparrow}$) again admits the structure of an orthogonalizable $\Lambda_{\uparrow}$-space, using the function $\rho_{\uparrow}^{\tilde{f}}\co \mathbf{HN}_k(\tilde{f};\xi)\to \R\cup\{-\infty\}$ that sends a class $a$ to the infimal filtration level at which $a$ is represented in the Novikov chain complex. The values $\rho_{\uparrow}^{\tilde{f}}(e_1),\ldots\rho_{\uparrow}^{\tilde{f}}(e_m)$ for an arbitrary $\rho_{\uparrow}^{\tilde{f}}$-orthogonal basis are the left endpoints of the infinite intervals of \cite{UZ}'s version of the sublevel barcode of $\tilde{f}$; the multiset consisting of the equivalence classes of these intervals modulo translation by $\Gamma$ is independent of the choice of $\rho_{\uparrow}^{\tilde{f}}$-orthogonal basis.  

Up to canonical isomorphisms induced by continuation maps such as those recalled in Sections \ref{morseintro} and \ref{novsect}, the $\Lambda_{\uparrow}$-vector space $\mathbf{HN}_k(\tilde{f};\xi)$ depends only on the cohomology class $\xi$ and not on the Morse function $\tilde{f}$ whose derivative is a pullback of a de Rham representative of $\xi$, so (with $\xi$ fixed) we can regard each $\rho_{\uparrow}^{\tilde{f}}$ as being defined on a common vector space $V_{\xi}\cong \mathbf{HN}_k(\tilde{f};\xi)$.  Moreover, based (\ref{filtchange}), it is not hard to see that, within the class of such $\tilde{f}$, the assignment $\tilde{f}\mapsto \rho_{\uparrow}^{\tilde{f}}$ is continuous with respect to the $C^0$ topology on the space of such $\tilde{f}$  and the $L^{\infty}$ norm on maps $V_{\xi}\setminus\{0\}\to\R$.  Hence we can define $\rho^{\tilde{f}}_{\uparrow}$ by continuity even if we drop the Morse condition on $\tilde{f}$.  Without the Morse hypothesis, continuity implies that $(V_{\xi},\rho^{\tilde{f}}_{\uparrow})$ will still be a normed $\Lambda_{\uparrow}$-space.  Hence if $\Gamma$ is discrete it will still be orthogonalizable by Proposition \ref{discretenormed}; if $\Gamma$ is not discrete I do not know whether this conclusion still always holds.

	%

\subsection{Duals}\label{dualsec}

If $M$ is a left module over a not-necessarily-commutative ring $R$, then $M^*:=\mathrm{Hom}_R(M,R)$ is naturally regarded as a \emph{right} $R$-module, with scalar multiplication of an element $\phi\in M^*$ by $r\in R$ being given by, for all $m\in M$, $(\phi r)(m)=\phi(m)r$.  In the case that $R=\Lambda$ (or more generally if $R$ is a group algebra) we can then convert this right module into a left module by conjugation; doing so yields the notion of module duality that is most useful for this paper.

\begin{dfn}
If $M$ is a left (resp. right) $\Lambda$-module we define its dual module ${}^{\vee}\!M$ to be the left (resp. right) $\Lambda$-module \[ {}^{\vee}\!M=\overline{\mathrm{Hom}_{\Lambda}(M,\Lambda)}.\]  Similarly, if $V$ is a $\Lambda_{\uparrow}$-vector space we define ${}^{\vee}\!V$ as the $\Lambda^{\downarrow}$-vector space \[ {}^{\vee}\!V=\overline{\mathrm{Hom}_{\Lambda_{\uparrow}}(V,\Lambda_{\uparrow})}.\]
\end{dfn}

If $A\co M\to N$ is a morphism of $\Lambda$-modules, we obtain an adjoint ${}^{\vee}\!A\co {}^{\vee}\!N\to {}^{\vee}\!M$ by taking the usual transpose $\phi\mapsto \phi\circ A$ from $\mathrm{Hom}_{\Lambda}(N,\Lambda)\to \mathrm{Hom}_{\Lambda}(M,\Lambda)$ and regarding this as a map between the conjugate modules ${}^{\vee}\!N$ and ${}^{\vee}\!M$.  If $M$ and $N$ happen to be free finite-rank (left) modules, with respective $\Lambda$-bases $\{e_1,\ldots,e_m\}$, $\{f_1,\ldots,f_n\}$, then the usual dual bases $\{e_{1}^{*},\ldots,e_{m}^{*}\}$,$\{f_{1}^{*},\ldots,f_{n}^{*}\}$ for the right modules $\mathrm{Hom}_{\Lambda}(M,\Lambda)$ and $\mathrm{Hom}_{\Lambda}(N,\Lambda)$ serve just as well as bases for the left modules ${}^{\vee}\!M$ and ${}^{\vee}\!N$; the conjugation however has the effect that coefficients in representations with respect to these bases are conjugated and hence the matrix representing ${}^{\vee}\!A$ with respect to the dual bases is the \emph{conjugate} transpose  of the matrix representing $A$ with respect to the original bases.

One of course has ${}^{\vee}\!(B\circ A)={}^{\vee}\!A\circ {}^{\vee}\!B$ for composable morphisms $A$ and $B$; in particular, if $A$ is invertible then so too is ${}^{\vee}\!A$, with $({}^{\vee}\!A)^{-1}={}^{\vee}\!(A^{-1})$.  Consequently we may write ${}^{\vee}\!A^{-1}$ without fear of confusion about order of operations.

Similarly, if $A\co V\to W$ is a morphism of $\Lambda_{\uparrow}$-vector spaces, we obtain an adjoint morphism of $\Lambda^{\downarrow}$-vector spaces ${}^{\vee}\!A\co {}^{\vee}\!W\to {}^{\vee}\!V$ by applying the conjugation functor to the transpose $A^*\co \mathrm{Hom}_{\Lambda_{\uparrow}}(W,\Lambda_{\uparrow})\to \mathrm{Hom}_{\Lambda_{\uparrow}}(W,\Lambda_{\uparrow})$.  Note that, although ${}^{\vee}\!A$ and $A^*$ are set-theoretically the same function, ${}^{\vee}\!A$ is a morphism of $\Lambda^{\downarrow}$-vector spaces while $A^{*}$ is a morphism of $\Lambda_{\uparrow}$-vector spaces. 

Of course, if instead $B\co X\to Y$ is a morphism of $\Lambda^{\downarrow}$-vector spaces, the same construction as in the previous paragraph gives an adjoint morphism ${}^{\vee}\!B\co {}^{\vee}\!Y\to {}^{\vee}\!X$ of $\Lambda_{\uparrow}$-vector spaces.

If $V$ is a $\Lambda_{\uparrow}$-vector space and $v\in V$ then we obtain a $\Lambda^{\downarrow}$-linear map $\iota_v\co {}^{\vee}V\to \Lambda^{\downarrow}$ defined by $\iota_v(\phi)=\overline{\phi(v)}$ for $\phi\in {}^{\vee}V$.  Now ${}^{\vee}\!({}^{\vee}\!V)=\overline{\mathrm{Hom}_{\Lambda^{\downarrow}}({}^{\vee}\!V,\Lambda^{\downarrow})}$, so $\iota_v$ can be regarded as an element of the $\Lambda_{\uparrow}$-vector space ${}^{\vee}\!({}^{\vee}\!V)$, and the map $\alpha_V:v\mapsto\iota_v$ is a $\Lambda_{\uparrow}$-linear map $V\to {}^{\vee}\!({}^{\vee}\!V)$.  This map $\alpha_V$ is obviously injective, so if $V$ is finite-dimensional $\alpha_V$ must be a $\Lambda_{\uparrow}$-vector space isomorphism $V\cong {}^{\vee}\!({}^{\vee}\!V)$ by dimensional considerations.  Indeed, if $\{e_1,\ldots,e_n\}$ is a basis for $V$ with dual basis $\{e_{1}^{*},\ldots,e_{n}^{*}\}$ for ${}^{\vee}\!V$, then the dual basis for ${}^{\vee}\!({}^{\vee}\!V)$ of  $\{e_{1}^{*},\ldots,e_{n}^{*}\}$ is evidently $\{\alpha_V(e_1),\ldots,\alpha_V(e_n)\}$.  

If $A\co V\to W$ is a morphism of finite-dimensional $\Lambda_{\uparrow}$-vector spaces, we obtain a double-adjoint ${}^{\vee}\!({}^{\vee}\!A)\co {}^{\vee}\!({}^{\vee}\!V)\to {}^{\vee}\!({}^{\vee}\!W)$, and it is straightforward to check that ${}^{\vee}\!({}^{\vee}\!A)$ coincides with $A$ under the isomorphisms $\alpha_V\co V\cong{}^{\vee}\!({}^{\vee}\!V)$ and $\alpha_W\co W\cong {}^{\vee}\!({}^{\vee}\!V)$.

By the same token, if $H$ is a $\Lambda$-module, we obtain a $\Lambda$-module homomorphism $\alpha_H\co H\to {}^{\vee}\!({}^{\vee}\!H)$ by setting $\alpha_H(h)$ equal to the map ${}^{\vee}\!H\to \Lambda$ defined by $\phi\mapsto \overline{\phi(h)}$.  In contrast to the situation with finite-dimensional vector spaces over $\Lambda_{\uparrow}$ or $\Lambda^{\downarrow}$, the map $\alpha_H$ need not be an isomorphism for general finitely-generated modules $H$ over $\Lambda$; for instance, if $h$ is a torsion element of $H$ then $\alpha_H(h)$ will be zero.  If $A\co H_0\to H_1$ is a $\Lambda$-module homomorphism we have a commutative diagram \begin{equation}\label{ddnat} \xymatrix{ H_0\ar[r]^{A}\ar[d]_{\alpha_{H_0}} & H_1\ar[d]^{\alpha_{H_1}} \\ {}^{\vee}\!({}^{\vee}\!H_0)\ar[r]_{{}^{\vee}\!{}^{\vee}\!A} & {}^{\vee}\!({}^{\vee}\!H_1)},\end{equation} as both compositions send $x\in H_0$ to the map on ${}^{\vee}\!H_1$ given by $\psi\mapsto \overline{\psi(Ax)}$.



\subsection{PD structures}\label{pdsec}

One ingredient in our constructions is the following pair of definitions, which are designed to mimic maps obtained from Poincar\'e duality in covering spaces.  

\begin{dfn}\label{pdstr}
Let $H_*=\oplus_{k\in \Z}H_k$ be a graded left $\Lambda$-module such that each $H_k$ is finitely generated over $\Lambda$, and let $n\in \Z$.  A \textbf{weak $n$-PD structure} on $H_*$ consists of module homomorphisms $\mathcal{D}_k\co H_k\to {}^{\vee}\!H_{n-k}$ for all $k\in \Z$ obeying the symmetry condition \begin{equation}\label{pdsym}  (\mathcal{D}_kx)(y)=\pm \overline{(\mathcal{D}_{n-k}y)(x)} \quad\mbox{for all }x\in H_k,y\in H_{n-k} \end{equation} for some sign $\pm$ that may depend on $k$, such that  $1\otimes \mathcal{D}_k\co \mathcal{Q}(\Lambda)\otimes_{\Lambda}H_k\to \mathcal{Q}(\Lambda)\otimes_{\Lambda}{}^{\vee}\!H_{n-k}$ is an isomorphism where $\mathcal{Q}(\Lambda)$ is the fraction field of $\Lambda$.  A \textbf{strong $n$-PD structure} on $H_*$ consists of module homomorphisms $\mathcal{D}_k\co H_k\to {}^{\vee}\!H_{n-k}$ obeying (\ref{pdsym}) which are surjective and have kernel equal to the $\Lambda$-torsion submodule $tH_k$ of $H_k$.
\end{dfn}

\begin{remark}\label{trivpdstr}
 If $\Gamma=\{0\}$ and hence $\Lambda$ is equal to the field $\kappa$, then there is no $\Lambda$-torsion and $\mathcal{Q}(\Lambda)=\Lambda$, so either a weak or a strong $n$-PD structure just amounts to isomorphisms $H_k\cong {}^{\vee}\!H_{n-k}$ obeying (\ref{pdsym}).  \end{remark}

For general $\Gamma$, from the facts (\cite[Proposition 4.78 and Theorem 4.80]{Rot}) that $\mathcal{Q}(\Lambda)$ is flat over $\Lambda$ and that the coefficient extension map $H_k\to \mathcal{Q}(\Lambda)\otimes_{\Lambda}H_k$ has kernel equal to $tH_k$, one easily sees that a strong $n$-PD structure is also a weak $n$-PD structure. For a counterexample to the converse if $\Gamma\neq\{0\}$ (and hence $\Lambda$ is not a field), one could start with a strong $n$-PD structure and then multiply all $\mathcal{D}_k$ by some nonzero, non-invertible, self-conjugate element of $\Lambda$ (such as $T^g+T^{-g}$ where $g\in\Gamma\setminus\{0\}$); the resulting map would no longer be surjective but would still become an isomorphism after tensoring with $\mathcal{Q}(\Lambda)$.  We will revisit this case in Example \ref{weakdualstrict}.

The motivating topological model for Definition \ref{pdstr} involves a regular covering space $p\co \tilde{X}\to X$ where $X$ is a smooth closed oriented manifold and the deck transformation group of $p$ is isomorphic to $\Gamma$.  The $\Gamma$ action on $\tilde{X}$ makes the singular chain complex of $\tilde{X}$ into complex of $\Lambda$-modules and for $H_k$ we can take the $k$th homology of this complex.\footnote{$H_k$ is finitely generated because $X$, being a smooth closed manifold, admits a finite cell decomposition which then lifts to a $\Gamma$-equivariant cell decomposition of $\tilde{X}$, and $H_k$ is isomorphic to the $k$th homology of the complex of finitely-generated $\Lambda$-modules associated to this cell decomposition.}  Poincar\'e duality as in \cite[Theorem 4.65]{Ran} then provides an isomorphism $H_k\cong H^{n-k}$ where $H^{n-k}$ is the cohomology of the dual complex to the singular chain complex.  Now Definition \ref{pdstr} does not reference anything that directly plays the role of the cohomology $H^{n-k}$; rather we use ${}^{\vee}\!H_{n-k}$, the dual of the homology, which if $\Gamma\neq\{0\}$ is a different thing.  More specifically, there is an evaluation map $\mathrm{ev}\co H^{n-k}\to{}^{\vee}\!H_{n-k}$ from cohomology to the dual of homology, and the maps $\mathcal{D}_k\co H_k\to {}^{\vee}\!H_{n-k}$ should be understood as the composition of the Poincar\'e duality isomorphism with $\mathrm{ev}$.   The mapping $(x,y)\mapsto (\mathcal{D}_kx)(y)$ is the homology intersection pairing defined in \cite[Definition 4.66]{Ran}; as noted there, one has $(\mathcal{D}_kx)(y)=(-1)^{k(n-k)}\overline{(\mathcal{D}_{n-k}y)(x)}$ so the symmetry condition (\ref{pdsym}) holds.

At this point we encounter the basic trichotomy of algebraic complexity in our analysis, depending on nature of the subgroup $\Gamma$ of $\mathbb{R}$.  If $\Gamma=\{0\}$ then most of the discussion up to this point collapses to something rather more simple: all of $\Lambda,\Lambda_{\uparrow},\Lambda^{\downarrow},\Lambda_{\updownarrow}$ are just equal to the field $\kappa$; ``conjugation'' is the identity map on $\kappa$;  the map $\mathrm{ev}$ of the previous paragraph is an isomorphism; and, as already noted in Remark \ref{trivpdstr}, either notion of a $n$-PD structure in Definition \ref{pdstr} just amounts to isomorphisms $H_k\to {}^{\vee}\!H_{n-k}$ that satisfy (\ref{pdsym}) (and in this case the conjugation symbol in (\ref{pdstr}) may be ignored). 

The next simplest possibility is that $\Gamma$ is a discrete, nontrivial subgroup of $\R$; as is well-known and not hard to show, this implies that $\Gamma$ is isomorphic to $\mathbb{Z}$ and hence that the group ring $\Lambda$ is isomorphic to the Laurent polynomial ring $\kappa[t,t^{-1}]$.  Thus in this case $\Lambda$ is a PID, and so the universal coefficient theorem \cite[Theorem 7.59]{Rot} applies to show that $\mathrm{ev}$ is surjective and to compute its (generally nontrivial) kernel.  

The remaining possibility is that $\Gamma\leq \R$ is a dense subgroup, in which case (since we assume $\Gamma$ to be finitely generated) $\Lambda$ is isomorphic to a multivariable Laurent polynomial ring $\kappa[t_1,t_{1}^{-1},\ldots,t_d,t_{d}^{-1}]$ with $d>1$, which is not a PID.  Finitely generated modules over $\Lambda$ can therefore be rather complicated, and moreover the usual universal coefficient theorem does not apply.  Instead, there is a universal coefficient spectral sequence described in \cite[Section 2]{Lev}, and if some differential on an $E^r$ page of this spectral sequence with $r\geq 2$ does not vanish on the bottom row then the map $\mathrm{ev}$ will not be surjective.

Our definition of a weak $n$-PD structure is designed to be flexible enough to apply in relevant situations even if $\Gamma$ is a dense subgroup of $\mathbb{R}$, while our definition of a strong $n$-PD structure is meant to allow for sharper statements such as Corollary \ref{pndual} that hold when $\Gamma$ is discrete.  To see that such structures do indeed arise from Poincar\'e duality, we prove:

\begin{prop}\label{cohtopdstr}
Let $C_*=\oplus_{k\in \Z}C_k$ be a graded left $\Lambda$-module with each $C_k$ free, let $\partial\co C_{*}\to C_{*-1}$ obey $\partial\circ\partial=0$, and assume that each homology module $H_k=\frac{\ker(\partial\co C_k\to C_{k-1})}{\Img(\partial\co C_{k+1}\to C_k)}$ is finitely generated as a $\Lambda$-module.  Write $C^k={}^{\vee}\!(C_k)$, and $\delta\co C^*\to C^{*+1}$ for the map given on $C^k$ by ${}^{\vee}\!((-1)^{k}\partial|_{C_{k+1}})$.  Assume that for each $k$ we have an isomorphism $D_k\co H_k\to H^{n-k}$ where $H^{n-k}=\frac{\ker(\delta\co C^k\to C^{k+1})}{\Img(\delta\co C^{k-1}\to C^k)}$, and let $\mathrm{ev}\co H^{n-k}\to {}^{\vee}\!H_{n-k}$ be the evaluation map: $\mathrm{ev}([\phi])([c])=\phi(c)$ whenever $\phi\in C^{n-k}$ has $\delta\phi=0$ and $c\in C_{n-k}$ has $\partial c=0$.  Assume moreover that the maps $\mathcal{D}_k=\mathrm{ev}\circ D_k$ obey (\ref{pdsym}). Then:
\begin{itemize}\item[(i)] If $\Gamma$ is discrete then the maps $\mathcal{D}_k=\mathrm{ev}\circ D_k$ define a strong $n$-PD structure.
\item[(ii)] Regardless of whether $\Gamma$ is discrete, if each $C_k$ is finitely generated then the maps $\mathcal{D}_k=\mathrm{ev}\circ D_k$ define a weak $n$-PD structure.
\end{itemize}\end{prop}

\begin{proof}[Proof of Proposition \ref{cohtopdstr}(i)]  
Since $\Gamma$ is discrete and so $\Lambda$ is a PID the universal coefficient theorem gives (after applying the conjugation functor) a short exact sequence \[ \xymatrix{0\ar[r] &\overline{\mathrm{Ext}^{1}_{\Lambda}(H_{n-k-1},\Lambda)}\ar[r]^{\,\,\,\,\,\,\quad j} & H^{n-k}\ar[r]^{\mathrm{ev}} & {}^{\vee}\!H_{n-k}\ar[r] & 0 }.\]  Since we assume that $D_k\co H_k\to H^{n-k}$ is an isomorphism it follows that $\mathcal{D}_k=\mathrm{ev}\circ D_k$ is a surjection with kernel equal to $D_{k}^{-1}(\Img j)$. Since ${}^{\vee}\!H_{n-k}=\overline{\mathrm{Hom}_{\Lambda}(H_{n-k},\Lambda)}$ is torsion-free this kernel must contain all $\Lambda$-torsion elements of $H_k$.  Conversely, all elements of $\ker\mathcal{D}_k=D_{k}^{-1}(\mathrm{Im}j)$ are torsion since $\mathrm{Ext}^{1}_{\Lambda}(H_{n-k-1},\Lambda)$ is a torsion $\Lambda$-module (as can be seen by direct calculation from the classification of finitely generated modules over a PID, noting that $H_{n-k-1}$ is assumed to be finitely generated).
\end{proof}

For the proof of the second part of the proposition it is helpful to have the following lemma.
  
\begin{lemma}\label{extend-to-frac}
Let $A$ be an integral domain with fraction field $Q$, and let $M$ be a finitely generated $A$-module.  Then the $Q$-vector space homomorphism \[ \alpha\co \mathrm{Hom}_{A}(M,A)\otimes_A Q\to \mathrm{Hom}_A(M,Q),\] defined on simple tensors $\phi\otimes q$ by $\left(\alpha(\phi\otimes q)\right)(m)=q\phi(m)$ for all $m\in M$, is an isomorphism.
\end{lemma}

\begin{proof}
First note that, because $M$ is finitely-generated, for each $\phi\in \mathrm{Hom}_A(M,Q)$, there is $a\in A$ such that $a\phi\in \mathrm{Hom}_A(M,A)$.  Indeed, if $\{x_1,\ldots,x_m\}$ is a generating set for $M$ with $\phi(x_i)=\frac{p_i}{q_i}$ with $p_i,q_i\in A$ and $q_i\neq 0$, we can take $a=\prod_{i=1}^{m}q_i$.  Thus the quotient $\frac{\mathrm{Hom}_A(M,Q)}{\mathrm{Hom}_A(M,A)}$ is a torsion $A$-module, and hence its tensor product with the fraction field $Q$ vanishes.  Now $Q$ is a flat $A$-module, so applying the exact functor $\_\otimes_A Q$ to the short exact sequence \[ 0\to \mathrm{Hom}_A(M,A)\to \mathrm{Hom}_A(M,Q)\to \frac{\mathrm{Hom}_A(M,Q)}{\mathrm{Hom}_A(M,A)}\to 0\] shows that the map \begin{equation}\label{domtofrac}  \mathrm{Hom}_A(M,A)\otimes_A Q\to \mathrm{Hom}_A(M,Q)\otimes_A Q \end{equation} induced by inclusion of $A$ into $Q$ is an isomorphism. But there is an isomorphism of $Q$-vector spaces $\mathrm{Hom}_A(M,Q)\otimes_A Q\to \mathrm{Hom}_A(M,Q)$ given on simple tensors by $\phi\otimes q\mapsto q\phi$ (this follows for instance from the proof of \cite[Corollary 4.79(ii)]{Rot}); composing this with (\ref{domtofrac}) yields the result.
\end{proof}

\begin{proof}[Proof of Proposition \ref{cohtopdstr}(ii)]
Write $(C^*,\delta)$ for the usual (unconjugated) dual complex to $(C_*,\partial)$, so $C^i=\mathrm{Hom}_{\Lambda}(C_i,\Lambda)$ and Lemma \ref{extend-to-frac} gives isomorphisms $C^i\otimes_{\Lambda}\mathcal{Q}(\Lambda)\cong \mathrm{Hom}_{\Lambda}(C_i,\mathcal{Q}(\Lambda))$.  Since $\mathcal{Q}(\Lambda)$ is a flat $\Lambda$-module, the obvious map $H^{n-k}(C^*)\otimes_{\Lambda}\mathcal{Q}(\Lambda)\to H^{n-k}(C^{*}\otimes_{\Lambda}\mathcal{Q}(\Lambda))$ is an isomorphism.   So we have  an isomorphism $H^{n-k}(C^*)\otimes_{\Lambda}\mathcal{Q}(\Lambda)\cong H^{n-k}(\mathrm{Hom}_{\Lambda}(C_*,\mathcal{Q}(\Lambda)))$.  

In turn, since $\mathcal{Q}(\Lambda)$ is an injective $\Lambda$-module, the evaluation map $H^{n-k}(\mathrm{Hom}_{\Lambda}(C_*,\mathcal{Q}(\Lambda)))\to \mathrm{Hom}_{\Lambda}(H_{n-k}(C_*),\mathcal{Q}(\Lambda))$ is an isomorphism.  Finally, another application of Lemma \ref{extend-to-frac}
gives an isomorphism $ \mathrm{Hom}_{\Lambda}(H_{n-k}(C_*),\mathcal{Q}(\Lambda))\cong \mathrm{Hom}_{\Lambda}(H_{n-k}(C_*),\Lambda)\otimes_{\Lambda} \mathcal{Q}(\Lambda)$.

Stringing together all of the above isomorphisms gives an isomorphism $H^{n-k}(C^*)\otimes_{\Lambda}\mathcal{Q}(\Lambda)\to  \mathrm{Hom}_{\Lambda}(H_{n-k}(C_*),\Lambda)\otimes_{\Lambda} \mathcal{Q}(\Lambda)$ which is easily seen to coincide with the coefficient extension of the evaluation map.  Conjugating then gives that $1_{\mathcal{Q}(\Lambda)}\otimes \mathrm{ev}\co \mathcal{Q}(\Lambda)\otimes_{\Lambda}H^{n-k}\to \mathcal{Q}(\Lambda)\otimes_{\Lambda}{}^{\vee}\!H_{n-k}$ is an isomorphism.  Composing this isomorphism with the coefficient extension to $\mathcal{Q}(\Lambda)$ of the isomorphism $D_k\co H_k\to H^{n-k}$ from the statement of the proposition, we see that the coefficient extension to $\mathcal{Q}(\Lambda)$ of $\mathcal{D}_k=\mathrm{ev}\circ D_k$ gives an isomorphism $\mathcal{Q}(\Lambda)\otimes_{\Lambda}H_k\to\mathcal{Q}(\Lambda)\otimes_{\Lambda} {}^{\vee}\!H_{n-k}$, as desired.
\end{proof}

\begin{remark}\label{weirdsign}
	The sign in the definition of the differential $\delta\co C^*\to C^{*+1}$ in Proposition \ref{cohtopdstr} is opposite to the conventional sign on the differential of a dual complex (see, \emph{e.g.}, \cite[Remark VI.10.28]{Dold}); of course this does not affect the (co)homology-level conclusions of Proposition \ref{cohtopdstr} or its proof, but this sign will reappear in Section \ref{cpnstr}.  
	
	We can justify our sign convention in terms of conjugation as follows.  The  sign on the differential on $\mathrm{Hom}(C_{*},\Lambda)$ that is used in \cite{Dold} is defined so as to make the evaluation map $\mathrm{Hom}(C_{-*},\Lambda)\otimes_{\Lambda} C_{*}\to \Lambda$ a chain map, where the codomain is regarded as a chain complex concentrated in degree zero (with trivial differential).  But with $(C_*,\partial)$ a complex of left $\Lambda$-modules, the various $C^k=\overline{\mathrm{Hom}(C_k,\Lambda)}$ are naturally also regarded as left modules, and in place of the usual evaluation map we have an evaluation map $\overline{C_{*}}\otimes_{\Lambda}C^{-*}\to \Lambda$. Our convention in Proposition \ref{cohtopdstr} makes this latter evaluation map---with elements of $C^*$ acting on the right---a chain map.
	
	In truth, this justification is somewhat \emph{post hoc}; my actual motivation for the choice of sign in Proposition \ref{cohtopdstr} (and hence in the definition of a chain-level Poincar\'e--Novikov structure in Section \ref{cpnstr}) is that it works better with the orientation conventions in Morse and Floer theory in Appendix \ref{orsect}.
	\end{remark}

\section{Filtered matched pairs}\label{fmpsec}

We now finally define the basic structures that will ultimately give rise to our version of the open and closed intervals in interlevel barcodes.  Recall the notions of a normed $\Lambda_{\uparrow}$-space and a normed $\Lambda^{\downarrow}$-space from Definition \ref{normdef}.

\begin{dfn}\label{fmpdfn}  Fix as before a field $\kappa$ and a finitely generated subgroup $\Gamma<\R$.
A \textbf{filtered matched pair} is a diagram \[ \xymatrix{ & (V^{\downarrow},\rho^{\downarrow}) \\ H\ar[ru]^{\phi^{\downarrow}}\ar[rd]_{\phi_{\uparrow}} & \\ & (V_{\uparrow},\rho_{\uparrow}) } \] where:
\begin{itemize} \item $H$ is a finitely-generated left $\Lambda$-module;
\item $(V_{\uparrow},\rho_{\uparrow})$ is an normed $\Lambda_{\uparrow}$-space, and the $\Lambda$-module homomorphism $\phi_{\uparrow}\co H\to V_{\uparrow}$ has the property that $1_{\Lambda_{\uparrow}}\otimes\phi_{\uparrow} \co \Lambda_{\uparrow}\otimes_{\Lambda}H\to V_{\uparrow}$ is an isomorphism; and 
\item $(V^{\downarrow},\rho^{\downarrow})$ is an normed $\Lambda^{\downarrow}$-space, and the $\Lambda$-module homomorphism $\phi^{\downarrow}\co H\to V^{\downarrow}$ has the property that $1_{\Lambda^{\downarrow}}\otimes \phi^{\downarrow}\co \Lambda^{\downarrow}\otimes_{\Lambda}H\to V^{\downarrow}$ is an isomorphism.
\end{itemize}\end{dfn}

If $\Gamma=\{0\}$ then $\Lambda_{\uparrow}=\Lambda^{\downarrow}=\Lambda=\kappa$, so $\phi_{\uparrow}$ and $\phi^{\downarrow}$ are already isomorphisms; thus in this case a filtered matched pair provides not much more information than simply a $\kappa$-vector space isomorphism (namely $\phi^{\downarrow}\circ\phi_{\uparrow}^{-1}$) between $V_{\uparrow}$ and $V^{\downarrow}$.  However, if $\Gamma\neq\{0\}$ then $V_{\uparrow}$ and $V^{\downarrow}$ are not even vector spaces over the same field so they certainly cannot be isomorphic.  One could conjugate $V^{\downarrow}$ to obtain a vector space over the same field as $V_{\uparrow}$, but conjugation is incompatible with the $\Gamma$-action which is meant to encode geometrically significant information.  Definition \ref{fmpdfn}  allows for comparison between $V_{\uparrow}$ and $V^{\downarrow}$ by providing a venue, $H$, in which elements of these spaces can sometimes be ``matched'' while simultaneously respecting the $\Gamma$-actions on both spaces.

\begin{remark} With notation as in Definition \ref{fmpdfn}, since $\Lambda_{\uparrow}$ and $\Lambda^{\downarrow}$ are both field extensions of the fraction field $\mathcal{Q}(\Lambda)$, the map $\phi_{\uparrow}$ factors as $H\to \mathcal{Q}(\Lambda)\otimes_{\Lambda}H\to V_{\uparrow}$ where the first map is the coefficient extension map and the second is the restriction of $1_{\Lambda_{\uparrow}}\otimes \phi_{\uparrow}$ and in particular is injective.  A similar remark applies to $\phi^{\downarrow}$.  So since the coefficient extension map $H\to \mathcal{Q}(\Lambda)\otimes_{\Lambda}H$ has kernel equal to the $\Lambda$-torsion submodule $tH$ of $H$, it follows that $\ker\phi_{\uparrow}=\ker\phi^{\downarrow}=tH$ and that $\phi_{\uparrow}$ (resp. $\phi^{\downarrow}$) descends to an injection $\frac{H}{tH}\to \Lambda_{\uparrow}$ (resp. $\frac{H}{tH}\to \Lambda^{\downarrow}$) which becomes an isomorphism after tensoring with $\Lambda_{\uparrow}$ (resp. $\Lambda^{\downarrow}$).  Moreover the dimensions of $V_{\uparrow}$ and $V^{\downarrow}$ are both equal to $\dim_{\mathcal{Q}(\Lambda)}\mathcal{Q}(\Lambda)\otimes_{\Lambda}H$.
\end{remark}

Let us now extract quantitative data from a filtered matched pair $\mathcal{P}=(H,V_{\uparrow},\rho_{\uparrow},\phi_{\uparrow},V^{\downarrow},\rho^{\downarrow},\phi^{\downarrow})$.  In favorable situations these data will emerge from finding a special kind of basis for the torsion-free module $\frac{H}{tH}$.  In our current level of generality such a special basis may not exist---indeed if $\Gamma$ is a dense subgroup of $\R$ so that $\Lambda$ is not a PID, then $\frac{H}{tH}$ may well have no basis at all---but we can still make the following definition:

\begin{dfn}\label{gapdfn}
For a filtered matched pair $\mathcal{P}$ as above, let $d=\dim_{\mathcal{Q}(\Lambda)}\mathcal{Q}(\Lambda)\otimes_{\Lambda}H$ and let $i$ be an integer with $1\leq i\leq d$.  We define the \textbf{$i$th gap} of $\mathcal{P}$ as \[ G_i(\mathcal{P})=\sup\left\{\gamma\in \R\left|\begin{array}{c}(\exists x_1,\ldots,x_i\in H \mbox{ independent})\\ (\rho^{\downarrow}(\phi^{\downarrow}x_j)-\rho_{\uparrow}(\phi_{\uparrow}x_j)\geq \gamma\mbox{ for all }j=1,\ldots,i)\end{array}\right.\right\}.\]
\end{dfn}

(Here a collection of elements of a module $H$ over an integral domain is said to be independent if the elements become linearly independent after tensoring with the fraction field of the domain.)

For filtered matched pairs obtained from Poincar\'e--Novikov structures via the construction in Section \ref{pnsect}, the $G_i(\mathcal{P})$ which are positive will appear as lengths of the closed intervals in the associated barcode, while the absolute values of the negative $G_i(\mathcal{P})$ will appear as lengths of open intervals.  (The latter will be shifted down by $1$ in homological degree.)

One might hope to have $G_i(\mathcal{P})=\rho^{\downarrow}(\phi^{\downarrow}x_i)-\rho_{\uparrow}(\phi_{\uparrow}x_i)$ for some collection of elements $\{x_1,\ldots,x_d\}\subset H$.  We will see that this holds if the $x_i$ form a \emph{doubly-orthogonal basis} in the following sense:

\begin{dfn}\label{dodfn}
For a filtered matched pair $\mathcal{P}$ as above, a \textbf{doubly-orthogonal basis} is a subset $\{e_1,\ldots,e_d\}\subset H$ such that:
\begin{itemize}\item[(i)] $\{e_1,\ldots,e_d\}$ projects under the quotient map $H\to \frac{H}{tH}$ to a basis for the $\Lambda$-module $\frac{H}{tH}$, where $tH$ is the $\Lambda$-torsion module of $H$;
\item[(ii)] $\{\phi_{\uparrow}(e_1),\ldots,\phi_{\uparrow}(e_d)\}$ is a $\rho_{\uparrow}$-orthogonal basis for $V_{\uparrow}$; and
\item[(iii)] $\{\phi^{\downarrow}(e_1),\ldots,\phi^{\downarrow}(e_d)\}$ is a $\rho^{\downarrow}$-orthogonal basis for $V^{\downarrow}$.
\end{itemize}
\end{dfn}

We will prove the following in Section \ref{basisconstruct}.

\begin{theorem}\label{basistheorem} Suppose that the subgroup $\Gamma<\R$ is discrete.  Then any filtered matched pair has a doubly-orthogonal basis.
\end{theorem}

As alluded to earlier, if $\Gamma$ is not discrete then doubly-orthogonal bases might not exist for the trivial reason that $\frac{H}{tH}$ might not be a free module.  Also, as noted after Proposition \ref{discretenormed}, if $\Gamma$ is not discrete then $(V_{\uparrow},\rho_{\uparrow})$ and $(V^{\downarrow},\rho^{\downarrow})$ might not be orthogonalizable.  I do not know whether, in the indiscrete case, filtered matched pairs in which $\frac{H}{tH}$ is free and $(V_{\uparrow},\rho_{\uparrow})$ and $(V^{\downarrow},\rho^{\downarrow})$ are orthogonalizable always admit 
doubly-orthogonal bases.  While the individual steps of the algorithm implicitly described in the proof of Theorem \ref{basistheorem} can be implemented for any $\Gamma$, the proof that the algorithm terminates uses the assumption that $\Gamma$ is discrete.

Doubly-orthogonal bases are related to the gaps of Definition \ref{gapdfn} as follows:

\begin{prop}\label{gapchar}
Let $\{e_1,\ldots,e_d\}$ be a doubly-orthogonal basis for a filtered matched pair $\mathcal{P}$, ordered in such a way that $\rho^{\downarrow}(\phi^{\downarrow}(e_i))-\rho_{\uparrow}(\phi_{\uparrow}(e_i))\geq \rho^{\downarrow}(\phi^{\downarrow}(e_{i+1}))-\rho_{\uparrow}(\phi_{\uparrow}(e_{i+1}))$ for each $i\in \{1,\ldots,d-1\}$.  Then, for each $i$, \[ G_i(\mathcal{P})=\rho^{\downarrow}(\phi^{\downarrow}(e_i))-\rho_{\uparrow}(\phi_{\uparrow}(e_i)).\]
\end{prop}

\begin{proof}  That $G_i(\mathcal{P})\geq \rho^{\downarrow}(\phi^{\downarrow}(e_i))-\rho_{\uparrow}(\phi_{\uparrow}(e_i))$ follows immediately from the definition (take $x_j=e_j$). To prove the reverse inequality we must show that, whenever $\{x_1,\ldots,x_i\}$ is an independent set in $H$, there is some $j\leq i$ such that $\rho^{\downarrow}(\phi^{\downarrow}(x_j))-\rho_{\uparrow}(\phi_{\uparrow}(x_j))\leq \rho^{\downarrow}(\phi^{\downarrow}(e_i))-\rho_{\uparrow}(\phi_{\uparrow}(e_i))$.  

To see this, note that since $\{x_1,\ldots,x_i\}$ is an independent set, it is not contained in the $\Lambda$-span of $\{e_i,\ldots,e_{i-1}\}$.  Hence there are $j\leq i$ and $k\geq i$ such that, expressing $x_j$ in terms of our basis as $x_j=t+\sum_{\ell}\lambda_{\ell}e_{\ell}$ where $t\in tH$, the coefficient $\lambda_k\in \Lambda$ is nonzero.  We hence have, by the double-orthogonality of $\{e_1,\ldots,e_d\}$ and the fact that $\phi_{\uparrow}$ and $\phi^{\downarrow}$ vanish on $tH$, \begin{align*} \rho^{\downarrow}(\phi^{\downarrow}(x_j))-\rho_{\uparrow}(\phi_{\uparrow}(x_j))& \leq \left(\rho^{\downarrow}(\phi^{\downarrow}(e_k))-\nu^{\downarrow}(\lambda_k)\right)-\left(\rho_{\uparrow}(\phi_{\uparrow}(e_k))-\nu_{\uparrow}(\lambda_k\right) 
\\ &\leq \left(\rho^{\downarrow}(\phi^{\downarrow}(e_i))-\rho_{\uparrow}(\phi_{\uparrow}(e_i))\right)-\left(\nu^{\downarrow}(\lambda_k)-\nu_{\uparrow}(\lambda_k)\right) 
\\ &\leq \left(\rho^{\downarrow}(\phi^{\downarrow}(e_i))-\rho_{\uparrow}(\phi_{\uparrow}(e_i))\right),\end{align*}
where the second inequality uses that $k\geq i$ and the last inequality uses that, as is clear from (\ref{nuformula}), $\nu^{\downarrow}(\lambda)-\nu_{\uparrow}(\lambda)\geq 0$ for all nonzero elements $\lambda$ of $\Lambda$.
\end{proof}

It follows in particular that the multiset $\{\rho^{\downarrow}(\phi^{\downarrow}(e_i))-\rho_{\uparrow}(\phi_{\uparrow}(e_i))\}$ is the same for any choice of doubly-orthogonal basis.  Note that for this to hold it is essential that, in Definition \ref{dodfn}, we require that $\{e_1,\ldots,e_d\}$ project to a \textbf{basis} for the module $\frac{H}{tH}$ and not merely a maximal independent set (assuming that $\Gamma\neq \{0\}$ so that there is a distinction between these notions).  Indeed, if we start with a doubly-orthogonal basis $\{e_1,\ldots,e_d\}$ and then multiply, say, $e_1$ by a nonzero, non-invertible element $\lambda\in \Lambda$, the new set $\{\lambda e_1,e_2,\ldots,e_d\}$ would still have the property that its images under $\phi_{\uparrow}$ and $\phi^{\downarrow}$ are orthogonal bases of $V_{\uparrow}$ and $V^{\downarrow}$ respectively, but since $\nu^{\downarrow}(\lambda)-\nu_{\uparrow}(\lambda)$ is nonzero for non-invertible (\emph{i.e.}, non-monomial) elements of $\Lambda$ this operation would change the value of $\rho^{\downarrow}(\phi^{\downarrow}(e_1))-\rho_{\uparrow}(\phi_{\uparrow}(e_1))$.

The absolute values of the $\rho^{\downarrow}(\phi^{\downarrow}(e_i))-\rho_{\uparrow}(\phi_{\uparrow}(e_i))$ for a doubly-orthogonal basis $\{e_i\}$ are meant to correspond to the lengths of bars in a barcode; it is then natural to expect that the endpoints of these bars would be $ \rho_{\uparrow}(\phi_{\uparrow}(e_i))$ and $\rho^{\downarrow}(\phi^{\downarrow}(e_i))$.  As in \cite{UZ}, in the presence of the deck transformation group $\Gamma$ we would expect these bars to be well-defined only up to translation by the group $\Gamma$; this corresponds to the fact that multiplying each $e_i$ by the invertible element $T^{g_i}$ of $\Lambda$ for some $g_i\in \Gamma$ preserves the double-orthogonality of the basis.  Modulo this ambiguity, we will show in Proposition \ref{barinvt} that different choices of doubly-orthogonal bases yield the same collection of ``bars.''

To set up some notation, given a filtered matched pair $\mathcal{P}$ as above, define functions $\ell_{\uparrow}\co \frac{H}{tH}\to\R\cup\{-\infty\}$ and $\ell^{\downarrow}\co \frac{H}{tH}\to \R\cup\{\infty\}$ by, for a coset $[x]$ in $\frac{H}{tH}$ of $x\in H$, \[ \ell_{\uparrow}([x])=\rho_{\uparrow}(\phi_{\uparrow}(x)),\qquad \ell^{\downarrow}([x])=\rho^{\downarrow}(\phi^{\downarrow}(x)).\]  Also, if $a\in \R$, put \begin{align}\label{hupdown} H_{\uparrow}^{\leq a}&=\left\{\left.[x]\in \frac{H}{tH}\right|\ell_{\uparrow}([x])\leq a\right\},\qquad H_{\uparrow}^{< a}=\left\{\left.[x]\in \frac{H}{tH}\right|\ell_{\uparrow}([x])< a\right\},\\ H^{\downarrow}_{\geq a} &=\left\{\left.[x]\in \frac{H}{tH}\right|\ell^{\downarrow}([x])\geq a\right\}, \qquad H^{\downarrow}_{> a} =\left\{\left.[x]\in \frac{H}{tH}\right|\ell^{\downarrow}([x])> a\right\}.\nonumber\end{align}  It follows from Definition \ref{normdef} that each of $H_{\uparrow}^{\leq a},H_{\uparrow}^{>a},H^{\downarrow}_{\geq a},H^{\downarrow}_{>a}$ is a $\kappa$-vector subspace of $\frac{H}{tH}$. 

\begin{prop}\label{barinvt}
With notation as above, if $\{e_1,\ldots,e_d\}\subset H$ is a doubly-orthogonal basis for the filtered matched pair $\mathcal{P}$ and if $a,b\in \R$, then \begin{align}\label{invteqn} &\#\{i|(\exists g\in \Gamma)(\rho_{\uparrow}(\phi_{\uparrow}(e_i))=a+g\mbox{ and }\rho^{\downarrow}(\phi^{\downarrow}(e_i))=b+g)\} \\ &\qquad = \dim_{\kappa}\left(\frac{H_{\uparrow}^{\leq a}\cap H^{\downarrow}_{\geq b}}{(H_{\uparrow}^{<a}\cap H^{\downarrow}_{\geq b})+(H_{\uparrow}^{\leq a}\cap H^{\downarrow}_{>b})}\right).\nonumber\end{align}
\end{prop}

\begin{proof}
An element of $z\in \frac{H}{tH}$ belongs to $H_{\uparrow}^{\leq a}\cap H^{\downarrow}_{\geq b}$ if and only if  its coordinate expression $z=\sum_i\lambda_i[e_i]$ in terms of the basis $[e_1],\ldots,[e_d]$ has, for all $i$, both $\rho_{\uparrow}(\phi_{\uparrow}(e_i))-\nu_{\uparrow}(\lambda_i)\leq a$ and $\rho^{\downarrow}(\phi^{\downarrow}(e_i))-\nu^{\downarrow}(\lambda_i)\geq b$, \emph{i.e.} both $\nu_{\uparrow}(\lambda_i)\geq \rho_{\uparrow}(\phi_{\uparrow}(e_i))-a$ and $\nu^{\downarrow}(\lambda_i)\leq  \rho^{\downarrow}(\phi^{\downarrow}(e_i))-b$.  Expressing the coefficients $\lambda_i\in \Lambda$ as finite sums $\lambda_i=\sum_{g\in \Gamma}k_{g,i}T^{g}$ where $k_{g,i}\in \kappa$, this shows that $z=\sum_{i=1}^{d}\sum_{g\in\Gamma}  k_{g,i}T^g[e_i]$ belongs to $H_{\uparrow}^{\leq a}\cap H^{\downarrow}_{\geq b}$ iff $k_{g,i}$ vanishes for all $g,i$ other than those for which $g$ lies in the interval $[\rho_{\uparrow}(\phi_{\uparrow}(e_i))-a,\rho^{\downarrow}(\phi^{\downarrow}(e_i))-b]$.  

Similarly, $z=\sum_{i,g} k_{g,i}T^g[e_i]$ belongs to $H_{\uparrow}^{<a}\cap H^{\downarrow}_{\geq b}$ iff $k_{g,i}=0$ whenever $g\notin (\rho_{\uparrow}(\phi_{\uparrow}(e_i))-a,\rho^{\downarrow}(\phi^{\downarrow}(e_i))-b]$, and belongs to $H_{\uparrow}^{\leq a}\cap H^{\downarrow}_{>b}$ iff $k_{g,i}=0$ whenever $g\notin [\rho_{\uparrow}(\phi_{\uparrow}(e_i))-a,\rho^{\downarrow}(\phi^{\downarrow}(e_i))-b)$.  So a vector space complement to $(H_{\uparrow}^{<a}\cap H^{\downarrow}_{\geq b})+(H_{\uparrow}^{\leq a}\cap H^{\downarrow}_{>b})$ in $H_{\uparrow}^{\leq a}\cap H^{\downarrow}_{\geq b}$ is given by the space of all classes $\sum_{i,g} k_{g,i}T^g[e_i]$ having the property that the only $k_{g,i}$ which might be nonzero are those for which $g$ belongs to $[\rho_{\uparrow}(\phi_{\uparrow}(e_i))-a,\rho^{\downarrow}(\phi^{\downarrow}(e_i))-b]$ but belongs to neither $ (\rho_{\uparrow}(\phi_{\uparrow}(e_i))-a,\rho^{\downarrow}(\phi^{\downarrow}(e_i))-b]$ nor  $[\rho_{\uparrow}(\phi_{\uparrow}(e_i))-a,\rho^{\downarrow}(\phi^{\downarrow}(e_i))-b)$.  In other words, $k_{g,i}=0$ unless $g=\rho_{\uparrow}(\phi_{\uparrow}(e_i))-a=\rho^{\downarrow}(\phi^{\downarrow}(e_i))-b$.  So if $I_0$ is the subset of $\{1,\ldots,d\}$ appearing on the left in (\ref{invteqn}), then $\frac{H_{\uparrow}^{\leq a}\cap H^{\downarrow}_{\geq b}}{(H_{\uparrow}^{<a}\cap H^{\downarrow}_{\geq b})+(H_{\uparrow}^{\leq a}\cap H^{\downarrow}_{>b})}$ is isomorphic as a $\kappa$-vector space to the $\kappa$-span of $\{T^{\rho_{\uparrow}(\phi_{\uparrow}(e_i))-a}[e_i]|i\in I_0\}$.
\end{proof}

We accordingly make the following definition:

\begin{dfn}\label{hbsdef}
Suppose that $\mathcal{P}$ is a filtered matched pair which admits a doubly-orthogonal basis.  The \textbf{basis spectrum} $\Sigma(\mathcal{P})$ of $\mathcal{P}$ is the sub-multiset of $(\R/\Gamma)\times \R$ given by \[ \left\{\left([\rho_{\uparrow}(\phi_{\uparrow}(e_i))],\rho^{\downarrow}(\phi^{\downarrow}(e_i))-\rho_{\uparrow}(\phi_{\uparrow}(e_i))\right)|i\in\{1,\ldots,d\}\right\} \] for one and hence (by Proposition \ref{barinvt}) every doubly-orthogonal basis $\{e_1,\ldots,e_d\}$ of $\mathcal{P}$.
\end{dfn}

Thus Proposition \ref{barinvt} shows that, for $([a],\ell)\in (\R/\Gamma)\times \R$ with $[a]$ denoting the coset in $\R/\Gamma$ of the real number $a$, the multiplicity of $([a],\ell)$ in $\Sigma(\mathcal{P})$ equals $\dim_{\kappa}\left(\frac{H_{\uparrow}^{\leq a}\cap H^{\downarrow}_{\geq b}}{(H_{\uparrow}^{<a}\cap H^{\downarrow}_{\geq b})+(H_{\uparrow}^{\leq a}\cap H^{\downarrow}_{>b})}\right)$.  This characterization is similar to the definition of the quantity denoted $\delta_r^f(a,b)$ in \cite[Section 3]{B1}.

As was mentioned in Section \ref{algintro}, in the motivating topological example an element $([a],\ell)$ of $\Sigma(\mathcal{P})$ corresponds to a bar in the interlevel persistence barcode of form $[a,a+\ell]$ if $\ell\geq 0$ and $(a+\ell,a)$ if $\ell<0$ (with homological degree shifted down by $1$ in the latter case).  When $\Gamma$ is nontrivial, as is reflected in Proposition \ref{barinvt}, bars can be shifted by elements of $\Gamma$; as in \cite{UZ} our notation accounts for this ambiguity by having the first entry of an element of $\Sigma(\mathcal{P})$ only be defined modulo $\Gamma$ while the second (representing the difference of the endpoints) is a real number.  (A difference with \cite{UZ} is that here the second entry is allowed to be negative, reflecting that these ``homologically essential'' bars reflect invariant topological features and so do not disappear when their ``lengths'' go to zero; rather, they change into different types of bars.)

\subsection{Stability}\label{stabsec}

We now discuss the algebraic counterpart to the notion that a $C^0$-small change in a function should lead to a correspondingly small change in its interlevel barcode.  The $\Lambda$-module $H$ in the definition of a filtered matched pair is meant to correspond to (a graded piece of) the homology of the cover on which our function is defined, and thus does not change with the function.  Since the vector spaces $V_{\uparrow}$ and $V^{\downarrow}$ are isomorphic to $\Lambda_{\uparrow}\otimes_{\Lambda}H$ and $\Lambda^{\downarrow}\otimes_{\Lambda}H$, respectively, their vector space isomorphism types do not change either, but the filtration functions $\rho_{\uparrow}$ and $\rho^{\downarrow}$ will be sensitive to the function.

\begin{dfn}\label{tmorphdfn}
Let $t\geq 0$ and let 
\[
\xymatrix{ & & (V^{\downarrow},\rho^{\downarrow}) & & & &  (\hat{V}^{\downarrow},\hat{\rho}^{\downarrow})  \\ \mathcal{P}  \ar@{}[r]^(.35){}="a"^(.65){}="b" \ar@{=} "a";"b"    & H\ar[ru]^{\phi^{\downarrow}}\ar[rd]_{\phi_{\uparrow}} & & \mbox{and} & \hat{\mathcal{P}}  \ar@{}[r]^(.35){}="a"^(.65){}="b" \ar@{=} "a";"b"    & H\ar[ru]^{\psi^{\downarrow}}\ar[rd]_{\psi_{\uparrow}}   \\ & & (V_{\uparrow},\rho_{\uparrow}) & & & & (\hat{V}_{\uparrow},\hat{\rho}_{\uparrow}) }
\] 
be two filtered matched pairs with the same $\Lambda$-module $H$.  A \textbf{$t$-morphism} from $\mathcal{P}$ to $\hat{\mathcal{P}}$ consists of a $\Lambda_{\uparrow}$-vector space isomorphism $\alpha_{\uparrow}\co V_{\uparrow}\to \hat{V}_{\uparrow}$ and a $\Lambda^{\downarrow}$-vector space isomorphism $\alpha^{\downarrow}\co V^{\downarrow}\to\hat{V}^{\downarrow}$ such that:
\begin{itemize}\item $\psi_{\uparrow}=\alpha_{\uparrow}\circ\phi_{\uparrow}$;
\item $\psi^{\downarrow}=\alpha^{\downarrow}\circ\phi^{\downarrow}$; and \item $|\rho_{\uparrow}(\phi_{\uparrow}(x))-\hat{\rho}_{\uparrow}(\psi_{\uparrow}(x))|\leq t$   and  $|\rho^{\downarrow}(\phi^{\downarrow}(x))-\hat{\rho}^{\downarrow}(\psi^{\downarrow}(x))|\leq t$ for all $x\in H\setminus\{0\}$.\end{itemize} 
\end{dfn}


Let us first show stability for the gaps $G_i$ from Definition \ref{gapdfn}. 

\begin{prop}\label{gapcts}
Suppose that there exists a $t$-morphism from $\mathcal{P}$ to $\hat{\mathcal{P}}$.  Then $|G_i(\mathcal{P})-G_i(\hat{\mathcal{P}})|\leq 2t$ for all $i$.
\end{prop}

\begin{proof} This is almost immediate from the definitions.  If, as in Definition \ref{gapdfn}, $\{x_1,\ldots,x_i\}\subset H$ is an independent set and $\gamma$ is a real number such that $\rho^{\downarrow}(\phi^{\downarrow}(x_j))-\rho_{\uparrow}(\phi_{\uparrow}(x_j))\geq \gamma$ for all $j\in \{1,\ldots,i\}$, then by Definition \ref{tmorphdfn} we have $\hat{\rho}^{\downarrow}(\psi^{\downarrow}(x_j))\geq \rho^{\downarrow}(\phi^{\downarrow}(x_j))-t$ and $\hat{\rho}_{\uparrow}(\phi_{\uparrow}(x_j))\leq \rho_{\uparrow}(\phi_{\uparrow}(x_j))+t$.  Thus $\hat{\rho}^{\downarrow}(\psi^{\downarrow}(x_j))-\hat{\rho}_{\uparrow}(\phi_{\uparrow}(x_j))\geq \gamma-2t$.  So taking the $\sup$ over such $\gamma$ shows that $G_i(\hat{\mathcal{P}})\geq G_i(\mathcal{P})-2t$.  Combining this with the same statement with the roles of $\mathcal{P}$ and $\hat{\mathcal{P}}$ reversed proves the result.
\end{proof}

In the case that $\mathcal{P}$ and $\hat{\mathcal{P}}$ admit doubly-orthogonal bases, we now consider improving Proposition \ref{gapcts} to a statement about the basis spectra $\Sigma(\mathcal{P}),\Sigma(\hat{\mathcal{P}})$.  Now by Proposition \ref{gapchar}, $\Sigma(\mathcal{P})$ takes the form $\{([a_1],G_1(\mathcal{P})),\ldots,([a_d],G_d(\mathcal{P}))\}$ for suitable $[a_i]\in \R/\Gamma$, and likewise for $\Sigma(\hat{\mathcal{P}})$.  Proposition \ref{gapcts} thus bounds the difference between the second coordinates of these.  If $\Gamma<\R$ is dense, there is no meaningful nontrivial distance between the first coordinates (the natural topology on $\R/\Gamma$ is the trivial one), so even laying aside the point that Theorem \ref{basistheorem} is proven only for the discrete case, in the dense case Proposition \ref{gapcts} should be regarded as providing all the stability information that we could ask for.  Accordingly,  Proposition \ref{stab} below will assume that $\Gamma$ is discrete.

To compare the basis spectra let us first introduce the following definition.

\begin{dfn}
Let $\mathcal{S}$ and $\hat{\mathcal{S}}$ be two multisets whose elements belong to $(\mathbb{R}/\Gamma)\times \R$, and let $t\geq 0$. A \textbf{strong $t$-matching} between $\mathcal{S}$ and $\hat{\mathcal{S}}$ is a bijection $\sigma\co \mathcal{S}\to\hat{\mathcal{S}}$ such that, for all $([a],\ell)\in \mathcal{S}$, if we write $\sigma([a],\ell)=([\hat{a}],\hat{\ell})$, then there is $g\in \Gamma$ such that both $|g+\hat{a}-a|\leq t$ and $|(g+\hat{a}+\hat{\ell})-(a+\ell)|\leq  t$.
\end{dfn}

(Our use of the word ``strong'' is meant to distinguish this notion from the notion of matching that arises in sublevel persistence, which allows for deleting elements with $\ell\leq 2t$.)

\begin{theorem}\label{stab}
Suppose that $\Gamma$ is discrete and that there is a $t$-morphism from $\mathcal{P}$ to $\hat{\mathcal{P}}$, where $t\geq 0$. Then there is a strong $t$-matching between the basis spectra $\Sigma(\mathcal{P})$ and $\Sigma(\hat{\mathcal{P}})$.\end{theorem}
\begin{proof}
The overall outline of the proof bears some resemblance to \cite[Section 4]{CDGO} in that we use interpolation to reduce to a local result which is then proven by considering certain rectangle measures, though the details are different.

By replacing $(\hat{V}_{\uparrow},\hat{\rho}_{\uparrow})$ and $(\hat{V}^{\downarrow},\hat{\rho}^{\downarrow})$ by $(V_{\uparrow},\hat{\rho}_{\uparrow}\circ\alpha_{\uparrow})$ and $(V^{\downarrow},\hat{\rho}^{\downarrow}\circ\alpha^{\downarrow})$ we may as well assume that $\hat{V}_{\uparrow}=V_{\uparrow}$ and $\hat{V}^{\downarrow}=V^{\downarrow}$, and that $\alpha_{\uparrow}$ and $\alpha^{\downarrow}$ are the respective identity maps.  So we have functions $\rho_{\uparrow},\hat{\rho}_{\uparrow}\co V_{\uparrow}\to\R\cup\{-\infty\}$ such that $(V_{\uparrow},\rho_{\uparrow})$ and $(V_{\uparrow},\hat{\rho}_{\uparrow})$ are each orthogonalizable $\Lambda_{\uparrow}$-spaces, and likewise we have functions $\rho^{\downarrow},\hat{\rho}^{\downarrow}\co V^{\downarrow}\to\R\cup\{\infty\}$ such that $(V^{\downarrow},\rho^{\downarrow})$ and $(V^{\downarrow},\hat{\rho}^{\downarrow})$ are orthogonalizable $\Lambda^{\downarrow}$-spaces.  Moreover $\max|(\rho_{\uparrow}-\hat{\rho}_{\uparrow})\circ\phi_{\uparrow})|\leq t$ and $\max|(\rho^{\downarrow}-\hat{\rho}^{\downarrow})\circ\phi^{\downarrow}|\leq t$.

For each $s\in [0,1]$, put \[ \rho_{\uparrow}^{s}=(1-s)\rho_{\uparrow}+s\hat{\rho}_{\uparrow},\qquad \rho^{\downarrow}_{s}=(1-s)\rho^{\downarrow}+s\hat{\rho}^{\downarrow}.\] 
Then for all $s\in [0,1]$, $\phi_{\uparrow}\co H\to (V_{\uparrow},\rho_{\uparrow}^{s})$ and $\phi^{\downarrow}\co H\to (V^{\downarrow},\rho^{\downarrow}_{s})$ together provide a filtered matched pair which we denote by $\mathcal{P}_s$.  Moreover if $s,s'\in [0,1]$ then the identity maps provide a $(t|s-s'|)$-morphism between $\mathcal{P}_{s}$ and $\mathcal{P}_{s'}$.

It consequently suffices to prove the following ``local'' version of Theorem \ref{stab}:

\begin{lemma}\label{localstab}
Suppose that $\Gamma$ is discrete and that $\mathcal{P}$ is a filtered matched pair.  Then there is $\tau(\mathcal{P})>0$ such that, if $\hat{\mathcal{P}}$ is any other filtered matched pair such that there exists a $t$-morphism from $\mathcal{P}$ to $\hat{\mathcal{P}}$ with $t<\tau(\mathcal{P})$, then there is a strong $t$-matching between $\mathcal{P}$ and $\hat{\mathcal{P}}$.
\end{lemma}

Indeed, assuming Lemma \ref{localstab}, for all $s\in [0,1]$ there will be $\delta_s>0$ such that if $s'\in [0,1]$ obeys $|s-s'|<\delta_s$ then there is a strong $t|s-s'|$-matching between $\Sigma(\mathcal{P}_s)$ and $\Sigma(\mathcal{P}_{s'})$.  Using a finite cover of $[0,1]$ by intervals of form $(s-\delta_s,s+\delta_s)$, one can then find $0=s_0<s_1<\cdots<s_{N-1}<s_N=1$ such that there is a $t(s_{i+1}-s_i)$-matching between $\Sigma(\mathcal{P}_{s_i})$ and $\Sigma(\mathcal{P}_{s_{i+1}})$ for each $i\in\{0,\ldots,N-1\}$.  The composition of these matchings would then be the desired $t$-matching between $\Sigma(\mathcal{P}_0)$ and $\Sigma(\mathcal{P}_1)$.

Accordingly, to complete the proof of Theorem \ref{stab} we now prove Lemma \ref{localstab}.  Since $\Gamma$ is a discrete subgroup of $\R$, either $\Gamma=\{0\}$ or $\Gamma$ is infinite cyclic.  In the former case set $\lambda_0=\infty$, and in the latter case set $\lambda_0$ equal to the positive generator of $\Gamma$.  Let $\{([a_1],\ell_1),\ldots,([a_d],\ell_d)\}$ denote the basis spectrum of $\mathcal{P}$, and for the representatives $a_i$ of the $\Gamma$-cosets $[a_i]$ take the unique element of the coset that belongs to $\left(-\frac{\lambda_0}{2},\frac{\lambda_0}{2}\right]$. Choose $\tau(\mathcal{P})>0$ to be such that both $\tau(\mathcal{P})<\frac{\lambda_0}{2}$ and, for all $i\in \{1,\ldots,d\}$ and $j\neq i$, either $([a_j],\ell_j)=([a_i],\ell_i)$ or else the $\ell^{\infty}$-distance in $\R^2$ from $(a_i,a_i+\ell_i)$ to the $\Gamma$-orbit $\{(a_j+g,a_j+\ell_j+g)|g\in \Gamma\}$ is greater than $2\tau(\mathcal{P})$.


For $a,b\in \R$ let $H_{\uparrow}^{\leq a},H_{\uparrow}^{<a},H^{\downarrow}_{\geq b},H^{\downarrow}_{>b}$ be the $\kappa$-vector subspaces of $\frac{H}{tH}$ defined in (\ref{hupdown}).  If $a'<a$ and $b'>b$, and if $\max\{|a'-a|,|b'-b|\}<\lambda_0$, then a similar analysis to that in the proof of Proposition \ref{barinvt} shows that the $\kappa$-vector space \[ V_{[a',a]\times [b,b']}:=\frac{H_{\uparrow}^{\leq a}\cap H^{\downarrow}_{\geq b}}{(H_{\uparrow}^{< a'}\cap H^{\downarrow}_{\geq b})+(H_{\uparrow}^{\leq a}\cap H^{\downarrow}_{>b'})}\] has dimension equal to the number of $i$ (counting multiplicity) for which the set $O_i:=\{(a_i+g,a_i+\ell_i+g)|g\in \Gamma\}\subset \R^2$ contains an element of the rectangle $[a',a]\times [b,b']$.\footnote{Such an element is necessarily unique by the assumption that $\max\{|a'-a|,|b'-b|\}<\lambda_0$; without this assumption the correct statement would have been that $\dim_{\kappa}V_{[a',a]\times [b,b']}$ is the sum of the cardinalities of the sets $O_i\cap ([a',a]\times [b,b'])$.}  If $a'<x'\leq a<x$ and $b'>y'\geq b>y$ then the inclusions $H_{\uparrow}^{\leq a}\subset H_{\uparrow}^{\leq x}$, $H_{\uparrow}^{<a'}\subset H_{\uparrow}^{<x'}$, $H^{\downarrow}_{\geq b}\subset H^{\downarrow}_{\geq y}$, and $H^{\downarrow}_{>b'}\subset H^{\downarrow}_{>y'}$ induce a map $V_{[a',a]\times [b,b']}\to V_{[x',x]\times[y,y']}$, and the rank of this map equals the number of $i$ for which some element of $O_i$ lies in both rectangles $[a',a]\times[b,b']$ and $[x',x]\times [y,y']$  (\emph{i.e.}, lies in the intersection $[x',a]\times[b,y']$).  

In particular, the inclusion-induced map $V_{[a',a]\times [b,b']}\to V_{[x',x]\times[y,y']}$ is an isomorphism provided that every element of $\cup_iO_i$ that lies in either one of $[a',a]\times[b,b']$ or $[x',x]\times [y,y']$ in fact lies in both of them.  For example, if $([a_i],\ell_i)\in \Sigma(\mathcal{P})$ we can apply this with $x'=a=a_i$ and $b=y'=a_i+\ell_i$ to see that, if $t$ is smaller than the number $\tau(\mathcal{P})$ defined above, then the inclusion-induced map \[ V_{[a_i-2t,a_i]\times [a_i+\ell_i,a_i+\ell_i+2t]}\to V_{[a_i,a_i+2t]\times [a_{i}+\ell_i-2t,a_i+\ell_i]} \] is an isomorphism between $\kappa$-vector spaces of dimension equal to the multiplicity of $([a_i],\ell_i)$ in $\Sigma(\mathcal{P})$.  

Now assuming that there is a $t$-morphism from $\mathcal{P}$ to $\hat{\mathcal{P}}$ with $t<\tau(\mathcal{P})$, for closed intervals $I,J\subset \R$ let $\hat{V}_{I\times J}$ be the vector space analogous to $V_{I\times J}$ but constructed from $\hat{\mathcal{P}}$ instead of from $\mathcal{P}$, and similarly define $\hat{H}_{\uparrow}^{\leq a},\hat{H}_{\uparrow}^{<a},\hat{H}^{\downarrow}_{\geq b},\hat{H}^{\downarrow}_{>b}$.  Because $|\rho_{\uparrow}\circ\phi_{\uparrow}-\hat{\rho}_{\uparrow}\circ\hat{\psi}_{\uparrow}|\leq t$ and $|\rho_{\uparrow}\circ\phi_{\uparrow}-\hat{\rho}_{\uparrow}\circ\hat{\psi}_{\uparrow}|\leq t$, for all $a\in \R$ we have inclusions $H_{\uparrow}^{\leq a-t}\subset  
\hat{H}_{\uparrow}^{\leq a}\subset H_{\uparrow}^{\leq a+t}$, $H_{\uparrow}^{< a-t}\subset  
\hat{H}_{\uparrow}^{<a}\subset H_{\uparrow}^{<a+t}$, $H^{\downarrow}_{\geq b+t}\subset  
\hat{H}^{\downarrow}_{\geq b}\subset H^{\downarrow}_{\geq b-t}$, and $H^{\downarrow}_{> b+t}\subset  
\hat{H}^{\downarrow}_{>b}\subset H^{\downarrow}_{>b-t}$.  So for any rectangle $[a',a]\times[b,b']$ we have a composition of maps induced by these inclusions: \[ V_{[a'-t,a-t]\times[b+t,b'+t]}\to \hat{V}_{[a',a]\times[b,b']}\to V_{[a'+t,a+t]\times[b-t,b'-t]}.\]  

So if $([a_i],\ell_i)$ belongs to $\Sigma(\mathcal{P})$ with positive multiplicity $m_i$, the isomorphism  of $m_i$-dimensional vector spaces   $V_{[a_i-2t,a_i]\times [a_i+\ell_i,a_i+\ell_i+2t]}\to V_{[a_i,a_i+2t]\times [a_{i}+\ell_i-2t,a_i+\ell_i]}$ factors through the $\kappa$-vector space $\hat{V}_{[a_i-t,a_i+t]\times [a_i+\ell_i-t,a_i+\ell_i+t]}$, whence the latter vector space has dimension at least $m_i$. Temporarily denote $\dim_{\kappa}\hat{V}_{[a_i-t,a_i+t]\times [a_i+\ell_i-t,a_i+\ell_i+t]}$ as $\hat{d}_i$; let us now show that $\hat{d}_i=m_i$. Using again that $2t<\lambda_0$, the dimension $\hat{d}_i$
is equal to sum of the multiplicities of all elements $([\hat{a}_j],\hat{\ell}_j)$ of $\Sigma(\hat{\mathcal{P}})$ such that $a_j$ can be chosen within its $\Gamma$-coset to obey $(\hat{a}_j,\hat{a}_j+\hat{\ell}_j)\in [a_i-t,a_i+t]\times [a_i+\ell_i-t,a_i+\ell_i+t]$.  Because $t<\tau(\mathcal{P})$, the orbits of the rectangles $[a_i-t,a_i+t]\times [a_i+\ell_i-t,a_i+\ell_i+t]$ under the diagonal action of $\Gamma$ are pairwise disjoint as $i$ varies, in view of which $\sum_i\hat{d}_i$ is no larger than the sum of the multiplicities of all elements of $\Sigma(\hat{\mathcal{P}})$.  But since $\Sigma(\mathcal{P})$ and $\Sigma(\hat{\mathcal{P}})$ are each in one-to-one correspondence with a basis for the free rank-$d$ $\Gamma$-module $\frac{H}{tH}$, it follows that each of these multisets has the sum of the multiplicities of its elements equal to $d$.  So we have shown on the one hand that $\sum_{i}\hat{d}_i\leq \sum_i m_i=d$, and on the other hand that, for each $i$, $\hat{d}_i\geq m_i$, whence indeed $\hat{d}_i=m_i$ for all $i$.  Furthermore, \emph{every} one of the $d$ elements (counted with multiplicity) of $\Sigma(\hat{\mathcal{P}})$ is among those enumerated by the various $\hat{d}_i$, for otherwise it could not hold both that $\hat{d}_i=m_i$ and that  $\sum_i m_i=d$.

It follows that there is a bijection between $\Sigma(\hat{\mathcal{P}})$ and $\Sigma(\mathcal{P})$ that sends $([\hat{a}],\hat{\ell})\in \Sigma(\hat{\mathcal{P}})$ to (one of the $d_i$ copies of) the unique $([a_i],\ell_i)$ for which there is $g\in \Gamma$ such that $(\hat{a}+g,\hat{a}+\hat{\ell}+g)\in [a_i-t,a_i+t]\times [a_i+\ell_i-t,a_i+\ell_i+t]$.  This bijection is then a $t$-matching, completing the proof of Lemma \ref{localstab} and hence of Theorem \ref{stab}.
\end{proof}

\begin{remark}
Conversely, if there is a $t$-matching between $\Sigma(\mathcal{P})$ and $\Sigma(\hat{\mathcal{P}})$ then it is easy to construct a $t$-morphism between $\mathcal{P}$ and $\hat{\mathcal{P}}$, by taking $\alpha_{\uparrow}$ and $\alpha^{\downarrow}$  to intertwine the basis elements corresponding to the elements of $(\mathbb{R}/\Gamma)\times \R$ that correspond under the $t$-matching.  Thus the above stability theorem implies an isometry theorem (on the space of filtered matched pairs with appropriate morphism and matching pseudometrics) like the familiar one from sublevel persistence (as in \cite{CDGO}).
\end{remark}

\begin{remark}
As in \cite[Remark 8.5]{UZ}, a ``two-parameter'' version of the notion of a $t$-morphism could be defined by declaring, for any real numbers $s,t$ with $s+t\geq 0$, a $(s,t)$-morphism 
to be given by data as in Definition \ref{tmorphdfn} except with the last bullet point replaced by the conditions \[ \rho_{\uparrow}\circ\phi_{\uparrow}-s\leq \hat{\rho}_{\uparrow}\circ\psi_{\uparrow}\leq \rho_{\uparrow}\circ\phi_{\uparrow}+t,\qquad
\rho^{\downarrow}\circ\phi^{\downarrow}-s\leq \hat{\rho}^{\downarrow}\circ\psi^{\downarrow}\leq \rho^{\downarrow}\circ\phi^{\downarrow}\circ+t.\]  The existence of an $(s,t)$-morphism then implies that of a ``$(s,t)$-matching'' in the sense of a bijection between $\Sigma(\mathcal{P})$ and $\Sigma(\hat{\mathcal{P}})$ such that, for corresponding elements $([a],\ell)\in\Sigma(\mathcal{P})$ and $([\hat{a}],\hat{\ell})$ and suitable choices of $a,\hat{a}$ within their $\Gamma$-cosets, one has $a-s\leq \hat{a}\leq a+t$ and $a+\ell-s\leq \hat{a}+\hat{\ell}\leq a+\ell+t$.  Indeed, this can be inferred formally from what we have done by considering the filtered matched pair $\tilde{\mathcal{P}}$ given by setting $\tilde{\rho}_{\uparrow}=\hat{\rho}_{\uparrow}+\frac{s-t}{2}$ and $\tilde{\rho}^{\downarrow}=\hat{\rho}^{\downarrow}+\frac{s-t}{2}$.  There is then a $\frac{s+t}{2}$-morphism between $\mathcal{P}$ and $\tilde{\mathcal{P}}$ and hence (by Theorem \ref{stab}) a $\frac{s+t}{2}$-matching between $\Sigma(\mathcal{P})$ and $\Sigma(\tilde{\mathcal{P}})$, and then our claim follows by the obvious shift that relates $\Sigma(\tilde{\mathcal{P}})$ to $\Sigma(\hat{\mathcal{P}})$.

For example (set $s=0$), if $\Gamma=\{0\}$ then the basis spectra are monotone in the sense that if $\hat{\rho}_{\uparrow}\circ\phi_{\uparrow}\geq\rho_{\uparrow}\circ\psi_{\uparrow}$ and $\hat{\rho}^{\downarrow}\circ\psi^{\downarrow}\geq \rho^{\downarrow}\circ\psi^{\downarrow}$ then there is a bijection from $\Sigma(\mathcal{P})$ to $\Sigma(\hat{\mathcal{P}})$ that sends each $(a,\ell)$ to $(\hat{a},\hat{\ell})$ with both $\hat{a}\geq a$ and $\hat{a}+\hat{\ell}\geq a+\ell$. 
\end{remark}

\begin{remark}\label{twostab}
Another modest generalization of the notion of a $t$-morphism would involve allowing $(V_{\uparrow},\rho_{\uparrow})$ and $(V^{\downarrow},\rho^{\downarrow})$ to vary independently of each other, so we would choose two nonnegative numbers $t_{\uparrow},t^{\downarrow}$ and require  $|\rho_{\uparrow}(\phi_{\uparrow}(x))-\hat{\rho}_{\uparrow}(\psi_{\uparrow}(x))|\leq t_{\uparrow}$ for all $x\in V_{\uparrow}\setminus\{0\}$  and  $|\rho^{\downarrow}(\phi^{\downarrow}(x))-\hat{\rho}^{\downarrow}(\psi^{\downarrow}(x))|\leq t^{\downarrow}$ for all $x\in V^{\downarrow}\setminus\{0\}$.  Of course one could then apply Theorem \ref{stab} with $t=\max\{t_{\uparrow},t^{\downarrow}\}$, but if $t_{\uparrow}\neq t^{\downarrow}$ a slightly stronger statement holds.  Namely, the $t$-matching provided by Theorem \ref{stab} can be taken to have the property that, for matched elements $([a_i],\ell_i),([\hat{a}_i],\hat{\ell}_i)$ and suitable representatives $a_i,\hat{a}_i$ of the $\Gamma$-cosets, we have $|\hat{a}_i-a_i|\leq t_{\uparrow}$ and $|(\hat{a}_i+\hat{\ell}_i)-(a_i+\ell_i)|\leq t^{\downarrow}$.  If $t_{\uparrow},t^{\downarrow}$ are both nonzero this can be deduced formally from Theorem \ref{stab} by multiplying $\rho^{\downarrow},\hat{\rho}^{\downarrow}$ by $\frac{t_{\uparrow}}{t^{\downarrow}}$, applying  Theorem \ref{stab} with $t=t_{\uparrow}$, and then rescaling again to return to the original $\rho^{\downarrow},\hat{\rho}^{\downarrow}$.  If instead $t^{\downarrow}=0$ (say), one can observe that, given the discreteness of $\Gamma$, the existence of a matching of the desired form for all $0<t^{\downarrow}<t_{\uparrow}$ implies the same for $t^{\downarrow}=0$, since if $|(\hat{a}_i+\hat{\ell}_i)-(a_i+\ell_i)|$ is sufficiently small then it must be zero. 

For example, one might take $V_{\uparrow}$ to arise from the sublevel persistence of some function which is allowed to vary, and take $V^{\downarrow}$ to arise from the superlevel persistence of a different function which is held fixed.  The basis spectra $\Sigma(\mathcal{P}_s)$ of the resulting filtered matched pairs $\mathcal{P}_s$ would then vary continuously with the first function, with the values $a+\ell$ for $([a],\ell)\in \Sigma(\mathcal{P}_s)$ remaining fixed as $s$ varies.
\end{remark}

\subsection{Duality}\label{fmpdual}

Given a filtered matched pair $\mathcal{P}$ consisting of  data \begin{equation}\label{pdisp}\xymatrix{ & (V^{\downarrow},\rho^{\downarrow}) \\ H\ar[ru]^{\phi^{\downarrow}}\ar[rd]_{\phi_{\uparrow}} & \\ & (V_{\uparrow},\rho_{\uparrow}) }\end{equation} satisfying the conditions in Definition \ref{fmpdfn}, we shall now define the \textbf{dual matched pair} ${}^{\vee}\!\mathcal{P}$ with corresponding data \begin{equation}\label{dualm}  \xymatrix{ & ({}^{\vee}\!(V_{\uparrow}),{}^{\vee}\!\rho_{\uparrow}) \\ {}^{\vee}\!H\ar[ru]^{\delta^{\downarrow}(\phi_{\uparrow})}\ar[rd]_{\delta_{\uparrow}(\phi^{\downarrow})}& \\ & ({}^{\vee}\!(V^{\downarrow}),{}^{\vee}\!\rho^{\downarrow})}\end{equation}  (The arrows are not misplaced---recall that if $W$ is a $\Lambda_{\uparrow}$-vector space then ${}^{\vee}\!W=\overline{\mathrm{Hom}_{\Lambda_{\uparrow}}(W,\Lambda_{\uparrow})}$ is a $\Lambda^{\downarrow}$-vector space, and vice versa.)

The construction proceeds in a fairly straightforward way.  First, we need to define the functions ${}^{\vee}\!\rho_{\uparrow}\co {}^{\vee}\!(V_{\uparrow})\to \R\cup\{\infty\}$ and ${}^{\vee}\!\rho^{\downarrow}\co {}^{\vee}\!(V^{\downarrow})\to \R\cup\{-\infty\}$ in such a way that $({}^{\vee}\!(V_{\uparrow}),{}^{\vee}\!\rho_{\uparrow})$ is an orthogonalizable $\Lambda^{\downarrow}$-space and   $({}^{\vee}\!(V^{\downarrow}),{}^{\vee}\!\rho_{\downarrow})$ is an orthogonalizable $\Lambda_{\uparrow}$-space.  We essentially follow \cite[Section 2.4]{UZ}, in which, for a normed $\Lambda_{\uparrow}$-space $(V,\rho)$, a function $\rho^*\co \mathrm{Hom}_{\Lambda_{\uparrow}}(V,\Lambda_{\uparrow})\to \R\cup\{-\infty\}$ was defined by $\rho^*(\zeta)=\sup_{0\neq x\in V}(-\rho(x)-\nu_{\uparrow}(\zeta(x)))$. As verified in \cite{UZ}, $(\mathrm{Hom}_{\Lambda_{\uparrow}}(V,\Lambda_{\uparrow}),\rho^*)$ is then a normed $\Lambda_{\uparrow}$-space, and hence by Remark \ref{switch} $({}^{\vee}V,-\rho^*)$ is a normed $\Lambda^{\downarrow}$-space.

Thus, for a normed $\Lambda_{\uparrow}$-space $(V_{\uparrow},\rho_{\uparrow})$ we obtain a normed $\Lambda^{\downarrow}$-space $({}^{\vee}\!(V_{\uparrow}),{}^{\vee}\!\rho_{\uparrow})$ by setting, for each $\zeta\in {}^{\vee}\!(V_{\uparrow})=\overline{\mathrm{Hom}_{\Lambda_{\uparrow}}(V_{\uparrow},\Lambda_{\uparrow})}$, \[ {}^{\vee}\!\rho_{\uparrow}(\zeta)=\inf_{0\neq x\in V_{\uparrow}}(\nu_{\uparrow}(\zeta(x))+\rho_{\uparrow}(x)).\]  Symmetrically, for a normed $\Lambda^{\downarrow}$-space $(V^{\downarrow},\rho^{\downarrow})$ we obtain a normed $\Lambda_{\uparrow}$-space  $({}^{\vee}(V^{\downarrow}),{}^{\vee}\!\rho^{\downarrow})$ by setting, for each $\zeta\in {}^{\vee}\!(V^{\downarrow})$, \[ {}^{\vee}\!\rho^{\downarrow}(\zeta)=\sup_{0\neq x\in V^{\downarrow}}(\nu^{\downarrow}(\zeta(x))+\rho^{\downarrow}(x)).\]  This suffices to define the notation ${}^{\vee}\!\rho_{\uparrow}$, ${}^{\vee}\!\rho^{\downarrow}$ that appears in (\ref{dualm}).

\begin{remark}\label{dualbasisrem}
After adjusting signs for the effect of conjugation, \cite[Proposition 2.20]{UZ} shows that if $\{v_1,\ldots,v_d\}$ is an orthogonal basis for $V_{\uparrow}$ (resp. for $V^{\downarrow}$), then, with respect to ${}^{\vee}\!\rho_{\uparrow}$ or ${}^{\vee}\!\rho^{\downarrow}$, the dual basis $\{v_{1}^{*},\ldots,v_{d}^{*}\}$ is likewise an orthogonal basis for ${}^{\vee}\!(V_{\uparrow})$ (resp. for ${}^{\vee}\!(V^{\downarrow})$), and moreover that  ${}^{\vee}\!\rho_{\uparrow}(v_{i}^{*})=\rho_{\uparrow}(v_i)$ (resp. ${}^{\vee}\!\rho^{\downarrow}(v_{i}^{*})=\rho^{\downarrow}(v_i)$) for each $i$.
\end{remark}

We next define the maps $\delta^{\downarrow}(\phi_{\uparrow})$, $\delta_{\uparrow}(\phi^{\downarrow})$ that appear in (\ref{dualm}).  In this direction, it is helpful to note the following:

\begin{prop}\label{pullout} If $H$ is a finitely-generated $\Lambda$-module then the $\Lambda^{\downarrow}$-linear map $\beta\co \Lambda^{\downarrow}\otimes_{\Lambda}{}^{\vee}\!H \to {}^{\vee}(\Lambda_{\uparrow}\otimes_{\Lambda}H)$, defined by extending linearly from \[ (\beta(\mu\otimes\phi))(\lambda\otimes x)=\bar{\mu}\lambda\phi(x) \] whenever $\mu\in \Lambda^{\downarrow},\phi\in {}^{\vee}\!H,\lambda\in \Lambda_{\uparrow},x\in H$, is an isomorphism.  Likewise, the same formula defines an isomorphism $\beta$ of $\Lambda_{\uparrow}$-vector spaces $\Lambda_{\uparrow}\otimes_{\Lambda}{}^{\vee}\!H\to {}^{\vee}(\Lambda^{\downarrow}\otimes_{\Lambda}H)$.
\end{prop}

\begin{remark}\label{betabasis}
If $H$ is a free module, say with basis $\{e_1,\ldots,e_r\}$, then we obtain a basis $\{1\otimes e_1,\ldots,1\otimes e_r\}$ for $\Lambda_{\uparrow}\otimes_{\Lambda}H$ and then dual bases $\{e_{1}^{\vee},\ldots,e_{r}^{\vee}\}$ for ${}^{\vee}\!H$ and $\{(1\otimes e_1)^{\vee},\ldots,(1\otimes e_r)^{\vee}\}$ for ${}^{\vee}\!(\Lambda_{\uparrow}\otimes_{\Lambda} H)$. Then $\beta\co \Lambda^{\downarrow}\otimes_{\Lambda}{}^{\vee}\!H \to {}^{\vee}(\Lambda_{\uparrow}\otimes_{\Lambda}H)$ is the isomorphism that sends $1\otimes (e_{i}^{\vee})\in \Lambda^{\downarrow}\otimes_{\Lambda}{}^{\vee}\!H$ to $(1\otimes e_i)^{\vee}\in  {}^{\vee}(\Lambda_{\uparrow}\otimes_{\Lambda}H)$.  However the argument that we give below applies regardless of whether $H$ is free.
\end{remark}

\begin{proof}
Applying Proposition \ref{flipconj} with $M$ equal to $\mathrm{Hom}_{\Lambda}(H,\Lambda)$ (so by definition ${}^{\vee}H=\bar{M}$), our map is the composition of the inverse of the map $\alpha$ from that proposition with the map $\gamma\co\overline{\mathrm{Hom}_{\Lambda}(H,\Lambda)\otimes_{\Lambda}\Lambda_{\uparrow}}\to  \overline{\mathrm{Hom}_{\Lambda_{\uparrow}}(\Lambda_{\uparrow}\otimes_{\Lambda}H,\Lambda_{\uparrow})}$ that sends a simple tensor $\phi\otimes \mu$ to the map defined on simple tensors by $\lambda\otimes x\mapsto \mu\lambda\phi(x)$.  Let us remove the conjugation symbols and regard this map $\gamma$ as a $\Lambda_{\uparrow}$-linear map $\mathrm{Hom}_{\Lambda}(H,\Lambda)\otimes_{\Lambda}\Lambda_{\uparrow}\to \mathrm{Hom}_{\Lambda_{\uparrow}}(\Lambda_{\uparrow}\otimes_{\Lambda}H,\Lambda_{\uparrow})$.  (Recall that the conjugation functor does not change the underlying set-theoretic map.)   

Passing without comment through various identifications $R\otimes_RN\cong N$ for $R$-modules $N$, and recalling that $\mathcal{Q}(\Lambda)$ is the fraction field of $\Lambda$, the map $\gamma$ factors as \begin{align*} \left(\mathrm{Hom}_{\Lambda}(H,\Lambda)\otimes_{\Lambda}\mathcal{Q}(\Lambda)\right)\otimes_{\mathcal{Q}(\Lambda)}\Lambda_{\uparrow}&\to \mathrm{Hom}_{\mathcal{Q}(\Lambda)}(\mathcal{Q}(\Lambda)\otimes_{\Lambda}H,\mathcal{Q}(\Lambda))\otimes_{\mathcal{Q}(\Lambda)}\Lambda_{\uparrow} \\ & \to \mathrm{Hom}_{\Lambda_{\uparrow}}(\Lambda_{\uparrow}\otimes_{\Lambda}H,\Lambda_{\uparrow}) \end{align*} where both maps are defined in a straightforward way that is very similar to the definition of $\gamma$.  The first map above is an isomorphism by \cite[Lemma 4.87]{Rot}\footnote{This lemma applies because our module $H$, being finitely generated over the Noetherian ring $\Lambda$, is finitely presented.}, while the second map is the canonical isomorphism between the coefficient extension to $\Lambda_{\uparrow}$ of the dual of the  $\mathcal{Q}(\Lambda)$-vector space $W:=\mathcal{Q}(\Lambda)\otimes_{\Lambda}H$ with the dual of the coefficient extension $\Lambda_{\uparrow}\otimes_{\mathcal{Q}(\Lambda)}W$.  Thus $\gamma$ is an isomorphism, and hence $\beta=\gamma\circ\alpha^{-1}$ is also an isomorphism.  The isomorphism $\Lambda_{\uparrow}\otimes_{\Lambda}{}^{\vee}\!H\cong {}^{\vee}(\Lambda^{\downarrow}\otimes_{\Lambda}H)$ holds by the same argument.
\end{proof}  

Given the filtered matched pair $\mathcal{P}$ with data as in (\ref{pdisp}), so that in particular $1_{\Lambda_{\uparrow}}\otimes \phi_{\uparrow}\co \Lambda_{\uparrow}\otimes_{\Lambda}H\to V_{\uparrow}$ and $1_{\Lambda^{\downarrow}}\otimes \phi^{\downarrow}\co \Lambda^{\downarrow}\otimes_{\Lambda}H\to V^{\downarrow}$ are isomorphisms, we \textbf{define the maps $\delta_{\uparrow}(\phi^{\downarrow}),\delta^{\downarrow}(\phi_{\uparrow})$ in (\ref{dualm}) to be the respective compositions} \[ \xymatrix{ {}^{\vee}\!H\ar[r]& \Lambda_{\uparrow}\otimes_{\Lambda}{}^{\vee}\!H\ar[r]^{\beta}& {}^{\vee}\!(\Lambda^{\downarrow}\otimes_{\Lambda}H)\ar[rr]^{{}^{\vee}\!(1_{\Lambda^{\downarrow}}\otimes \phi^{\downarrow})^{-1}} & & {}^{\vee}\!(V^{\downarrow})} \] and \begin{equation}\label{dualpss}  \xymatrix{ {}^{\vee}\!H\ar[r]& \Lambda^{\downarrow}\otimes_{\Lambda}{}^{\vee}\!H\ar[r]^{\beta}& {}^{\vee}\!(\Lambda_{\uparrow}\otimes_{\Lambda}H)\ar[rr]^{{}^{\vee}\!(1_{\Lambda_{\uparrow}}\otimes \phi_{\uparrow})^{-1}}  & & {}^{\vee}\!(V_{\uparrow})} \end{equation} where in each case the first map is coefficient extension and the map $\beta$ is the isomorphism from Proposition \ref{pullout}.  Replacing $\delta_{\uparrow}(\phi^{\downarrow})$ by $1_{\Lambda_{\uparrow}}\otimes \delta_{\uparrow}(\phi^{\downarrow})$ and $\delta^{\downarrow}(\phi_{\uparrow})$ by $1_{\Lambda^{\downarrow}}\otimes \delta^{\downarrow}(\phi_{\uparrow})$ has the effect of converting the first maps in the above composition to the respective identities, and so since in each case the second and third maps are isomorphisms it follows that $1_{\Lambda_{\uparrow}}\otimes \delta_{\uparrow}(\phi^{\downarrow})$ and $1_{\Lambda^{\downarrow}}\otimes \delta^{\downarrow}(\phi_{\uparrow})$ are isomorphisms.  This suffices to show that, with our definitions, the dual ${}^{\vee}\!\mathcal{P}$ from (\ref{dualm}) is a filtered matched pair.

Here is, perhaps, a more transparent characterization of the maps $\delta_{\uparrow}(\phi^{\downarrow})$ and $\delta^{\downarrow}(\phi^{\uparrow})$.

\begin{prop}\label{deltachar}
For each $\zeta\in {}^{\vee}\!H$, the element $\delta_{\uparrow}(\phi^{\downarrow})(\zeta)\in {}^{\vee}\!(V^{\downarrow})$ is uniquely characterized by the fact that \begin{equation}\label{deltachareqn} \mbox{for all $h\in H$,}\quad \left(\delta_{\uparrow}(\phi^{\downarrow})(\zeta)\right)(\phi^{\downarrow}h)=\zeta(h).\end{equation}  Similarly, $\delta^{\downarrow}(\phi_{\uparrow})(\zeta)\in {}^{\vee}\!(V_{\uparrow})$ is uniquely characterized by the fact that \[ \mbox{for all $h\in H$,}\quad \left(\delta^{\downarrow}(\phi_{\uparrow})(\zeta)\right)(\phi_{\uparrow}h)=\zeta(h).\]
\end{prop}

\begin{proof} The proofs of the two sentences are identical so we just prove the first.  There can be at most one $\eta\in {}^{\vee}\!(V^{\downarrow})$ obeying $\eta(\phi^{\downarrow}h)=\zeta(h)$ for all $h\in H$ because the image of $H$ under $\phi^{\downarrow}$ spans $V^{\downarrow}$ over $\Lambda^{\downarrow}$, so we just need to check that (\ref{deltachareqn}) holds.  By definition, \begin{align*} 
\left(\delta_{\uparrow}(\phi^{\downarrow})(\zeta)\right)(\phi^{\downarrow}h)&=\left({}^{\vee}(1_{\Lambda^{\downarrow}}\otimes \phi^{\downarrow})^{-1}\beta(1\otimes \zeta)\right)(\phi^{\downarrow}h)
\\ &= \left(\beta(1\otimes\zeta)\right)(1\otimes h)=\zeta(h),\end{align*} as desired.
\end{proof}

When $\mathcal{P}$ has a doubly-orthogonal basis, we readily obtain a similar such basis for ${}^{\vee}\!\mathcal{P}$, as follows.

\begin{prop}\label{bardual}
If $\mathcal{P}$ is a filtered matched pair having a doubly-orthogonal basis $\{e_1,\ldots,e_d\}$, then the dual filtered matched pair has a doubly-orthogonal basis $\{e_{1}^{*},\ldots,e_{d}^{*}\}$ such that, with notation as in (\ref{pdisp}) and (\ref{dualm})) ${}^{\vee}\!\rho_{\uparrow}(\delta^{\downarrow}(\phi_{\uparrow})e_{i}^{*})=\rho_{\uparrow}(\phi_{\uparrow}e_i)$ and  ${}^{\vee}\!\rho^{\downarrow}(\delta_{\uparrow}(\phi^{\downarrow})e_{i}^{*})=\rho^{\downarrow}(\phi^{\downarrow}e_i)$ for all $i$.  Hence the basis spectra of ${}^{\vee}\!\mathcal{P}$ and $\mathcal{P}$ are related by \[ \Sigma({}^{\vee}\!\mathcal{P})=\{([a_i+\ell_i],-\ell_i)|([a_i],\ell_i)\in \Sigma(\mathcal{P})\}.\]
\end{prop}

\begin{proof}
The existence of the doubly-orthogonal basis $\{e_1,\ldots,e_d\}$ for $\mathcal{P}$ implies in particular that $\frac{H}{tH}$ is free with basis consisting of the cosets $[e_i]$ of $e_i$.  Since, $\Lambda$ being an integral domain, we have $\mathrm{Hom}_{\Lambda}(tH,\Lambda)=\{0\}$, by dualizing the canonical exact sequence $tH\to H\to \frac{H}{tH}\to 0$ we see that the pullback map ${}^{\vee}\!\left(\frac{H}{tH}\right)\to {}^{\vee}\!H$ is an isomorphism.  So ${}^{\vee}\!H$ has a basis $\{e_{i}^{*}\}$ consisting of the pullbacks to $H$ of the elements of the dual basis to $\{[e_1],\ldots,[e_d]\}$ for  ${}^{\vee}\!\left(\frac{H}{tH}\right)$.   It then follows immediately from Proposition \ref{deltachar} that the pairing of $\delta_{\uparrow}(\phi^{\downarrow})e_{i}^{*}$ with $\phi^{\downarrow}e_j$ is $1$ if $i=j$ and $0$ otherwise.  So $\{\delta_{\uparrow}(\phi^{\downarrow})e_{1}^{*},\ldots,\delta_{\uparrow}(\phi^{\downarrow})e_d^*\}$ is the dual basis for the $\Lambda_{\uparrow}$-vector space ${}^{\vee}(V^{\downarrow})$ to the orthogonal basis $\{\phi^{\downarrow}e_1,\ldots,\phi^{\downarrow}e_d\}$ for $V^{\downarrow}$.  Similarly, $\{\delta^{\downarrow}(\phi_{\uparrow})e_1^*,\ldots,\delta^{\downarrow}(\phi_{\uparrow})e_d^*\}$ is the dual basis to the orthogonal basis $\{\phi_{\uparrow}e_1,\ldots\phi_{\uparrow}e_d\}$ for  $V_{\uparrow}$.  By Remark \ref{dualbasisrem}, the bases $\{\delta_{\uparrow}(\phi^{\downarrow})e_1^*,\ldots,\delta_{\uparrow}(\phi^{\downarrow})e_d^*\}$ and  $\{\delta^{\downarrow}(\phi_{\uparrow})e_1^*,\ldots,\delta^{\downarrow}(\phi_{\uparrow})e_d^*\}$ are orthogonal, with their values of ${}^{\vee}\!\rho^{\downarrow}$ and ${}^{\vee}\!\rho_{\uparrow}$ as claimed in the proposition.  The last sentence of the proposition follows directly from this and from the definition of the basis spectrum.
\end{proof}

Our constructions behave well with respect to double duals:

\begin{prop} \label{ddcom} For a filtered matched pair $\mathcal{P}$ with notation as in (\ref{pdisp}) we have commutative diagrams \[ \xymatrix{ H\ar[r]^{\phi_{\uparrow}}\ar[d]_{\alpha_H} & V_{\uparrow}\ar[d]^{\alpha_{V_{\uparrow}}} \\ {}^{\vee\vee}\!H\ar[r]_{\delta_{\uparrow}(\delta^{\downarrow}(\phi_{\uparrow}))} & {}^{\vee\vee}\!V_{\uparrow} },\qquad \xymatrix{ H\ar[r]^{\phi^{\downarrow}}\ar[d]_{\alpha_H} & V^{\downarrow}\ar[d]^{\alpha_{V^{\downarrow}}} \\ {}^{\vee\vee}\!H\ar[r]_{\delta^{\downarrow}(\delta_{\uparrow}(\phi^{\downarrow}))} & {}^{\vee\vee}\!V^{\downarrow} }\] where $\alpha_H,\alpha_{V_{\uparrow}},\alpha_{V^{\downarrow}}$ are the canonical maps to double duals from Section \ref{dualsec}.
\end{prop}

\begin{proof}
The proofs are the same for the two diagrams so we just check the commutativity of the first one.

Let $h\in H$.  By Proposition \ref{deltachar}, the element $\eta:=(\delta_{\uparrow}(\delta^{\downarrow}(\phi_{\uparrow})))(\alpha_Hh)$ of ${}^{\vee\vee}V_{\uparrow}$ is uniquely characterized by the fact that, whenever $\zeta\in {}^{\vee}\!H$, we have $\eta(\delta^{\downarrow}(\phi_{\uparrow})\zeta)=(\alpha_Hh)(\zeta)$.  So we just need to check that $\alpha_{V_{\uparrow}}(\phi_{\uparrow}h)$ satisfies this same property.  
For $\zeta\in {}^{\vee}\!H$, one has \begin{align*} (\alpha_{V_{\uparrow}}(\phi_{\uparrow}h))(\delta^{\downarrow}(\phi_{\uparrow})\zeta)&=\overline{(\delta^{\downarrow}(\phi_{\uparrow}\zeta))(\phi_{\uparrow}h)}
\\ &= \overline{\zeta(h)}=(\alpha_Hh)(\zeta),\end{align*} as desired, where in the penultimate equality we have again used Proposition \ref{deltachar}.
\end{proof}

If $\mathcal{P}$ is a filtered matched pair with notation again as in (\ref{pdisp}), then with $d=\dim_{\mathcal{Q}(\Lambda)}\mathcal{Q}(\Lambda)\otimes_{\Lambda}H$ we have the gaps $G_1(\mathcal{P}),\ldots,G_d(\mathcal{P})$ of Definition \ref{gapdfn}.  The dimension of $\mathcal{Q}(\Lambda)\otimes_{\Lambda}{}^{\vee}\!H$ is also equal to $d$ (for instance this follows from \cite[Lemma 4.87]{Rot}), so there are also gaps $G_1({}^{\vee}\!\mathcal{P}),\ldots,G_{d}({}^{\vee}\!\mathcal{P})$.  If $\mathcal{P}$ admits a doubly-orthogonal basis (as always holds if $\Gamma$ is discrete) then Propositions \ref{gapchar} and \ref{bardual} imply that we have $G_i(\mathcal{P})=-G_{d+1-i}({}^{\vee}\!\mathcal{P})$ for each $i=1,\ldots,d$.  More generally we at least have an inequality:

\begin{prop} \label{gapdual} With notation as above, for each $i\in\{1,\ldots,d\}$ we have $G_{i}(\mathcal{P})\leq -G_{d+1-i}({}^{\vee}\!\mathcal{P})$.  
\end{prop}


\begin{proof}
Suppose that $c>-G_{d+1-i}({}^{\vee}\!\mathcal{P})$; it suffices to show that for any independent set $\{h_1,\ldots,h_i\}\subset H$ there is some $\ell\in \{1,\ldots,i\}$ with $\rho^{\downarrow}(\phi^{\downarrow}h_{\ell})-\rho_{\uparrow}(\phi_{\uparrow}h_{\ell})<c$.

Now the assumption that $c>-G_{d+1-i}({}^{\vee}\!\mathcal{P})$ implies that there is an independent set $\{\zeta_1,\ldots,\zeta_{d+1-i}\}\subset {}^{\vee}\!H$ such that, for each $m\in\{1,\ldots,d+1-i\}$, we have \begin{equation}\label{dualc} {}^{\vee}\!\rho_{\uparrow}(\delta^{\downarrow}(\phi_{\uparrow})\zeta_m)-{}^{\vee}\!\rho^{\downarrow}(\delta_{\uparrow}(\phi^{\downarrow})\zeta_m)>-c.\end{equation}  For any independent set $\{h_1,\ldots,h_i\}\subset H$, the annihilator of $\{1\otimes h_1,\ldots,1\otimes h_i\}\subset \mathcal{Q}(\Lambda)\otimes_{\Lambda}H$ is a $(d-i)$-dimensional subspace of $\mathrm{Hom}_{\mathcal{Q}(\Lambda)}(\mathcal{Q}(\Lambda)\otimes_{\Lambda}H,\mathcal{Q}(\Lambda))$.  Since the $(d+1-i)$-element set $\{\zeta_1,\ldots,\zeta_{d+1-i}\}$ becomes linearly independent on coefficient extension to $\mathcal{Q}(\Lambda)$, it follows that at least one $\zeta_{m}$ with $1\leq m\leq d+1-i$ must not vanish identically on $\{h_1,\ldots,h_i\}$.  So choose $\ell\in\{1,\ldots,i\}$ such that $\zeta_{m}(h_{\ell})\neq 0$.  

By the definition of ${}^{\vee}\!\rho_{\uparrow}$ and Proposition \ref{deltachar}, we have \[ 
{}^{\vee}\!\rho_{\uparrow}(\delta^{\downarrow}(\phi_{\uparrow})\zeta_m)\leq \nu_{\uparrow}((\delta^{\downarrow}(\phi_{\uparrow})\zeta_m)(\phi_{\uparrow}h_{\ell}))+\rho_{\uparrow}(\phi_{\uparrow}h_{\ell})=\nu_{\uparrow}(\zeta_m(h_{\ell}))+\rho_{\uparrow}(\phi_{\uparrow}h_{\ell});\] similarly, the definition of ${}^{\vee}\!\rho^{\downarrow}$ yields \[  {}^{\vee}\!\rho^{\downarrow}(\delta_{\uparrow}(\phi^{\downarrow})\zeta_m)\geq \nu^{\downarrow}((\delta_{\uparrow}(\phi^{\downarrow})\zeta_m)(\phi^{\downarrow}h_{\ell}))+\rho^{\downarrow}(\phi^{\downarrow}h_{\ell})=\nu^{\downarrow}(\zeta_m(h_{\ell}))+\rho^{\downarrow}(\phi^{\downarrow}h_{\ell}).\]  
Hence (\ref{dualc}) gives \[ -c<\left(\nu_{\uparrow}(\zeta_m(h_{\ell}))+\rho_{\uparrow}(\phi_{\uparrow}h_{\ell})\right)-\left(\nu^{\downarrow}(\zeta_m(h_{\ell}))+\rho^{\downarrow}(\phi^{\downarrow}h_{\ell})\right).\]  But any nonzero $\lambda\in \Lambda$, such as $\lambda=\zeta_m(h_{\ell})$, obeys $\nu_{\uparrow}(\lambda)-\nu^{\downarrow}(\lambda)\leq 0$, so we infer from the above that \[ -c<-\left(\rho^{\downarrow}(\phi^{\downarrow}h_{\ell})-\rho_{\uparrow}(\phi_{\uparrow}h_{\ell})\right).\]  Since the number $c>-G_{d+1-i}({}^{\vee}\!\mathcal{P})$ and the independent $i$-element set $\{h_1,\ldots,h_i\}\subset H$ were arbitrary this confirms that $G_{i}(\mathcal{P})\leq -G_{d+1-i}({}^{\vee}\!\mathcal{P})$.
\end{proof}

It would be interesting to know whether equality always holds in Proposition \ref{gapdual}.

\subsection{Constructing doubly-orthogonal bases}\label{basisconstruct}

In this section we prove Theorem \ref{basistheorem} asserting that, if $\Gamma$ is discrete, any filtered matched pair admits a doubly-orthogonal basis. The proof is rather involved, and the ideas introduced in it are not used elsewhere in the paper, so it might be omitted on first reading. In the case that $\Gamma=\{0\}$ the result can be established with much less effort than we expend for the general discrete case: in the notation used below, the maps $\phi_{\uparrow}$ and $\phi^{\downarrow}$ are just isomorphisms of $\kappa$-vector spaces when $\Gamma=\{0\}$, and one can then apply \cite[Theorem 3.4]{UZ} (or \cite[Lemma 2]{Bar}) to the map $\phi_{\uparrow}\circ\phi^{\downarrow -1}$ between the orthogonalizable $\Lambda_{\uparrow}$-spaces $(\bar{V}^{\downarrow},-\rho^{\downarrow})$ and $(V_{\uparrow},\rho_{\uparrow})$ to obtain the desired basis.

   Assume the filtered matched pair to be given by maps $\phi_{\uparrow}\co H\to V_{\uparrow}$ and $\phi^{\downarrow}\co H\to V^{\downarrow}$, where (in view of Proposition \ref{discretenormed}) $(V_{\uparrow},\rho_{\uparrow})$ is an orthogonalizable $\Lambda_{\uparrow}$-space and $(V^{\downarrow},\rho^{\downarrow})$ is an orthogonalizable $\Lambda^{\downarrow}$-space.  Write $F$ for the quotient $\frac{H}{tH}$; then $F$ is a free $\Lambda$-module because $H$ is a finitely generated module over $\Lambda$, which is a PID because $\Gamma$ is discrete.  Since $\phi_{\uparrow}$ and $\phi^{\downarrow}$ each factor through the quotient projection $H\to F$, and this quotient projection becomes an isomorphism after tensoring with $\Lambda_{\uparrow}$ or $\Lambda^{\downarrow}$, we may for simplicity just consider the case of a filtered matched pair of the form \[ \xymatrix{ & (V^{\downarrow},\rho^{\downarrow}) \\ F\ar[ru]^{\phi^{\downarrow}}\ar[rd]_{\phi_{\uparrow}} & \\ & (V_{\uparrow},\rho_{\uparrow}) } \] where  $F$ is a finitely-generated free $\Lambda$-module, as a doubly-orthogonal basis for this filtered matched pair will lift via the quotient projection to a doubly-orthogonal basis for the original one.

We thus need to show that there is a basis $\{e_1,\ldots,e_d\}$ for the free module $F$ such that, simultaneously, $\{\phi_{\uparrow}e_1,\ldots,\phi_{\uparrow}e_d\}$ is $\rho_{\uparrow}$-orthogonal for $V_{\uparrow}$ and $\{\phi^{\downarrow}e_1,\ldots,\phi^{\downarrow}e_d\}$ is $\rho^{\downarrow}$-orthogonal for $V^{\downarrow}$.   A first step, accomplished in Lemma \ref{triconstruct}, is to show that, separately, $(V_{\uparrow},\rho_{\uparrow})$ and $(V^{\downarrow},\rho^{\downarrow})$ each admit orthogonal bases that are  images under $\phi_{\uparrow}$ and $\phi^{\downarrow}$, respectively, of bases $\{x_1,\ldots,x_d\}$ and $\{y_1,\ldots,y_d\}$ of $F$; subsequent to this, we will iteratively modify these two bases, preserving  the respective orthogonality criteria, until they are the same basis for $F$.  For both of these purposes, it is helpful to record some general operations which preserve the property of orthogonality.

\begin{prop}\label{orthpres}
Let $(V_{\uparrow},\rho_{\uparrow})$ be an orthogonalizable $\Lambda_{\uparrow}$-space and let $\{v_1,\ldots,v_d\}$ be a $\rho_{\uparrow}$-orthogonal basis for $V_{\uparrow}$.
\begin{itemize} \item[(i)] If $\mu_1,\ldots,\mu_d$ are nonzero elements of $\Lambda_{\uparrow}$ then $\{\mu_1v_1,\ldots,\mu_dv_d\}$ is still a $\rho_{\uparrow}$-orthogonal basis for $V_{\uparrow}$.
\item[(ii)] If $w_1,\ldots,w_d\in V_{\uparrow}$ obey $\rho_{\uparrow}(w_i)<\rho_{\uparrow}(v_i)$ for all $i$, then $\{v_1+w_1,\ldots,v_d+w_d\}$ is still a $\rho_{\uparrow}$-orthogonal basis for $V_{\uparrow}$.
\item[(iii)] If $v\in \mathrm{span}_{\Lambda_{\uparrow}}\{v_2,\ldots,v_d\}$ obeys $\rho_{\uparrow}(v)\leq \rho_{\uparrow}(v_1)$, and if $\lambda\in \Lambda_{\uparrow}$ with $\nu_{\uparrow}(\lambda)=0$, then $\{\lambda v_1+v,v_2,\ldots,v_d\}$ is still a $\rho_{\uparrow}$-orthogonal basis for $V_{\uparrow}$.
\end{itemize}
\end{prop}

\begin{remark}\label{orthrev}
By Remark \ref{switch}, this proposition implies analogous results for orthogonalizable $\Lambda^{\downarrow}$-spaces, replacing the various apperances of $\rho_{\uparrow}$ with $(-\rho^{\downarrow})$.  Specifically, the $\Lambda^{\downarrow}$ version of (ii) (which will be used in the proof of Lemma \ref{triconstruct}) says that, if $(V^{\downarrow},\rho^{\downarrow})$ has an orthogonal basis $\{v_1,\ldots,v_d\}$ and if $\rho^{\downarrow}(w_i)>\rho^{\downarrow}(v_i)$ for each $i$ then $\{v_1+w_1,\ldots,v_d+w_d\}$ is still an orthogonal basis.
\end{remark}
\begin{proof}
(i) follows from noting that, due to the general relation $\nu_{\uparrow}(\lambda\mu)=\nu_{\uparrow}(\lambda)+\nu_{\uparrow}(\mu)$, for any $\lambda_1,\ldots,\lambda_d\in \Lambda_{\uparrow}$ we have \[ \nu_{\uparrow}\left(\sum_{i=1}^{d}\lambda_i\mu_iv_i\right)=\max_i\left(\rho_{\uparrow}(v_i)-\nu_{\uparrow}(\mu_i)-\nu_{\uparrow}(\lambda_i)\right)=\max_i\left(\rho_{\uparrow}(\mu_iv_i)-\nu_{\uparrow}(\lambda_i)\right).\]

For (ii), recall that, as noted in Remark \ref{equalcase}, one has $\rho_{\uparrow}(x+y)=\max\{\rho_{\uparrow}(x)+\rho_{\uparrow}(y)\}$ whenever $\rho_{\uparrow}(x)\neq\rho_{\uparrow}(y)$.  Given $\lambda_1,\ldots,\lambda_d\in \Lambda_{\uparrow}$, the hypothesis $\rho_{\uparrow}(w_i)<\rho_{\uparrow}(v_i)$ for all $i$ implies that \[ \rho_{\uparrow}\left(\sum_i\lambda_iw_i\right)\leq \max_i\left(\rho_{\uparrow}(w_i)-\nu_{\uparrow}(\lambda_i)\right)<\max_i\left(\rho_{\uparrow}(v_i)-\nu_{\uparrow}(\lambda_i)\right)=\rho_{\uparrow}\left(\sum_i\lambda_iv_i\right)\] and hence \begin{align*} \rho_{\uparrow}\left(\sum_i\lambda_i(v_i+w_i)\right)& = \rho_{\uparrow}\left(\sum_{i}\lambda_iv_i+\sum_i\lambda_iw_i\right)=\rho_{\uparrow}\left(\sum_i\lambda_iv_i\right) \\ &=\max_i\left(\rho_{\uparrow}(v_i)-\nu_{\uparrow}(\lambda_i)\right)=\max_i\left(\rho_{\uparrow}(v_i+w_i)-\nu_{\uparrow}(\lambda_i)\right).\end{align*}
Thus the set $\{v_1+w_1,\ldots,v_d+w_d\}$ is $\rho_{\uparrow}$-orthogonal.  A  $\rho_{\uparrow}$-orthogonal set of nonzero vectors is automatically linearly independent since the orthogonality condition implies that any nontrivial linear combination of the vectors has $\rho_{\uparrow}>-\infty$ while $\rho_{\uparrow}(0)=-\infty$.  So since by hypothesis $\dim V_{\uparrow}=d$ it follows that our $\rho_{\uparrow}$-orthogonal set $\{v_1+w_1,\ldots,v_d+w_d\}$ is a basis for $V_{\uparrow}$.

We now prove (iii).  First note that, by first applying (i) with $\mu_1=\lambda$ and $\mu_2=\cdots=\mu_d=1$, it suffices to prove (iii) in the case that $\lambda=1$.  Because $\rho_{\uparrow}(v)\leq \rho_{\uparrow}(v_1)$ and $v\in \mathrm{span}_{\Lambda_{\uparrow}}\{v_2,\ldots,v_d\}$ where $\{v_1,\ldots,v_d\}$ is orthogonal, one has $\rho_{\uparrow}(v_1+v)=\rho_{\uparrow}(v_1)$.  Let $\lambda_1,\ldots,\lambda_d\in \Lambda_{\uparrow}$.  If   $\rho_{\uparrow}(v_1+v)-\nu_{\uparrow}(\lambda_1)<\max_{i\geq 2}(\rho_{\uparrow}(v_i)-\nu_{\uparrow}(\lambda_i))$ then, using Remark \ref{equalcase}, \begin{align*} & \rho_{\uparrow}\left(\lambda_1(v_1+v)+\sum_{i\geq 2}\lambda_iv_i\right)=\rho_{\uparrow}\left(\sum_{i\geq 2}\lambda_iv_i\right)=\max_{i\geq 2}(\rho_{\uparrow}(v_i)-\nu_{\uparrow}(\lambda_i)) \\ & \qquad = \max\left\{\rho_{\uparrow}(v_1+v)-\nu_{\uparrow}(\lambda_1),\rho_{\uparrow}(v_2)-\nu_{\uparrow}(\lambda_2),\ldots,\rho_{\uparrow}(v_d)-\nu_{\uparrow}(\lambda_d)\right\}.\end{align*}  The remaining case is that $\rho_{\uparrow}(v_1+v)-\nu_{\uparrow}(\lambda_1)\geq \max_{i\geq 2}(\rho_{\uparrow}(v_i)-\nu_{\uparrow}(\lambda_i))$. Then   $\lambda_1v+\sum_{i=2}^{d}\lambda_iv_i\in\mathrm{span}_{\Lambda_{\uparrow}}\{v_2,\ldots,v_d\}$ with $\rho_{\uparrow}\left(\lambda_1v+\sum_{i=2}^{d}\lambda_iv_i\right)\leq \rho_{\uparrow}(\lambda_1v_1)$, and so by the orthogonality of $\{v_1,\ldots,v_d\}$ we have \begin{align*}  &\rho_{\uparrow}\left(\lambda_1(v_1+v)+\sum_{i=2}^{d}\lambda_iv_i\right)=\rho_{\uparrow}\left(\lambda_1v_1+\left(\lambda_1v+\sum_{i=2}^{d}\lambda_iv_i\right)\right)=\rho_{\uparrow}(\lambda_1v_1)\\ &=\rho_{\uparrow}(v_1)-\nu_{\uparrow}(\lambda_1)=\rho_{\uparrow}(v_1+v)-\nu_{\uparrow}(\lambda_1) \\& =\max\left\{\rho_{\uparrow}(v_1+v)-\nu_{\uparrow}(\lambda_1),\rho_{\uparrow}(v_2)-\nu_{\uparrow}(\lambda_2),\ldots,\rho_{\uparrow}(v_d)-\nu_{\uparrow}(\lambda_d)\right\}.\end{align*}  So indeed the set $\{v_1+v,v_2,\ldots,v_d\}$ (which is obviously a basis) is $\rho_{\uparrow}$-orthogonal.
\end{proof}

We now begin constructing the desired bases.  Throughout the following we continue to assume that $F$ is a free $\Lambda$-module of rank $d$, equipped with $\Lambda$-module homomorphisms $\phi_{\uparrow}\co F\to V_{\uparrow}$ and $\phi^{\downarrow}\co F\to V^{\downarrow}$ where $(V_{\uparrow},\rho_{\uparrow})$ is an orthogonalizable $\Lambda_{\uparrow}$-space and $(V^{\downarrow},\rho^{\downarrow})$ is an orthogonalizable $\Lambda^{\downarrow}$-space, such that $1_{\Lambda_{\uparrow}}\otimes\phi_{\uparrow}$ and $1_{\Lambda^{\downarrow}}\otimes \phi^{\downarrow}$ are vector space isomorphisms.

\begin{lemma} \label{triconstruct} There are bases $\{x_1,\ldots,x_d\}$ and $\{y_1,\ldots,y_d\}$ for the free $\Lambda$-module $F$ such that:
\begin{itemize} \item[(i)] $\{\phi_{\uparrow}x_1,\ldots,\phi_{\uparrow}x_d\}$ is an orthogonal basis for $(V_{\uparrow},\rho_{\uparrow})$;
\item[(ii)] $\{\phi^{\downarrow}y_1,\ldots,\phi^{\downarrow}y_d\}$ is an orthogonal basis for $(V^{\downarrow},\rho^{\downarrow})$; and 
\item[(iii)] For each $i\in\{1,\ldots,d\}$, $x_i-y_i\in \mathrm{span}_{\Lambda}\{y_1,\ldots,y_{i-1}\}$.
\end{itemize}
\end{lemma}

\begin{proof}
Let us first construct $\{y_1,\ldots,y_d\}$.  Let $\{f_1,\ldots,f_d\}$ be an arbitrary basis for the free module $F$.  Then $\{\phi^{\downarrow}f_1,\ldots,\phi^{\downarrow}f_d\}$ is a (probably not orthogonal) basis for $V^{\downarrow}$.  In view of Remark \ref{switch}, \cite[Theorem 2.16]{UZ} shows that there is an orthogonal basis $\{v_1,\ldots,v_d\}$ for $(V^{\downarrow},\rho^{\downarrow})$ such that, for each $i$, we have \[ v_i=\phi^{\downarrow}f_i+\sum_{j<i}\mu_{ij}\phi^{\downarrow}f_j \] for suitable $\mu_{ij}\in \Lambda^{\downarrow}$.  These $v_i$ might not belong to the image of $\phi^{\downarrow}$ because the $\mu_{ij}$ may not belong to the subset $\Lambda$ of $\Lambda^{\downarrow}$, but we can remedy this by using Proposition \ref{orthpres} (and its variant from Remark \ref{orthrev}).  

Specifically, let $a_0=\max_i \rho^{\downarrow}(v_i)$, and let $d_0=\min_j\rho^{\downarrow}(\phi^{\downarrow}f_j)-a_0$.  For a general element $\mu\in \Lambda^{\downarrow}$, so that $\mu$ takes the form $\sum_{g\in \Gamma}c_gT^g$ where $c_g\in\kappa$ and, for any $g_0\in \R$, only finitely many of the $c_g$ with $g\geq g_0$ are nonzero, let us write \[ \widehat{\mu}=\sum_{g\in \Gamma:g\geq d_0}c_gT^g \] for the sum of only those terms in the generalized power series defining $\mu$ for which the power of $T$ is at least $d_0$.  In particular, the sum defining $\widehat{\mu}$ above is a finite sum, so that $\widehat{\mu}\in \Lambda$.  Moreover, $\nu^{\downarrow}(\widehat{\mu}-\mu)<d_0$.  Hence for each $i$ and $j$ and each $\mu\in \Lambda^{\downarrow}$ we have \[ \rho^{\downarrow}\left((\widehat{\mu}-\mu)\phi^{\downarrow}f_j\right)>\rho^{\downarrow}(\phi^{\downarrow}f_j)-d_0\geq a_0\geq \rho^{\downarrow}(v_i).\]  So by Remark \ref{orthrev}, if we let \[ \widehat{v}_i=v_i+\sum_{j<i}(\widehat{\mu}_{ij}-\mu_{ij})\phi^{\downarrow}f_j=\phi^{\downarrow}f_i+\sum_{j<i}\widehat{\mu}_{ij}\phi^{\downarrow}f_j \] then $\{\widehat{v}_1,\ldots,\widehat{v}_d\}$ will be an orthogonal basis for $V^{\downarrow}$.  So define the elements $y_i\in F$ by \[ y_i=f_i+\sum_{j<i}\widehat{\mu}_{ij}f_j;\] by construction, $v_i=\phi^{\downarrow}y_i$ and so $\{\phi^{\downarrow}y_1,\ldots,\phi^{\downarrow}y_d\}$ is an orthogonal basis for $V^{\downarrow}$.  Moreover, since $\{y_1,\ldots,y_d\}$ is related to $\{f_1,\ldots,f_d\}$ by a triangular $\Lambda$-valued matrix with ones on the diagonal, the fact that $\{f_1,\ldots,f_d\}$ is a basis for the $\Lambda$-module $F$ (and not merely a maximal independent set) implies that $\{y_1,\ldots,y_d\}$ is also a basis for $F$.

Now that we have our basis $\{y_1,\ldots,y_d\}$, the construction of $\{x_1,\ldots,x_d\}$ proceeds in essentially the same way (adjusting for signs), using $\{y_1,\ldots,y_d\}$ as the initial basis instead of $\{f_1,\ldots,f_d\}$.  \cite[Theorem 2.16]{UZ} gives an orthogonal basis $\{w_1,\ldots,w_d\}$ for $V_{\uparrow}$ with each $w_i$ having the form \[ w_i=\phi_{\uparrow}y_i+\sum_{j<i}\lambda_{ij}\phi_{\uparrow}y_j \] with $\lambda_i\in \Lambda_{\uparrow}$.  If we let $\widehat{\lambda}_{ij}$ denote the element of $\Lambda$ that results from deleting from $\lambda_{ij}$ all terms involving powers $T^d$ with $d>\max_j\rho_{\uparrow}(\phi_{\uparrow}y_j)-\min_i\rho_{\uparrow}(w_i)$, then the elements $x_i=y_i+\sum_{j<i}\widehat{\lambda}_{ij}y_j$ evidently comprise a basis for $F$ which satisfies property (iii) in the lemma, and Proposition \ref{orthpres} (ii) implies that $\{\phi_{\uparrow}x_1,\ldots,\phi_{\uparrow}x_d\}$ is an orthogonal basis for $(V_{\uparrow},\rho_{\uparrow})$.
\end{proof}

With Lemma \ref{triconstruct}'s bases $\{x_1,\ldots,x_d\}$ and $\{y_1,\ldots,y_d\}$ for $F$ in hand, we set about describing a linear algebraic procedure to modify these until they equal a common doubly-orthogonal basis.  Let $M$ denote the $d\times d$ matrix with coefficients in $\Lambda$ relating our two bases in the sense that $x_j=\sum_{i=1}^{d}M_{ij}y_i$.  Thus by Lemma \ref{triconstruct}(iii), $M$ is \textbf{unitriangular}, \emph{i.e.} upper triangular with all diagonal entries equal to $1$.  Changes to the basis $\{x_1,\ldots,x_d\}$ (resp. $\{y_1,\ldots,y_d\}$) of course correspond to column (resp. row) operations on $M$; we would like to confine ourselves to row or column operations that correspond to basis changes that preserve the property that $\{\phi_{\uparrow}x_1,\ldots,\phi_{\uparrow}x_d\}$ and $\{\phi^{\downarrow}y_1,\ldots,\phi^{\downarrow}y_d\}$ are respectively $\rho_{\uparrow}$- and $\rho^{\downarrow}$-orthogonal.  To assist with this we keep track of, along with the change of basis matrix $M$, the values $\xi_i=\rho_{\uparrow}(\phi_{\uparrow}x_i)$ and $\eta_j=\rho^{\downarrow}(\phi^{\downarrow}y_j)$.

\begin{dfn} A \textbf{labeled basis change matrix} is a triple $\mathcal{M}=(M,\vec{\xi},\vec{\eta})$ where $M$ is an invertible matrix with coefficients in $\Lambda$, say of size $d$-by-$d$, and $\vec{\xi},\vec{\eta}\in \mathbb{R}^d$.
\end{dfn}  

It turns out to be helpful to introduce the following quantity:

\begin{dfn} \label{misdef} Let $\mathcal{M}=(M,\vec{\xi},\vec{\eta})$ be a labeled basis change matrix with the property that $M$ is unitriangular.  The \textbf{misalignment} of $\mathcal{M}$ is the number \[ m(\mathcal{M})=\sum_{i<j}|(\eta_i-\xi_i)-(\eta_j-\xi_j)|.\]
\end{dfn}

We  now consider some transformations of labeled basis change matrices $(M,\vec{\xi},\vec{\eta})$ which correspond to modifications of the bases $\{x_1,\ldots,x_d\}$ and/or $\{y_1,\ldots,y_d\}$ that preserve the desired orthogonality properties:

\begin{itemize}
\item[(A1)] Reordering the $y_i$, say by swapping $y_i$ with $y_j$, corresponds to swapping the $i$th and $j$th rows of $M$, and also swapping $\eta_i$ with $\eta_j$.
\item[(A2)] Multiplying the basis element $x_j$ by a unit in $\Lambda$ (which is necessarily of form $aT^g$ where $a\in \kappa^{\times}$ and $g\in \Gamma$) corresponds to multiplying the $j$th column of $M$ by $aT^g$ while subtracting $g$ from $\xi_j$.
\item[(A3)] By Proposition \ref{orthpres} and Remark \ref{orthrev}, the basis $\{y'_1,\ldots,y'_d\}$ obtained by, for some distinct $i$ and $j$, setting $y'_j=y_j+\mu y_i$ and all other $y'_k=y_k$, will continue to have the property that $\{\phi^{\downarrow}y'_1,\ldots,\phi^{\downarrow}y'_d\}$ is $\rho^{\downarrow}$-orthogonal provided that $\rho^{\downarrow}(\mu y_i)\geq \rho^{\downarrow}(y_j)$.  This corresponds to subtracting  $\mu$ times row $j$ from row $i$ while leaving all $\eta_k,\xi_k$ unchanged, subject to the condition that $\nu^{\downarrow}(\mu)\leq \eta_i-\eta_j$.
\item[(A4)] Similarly to (A3), Proposition \ref{orthpres} implies that the basis $\{x'_1,\ldots,x'_d\}$ obtained by, for some distinct $i$ and $j$, setting $x'_j=x_j+\mu x_i$ and all other $x'_k=x_k$ will continue to have $\{\phi_{\uparrow}x'_1,\ldots,\phi_{\uparrow}x'_d\}$ $\rho_{\uparrow}$-orthogonal provided that $\rho_{\uparrow}(\mu x_i)\leq \rho_{\uparrow}(x_j)$.  This corresponds to adding $\mu$ times column $i$ to column $j$ while leaving all $\eta_k,\xi_k$ unchanged, subject to the condition that $\nu_{\uparrow}(\mu)\geq \xi_i-\xi_j$.
\item[(A5)] Generalizing (A4), let $Z=\left(\begin{array}{cc} d_i & \mu \\ \lambda & d_j\end{array}\right)$ be a $2\times 2$ matrix that is invertible over $\Lambda$, with $\nu_{\uparrow}(d_i)=\nu_{\uparrow}(d_j)=0$ and  $\nu_{\uparrow}(\lambda)>\xi_j-\xi_i$ while $\nu_{\uparrow}(\mu)\geq\xi_i-\xi_j$ (corresponding to the conditions $\rho_{\uparrow}(\lambda x_j)<\rho_{\uparrow}(x_i)$ and $\rho_{\uparrow}(\mu x_i)\leq\rho_{\uparrow}(x_j)$).  Proposition \ref{orthpres} implies that the basis $\{x'_1,\ldots,x'_d\}$ obtained by setting $x'_i=d_ix_i+\lambda x_j$, $x'_j=\mu x_i+d_jx_j$, and all other $x'_k=x_k$ will continue to have $\{\phi_{\uparrow}x'_1,\ldots,\phi_{\uparrow}x'_d\}$ $\rho_{\uparrow}$-orthogonal.  The corresponding operation on labeled basis change matrices $(M,\vec{\xi},\vec{\eta})$ leaves $\vec{\xi},\vec{\eta}$ unchanged and multiplies $M$ on the right by the $d\times d$ matrix $E$ which coincides with the identity except that $E_{ii}=d_i,E_{ij}=\mu,E_{ji}=\lambda,E_{jj}=d_j$ (again, subject to the conditions that $\nu_{\uparrow}(d_i)=\nu_{\uparrow}(d_j)=0$, $\nu_{\uparrow}(\lambda)>\xi_j-\xi_i$, and $\nu_{\uparrow}(\mu)\geq \xi_i-\xi_j$).
\end{itemize}

In view of Lemma \ref{triconstruct}, to prove Theorem \ref{basistheorem} it thus suffices to show that, if $\mathcal{M}=(M,\vec{\xi},\vec{\eta})$ is any labeled basis change matrix such that $M$ is unitriangular, there is a sequence of operations of types (A1),(A2),(A3),(A4),(A5) as above which convert $M$ to the identity (for the corresponding bases $\{x'_1,\ldots,x'_d\}$ and $\{y'_1,\ldots,y'_d\}$ will then be equal to each other, and will continue to have the properties that $\{\phi_{\uparrow}x'_1,\ldots,\phi_{\uparrow}x'_d\}$ is $\rho_{\uparrow}$-orthogonal and $\{\phi^{\downarrow}y'_1,\ldots,\phi^{\downarrow}y'_d\}$ is $\rho^{\downarrow}$-orthogonal; thus $\{x'_1,\ldots,x'_d\}=\{y'_1,\ldots,y'_d\}$ will be a doubly orthogonal basis).   To facilitate an inductive argument, we will prove a  more specific statement that includes information about the labels $\vec{\xi}$ and $\vec{\eta}$.  Let us first make the following definition.

\begin{dfn}
Let $\alpha=(\alpha_1,\ldots,\alpha_d)\in\R^d$.  A \textbf{simple compression} of $\vec{\alpha}$ is an element $\vec{\beta}=(\beta_1,\ldots,\beta_d)\in \R^d$ such that, for some distinct $k,\ell\in\{1,\ldots,d\}$, we have \[ \beta_k+\beta_{\ell}=\alpha_k+\alpha_{\ell},\quad |\beta_k-\beta_{\ell}|<|\alpha_k-\alpha_{\ell}|,\,\,\mbox{ and }\beta_i=\alpha_i\mbox{ for all }i\notin\{k,\ell\}.\]  A \textbf{compression} of $\alpha$ is an element $\vec{\alpha}'\in\R^d$ such that, for some $m\geq 1$, there are $\vec{\alpha}^{(0)}=\alpha,\,\vec{\alpha}^{(1)},\ldots,\vec{\alpha}^{(m)}=\vec{\alpha}'$ such that $\vec{\alpha}^{(i+1)}$ is a simple compression of $\vec{\alpha}^{(i)}$ for each $i\in\{0,\ldots,m-1\}$.\end{dfn}

It will be helpful to know the following (\emph{cf}. Definition \ref{misdef}):

\begin{prop}\label{compmis}
Suppose that $\vec{\alpha}'$ is a compression of $\vec{\alpha}$.  Then \[ \sum_{i<j}|\alpha'_i-\alpha'_j|<\sum_{i<j}|\alpha_i-\alpha_j|.\]
\end{prop}

\begin{proof}
By transitivity it suffices to prove this in the case of a simple compression, so assume that, for some $k,\ell$, say with $\alpha'_k\leq \alpha'_{\ell}$, we have $\alpha'_k+\alpha'_{\ell}=\alpha_k+\alpha_{\ell}$, $|\alpha'_k-\alpha'_{\ell}|<|\alpha_k-\alpha_{\ell}|$, and $\alpha'_i=\alpha_i$ for all $i\notin\{k,\ell\}$.  For any $t\in \R$, if $t$ does not lie in the open interval $(\alpha'_k,\alpha'_{\ell})$ then \[ |t-\alpha'_k|+|t-\alpha'_{\ell}|=|2t-(\alpha'_{k}+\alpha'_{\ell})|=|2t-(\alpha_k+\alpha_{\ell})|\leq |t-\alpha_k|+|t-\alpha_{\ell}|;\] on the other hand, if $t$ does lie in $(\alpha'_k,\alpha'_{\ell})$ then \[ |t-\alpha'_k|+|t-\alpha'_{\ell}|=|\alpha'_k-\alpha'_{\ell}|<|\alpha_k-\alpha_{\ell}|\leq |\alpha_k-t|+|t-\alpha_{\ell}|.\]  So in fact $|t-\alpha'_{k}|+|t-\alpha'_{\ell}|\leq |t-\alpha_k|+|t-\alpha_{\ell}|$ for all $t\in \R$.  Applying this with $t=\alpha_{i}$ for $i\notin \{k,\ell\}$, we see that sum of just those terms in $\sum_{i<j}|\alpha'_i-\alpha'_j|$ for which $\{i,j\}\neq\{k,\ell\}$ is bounded above by the corresponding sum in $\sum_{i<j}|\alpha_i-\alpha_j|$.  So since $|\alpha'_k-\alpha'_{\ell}|<|\alpha_k-\alpha_{\ell}|$ the result follows.
\end{proof}

Here then is the main technical ingredient in the proof of Theorem \ref{basistheorem}.

\begin{prop}\label{reduction}
Let $\mathcal{M}=(M,\vec{\xi},\vec{\eta})$ be a labeled basis change matrix such that $M$ is a $d\times d$ unitriangular matrix, and assume that $\Gamma$ is discrete.  Then, by a sequence of operations of types (A1),(A2),(A3),(A4),(A5), $\mathcal{M}$ may be converted to a labeled basis change matrix of the form $\mathcal{M}'=(I,\vec{\xi}',\vec{\eta}')$ where $I$ is the $d\times d$ identity, and the tuple $(\eta'_1-\xi'_1,\ldots,\eta'_d-\xi'_d)$ either is equal to or is a compression of $(\eta_1-\xi_1,\ldots,\eta_d-\xi_d)$.
\end{prop}

\begin{proof}
The proof is by induction on the dimension, so we assume that the result holds for all unitriangular $k\times k$ labeled basis change matrices whenever $k<d$.  

If $M$ is $d\times d$ and unitriangular and $\ell\in \{2,\ldots,d\}$, let $M(\ell)$ denote the lower right $(d-\ell+1)\times (d-\ell+1)$ block of $M$ (so the entries of $M(\ell)$ are the $M_{ij}$ with both $i,j\geq \ell$).  Then $M(\ell)$ is still unitriangular.  
Given an operation on labeled $(d-\ell+1)\times(d-\ell+1)$ basis change matrices of type (A1)-(A5), that operation extends in obvious fashion to an operation, which is again of type (A1)-(A5), on labeled $d\times d$ basis change matrices: if the original operation is given at the level of matrices by left or right multiplication by a $(d-\ell+1)\times (d-\ell+1)$ matrix $E$ then the extended operation acts by left or right multiplication by the  block matrix $\left(\begin{array}{cc} I & 0\\ 0 & E\end{array}\right)$, and the  operation on labels $\vec{\xi},\vec{\eta}$ extends trivially.  The extended operation has the same effect on the lower right block $M(\ell)$ as did the original operation; it has no effect on the upper left $(\ell-1)\times (\ell-1)$ block of $M$; and while it might affect the other blocks for general matrices $M$ it preserves the condition that, in the lower left block, $M_{ij}=0$ for all $i\geq \ell$ and $j<\ell$.

With $\mathcal{M}=(M,\vec{\xi},\vec{\eta})$ as in the statement of the proposition, we may apply the inductive hypothesis to the $(d-1)\times (d-1)$ submatrix $M(2)$ with labels $(\xi_2,\ldots,\xi_d)$ and $(\eta_2,\ldots,\eta_d)$, and extend the sequence of operations to $M$ in the manner described in the previous paragraph.  This reduces us to the case that $M(2)$ is the identity matrix, so that $M$ has the form: \begin{equation}\label{1cleared} M=\left(\begin{array}{ccccc}1 & f_2 & f_3 & \cdots & f_d \\ 0 & 1 & 0 & \cdots & 0 \\ 0 & 0 & 1 & \cdots  & 0 \\ \vdots & \vdots & \vdots & \ddots & \vdots \\ 0 & 0 & 0 & \cdots & 1\end{array}\right),\end{equation} and our task now is to eliminate the entries $f_i$.     For $k\in\{1,\ldots,d\}$, let us say that a $d\times d$-matrix $M$ is $k$-\textbf{cleared} if $M_{ij}=\delta_{ij}$ (Kronecker delta) whenever either $i\geq 2$ or   $j\leq k$.  Thus a matrix as in (\ref{1cleared}) is $1$-cleared, and it is $k$-cleared iff $f_j=0$ for all $j\leq k$.  

The main steps in our reduction procedure are contained in the proof of the following lemma:

\begin{lemma}\label{mainbasisstep}
Let $\mathcal{M}=(M,\vec{\xi},\vec{\eta})$ be a labeled basis change matrix such that the $d\times d$ matrix $M$ is $k$-cleared, for some $k\in \{1,\ldots,d-1\}$.  Then there is a sequence of operations of type (A1)-(A5)  which converts $\mathcal{M}$ to a labeled basis change matrix $\mathcal{M}'=(M',\vec{\xi}',\vec{\eta}')$ such that either:
\begin{itemize} \item $M'$ is $(k+1)$-cleared, and $\vec{\xi}'=\vec{\xi}$ and $\vec{\eta}'=\vec{\eta}$; or \item $M'$ is  $k$-cleared, and  $(\eta'_1-\xi'_1,\ldots,\eta'_d-\xi'_d)$ is a compression of $(\eta_1-\xi_1,\ldots,\eta_d-\xi_d)$.
\end{itemize}
\end{lemma}

We now complete the proof of Proposition \ref{reduction} (and hence, as noted before the proposition, of Theorem \ref{basistheorem}) assuming Lemma \ref{mainbasisstep}, deferring the proof of the latter to the end of the section.

The main remaining observation is that, due to the discreteness of $\Gamma$, if we iteratively apply Lemma \ref{mainbasisstep} then there are only finitely many values that the misalignment $m$ can take during the iteration.  To see this, for any labeled basis change matrix $\mathcal{M}=(M,\vec{\xi},\vec{\eta})$ let \[ \mathcal{S}(\mathcal{M})=\{(\eta_i-\xi_j)-(\eta_k-\xi_{\ell})+g|1\leq i,j,k,\ell\leq d,\,g\in \Gamma\}.\]  Thus $\mathcal{S}(\mathcal{M})$ is a finite union of cosets of the discrete subgroup $\Gamma<\R$, and so $\mathcal{S}(\mathcal{M})$ is a discrete subset of $\R$.   Among the operations (A1)-(A5), the only ones that modify the labels   $\vec{\xi},\vec{\eta}$ are (A1), which swaps two of the $\eta_i$, and (A2), which subtracts an element of $\Gamma$ from one of the $\xi_i$. Since $\Gamma$ is an additive subgroup of $\R$ it follows that $\mathcal{S}(\mathcal{M}')=\mathcal{S}(\mathcal{M})$ whenever $\mathcal{M}'$ is obtained from $\mathcal{M}$ by an operation from among (A1)-(A5), and hence also whenever $\mathcal{M}'$ is obtained from $\mathcal{M}$ by a sequence of such operations.

Given $\mathcal{M}=(M,\vec{\xi},\vec{\eta})$ as in the statement of the proposition, as noted already we can apply the inductive hypothesis to $M(2)$ to reduce to the case that $M$ is $1$-cleared, setting us up to apply Lemma \ref{mainbasisstep}.  Let $\mathcal{M}^{(0)}=\mathcal{M}$, and, assuming inductively that $\mathcal{M}^{(i)}=(M^{(i)},\vec{\xi}^{(i)},\vec{\eta}^{(i)})$ is $k_i$-cleared with $k_i<d$, let $\mathcal{M}^{(i+1)}$ result from applying Lemma \ref{mainbasisstep} to $\mathcal{M}^{(i)}$. By the preceding paragraph, $\mathcal{S}(\mathcal{M}^{(i)})=\mathcal{S}(\mathcal{M}^{(0)})$ for all $i$. Also, by Proposition \ref{compmis}, the misalignments $m(\mathcal{M}^{(i)})$ obey $m(\mathcal{M}^{(i+1)})\leq m(\mathcal{M}^{(i)})$, with equality only when $M^{(i+1)}$ is $(k_i+1)$-cleared.  Now the misalignments $m(\mathcal{M}^{(i)})$ form a nonincreasing sequence within the set of finite sums of absolute values of elements of the discrete set $\mathcal{S}(\mathcal{M}^{(0)})$; this set of sums is itself a discrete set of nonnegative numbers.  But any nonincreasing sequence in a discrete set of nonnegative numbers must eventually stabilize: there is $i_0$ such that $m(\mathcal{M}^{(i+1)})=m(\mathcal{M}^{i})$ whenever $i\geq i_0$.  After this point in the iteration, the first alternative in Lemma \ref{mainbasisstep} must always hold, and so since $M^{(i_0)}$ is $k_{i_0}$-cleared  it follows that $M^{(i_0+d-k_{i_0})}$ is $d$-cleared, \emph{i.e.} is equal to the identity matrix.  Thus the proposition holds with $\mathcal{M}'=\mathcal{M}^{(i_0+d-k_{i_0})}$, completing the proof modulo Lemma \ref{mainbasisstep}.
\end{proof}

\begin{proof}[Proof of Lemma \ref{mainbasisstep}]
Our input is a labeled basis change matrix $\mathcal{M}=(M,\vec{\xi},\vec{\eta})$ such that $M_{ij}=\delta_{ij}$ for all $i\geq 2$, and $M_{1j}=\delta_{1j}$ for all $j\leq k$, so that the first (possibly) nonzero entry off of the diagonal in the first row of $M$ is $f:=M_{1,k+1}$.  Our task is to use operations (A1)-(A5)  either to eliminate $f$ while preserving $\vec{\xi}$ and $\vec{\eta}$, or else to preserve the property of being $k$-cleared while compressing the vector $\vec{\eta}-\vec{\xi}$. 

By definition, $f$ is an element of the group algebra $\Lambda=\kappa[\Gamma]$, so takes the form $f=\sum_{g\in \Gamma}a_gT^g$ for some $a_g\in \kappa$ (only finitely many of which are nonzero).  Let us split this sum up as \[ f=\underbrace{\sum_{g\leq \eta_1-\eta_{k+1}}a_gT^g}_{f_-}+\underbrace{\sum_{\eta_1-\eta_{k+1}<g<\xi_1-\xi_{k+1}}a_gT^g}_{f_0} + \underbrace{\sum_{g\geq \xi_1-\xi_{k+1}}a_gT^g}_{f_+}.\]  (If $\eta_{1}-\eta_{k+1}\geq \xi_{1}-\xi_{k+1}$ such a decomposition is not unique, but we  choose one.)  Thus $\nu^{\downarrow}(f_-)\leq \eta_1-\eta_{k+1}$, and $\nu_{\uparrow}(f_+)\geq \xi_1-\xi_{k+1}$.  So we can apply operation (A3), subtracting $f_-$ times row $k+1$ from row $1$; since row $k+1$ is the $(k+1)$th standard basis vector this has the sole effect of changing $M_{1,k+1}$ from $f$ to $f_0+f_+$.  Then apply operation (A4), subtracting $f_+$ times column $1$ from column $k+1$, after which $M_{1,k+1}$ will be equal to $f_0$ but the rest of the matrix $M$, and also the labels $\vec{\xi}$ and $\vec{\eta}$, will be unchanged from what they were at the start of the proof.  By construction, we have \begin{equation}\label{nuf0} \nu_{\uparrow}(f_0)>\eta_1-\eta_{k+1}\mbox{  and }\nu^{\downarrow}(f_0)<\xi_1-\xi_{k+1}.\end{equation}

If $f_0=0$ (as will automatically be the case if $\eta_{1}-\eta_{k+1}\geq \xi_{1}-\xi_{k+1}$) then we are done: the matrix is now $(k+1)$-cleared and the labels $\vec{\xi}$ and $\vec{\eta}$ are unchanged, so the first alternative in the conclusion of the lemma holds.

So suppose for the rest of the proof that $f_0\neq 0$.  The lowest degree term in $f_0$ then takes the form $aT^g$ where $a\in \kappa^{\times}$ and $\eta_1-\eta_{k+1}<g< \xi_1-\xi_{k+1}$; let us factor this out and express $f_0$ in the form \[ f_0=aT^g(1-r) \mbox{ where }a\in\kappa^{\times},\,\eta_1-\eta_{k+1}<g< \xi_1-\xi_{k+1}\,\,\nu_{\uparrow}(r)>0.\]  (This includes the possibility that $r=0$, as $\nu_{\uparrow}(0)=\infty$.)   Since $\nu_{\uparrow}$ take products to sums, and since $\nu_{\uparrow}(r)>0$, we may choose $N\in \N$ such that \begin{equation}\label{nupower} \nu_{\uparrow}(r^{N+1})>\xi_1-\xi_{k+1}-g. \end{equation}  We shall apply (A5) with $i=1$, $j=k+1$, and \[ Z=\left(\begin{array}{cc} -a(1-r) & T^gr^{N+1}\\ T^{-g} & a^{-1}(1+r+\cdots+r^{N})\end{array}\right).\]  Note that this $Z$ is invertible over $\Lambda$: the relation $(1-r)(1+r+\cdots+r^{N})=1-r^{N+1}$ implies that $\det Z=-1$.  Also, $\nu_{\uparrow}(T^{-g})=-g>\xi_{k+1}-\xi_1$ and $\nu_{\uparrow}(T^gr^{N+1})=g+\nu_{\uparrow}(r^{N+1})>\xi_1-\xi_{k+1}$ so the conditions on $Z$ required in (A5) are indeed satisfied.  Applying this transformation leaves all of the data in $\mathcal{M}$ unchanged except that the principal submatrix of $M$ corresponding to rows and columns $1$ and $k+1$ becomes \[ \left(\begin{array}{cc} 0 & T^gr^{N+1}+T^g(1-r^{N+1}) \\ T^{-g} & a^{-1}(1+r+\cdots+r^N)\end{array}\right)=\left(\begin{array}{cc} 0 & T^g \\ T^{-g} & a^{-1}(1+r+\cdots+r^N)\end{array}\right).\]  
  Next, apply (A1) (swapping rows $1$ and $k+1$), and then apply (A2), multiplying column $1$ by $T^g$ and column $k+1$ by $T^{-g}$. The result of this sequence of operations is a labeled basis change matrix $\mathcal{M}''=(M'',\vec{\xi}'',\vec{\eta}'')$ where $M''$ is a unitriangular matrix of form 
   \begin{equation}\label{bigmatrix} \left(\begin{array}{cccccc} 
	1 & \cdots & a^{-1}T^{-g}(1+r+\cdots+r^N) & 0& \cdots & 0 \\
	0 & 1 & 0 & \cdots & \cdots & 0\\ 
	\vdots & & & & & \vdots \\
	0 & \cdots & 1 & \star & \cdots & \star \\ 
	\vdots & & & \ddots & & \vdots \\
	\vdots & & & & \ddots & \vdots \\
	0 & \cdots & \cdots & \cdots & 0 & 1\end{array}\right) \end{equation} and where the labels $\vec{\xi}'',\vec{\eta}''$ coincide with $\vec{\xi},\vec{\eta}$ except that \begin{equation}\label{primeprime} \xi''_{1}=\xi_1-g,\quad \xi''_{k+1}=\xi_{k+1}+g,\quad \eta''_{1}=\eta_{k+1},\quad \eta''_{k+1}=\eta_1.\end{equation}  
In particular, we have \begin{equation}\label{sumsame}   
(\eta''_1-\xi''_1)+(\eta''_{k+1}-\xi''_{k+1})=(\eta_1-\xi_1)+(\eta_{k+1}-\xi_{k+1}).
\end{equation}  Also, \[ (\eta''_1-\xi''_1)-(\eta_1-\xi_1)=(\eta_{k+1}-\eta_1)+g>0 \] and \[ (\eta''_1-\xi''_1)-(\eta_{k+1}-\xi_{k+1})=\xi_{k+1}-\xi_1+g<0,\] where we have twice used (\ref{nuf0}), as $f_0=aT^g(1-r)$ has $g=\nu_{\uparrow}(f_0)\leq \nu^{\downarrow}(f_0)$.

By (\ref{sumsame}), since $\eta''_1-\xi''_1$ lies strictly between $\eta_1-\xi_1$ and $\eta_{k+1}-\xi_{k+1}$, we see that $\eta''_{k+1}-\xi''_{k+1}$ also lies strictly between $\eta_1-\xi_1$ and $\eta_{k+1}-\xi_{k+1}$ and hence that $|(\eta''_{1}-\xi''_1)-(\eta''_{k+1}-\xi''_{k+1})|<|(\eta_1-\xi_1)-(\eta_{k+1}-\xi_{k+1})|$.  Thus $\vec{\eta}''-\vec{\xi}''$ is a simple compression of $\vec{\eta}-\vec{\xi}$. 

Now $\mathcal{M}''$ is, perhaps, not $k$-cleared because our row swap may have introduced nonzero terms above the diagonal in row $k+1$.  However, we can apply the inductive hypothesis to the $(d-k)\times (d-k)$ lower right block $M''(k+1)$ with labels $\vec{\xi}''(k+1)=(\xi''_{k+1},\ldots,\xi''_d)$ and $\vec{\eta}''(k+1)=(\eta''_1,\ldots,\eta''_d)$ to obtain operations of type (A1)-(A5) that convert $M''(k+1)$ to the $(d-k)\times(d-k)$ identity matrix, with new labels $\vec{\xi}'(k+1)=(\xi'_{k+1},\ldots,\xi'_d),\vec{\eta}'(k+1)=(\eta'_{k+1},\ldots,\eta'_d)$ such that $\vec{\eta}'(k+1)-\vec{\xi}'(k+1)$ either is equal to or is a compression of $\vec{\eta}''(k+1)-\vec{\xi}''(k+1)$.  When these operations on the lower right block $M''(k+1)$ are extended in the natural way to operations on the entire matrix $M''$, they do not affect the fact that columns $1$ through $k$ and rows $2$ through $k$ of $M''$ coincide with the corresponding columns and rows of the identity matrix, while they convert the lower $(d-k)\times (d-k)$ block to the identity, and so the resulting matrix $M'$ is still $k$-cleared.  Moreover the labels $\vec{\xi}''$ and $\vec{\eta}''$ are transformed, respectively, to $\vec{\xi}':=(\xi''_1,\ldots,\xi''_{k},\xi'_{k+1},\ldots,\xi'_d)$ and $\vec{\eta}'=
(\eta''_1,\ldots,\eta''_k,\eta'_{k+1},\ldots,\eta'_d)$.  The fact that $\vec{\eta}'(k+1)-\vec{\xi}'(k+1)$ either is equal to or is a compression of $\vec{\eta}''(k+1)-\vec{\xi}''(k+1)$ implies that $\vec{\eta'}-\vec{\xi}'$ either is equal to or is a compression of $\vec{\eta}''-\vec{\xi}''$.  So since $\vec{\eta}''-\vec{\xi}''$ is a simple compression of $\vec{\eta}-\vec{\xi}$, it follows that $\vec{\eta}'-\vec{\xi}'$ is a compression  of $\vec{\eta}-\vec{\xi}$.  Our new labeled basis change matrix $\mathcal{M}'=(M',\vec{\xi}',\vec{\eta}')$ thus satisfies the second alternative of the conclusion of Lemma \ref{mainbasisstep}, completing the proof.\end{proof}

\section{Poincar\'e-Novikov structures}\label{pnsect}

We now introduce a new type of structure that gives rise to filtered matched pairs that satisfy a version of Ponicar\'e duality.  We continue to work with respect to a fixed field $\kappa$ and finitely generated subgroup $\Gamma<\R$.

\begin{dfn}\label{pnstr}
A weak (resp. strong) $n$-\textbf{Poincar\'e-Novikov structure} $\mathcal{N}$ consists of the following data: \begin{itemize} \item A graded $\Lambda$-module $H_*=\oplus_{k\in \Z}H_k$ with each $H_k$ finitely generated over $\Lambda$, and maps $\mathcal{D}_k\co H_k\to {}^{\vee}\!H_{n-k}$ that comprise the data of a weak (resp. strong) $n$-PD structure (see Definition \ref{pdstr}); 
\item for each $k\in \Z$, a normed $\Lambda_{\uparrow}$-space $(V_k,\rho_k)$; and
\item for each $k\in \Z$, a homomorphism of $\Lambda$-modules $S_k\co H_k\to V_k$ such that $1_{\Lambda_{\uparrow}}\otimes S_k\co \Lambda_{\uparrow}\otimes_{\Lambda}H_k\to V_k$ is an isomorphism of $\Lambda_{\uparrow}$-vector spaces.
\end{itemize}
\end{dfn}

If $\mathcal{N}$ is as in Definition \ref{pnstr}, for each $k\in \Z$ we obtain a filtered matched pair $\mathcal{P}(\mathcal{N})_k$ of the form \begin{equation}\label{pnkdef} \xymatrix{ & & ({}^{\vee}(V_{n-k}),{}^{\vee}\rho_{n-k}) \\ \mathcal{P}(\mathcal{N})_k  \ar@{}[r]^(.45){}="a"^(.7){}="b" \ar@{=} "a";"b"    & H_k\ar[ru]^{\widetilde{S}_k} \ar[rd]_{S_k} & \\ & & (V_k,\rho_k)} \end{equation} where ${}^{\vee}\rho_{n-k}$ is defined from $\rho_{n-k}$ as in Section \ref{fmpdual} and where $\widetilde{S}_k=\delta^{\downarrow}(S_{n-k})\circ \mathcal{D}_k$. Thus, in view of Proposition \ref{deltachar}, the map $\widetilde{S}_k\co H_k\to {}^{\vee}\!(V_{n-k})$ is characterized by the property that, for all $x\in H_k$ and $y\in H_{n-k}$, we have \begin{equation}\label{tildesk}  (\widetilde{S}_kx)(S_{n-k}y)=(\mathcal{D}_kx)(y).\end{equation}

We now introduce a barcode associated to a Poincar\'e-Novikov structure. Elements of barcodes in this paper will be taken to be intervals modulo $\Gamma$-translation, expressed as follows:

\begin{notation}\label{intnot}
If $I$ is a nonempty interval in $\R$ (of any type $[a,b),(a,b],[a,b],(a,b)$, allowing $b=\infty$ in the first and fourth cases and $a=-\infty$ in the second and fourth cases), we denote by $I^{\Gamma}$ the equivalence class of $I$ under the equivalence relation $\sim$ on the collection $\mathcal{I}$ of all   intervals in $\R$ according to which, for $I,J\in\mathcal{I}$, $I\sim J$ if and only if there is $g\in \Gamma$ such that the translation $x\mapsto x+g$ maps $I$ bijectively to $J$.
\end{notation}

 Our conventions in the following definition are motivated by the outcome of the calculation in Section \ref{mabsect}; see also Definition \ref{fullbar}.

\begin{dfn}\label{pnbar}
Let $\mathcal{N}$ be a weak $n$-Poincar\'e-Novikov structure such that each $\mathcal{P}(\mathcal{N})_k$ admits a doubly-orthogonal basis.  The \textbf{essential barcode} of $\mathcal{N}$ is the collection $\{\mathcal{B}_k(\mathcal{N})|k\in\mathbb{Z}\}$  of multisets of intervals modulo $\Gamma$-translation, where $\mathcal{B}_k(\mathcal{N})$ consists of the elements $[a,a+\ell]^{\Gamma}$ for each $([a],\ell)\in \Sigma(\mathcal{P}(\mathcal{N})_k)$ with $\ell\geq 0$, together with the elements $(a+\ell,a)^{\Gamma}$ for each $([a],\ell)\in \Sigma(\mathcal{P}(\mathcal{N})_{k+1})$ with $\ell<0$.  (Here $\mathcal{P}(\mathcal{N})_k$ is defined in (\ref{pnkdef}).)\end{dfn}

Under weaker hypotheses we can adapt the gaps of Definition \ref{gapdfn} as follows:
\begin{dfn}\label{pngap}
If $\mathcal{N}$ is a weak $n$-Poincar\'e-Novikov structure with modules $H_k$ as in Definition \ref{pnstr} having $d_k=\dim_{\mathcal{Q}(\Lambda)}\mathcal{Q}(\Lambda)\otimes_{\Lambda}H_k$, for any $k\in \Z$ and $i\in \{1,\ldots,d_k\}$ we define \[ \mathcal{G}_{k,i}(\mathcal{N})=G_i(\mathcal{P}(\mathcal{N})_k).\]
\end{dfn}

It is straightforward to adapt the stability results for filtered matched pairs from Section \ref{stabsec} to prove analogous results for Poincar\'e-Novikov structures.  First we define the appropriate notion of approximate isomorphism:

\begin{dfn}\label{pntmorph}  Let $\mathcal{N}$ and $\mathcal{N}'$ be two weak $n$-Poincar\'e-Novikov structures, having the same modules $H_k$ and duality morphisms $\mathcal{D}_k\co H_k\to {}^{\vee}\!H_{n-k}$ and (possibly) different normed $\Lambda_{\uparrow}$-spaces $(V_k,\rho_k)$, $(V'_k,\rho'_k)$ and $\Lambda$-module homomorphisms $S_k\co H_k\to V_k$ and $S'_k\co H_k\to V'_k$.  If $t\geq 0$, a $t$-\textbf{morphism} from $\mathcal{N}$ to $\mathcal{N}'$ consists of, for each $k$,  $\Lambda_{\uparrow}$-linear isomorphisms $a_k\co V_k\to V'_k$ such that $S'_k=a_k\circ S_k$ and, for all $x\in V_k$, $|\rho'_k(a_kx)-\rho_k(x)|\leq t$.
\end{dfn}

\begin{prop}\label{pntofmp}
If there exists a $t$-morphism between the weak $n$-Poincar\'e-Novikov structures $\mathcal{N}$ and $\mathcal{N}'$, then for each $k$ there is a $t$-morphism (in the sense of Definition \ref{tmorphdfn}) between the corresponding filtered matched pairs $\mathcal{P}(\mathcal{N})_k$ and $\mathcal{P}(\mathcal{N}')_k$.
\end{prop}

\begin{proof}
For the maps $\alpha_{\uparrow}$ of Definition \ref{tmorphdfn} we may use $a_k$ from Definition \ref{pntmorph}, as by definition $a_k$ obeys $S'_k=a_k\circ S_k$ and $|\rho'_k\circ S'_k-\rho_k\circ S_k|\leq t$.  For the map $\alpha^{\downarrow}\co {}^{\vee}\!V_{n-k}\to {}^{\vee}\!V'_{n-k}$, since $a_{n-k}$ is a vector space isomorphism we may use the map ${}^{\vee}a_{n-k}^{-1}$.  Indeed, if $x\in H_k$ and $y\in H_{n-k}$ then by (\ref{tildesk}) we have $(\tilde{S}_{k}x)(S_{n-k}y)=(\tilde{S}'_kx)(S'_{n-k}y)=(\mathcal{D}_kx)(y)$ and hence \[ ({}^{\vee}a_{n-k}^{-1}\tilde{S}_kx)(S'_{n-k}y)=(\tilde{S}_kx)(a_{n-k}^{-1}S'_{n-k}y)=(\tilde{S}_{k}x)(S_{n-k}y)=(\tilde{S}'_kx)(S'_{n-k}y),\] confirming that ${}^{\vee}a_{n-k}^{-1}\tilde{S}_k=\tilde{S}'_k$ since the image of $S'_{n-k}$ spans $V'_{n-k}$ over $\Lambda_{\uparrow}$.  Moreover if $\zeta\in {}^{\vee}\!V'_{n-k}$ then \[ {}^{\vee}\!\rho_{n-k}({}^{\vee}\!a_{n-k}\eta)=\inf_{0\neq x\in V_{n-k}}\left(\nu_{\uparrow}(\zeta(a_{n-k}x))+\rho_k(x)\right) \] which differs from \[ {}^{\vee}\!\rho'_{n-k}(\eta)=\inf_{0\neq y\in V'_{n-k}}(\nu_{\uparrow}(\zeta(y))+\rho'_k(y)) \] by at most $t$ since $a_{n-k}$ is surjective and $|\rho'_{n-k}\circ a_{n-k}-\rho_{n-k}|\leq t$ by hypothesis.  This shows that $|{}^{\vee}\!\rho'_{n-k}\circ {}^{\vee}\!a_{n-k}^{-1}-{}^{\vee}\!\rho_{n-k}|\leq t$, so the continuity requirement on $\alpha^{\downarrow}={}^{\vee}\!a_{n-k}^{-1}$ is satisfied.
\end{proof}

We can now read off the following stability results.

\begin{cor}\label{pngapstab}
If $\mathcal{N}$ and $\mathcal{N}'$ are weak $n$-Poincar\'e-Novikov structures between which there exists a $t$-morphism, then for all $k$ and $i$ we have $|\mathcal{G}_{k,i}(\mathcal{N})-\mathcal{G}_{k,i}(\mathcal{N}')|\leq 2t$.
\end{cor}
\begin{proof}
This is immediate from Propositions \ref{gapcts} and \ref{pntofmp}.
\end{proof}

\begin{cor}\label{pnstab}
Assume that $\Gamma$ is discrete and that $\mathcal{N}$ and $\mathcal{N}'$ are two weak $n$-Poincar\'e-Novikov structures between which there exists a $t$-morphism. Then there is a bijection $\sigma\co \cup_k\mathcal{B}_k(\mathcal{N})\to\cup\mathcal{B}_k(\mathcal{N}')$ such that each $[a,b]^{\Gamma}\in \mathcal{B}_k(\mathcal{N})$ has $\sigma([a,b]^{\Gamma})$ equal either to some $[a',b']^{\Gamma}\in\mathcal{B}_k(\mathcal{N}')$ or to some $(b',a')^{\Gamma}\in\mathcal{B}_{k-1}(\mathcal{N}')$ where in either case $|a'-a|\leq t$ and $|b'-b|\leq t$, and similarly each $(a,b)^{\Gamma}\in\mathcal{B}_k(\mathcal{N}')$ has $\sigma((a,b)^{\Gamma})$ equal either to some $(a',b')^{\Gamma}\in\mathcal{B}_k(\mathcal{N}')$ or to some $[b',a']^{\Gamma}\in\mathcal{B}_{k+1}(\mathcal{N}')$ where in either case $|a'-a|\leq t$ and $|b'-b|\leq t$,
\end{cor}
\begin{proof}
This is immediate from Theorem \ref{stab}, Proposition \ref{pntofmp}, and the conventions in Definition \ref{pnbar}.
\end{proof}


We now develop consequences of the duality analysis in Section \ref{fmpdual}.  First, we find the duals of the filtered matched pairs $\mathcal{P}(\mathcal{N})_k$ associated to a Poincar\'e-Novikov structure via (\ref{pnkdef}).

\begin{prop}\label{dualpnchar} For a weak $n$-Poincar\'e Novikov structure with notation as above, the dual filtered matched pair  ${}^{\vee}\!\mathcal{P}(\mathcal{N})_k$ (in the sense of Section \ref{fmpdual}) takes the form \[ \xymatrix{ & {}^{\vee}\! V_k \\ {}^{\vee}\!H_k \ar[ru]^{\delta^{\downarrow}(S_k)} \ar[rd]_{F_k} \\ & {}^{\vee}\!({}^{\vee}V_{n-k}) } \]
where the map $F_k$ satisfies \[ F_k\circ \mathcal{D}_{n-k}=\pm \alpha_{V_{n-k}}\circ S_{n-k} \] where $\pm$ is the sign from (\ref{pdsym}) and $\alpha_{V_{n-k}}\co V_{n-k}\to {}^{\vee}\!({}^{\vee}V_{n-k})$ is the isomorphism from Section \ref{dualsec}.
\end{prop}

\begin{proof} That the map ${}^{\vee}\!H_k\to {}^{\vee}\!V_k$ in ${}^{\vee}\!\mathcal{P}(\mathcal{N})_k$ is equal to $\delta^{\downarrow}(S_{k})$ is true by definition.  The other map $F_k$ is equal by definition to $\delta_{\uparrow}(\tilde{S}_k)=\delta_{\uparrow}(\delta^{\downarrow}(S_{n-k})\circ \mathcal{D}_k)$, so we are to show that, for all $y\in H_{n-k}$, \begin{equation}\label{dualnoveqn} \delta_{\uparrow}(\delta^{\downarrow}(S_{n-k})\circ \mathcal{D}_k)(\mathcal{D}_{n-k}y)=\pm \alpha_{V_{n-k}}(S_{n-k}y)\end{equation} as elements of ${}^{\vee}\!({}^{\vee}\!V_{n-k})$.  Since both $\mathcal{D}_k\co H_k\to {}^{\vee}\!H_{n-k}$ and $\delta^{\downarrow}(S_{n-k})\co {}^{\vee}\!H_{n-k}\to {}^{\vee}\!V_{n-k}$ become isomorphisms after tensoring with the field $\Lambda^{\downarrow}$, it suffices to show that the two sides of (\ref{dualnoveqn}) are equal when evaluated on arbitrary elements of the image of $\delta^{\downarrow}(S_{n-k})\circ \mathcal{D}_k$.  To check this, we find from Proposition \ref{deltachar} that for any $x\in H_k$, \[ 
\left(\delta_{\uparrow}(\delta^{\downarrow}(S_{n-k})\circ \mathcal{D}_k)(\mathcal{D}_{n-k}y)\right)((\delta^{\downarrow}(S_{n-k})\circ \mathcal{D}_k)x)=(\mathcal{D}_{n-k}y)(x),\] while \[ \left(\alpha_{V_{n-k}}(S_{n-k}y)\right)((\delta^{\downarrow}(S_{n-k})\circ \mathcal{D}_k)x)=\overline{\left(\delta^{\downarrow}(S_{n-k})(\mathcal{D}_kx)\right)(S_{n-k}y)}=\overline{(\mathcal{D}_kx)(y)}.\] Thus the desired equality is 
an expression of the symmetry property (\ref{pdsym}).\end{proof}

\begin{cor}\label{pndual}
If $\mathcal{N}$ is a strong $n$-Poincar\'e-Novikov structure such that each $\mathcal{P}(\mathcal{N})_k$ admits a doubly-orthogonal basis, there is a bijection from $\cup_k\mathcal{B}_k(\mathcal{N})$ to itself which pairs each element $[a,b]^{\Gamma}\in\mathcal{B}_k(\mathcal{N})$ (resp. $(a,b)^{\Gamma}\in \mathcal{B}_k(\mathcal{N})$) such that $a<b$ with an element $(a,b)^{\Gamma}$ (resp. $[a,b]^{\Gamma}$) of $\mathcal{B}_{n-1-k}(\mathcal{N})$, and which pairs any element of form $[a,a]^{\Gamma}\in\mathcal{B}_k(\mathcal{N})$ with an element $[a,a]^{\Gamma}$ of $\mathcal{B}_{n-k}(\mathcal{N})$.
\end{cor}

\begin{proof} Note that, under our usual correspondence between elements $([a],\ell)\in(\R/\Gamma)\times\R$ and intervals modulo $\Gamma$-translation $[a,a+\ell]^{\Gamma}$ (for $\ell\geq 0$) and $(a+\ell,a)^{\Gamma}$ (for $\ell<0$), the involution $([a],\ell)\leftrightarrow([a+\ell],-\ell)$ corresponds to interchanging $[a,b]^{\Gamma}$ with $(a,b)^{\Gamma}$ (except in the case that $\ell=0$ in which case it has no effect). So
in view of Proposition \ref{bardual}, the statement is equivalent to the claim that, for each $k$, the basis spectra $\Sigma(\mathcal{P}(\mathcal{N})_{n-k})$ and $\Sigma({}^{\vee}\!\mathcal{P}(\mathcal{N})_k)$ coincide.   By hypothesis, with notation as in Definition \ref{pnstr}, $\mathcal{D}_{n-k}\co H_{n-k}\to{}^{\vee}\!H_k$ is a surjection with kernel equal to the torsion module $tH_{n-k}$.  Hence by Proposition \ref{dualpnchar}, a doubly-orthogonal basis for ${}^{\vee}\mathcal{P}(\mathcal{N})_{k}$ lifts via $\mathcal{D}_{n-k}$ to a doubly orthogonal basis for the filtered matched pair \begin{equation}\label{otherfmp} \xymatrix{ & {}^{\vee}\! V_k \\ H_{n-k} \ar[ru]^{\delta^{\downarrow}(S_k)\circ\mathcal{D}_{n-k}} \ar[rd]_{\pm \alpha_{V_{n-k}}\circ S_{n-k}} \\ & {}^{\vee}\!({}^{\vee}V_{n-k}) 
}.\end{equation}  By definition, the map $\delta^{\downarrow}(S_k)\circ\mathcal{D}_{n-k}$ is equal to the map $\widetilde{S}_{n-k}$ in the definition of $\mathcal{P}(\mathcal{N})_{n-k}$.  Also, it follows from Remark \ref{dualbasisrem} that any orthogonal basis $\{v_1,\ldots,v_d\}$ for $V_{n-k}$ will be mapped by the isomorphism $\pm\alpha_{V_{n-k}}\co V_{n-k}\to {}^{\vee\vee}\!V_{n-k}$ to an orthogonal basis for  ${}^{\vee\vee}\!V_{n-k}$, with ${}^{\vee\vee}\rho_{\uparrow}(\pm\alpha_{V_{n-k}}v_i)=\rho_{\uparrow}(v_i)$ for each $i$.   Thus any doubly-orthogonal basis for $\mathcal{P}(\mathcal{N})_{n-k}$ is also a doubly-orthogonal basis for the filtered matched pair in (\ref{otherfmp}), and the values of the various versions of the functions $\rho_{\uparrow},\rho^{\downarrow}$ coincide under this correspondence.  So indeed $\Sigma(\mathcal{P}(\mathcal{N})_{n-k})=\Sigma({}^{\vee}\!\mathcal{P}(\mathcal{N})_k)$ 
\end{proof}

Under weaker hypotheses on $\mathcal{N}$ we still have a duality inequality for the gaps $\mathcal{G}_{k,i}(\mathcal{N})$:

\begin{cor}\label{weakdual}
Let $\mathcal{N}$ be a weak $n$-Poincar\'e-Novikov structure and for each $k$ let $d_k$ be as in Definition \ref{pngap}.  Then $d_{n-k}=d_k$, and for all $i\in \{1,\ldots,d_k\}$ we have $\mathcal{G}_{n-k,i}(\mathcal{N})+\mathcal{G}_{k,d_k+1-i}(\mathcal{N})\leq 0$.
\end{cor}

\begin{proof}  That $d_k=d_{n-k}$ follows from the fact that, with notation as in Definition \ref{pnstr}, $\mathcal{Q}(\Gamma)\otimes_{\Lambda}H_k$ is isomorphic (by the coefficient extension of $\mathcal{D}_k$) to $\mathcal{Q}(\Lambda)\otimes_{\Lambda}{}^{\vee}\!H_{n-k}$.

By definition we have $\mathcal{G}_{n-k,i}(\mathcal{N})=G_i(\mathcal{P}(\mathcal{N})_{n-k})$ and $\mathcal{G}_{k, d_{k}+1-i}(\mathcal{N})=G_{d_k+1-i}(\mathcal{P}(\mathcal{N})_k)$, while Proposition \ref{gapdual} shows that $G_{i}({}^{\vee}\!\mathcal{P}(\mathcal{N})_k)+G_{d_k+1-i}(\mathcal{P}(\mathcal{N})_k)\leq 0$.  So it suffices to show that $G_i(\mathcal{P}(\mathcal{N})_{n-k})\leq G_{i}({}^{\vee}\!\mathcal{P}(\mathcal{N})_k)$.  By Proposition \ref{dualpnchar}, the maps $\tilde{S}_{n-k}\co H_{n-k}\to {}^{\vee}\!V_k$ and $S_{n-k}\co H_{n-k}\to V_{n-k}$ that define $\mathcal{P}(\mathcal{N})_{n-k}$ factor through the analogous maps $\delta^{\downarrow}(S_k)$ and $F_k$ for ${}^{\vee}\!\mathcal{P}(\mathcal{N})_k$ as, respectively, \[ \xymatrix{ H_{n-k}\ar[r]^{\mathcal{D}_{n-k}} & {}^{\vee}\!H_k\ar[r]^{\delta^{\downarrow}(S_k)} & {}^{\vee}\!V_k } \] and \[ \xymatrix{
H_{n-k}\ar[r]^{\mathcal{D}_{n-k}} & {}^{\vee}\!H_k\ar[r]^{F_k} & {}^{\vee}\!({}^{\vee}\!V_{n-k})\ar[r]^{\pm\alpha_{V_{n-k}}^{-1}} & V_{n-k}}.\] Moreover by Remark \ref{dualbasisrem} we have ${}^{\vee\vee}\!\rho_{n-k}(x)=\rho_{n-k}(\pm\alpha_{V_{k}}^{-1}x)$ for all $x\in {}^{\vee}\!({}^{\vee}\!V_{n-k})$.  Our hypotheses imply that $\mathcal{D}_{n-k}$, while not necessarily surjective, maps independent subsets of $H_{n-k}$ to independent subsets of ${}^{\vee}\!H_k$.  Consequently the set of which Definition \ref{gapdfn} defines $G_{i}({}^{\vee}\!\mathcal{P}(\mathcal{N})_k)$ as the supremum contains the corresponding set for $G_i(\mathcal{P}(\mathcal{N})_{n-k})$.  So indeed $G_i(\mathcal{P}(\mathcal{N})_{n-k})\leq G_{i}({}^{\vee}\!\mathcal{P}(\mathcal{N})_k)$.  
\end{proof}

\begin{ex}\label{weakdualstrict}
We provide an example showing that, if $\mathcal{N}$ is not a strong Poincar\'e--Novikov structure, then the inequality in Corollary \ref{weakdual} may be strict. In the notation of Definition \ref{pnstr}, take $n=0$ and $H_k=\left\{\begin{array}{ll} \Lambda & k=0 \\ \{0\} & k\neq 0\end{array}\right.$.  The only nontrivial $V_k$ will necessarily correspond to $k=0$, and we take $(V_0,\rho_0)=(\Lambda_{\uparrow},-\nu_{\uparrow})$, and use for the map $S_0\co H_0\to V_0$ the inclusion $\Lambda\to\Lambda_{\uparrow}$.  To define the PD structure, choose $\alpha\in \Lambda\setminus\{0\}$ such that $\bar{\alpha}=\alpha$, and define $\mathcal{D}_0\co H_0\to {}^{\vee}\!H_0$ by $(\mathcal{D}_0(\lambda))(\mu)=\alpha\bar{\lambda}\mu$.  Let $\mathcal{N}_{\alpha}$ denote the weak Poincar\'e--Novikov structure consisting of the data in this paragraph.  (The condition $\bar{\alpha}=\alpha$ ensures that (\ref{pdsym}) holds.)

There is an isomorphism $\Lambda^{\downarrow}\to {}^{\vee}\!\Lambda_{\uparrow}$ sending $\mu\in \Lambda^{\downarrow}$ to the map $(\lambda\mapsto \bar{\mu}\lambda)$. One can check that the filtration function ${}^{\vee}\!\rho_0$ on ${}^{\vee}\!\Lambda_{\uparrow}$ corresponds under this isomorphism to $-\nu^{\downarrow}\co \Lambda^{\downarrow}\to \R\cup\{\infty\}$, and that the map $\tilde{S}_0$ from (\ref{pnkdef}) corresponds to the map $\Lambda\to\Lambda^{\downarrow}$ given by $\alpha$ times the inclusion.  

Since $d_0=1$ and all other $d_k$ are zero, the only gap to consider is $\mathcal{G}_{0,1}(\mathcal{N}_{\alpha})$, and Corollary \ref{weakdual} asserts that $2\mathcal{G}_{0,1}(\mathcal{N}_{\alpha})\leq 0$.   Now from the definitions and the identifications in the previous paragraph one obtains \begin{align*} \mathcal{G}_{0,1}(\mathcal{N}_{\alpha})&=\sup\left\{\nu_{\uparrow}(\lambda)-\nu^{\downarrow}(\alpha\lambda)|\lambda\neq 0\right\} 
\\ & = -\nu^{\downarrow}(\alpha)+\sup\left\{\nu_{\uparrow}(\lambda)-\nu^{\downarrow}(\lambda)|\lambda\neq 0\right\}=-\nu^{\downarrow}(\alpha)
\end{align*}
 So if $\Gamma\neq \{0\}$ and $\alpha=T^g+T^{-g}$ for some $g>0$ then we have $\mathcal{G}_{0,1}(\mathcal{N}_{\alpha})=-g<0$ and the inequality of Corollary \ref{weakdual} is indeed strict.   
\end{ex}

The above example also illustrates the importance of the symmetry condition (\ref{pdsym}), which we used in Proposition \ref{dualpnchar}; if (\ref{pdsym}) had not been required to hold, we could have taken $\alpha$ to be an arbitrary nonzero element of $\Lambda$ and, by choosing $\alpha$ with $\nu^{\downarrow}(\alpha)<0$, obtained a counterexample to Proposition \ref{weakdual}.

\section{Hamiltonian Floer theory}\label{floersect}
We now turn to the case that motivated this work, namely Floer theory for Hamiltonian flows on compact symplectic manifolds.  We will not attempt to be self-contained, and refer to \cite[Chapters 19 and 20]{ohbook} for the general theory of Hamiltonian Floer homology (and to Part 1 of \cite{ohbook} for more elementary background about symplectic manifolds and Hamiltonian flows on them).  There are certainly other variants of Floer theory to which our methods would apply; for instance, one could consider Floer theory for (not necessarily Hamiltonian) symplectic isotopies as in \cite{LO}, to which some of the considerations about Novikov homology in Section \ref{novsect} would be relevant.  However, we will restrict attention to Hamiltonian Floer theory here.  

Let $(M,\omega)$ be a compact, $2m$-dimensional semipositive symplectic manifold.\footnote{See \cite[Definition 19.6.1]{ohbook} for the definition of semipositivity; this condition should not be essential  if one adapts virtual techniques as in \cite{BX},\cite{Par}.  However, we will use in Proposition \ref{floerpair} some facts related to the PSS map that, to the best of my knowledge, are not available in the literature in the non-semipositive case.}  Writing $S^1=\R/\Z$, let $\mathcal{L}_0M$ denote the space of contractible loops $\gamma\co S^1\to M$, and let $\widetilde{\mathcal{L}_0M}$ denote the covering space of $\mathcal{L}_0M$ whose fiber over $\gamma\in \mathcal{L}_0M$ consists of equivalence classes $[\gamma,w]$ of pairs $(\gamma,w)$ where $w\co D^2\to M$ satisfies $w(e^{2\pi it})=\gamma(t)$ for all $t$, with $(\gamma,v)$ and $(\gamma,w)$ equivalent iff both $\int_{D^2}v^*\omega=\int_{D^2}w^*\omega$ and $\langle c_1,[\bar{v}\#w]\rangle =0$.  Here $c_1\in H^2(M;\Z)$ is the first Chern class of $TM$   and $[\bar{v}\#w]$ is the homology class of the sphere that results from gluing the disks $w$ and $v$ along their boundaries, with the orientation on $v$ reversed.  

 For a generic nondegenerate smooth Hamiltonian function $H\co S^1\times M\to \R$, we then have an action functional $\mathcal{A}_H\co \widetilde{\mathcal{L}_0M}\to \R$ defined by \[ \mathcal{A}_H([\gamma,w])=-\int_{D^2}w^{*}\omega-\int_{0}^{1}H(t,\gamma(t))dt.\] Formally speaking, the Floer complex $CF_{*}(H)$ is the Novikov complex of this functional (defined with the assistance of a suitable almost complex structure $J$ in order to define the analogue of the gradient flow;  we suppress $J$ from the notation as different choices lead to naturally isomorphic filtered Floer homologies\footnote{We are following the approach in \cite{HS} to evade bubbling issues, by taking the almost complex structure  independent of the $S^1$-variable, which requires a further genericity condition on the Hamiltonian beyond nondegeneracy.}). The critical points of $\mathcal{A}_H$ are those $[\gamma,w]$ such that $\gamma$ is a closed orbit of the Hamiltonian flow of $H$ (using the sign convention of \cite{ohbook} that the Hamiltonian vector field of $H$ is given at time $t$ by $\omega(X_{H_t},\cdot)=d(H(t,\cdot))$).  Such a critical point has an associated index $\mu_H([\gamma,w])\in \Z$  which serves in Floer theory as a substitute for the Morse index (as the true Morse index is infinite); we will follow the convention in \cite{Sal} of defining $\mu_H$ to be $m-\mu_{CZ}(\Psi)$ where $\Psi$ is the path of symplectic matrices obtained from the linearization of the Hamiltonian flow along $\gamma$ in terms of a trivialization of $w^*TM$, and $\mu_{CZ}$ is the Conley--Zehnder index.  In this convention, the isomorphism between Floer and coefficient-extended singular homology preserves grading.

In contrast to the finite-dimensional Novikov context, two critical points $[\gamma,v]$ and $[\gamma,w]$ lying in the same fiber of the projection $\widetilde{\mathcal{L}_0M}\to\mathcal{L}_0M$ may have different Conley--Zehnder indices.  To be precise, if $A\in \pi_2(M)$ and $[\gamma,w]\in\widetilde{\mathcal{L}_0M}$ let $[\gamma,A\#w]$ be the result of gluing a representative of $A$ to $w$; one then has \begin{equation}\label{muchange} \mu_H([\gamma,A\#w])=\mu_H([\gamma,w])-2\langle c_1,A\rangle \end{equation} (where we have implicitly mapped $A$ to $H_2(M;\Z)$ by the Hurewicz map).     Similarly, by the definition of  $\mathcal{A}_H$, we have \begin{equation}\label{omegachange} \mathcal{A}_H([\gamma,A\#w])=\mathcal{A}_H([\gamma,w])-\langle [\omega],A\rangle,\end{equation} where $[\omega]$ is the de Rham cohomology class of the closed $2$-form $\omega$.

Let $I_{\omega}\co \pi_2(M)\to\R$ and $I_{c_1}\co \pi_2(M)\to\Z$ be the maps defined by evaluating, respectively, $[\omega]$ and $c_1$ on spheres, and denote \begin{equation}\label{gammas} \hat{\Gamma}=\frac{\pi_2(M)}{\ker(I_{\omega})\cap \ker(I_{c_1})},\qquad \Gamma'=\frac{\ker(I_{c_1})}{\ker(I_{\omega})\cap \ker(I_{c_1})},\end{equation} so that $I_{\omega}$ and $I_{c_1}$ descend to maps defined on $\hat{\Gamma},\Gamma'$.  Thus $\hat{\Gamma}$ is the full deck transformation group of the covering space $\widetilde{\mathcal{L}_0M}$ and $\Gamma'$ is the subgroup which preserves the Conley--Zehnder indices $\mu_H$ of critical points of $\mathcal{A}_H$.  For our subgroup $\Gamma$ of $\R$ (used in the definitions of $\Lambda,\Lambda_{\uparrow},\Lambda^{\downarrow}$ in Section \ref{basic}) we use the image under the monomorphism $I_{\omega}|_{\Gamma'}\co \Gamma'\to \R$.  (In particular, if $(M,\omega)$ is positively or negatively monotone, or weakly exact, then $\Gamma=\{0\}$.)  Modulo the identification of $\Gamma'$ with $\Gamma$ given by integration of $\omega$, our Novikov field $\Lambda_{\uparrow}$ then coincides with the field denoted $\Lambda_{\omega}^{(0)}$ in \cite[Proposition 19.1.5]{ohbook}.  

For $k\in\Z$, let $\tilde{P}_k(H)$ denote the set of critical points $[\gamma,w]$ of $\mathcal{A}_H$ having $\mu_{H}([\gamma,w])=k$.  The degree-$k$ part of the Floer complex, $CF_k(H)$, consists of formal sums \[ \sum_{[\gamma,w]\in\tilde{P}_k(H)}a_{[\gamma,w]}[\gamma,w] \] where $a_{[\gamma,w]}\in\kappa$ and, for all $c\in \R$, only finitely many $[\gamma,w]$ have both $\mathcal{A}_H([\gamma,w])\geq c$ and  $a_{[\gamma,w]}\neq 0$.  Any $\lambda\in \Lambda_{\uparrow}$ can be written as $\lambda=\sum_{A\in \ker(I_{c_1})}\lambda_AT^{I_{\omega}(A)}$ where for all $c\in \R$ only finitely many $A$ have both $\lambda_A\neq 0$ and $I_{\omega}(A)\leq c$.  Given (\ref{omegachange}) and the definition of the equivalence relation on $\widetilde{\mathcal{L}_0M}$, $CF_k(H)$ is a $\Lambda_{\uparrow}$-vector space with the scalar multiplication \[ \left(\sum_{A\in \ker(I_{c_1})}\lambda_AT^{I_{\omega}(A)}\right)\cdot \left(\sum_{[\gamma,w]\in\tilde{P}_k(H)}a_{[\gamma,w]}[\gamma,w]\right)=\sum_{[\gamma,w],A}\lambda_Aa_{[\gamma,w]}[\gamma,A\#w].\]  One has a Floer boundary operator $\partial_H\co CF_k(H)\to CF_{k-1}(H)$ defined as in \cite[Section 19.2]{ohbook}, and a filtration function $\ell_{\uparrow}^{H}\co CF_k(H)\to\R\cup\{-\infty\}$ defined by $\ell_{\uparrow}^{H}\left( \sum_{[\gamma,w]\in\tilde{P}_k(H)}a_{[\gamma,w]}[\gamma,w]\right)=\max\{\mathcal{A}_{H}([\gamma,w])|a_{[\gamma,w]}\neq 0\}$.  Writing $CF_{*}(H)=\oplus_kCF_k(H)$, we then have (in the language of Section \ref{pmdefsec}) a $\Lambda_{\uparrow}$-Floer-type complex $\mathbf{CF}(H)_{\uparrow}=(CF_*(H),\partial_H,\ell_{\uparrow}^{H})$.

With $HF_k(H)$ defined as the $k$th homology of the complex $CF_*(H)$, the general \cite[Proposition 6.6]{UZ} implies that $HF_k(H)$ is an orthogonalizable $\Lambda_{\uparrow}$-space, in the sense of Definition \ref{orthdef}, with respect to the filtration function $\rho_{\uparrow}^{H}$ which sends a homology class $h$ to the infimal (in fact minimal, if $h\neq 0$, by \cite[Theorem 1.4]{U08}) value of $\ell_{\uparrow}^{H}$ on cycles representing $h$.  The Poincar\'e--Novikov structure $\mathcal{N}(M,\omega,H)$ associated to $H$ will then be based on the orthogonalizable $\Lambda_{\uparrow}$-spaces $(HF_k(H),\rho_{\uparrow}^{H})$ together with data associated to the PSS isomorphism $QH_*(M,\omega)\to HF_*(H)$ from quantum homology and the (classical) Poincar\'e intersection pairing on $M$.

To be precise about this we should first fix conventions for quantum homology.  Recalling (\ref{gammas}), let $\hat{\Lambda}=\kappa[\hat{\Gamma}]$, so that $\hat{\Lambda}$ consists of finite sums $\sum_{A\in \hat{\Gamma}}b_AT^A$ where $b_A\in \kappa$), and let $\hat{\Lambda}_{\uparrow}$ denote the ring of (possibly infinite) sums $\sum_{A\in\hat{\Gamma}}b_AT^A$ such that for each $c\in \R$ there are only finitely many $A$ with $b_A\neq 0$ and $I_{\omega}(A)\leq c$, and such that moreover the set $\{ I_{c_1}(A)|b_A\neq 0\}\subset \Z$ is finite. We regard both $\hat{\Lambda}_{\uparrow}$ and its subring $\hat{\Lambda}$ as graded $\kappa$-algebras, with $k$th graded part consisting of sums $\sum_Ab_AT^A$ for which every $A$ appearing in the sum has $-2 I_{c_1}(A)=k$. (The $-2$ is motivated by (\ref{muchange}).) The integration isomorphism $I_{\omega}\co \Gamma'\cong \Gamma$ identifies our earlier rings $\Lambda=\kappa[\Gamma]$ and $\Lambda_{\uparrow}$ with the degree-zero parts of $\hat{\Lambda}$ and $\hat{\Lambda}_{\uparrow}$.  If $I_{c_1}\co \pi_2(M)\to \Z$ identically vanishes then these degree-zero parts are the entire rings (in which case we put $q=1$ in what follows).  Otherwise, letting $N$ denote the minimal Chern number of  $(M,\omega)$ (\emph{i.e.} the positive generator of the image of $I_{c_1}$), by choosing some $A_0\in \hat{\Gamma}$ with $I_{c_1}(A_0)=N$ we may identify $\hat{\Lambda}\cong \Lambda[q^{-1},q]$ and $\hat{\Lambda}_{\uparrow}\cong \Lambda_{\uparrow}[q^{-1},q]$ where $q=T^{A_0}$ is a variable of degree $-2N$.  
Note that, correspondingly, the action of $\Lambda_{\uparrow}$ on $CF_*(H)$ extends to an action of $\hat{\Lambda}_{\uparrow}$ by the obvious extension of the prescription $T^A[\gamma,w]=[\gamma,A\#w]$, and this action is consistent with the gradings in that the grading $j$ part of $\hat{\Lambda}_{\uparrow}$ sends $CF_k(H)$ to $CF_{j+k}(H)$.

The \textbf{quantum homology} of $(M,\omega)$, $QH_{*}(M,\omega)$, is then the tensor product of graded modules $\hat{\Lambda}_{\uparrow}\otimes_{\kappa}H_*(M;\kappa)$ where $H_*(M;\kappa)$ is singular homology with coefficients in $\kappa$.  Thus if $I_{c_1}=0$ then $QH_k(M,\omega)=\Lambda_{\uparrow}\otimes_{\kappa}H_k(M;\kappa)$; otherwise, with notation in the previous paragraph, we have \[ QH_k(M,\omega)=\bigoplus_{j\in \Z}q^{j}\Lambda_{\uparrow}\otimes_{\kappa}H_{k+2Nj}(M;\kappa),\] and multiplication by $q$ induces isomorphisms $QH_k(M,\omega)\cong QH_{k-2N}(M,\omega)$.  The PSS isomorphism $\Phi_{PSS}^{H}$ as in \cite{PSS},\cite[Section 20.3]{ohbook} is then an isomorphism of graded $\hat{\Lambda}_{\uparrow}$-modules $QH_{*}(M,\omega)\to HF_{*}(H)$.

We can now  make the following definition:

\begin{dfn}\label{pnfloer} Let $H\co S^1\times M\to \R$ be a generic non-degenerate Hamiltonian on a $2m$-dimensional semipositive symplectic manifold $(M,\omega)$.  
We define $\mathcal{N}(M,\omega,H)$ to be the Poincar\'e--Novikov structure having its data $n,H_k,\mathcal{D}_k,V_k,\rho_k,S_k$ as in Definition \ref{pnstr} given as follows (with notation as above):
\begin{itemize}   \item $n=2m$.
 \item $H_k$ is the $k$th graded part of the tensor product $\hat{\Lambda}\otimes_{\kappa}H_*(M;\kappa)$.  Thus if $I_{c_1}=0$ then $H_k=\Lambda\otimes_{\kappa}H_{k}(M;\kappa)$, and otherwise $H_k=\oplus_{j\in\Z}q^j\Lambda\otimes_{\kappa}H_{k+2Nj}(M;\kappa)$.
\item $\mathcal{D}_k\co H_{k}\to {}^{\vee}\!H_{2m-k}=\overline{\mathrm{Hom}_{\Lambda}(H_{2m-k},\Lambda)}$ is given by, for $x_{i,A}\in H_{k+2Ni}(M;\kappa)$ and $y_{j,B}\in H_{2m-k+2Nj}(M;\kappa)$, \begin{equation}\label{dquant} \left(\mathcal{D}_k\left(\sum_{A\in\Gamma',i\in \Z}q^iT^Ax_{i,A}\right)\right)\left(\sum_{B\in\Gamma',j\in \Z}q^jT^By_{j,B}\right)=\sum_{A,B\in \Gamma,\,i,j\in\Z}T^{I_{\omega}(-A+B)}x_{i,A}\cap y_{j,B},\end{equation}  Here $x_{i,A}\cap y_{j,B}$ is the usual signed count of intersections between generic representatives of $x_{i,A}$ and $y_{j,B}$, and is taken to vanish if these are not of complementary dimension.  (Hence the only contributions to the above sum have $j=-i$.)
\item $V_k=HF_k(H)$, and $\rho_k=\rho_{\uparrow}^{H}$.
\item $S_k\co H_k\to V_k$ is the restriction to $H_k=\left(\hat{\Lambda}\otimes_{\kappa}H_{*}(M,\omega)\right)_{k}\subset QH_{k}(M,\omega)$ of the degree-$k$ part of the $PSS$ isomorphism $\Phi_{PSS}^{H}\co QH_{*}(M,\omega)\to HF_{*}(H)$.\end{itemize}\end{dfn}

Note that because we use $\hat{\Lambda}$ and not $\hat{\Lambda}_{\uparrow}$ in the definition of $H_k$, all sums in (\ref{dquant}) are finite.  It is routine to check that $\mathcal{D}_k\co H_k\to {}^{\vee}\!H_{2m-k}$  is an isomorphism that satisfies the symmetry condition (\ref{pdsym}) based on the corresponding fact for the intersection pairing on $H_*(M;\kappa)$, so in the terminology of Definition \ref{pnstr} $\mathcal{N}(M,\omega,H)$ is a strong $2m$-Poincar\'e--Novikov structure.

\begin{remark}\label{unnatural}
In cases where both $I_{c_1}$ and $\Gamma$ are nontrivial, a somewhat unnatural aspect of our formulation is that, in order to identify $\hat{\Lambda}$ with $\Lambda[q^{-1},q]$,  we have chosen an element $A_0\in \hat{\Gamma}$ having $I_{c_1}(A_0)=N$ and interpreted $q$ as $T^{A_0}$. The set of possible such choices $A_0$ is a coset of $\Gamma'\cong \Gamma$ in $\hat{\Gamma}$, and  if both  $I_{c_1}$ and $\Gamma$ are nontrivial then (\ref{dquant}) does depend on this choice.  Correspondingly, the map $\mathfrak{c}$ that appears in Corollary \ref{describepn} below also depends on the choice of $A_0$.  An arguably more natural version of (\ref{dquant}) would omit reference to $q$ and allow $A,B$ to vary through $\hat{\Gamma}$ rather than $\Gamma'$, leading to a pairing \begin{equation}\label{modifiedpair} \left(\sum_{A\in\hat{\Gamma}}T^Ax_A,\sum_{B\in \hat{\Gamma}}T^By_B\right)\mapsto \sum_{A,B}T^{I_{\omega}(-A+B)}x_A\cap y_B.\end{equation}  However, if $I_{\omega}(\hat{\Gamma})$ properly contains $\Gamma=I_{\omega}(\Gamma')$ then the right-hand side of (\ref{modifiedpair}) typically does not belong to $\Lambda=\kappa[\Gamma]$, so other modifications also would be required for this to fit into the general theory, such as using the larger group $I_{\omega}(\hat{\Gamma})$ in the role of $\Gamma$ and extending coefficients in the Floer complexes accordingly.  Enlarging $\Gamma$ in such a fashion seems undesirable, especially if doing so causes $\Gamma$ to change from being discrete to being dense.  On the other hand, if $I_{\omega}(\hat{\Gamma})=I_{\omega}(\Gamma')$ then (\ref{modifiedpair}) agrees with (\ref{dquant}) provided that we use for $A_0$ the unique element of $\hat{\Gamma}$ having $I_{c_1}(A_0)=N$ and $I_{\omega}(A_0)=0$, which can be regarded as a natural choice in this case.
\end{remark}

With $\mathcal{N}(M,\omega,H)$ defined, the general constructions of Section \ref{pnsect} then yield, for each $k\in \Z$, a filtered matched pair $\mathcal{P}(\mathcal{N}(M,\omega,H))_{k}$, gaps  $\mathcal{G}_{k,1}(\mathcal{N}(M,\omega,H)),\ldots,\mathcal{G}_{k,\dim_{\Lambda_{\uparrow}}QH_k(M,\omega)}(\mathcal{N}(M,\omega,H))$, and, if $\Gamma$ is discrete, an essential barcode $\mathcal{B}_k(\mathcal{N}(M,\omega,H))$ (using Theorem \ref{basistheorem} to obtain suitable doubly-orthogonal bases).   The various maps in our definition of $\mathcal{P}(\mathcal{N}(M,\omega,H))$ are obtained by passing to homology from maps defined at chain level, so one could rephrase the definition to yield a chain-level Poincar\'e--Novikov structure as in Section \ref{cpnstr}. (The signs work out properly by virtue of Corollary \ref{floerdualdiff} and an argument like that in the proof of Proposition \ref{negdual}.)  If $\Gamma$ is discrete this structure has an associated ``full barcode,'' intended to be considered as a Floer-theoretic version of the interlevel persistence barcode, which can be read off via Theorem \ref{bigdecomp} from the essential barcode of $\mathcal{N}(M,\omega,H)$  together with the concise barcode (from \cite{UZ}) of the Floer-type complex $\mathbf{CF}(H)_{\uparrow}$.

\begin{remark} \label{unnat2} Related to Remark \ref{unnatural}, if $x\in H_k=(\hat{\Lambda}\otimes_{\kappa}H_*(M;\kappa))_{k}$, then with notation as in (\ref{pnkdef}) and Definition \ref{pnfloer} one has \[ \rho_{k-2N}(S_{k-2n}(qx))=\rho_{k}(S_kx)-I_{\omega}(A_0),\qquad {}^{\vee}\!\rho_{k-2N}(\tilde{S}_{k-2N}(qx))={}^{\vee}\!\rho_k(\tilde{S}_kx)+I_{\omega}(A_0). \]  (This is easiest to see from Corollary \ref{describepn}, but can also be inferred by directly unwinding the various definitions.)  In particular, the gaps $\mathcal{G}_{k-2N,i}(\mathcal{N}(M,\omega,H))$ are shifted by $2I_{\omega}(A_0)$ relative to the gaps $\mathcal{G}_{k,i}(\mathcal{N}(M,\omega,H))$
\end{remark}

 If $H_0,H_1\co S^1\times M\to \R$ are two different nondegenerate Hamiltonians, we have a continuation chain homotopy equivalence  $h_{H_0,H_1}\co CF_{*}(H_0)\to CF_{*}(H_1)$ whose induced map $\underline{h}_{H_0H_1}$ on homology obeys $\underline{h}_{H_0,H_1}\circ\Phi_{PSS}^{H_0}=\Phi_{PSS}^{H_1}$.  Standard estimates for the behavior of the filtration functions $\ell_{\uparrow}^{H_i}$ under $h_{H_0,H_1}$ then imply as in \cite[Proposition 21.3.6]{ohbook} that $\underline{h}_{H_0H_1}$ provides a $t$-morphism in the sense of Definition \ref{pntmorph} from $\mathcal{N}(M,\omega,H_0)$ to $\mathcal{N}(M,\omega,H_1)$, with $t=\int_{0}^{1}\max_M|H_1(t,\cdot)-H_0(t,\cdot)|dt$.  Hence Corollaries \ref{pngapstab} and \ref{pnstab} imply that the gaps and essential barcode of 
$\mathcal{N}(M,\omega,H)$ depend in Lipschitz fashion\footnote{Implicit in this is a metric on barcodes which can be inferred from Corollary \ref{pngapstab}, based on an inf over bijections which match open intervals in degree $k-1$ either to other open intervals in degree $k-1$ or closed intervals in degree $k$, and likewise match closed intervals in degree $k$ either to other closed intervals in degree $k$ or open intervals in degree $k-1$.} on the Hamiltonian $H$ with respect to the $C^0$-norm.  In particular we may extend these by continuity to arbitrary Hamiltonians $H$, including degenerate ones.  The analogous continuity statement for the concise barcodes of the $\mathbf{CF}(H_i)_{\uparrow}$ is \cite[Corollary 1.5]{UZ}.  Note that the behavior of short intervals under continuous variation is different in the concise and essential barcodes, as the bottleneck distance with respect to which the concise barcode is continuous allows for intervals to disappear after their lengths continuously shrink to zero.

\begin{remark}\label{htopyinv}
By using the ``optimal continuation maps'' denoted $h_{(\phi^{\rho},J)}$ in \cite[Section 21.6.2]{ohbook}, which relate the Floer complexes of mean-normalized Hamiltonians $H_0,H_1\co S^1\times M\to \R$ having fixed-endpoint-homotopic flows $\{\phi_{H_i}^{t}\}_{0\leq t\leq 1}$ and which preserve the filtrations $\ell_{\uparrow}^{H_i}$ by \cite[Corollary 21.6.5]{ohbook}, one obtains in this situation an isomorphism ($t$-morphism with $t=0$) between the Poincar\'e--Novikov structures  $\mathcal{N}(M,\omega,H_i)$.  Hence the gaps and (when applicable) the essential barcode can be regarded as being associated to an element $\tilde{\phi}$ of the universal cover $\widetilde{\mathrm{Ham}}(M,\omega)$ of the Hamiltonian diffeomorphism group by choosing any mean-normalized Hamiltonian whose flow represents $\tilde{\phi}$, and they vary in Lipschitz fashion with respect to the Hofer (pseudo-)norm on $\widetilde{\mathrm{Ham}}(M,\omega)$.  One can also deduce a similar invariance result for the concise barcodes using these optimal continuation maps.  

Also, any symplectomorphism $\psi$ of $(M,\omega)$ induces an isomorphism $\mathcal{N}(M,\omega,H)\cong \mathcal{N}(M,\omega,H\circ\psi^{-1})$ (provided that the ambiguity in the choice of $A_0$ mentioned in Remark \ref{unnatural} is resolved consistently for $\mathcal{N}(M,\omega,H)$ and $\mathcal{N}(M,\omega,H\circ\psi^{-1})$, \emph{e.g.} by requiring the $I_{\omega}(A_0)$ to agree), so the gaps and/or essential barcodes associated to an element of $\widetilde{\mathrm{Ham}}(M,\omega)$ are invariant under conjugation by the symplectomorphism group $\mathrm{Symp}(M,\omega)$.
\end{remark}

\begin{ex}\label{smallfloer}
Suppose that $H\co S^1\times M\to \R$ is independent of the $S^1$-variable, say $H(t,x)=h(x)$ for some Morse function $h\co M\to \R$.  The Hamiltonian flow of $H$ then has rest points at each critical point of $h$. If $h$ is sufficiently $C^2$-small then these critical points $x\in \mathrm{Crit}(h)$ will be the only fixed points of the time-one map $\phi_{H}^{1}$ of $H$, and we will have $\mu_{H}([\gamma_x,w_x])=\mathrm{ind}_{-h}(x)$ where $\gamma_x\co S^1\to M$ and $w_x\co D^2\to M$ are constant maps to $x$ and $\mathrm{ind}_{-h}(x)$ is the Morse index of $x$ with respect to the Morse function $-h$.    Moreover $\mathcal{A}_{H}([\gamma_x,w_x])=-h(x)$.  In this situation one may reasonably expect the Floer complex $CF_{*}(H)$ to coincide, under the map $x\mapsto [\gamma_x,w_x]$, with (the coefficient extension of) the Morse complex, $\hat{\Lambda}_{\uparrow}\otimes \mathbf{CM}_{*}(-h)$, and also for the PSS map to coincide with the standard isomorphism between singular and Morse homology.  For example, under additional hypotheses and after multiplying $H$ by a small constant, the statement about the Floer complex follows from \cite[Proposition 7.4]{HS}; if  one uses the virtual techniques of  \cite{Par} (and hence characteristic zero coefficients) to define the Floer complex then the identification of the Floer complex of $H$ with the coefficient-extended Morse complex of $-h$ is established quite generally in the proof of \cite[Theorem 10.7.1]{Par}, and Pardon's $S^1$-localization techniques could plausibly also be used to prove the statement about the PSS map.  At any rate, let us assume that we are in a case where these correspondences hold.

In this case our Poincar\'e--Novikov structure $\mathcal{N}(M,\omega,H)$ coincides with the coefficient extension of the Morse-theoretic Poincar\'e--Novikov structure for $-h$ constructed in Section \ref{connext}.  Hence, in view of Theorem \ref{introiso}, the gaps and (for discrete $\Gamma$) essential barcode of $\mathcal{N}(M,\omega,H)$ can be read off from the interlevel barcode of the function $-h$.  (One should adjust appropriately for the coefficient extension by $\hat{\Lambda}$, in particular noting Remark \ref{unnat2} if $I_{c_1}\neq 0$.)  For a concrete example, if we assume that the functions in Figure \ref{genustwo} are scaled to be small enough that our assumptions hold, then the difference between their interlevel barcodes implies a lower bound on the Hofer distance between the elements in $\widetilde{\mathrm{Ham}}$ associated to their time-one Hamiltonian flows.  The fact that the sublevel barcodes of these functions are identical implies that no positive lower bound can be deduced from their concise barcodes, so the essential barcodes provide new information.
\end{ex}

To better understand the invariants associated to $\mathcal{N}(M,\omega,H)$ let us work toward an explicit description, up to isomorphism, of the corresponding filtered matched pairs $\mathcal{P}(\mathcal{N}(M,\omega,H))_k$.    By Definition \ref{pnfloer} and (\ref{pnkdef}), for any $k\in \Z$,  $\mathcal{P}(\mathcal{N}(M,\omega,H))_k$ takes the form \begin{equation}\label{hfpn} \xymatrix{ & \left({}^{\vee}\!HF_{2m-k}(H),{}^{\vee}\!\rho_{\uparrow}^{H}\right) \\ \left(\hat{\Lambda}\otimes_{\kappa}H_{*}(M;\kappa)\right)_{k} \ar[ru]^{\tilde{S}_{k}}\ar[rd]_{S_{k}} & \\ & \left(HF_k(H),\rho_{\uparrow}^{H}\right) } \end{equation}  where $S_k$ is the restriction of the PSS map $\Phi^{H}_{PSS}$ and the $\Lambda$-module homomorphism $\tilde{S}_k$ is, by (\ref{tildesk}), uniquely characterized by the property that, for all $x\in \left(\hat{\Lambda}\otimes_{\kappa}H_{*}(M;\kappa)\right)_{k} $ and $y\in \left(\hat{\Lambda}\otimes_{\kappa}H_{*}(M;\kappa)\right)_{2m-k}$, we have $(\tilde{S}_{k}x)(\Phi_{PSS}^{H}y)=(\mathcal{D}_kx)(y)$.  

 To begin to motivate what follows, define $\hat{H}\co S^1\times M\to \R$ by $\hat{H}(t,x)=-H(1-t,x)$.  Thus the time-one maps of $\hat{H}$ and $H$ are inverse to each other, and a loop $\gamma\co S^1\to M$ is a closed orbit of the Hamiltonian flow of $H$ if and only if its time-reversal $\bar{\gamma}\co S^1\to M$, given by $\bar{\gamma}(t)=\gamma(1-t)$, is a closed orbit of the Hamiltonian flow of $\hat{H}$.  The association $\gamma\mapsto \bar{\gamma}$ lifts to an involution $\mathbb{I}\co\widetilde{\mathcal{L}_0M}\to\widetilde{\mathcal{L}_0M}$ given by $\mathbb{I}([\gamma,w])=[\bar{\gamma},\bar{w}]$ where for any $w\co D^2\to M$ we define $\bar{w}\co D^2\to M$ by $\bar{w}(z)=w(\bar{z})$.  It is easy to see that the action functionals are related by $\mathcal{A}_{\hat{H}}=-\mathcal{A}_H\circ \mathbb{I}$.  Thus, via the action of $\mathbb{I}$, the (formal) positive gradient flow of $\mathcal{A}_H$ corresponds to the negative gradient flow of $\hat{H}$.  So in analogy with  (\ref{novupdown}), we would expect that ${}^{\vee}\!HF_{2m-k}(H)$ could be replaced in (\ref{hfpn}) by $\overline{HF_k(\hat{H})}$.  This will be reflected in Proposition \ref{floerpair}; the effect of the involution $\mathbb{I}$ appears as the conjugation involved in the map to $\overline{HF_k(\hat{H})}$, which, at least if $I_{c_1}=0$, can be interpreted as the action of $\mathbb{I}$ on the homology of the restriction of the covering space $\widetilde{\mathcal{L}_0M}$ to constant loops, corresponding to the submodule $\hat{\Lambda}\otimes_{\kappa}H_*(M;\kappa)$ of quantum homology.

For any $k\in \Z$, write $\tilde{P}_{2m-k}(H)$ for the set of critical points $[\gamma,w]$ for $\mathcal{A}_H$ having $\mu_{H}([\gamma,w])=2m-k$.  Denote by $P_{2m-k}(H)=\{\gamma_1,\ldots,\gamma_r\}$ the image of $\tilde{P}_{2m-k}(H)$ under the covering projection $\pi\co \widetilde{\mathcal{L}_0M}\to\mathcal{L}_{0}M$, and for $i=1,\ldots,r$ choose an element $[\gamma_i,w_i]\in\pi^{-1}(\{\gamma_i\})$.  Then $CF_{2m-k}(H)$ has $\{[\gamma_1,w_1],\ldots,[\gamma_r,w_r]\}$ as a basis over $\Lambda_{\uparrow}$, and $CF_{k}(\hat{H})$ has $\{[\bar{\gamma}_1,\bar{w}_1],\ldots,[\bar{\gamma}_r,\bar{w}_r]\}$ as a basis over $\Lambda_{\uparrow}$.  We can accordingly define a bilinear pairing \[ L_{H,k}^{F}\co CF_{2m-k}(H)\times CF_{k}(\hat{H})\to \Lambda_{\uparrow} \] by, for $\lambda_i,\mu_i\in\Lambda_{\uparrow}$, \begin{equation}\label{lkfdef} L_{H,k}^F\left(\sum_{i=1}^{r}\lambda_i[\gamma_i,w_i],\sum_{i=1}^{r}\mu_i[\bar{\gamma}_i,\bar{w}_i]\right) = \sum_{i=1}^{r}\lambda_i\mu_i.\end{equation}  Because, for $A\in \Gamma'$, one has $[\gamma_i,A\#w_i]=T^{I_{\omega}(A)}[\gamma_i,w_i]$ while $[\bar{\gamma}_i,\overline{A\#w_i}]=[\bar{\gamma}_i,(-A)\#\bar{w}_i]=T^{-I_{\omega}(A)}[\bar{\gamma}_i,\bar{w}_i]$, we see that $L_{H,k}^{F}$ is independent of the choice of $w_i$.  Moreover  we have, for any $a\in CF_{2m-k+1}(H),b\in CF_{k}(\hat{H})$, \[ L_{H,k}^{F}(\partial_{H}a,b)=(-1)^kL_{H,k-1}^{F}(a,\partial_{\hat{H}}b);\] modulo sign, this reflects the fact that a Floer trajectory $u\co \R\times S^1\to M$ for $\hat{H}$ corresponds, by reversing the directions of both the $\R$ and the $S^1$ variables, to a Floer trajectory for $H$, and the sign is calculated in Corollary \ref{floerdualdiff}.

Therefore $L_{H,k}^{F}$ descends to a bilinear pairing $\underline{L}_{H,k}^{F}\co HF_{2m-k}(H)\times HF_k(\hat{H})\to \Lambda_{\uparrow}$. Because $L_{H,k}^{F}$ is a non-degenerate pairing, and because $\Lambda_{\uparrow}$ is a field, the universal coefficient theorem implies that $\underline{L}_{H,k}^{F}$ is likewise non-degenerate, so that the map $a\mapsto \underline{L}_{H,k}^{F}(\cdot,a)$ defines an isomorphism of $\Lambda_{\uparrow}$-vector spaces $\lambda_{H,k}^{F}\co HF_{k}(\hat{H})\to HF_{2m-k}(H)^{*}$.  Applying the conjugation functor of Section \ref{conjsect}, the same map $\lambda_{H,k}^{F}$ can be viewed as an isomorphism of $\Lambda^{\downarrow}$-vector spaces $\overline{HF_k(\hat{H})}\cong {}^{\vee}\!HF_{-k}(H)$.  

This fits into a situation considered in more generality in \cite{U10}; the correspondence $[\bar{\gamma},\bar{w}]\leftrightarrow [\gamma,w]$ identifies the Floer-type complex $\mathbf{CF}(H)_{\uparrow}$   with what \cite{U10} calls the ``opposite complex'' to $\mathbf{CF}(\hat{H})_{\uparrow}$, and the pairing $L_{H,k}^{F}$ is identified with the pairing $L$ from \cite{U10}.   From this we infer:

\begin{prop}\label{dualcompute}
With ${}^{\vee}\!\rho_{\uparrow}^{H}\co {}^{\vee}\!H_{2m-k}\to\R\cup\{\infty\}$ defined from $\rho_{\uparrow}^{H}$ as in Section \ref{fmpdual}, for each $a\in HF_{k}(\hat{H})$ we have \[ -\rho_{\uparrow}^{\hat{H}}(a)= {}^{\vee}\!\rho_{\uparrow}^{H}(\lambda_{H,k}^{F}(a)).\]  
\end{prop}

\begin{proof}
Denote by $\underline{\Delta}\co HF_{2m-k}(H)\times HF_k(\hat{H})\to \kappa$ the map given by composing $\underline{L}_{H,k}^{F}$ with the function $\tau\co \Lambda_{\uparrow}\to\kappa$ defined by $\tau\left(\sum_gc_gT^g\right)=c_0$.  Then by \cite[Corollary 1.4]{U10} (generalizing \cite[Lemma 2.2]{EP}), we have, for each $a\in HF_k(\hat{H})$, \[ -\rho_{\uparrow}^{\hat{H}}(a)=\inf\{\rho_{\uparrow}^{H}(b)|\underline{\Delta}(b,a)\neq 0\}.\]  By definition, \[  {}^{\vee}\!\rho_{\uparrow}^{H}(\lambda_{H,k}^{F}(a))=\inf\left\{\left.\rho_{\uparrow}^{H}(b)+\nu_{\uparrow}\left(\underline{L}_{H,k}^{F}(b,a)\right)\right|b\neq 0\right\}.\]  By the definitions of $\nu_{\uparrow}$ and $\underline{\Delta}$, if $\underline{\Delta}(b,a)\neq 0$ then $\nu_{\uparrow}\left(\underline{L}_{H,k}^{F}(b,a)\right)\leq 0$, from which it follows readily that ${}^{\vee}\!\rho_{\uparrow}^{H}(\lambda_{H,k}^{F}(a))\leq -\rho_{\uparrow}^{\hat{H}}(a)$.  To establish the reverse inequality, note that if $b\neq 0$ but $\underline{L}_{H,k}^{F}(b,a)=0$ then $\rho_{\uparrow}^{H}(b)+\nu_{\uparrow}\left(\underline{L}_{H,k}^{F}(b,a)\right)=\infty$, while if $\underline{L}_{H,k}^{F}(b,a)\neq 0$ then the element $b'=T^{-\nu_{\uparrow}(\underline{L}_{H,k}^{F}(b,a))}b$ satisfies $\nu_{\uparrow}(\underline{L}_{H,k}^{F}(b',a))=0$, $\underline{\Delta}(b',a)\neq 0$, and $\rho_{\uparrow}^{H}(b')=\rho_{\uparrow}^{H}(b)+\nu_{\uparrow}\left(\underline{L}_{H,k}^{F}(b,a)\right)$.  Hence $-\rho_{\uparrow}^{\hat{H}}(a)\leq {}^{\vee}\!\rho_{\uparrow}^{H}(\lambda_{H,k}^{F}(a))$.
\end{proof}

By Proposition \ref{dualcompute}, up to isomorphism the $ \left({}^{\vee}\!HF_{2m-k}(H),{}^{\vee}\!\rho_{\uparrow}^{H}\right) $ appearing in (\ref{hfpn}) can be replaced by $(\overline{HF_k(\hat{H})},-\rho_{\uparrow}^{\hat{H}})$.  The following proposition specifies the map that correspondingly replaces $\tilde{S}_k$ in (\ref{hfpn}).  We again identify $\hat{\Lambda}$ with $\Lambda[q^{-1},q]$; as noted in Remark \ref{unnatural}, if $I_{c_1}\neq 0$ this identification depends on the choice of $A_0\in \hat{\Gamma}$ with $I_{c_1}(A_0)=N$.  (If $I_{c_1}=0$, $q$ should be interpreted as equal to $1$ in what follows, reflecting that $\hat{\Lambda}=\Lambda$.)  

\begin{prop}\label{floerpair}
Define $\mathfrak{c}\co \Lambda[q^{-1},q]\to \Lambda[q^{-1},q]$ by, for $\lambda_i\in\Lambda$, $\mathfrak{c}\left(\sum_i\lambda_iq^i\right)=\sum_i\bar{\lambda}_iq^i$, and for $a\in \hat{\Lambda}\otimes_{\kappa}H_{*}(M;\kappa)$ let $\bar{a}=(\mathfrak{c}\otimes 1_{H_*(M;\kappa)})(a)$.  Then, for all $a\in (\hat{\Lambda} \otimes_{\kappa}H_*(M;\kappa))_{k}$, the map $\tilde{S}_k$ of (\ref{hfpn}) satisfies \[ \tilde{S}_k(a)=\lambda_{H,k}^{F}\left(\Phi_{PSS}^{\hat{H}}(\bar{a})\right).\]
\end{prop}

\begin{proof}
The map $a\mapsto \lambda_{H,k}^{F}(\Phi_{PSS}^{\hat{H}}(\bar{a}))$ is a homomorphism of $\Lambda$-modules from $(\hat{\Lambda}\otimes_{\kappa}H_*(M;\kappa))_{k}$ to the conjugated dual ${}^{\vee}\!HF_{2m-k}(H)$.  So since $\tilde{S}_k\co (\hat{\Lambda}\otimes_{\kappa}H_*(M;\kappa))_{k}\to {}^{\vee}\!HF_{2m-k}(H)$ is the unique $\Lambda$-module homomorphism obeying (\ref{tildesk}) with $n=2m$ and $S_{2m-k}$ equal to the restriction to $(\hat{\Lambda}\otimes_{\kappa}H_*(M;\kappa))_{2m-k}$ of the PSS isomorphism $\Phi_{PSS}^{H}\co QH_{2m-k}(M,\omega)\to HF_{2m-k}(H)$, it suffices to show that, for all $a\in   (\hat{\Lambda}\otimes_{\kappa}H_*(M;\kappa))_{k}$ and $b\in (\hat{\Lambda}\otimes_{\kappa}H_*(M;\kappa))_{2m-k}$ we have \begin{equation} \label{needhat}
\underline{L}_{H,k}^{F}(\Phi_{PSS}^{H}(b),\Phi_{PSS}^{\hat{H}}(\bar{a}))=(\mathcal{D}_ka)(b).
\end{equation}
In this regard, recall that $\Phi_{PSS}^{H}$ is defined via the identification of $H_*(M;\kappa)$ with the Morse homology $\mathbf{HM}_*(f;\kappa)$ of a Morse function $f$ on $M$.  (We assume familiarity with Morse homology in what follows; see also Section \ref{morseintro}.  Further explanation of the PSS maps, with a focus on orientations, appears in Section \ref{pssor}.)  Letting $\mathbf{CM}_*(f;\kappa)$ denote the Morse complex of $f$, $\Phi_{PSS}^{H}$ is defined as the map induced on homology by a chain map $\tilde{\Phi}^{f,H^+}_{PSS}\co \hat{\Lambda}_{\uparrow}\otimes_{\kappa}\mathbf{CM}_{*}(f;\kappa)\to CF_{*}(H)$ associated to a function $H^+\co \C\times M\to \R$ with $H^+(re^{2\pi it},x)=H(t,x)$ for $r>2$. Different choices of $H^+$ yield chain homotopic maps $\tilde{\Phi}^{f,H^+}_{PSS}$, and the usual continuation isomorphisms between the $\mathbf{HM}_{*}(\tilde{f})$ for different choices of $f$  identify the maps induced on homology by the corresponding $\tilde{\Phi}^{f,H}_{PSS}$. As these continuation isomorphisms also commute with the identifications between Morse and singular homology (\emph{cf}. (\ref{contcomm})) it follows that the same map $\Phi_{PSS}^{H}\co QH_{*}(M,\omega)\to HF_*(H)$ is induced by $\tilde{\Phi}^{f,H^+}_{PSS}$ for any choices of $f$ and $H^+$.  In particular, if we us $f$ in the construction of $\Phi_{PSS}^{H}$ then we may (and shall) use $-f$ in the construction of $\Phi_{PSS}^{\hat{H}}$. 

Because an index-$(2m-k)$ critical point for $f$ is the same thing as an index-$k$ critical point for $-f$, for each $k$ we have a nondegenerate $\kappa$-bilinear pairing $L^{M,\kappa}_{f,k}\co \mathbf{CM}_{2m-k}(f;\kappa)\times \mathbf{CM}_{k}(-f;\kappa)\to \kappa$ defined by $L^{M,\kappa}_{f,k}\left(\sum_i \alpha_ip_i,\sum_j\beta_jp_j\right)=\sum_i\alpha_i\beta_i$ for $\alpha_i,\beta_j\in \kappa$, if the index-$(2m-k)$ critical points of $f$ are $p_1,\ldots,p_r$.  The Morse boundary operators for  $\pm f$ are, up to a sign given by Corollary \ref{mor}, adjoint with respect to this pairing, so it descends to a nondegenerate pairing on homology.      In fact, by considering intersections of the descending manifolds for $f$ with those of $-f$, it is easy to see that the resulting pairing on homology coincides with the Poincar\'e intersection pairing under the canonical isomorphisms between $H_*(M;\kappa)$ and $\mathbf{HM}_*(\pm f;\kappa)$.  Now let us extend coefficients to obtain a non-degenerate pairing of $\Lambda_{\uparrow}$-vector spaces \[ L^{M}_{f,k}\co \left(\hat{\Lambda}_{\uparrow}\otimes_{\kappa}\mathbf{CM}_*(f;\kappa)\right)_{2m-k} \times \left(\hat{\Lambda}_{\uparrow}\otimes_{\kappa}\mathbf{CM}_*(-f;\kappa)\right)_{k}\to \Lambda_{\uparrow} \] by putting, for $\lambda_j,\mu_j\in \Lambda_{\uparrow}$, $x_j\in \mathbf{CM}_{2m-k+2Nj}(f;\kappa)$, and $y_j\in \mathbf{CM}_{k+2Nj}(-f;\kappa)$, \begin{equation}\label{lkmdef} L^{M}_{f,k}\left(\sum_{j\in \Z}q^j\lambda_j\otimes x_j,\sum_{j\in \Z}q^j\mu_j\otimes y_j\right) =\sum_{j\in \Z} \lambda_{-j}\mu_j L^{M,\kappa}_{f,k+2Nj}(x_{-j},y_j).\end{equation}   Passing to homology and identifying Morse homology with singular homology, we obtain an induced nondegenerate pairing $\underline{L}^{M}_{f,k}\co QH_{2m-k}(M,\omega)\times QH_{k}(M,\omega)\to \Lambda_{\uparrow}$ (which is in fact independent of $f$).  Given that $L^{M,\kappa}_{f,k}$ induces the Poincar\'e intersection pairing on homology, comparing to (\ref{dquant}) shows that the restriction of $\underline{L}^{M}_{f,k}$ to $(\hat{\Lambda}\otimes_{\kappa}H_*(M;\kappa))_{2m-k}\times (\hat{\Lambda}\otimes_{\kappa}H_*(M;\kappa))_{k}$ satisfies \[ (\mathcal{D}_ka)(b)=\underline{L}^{M}_{f,k}(b,\bar{a})\mbox{ for }a\in (\hat{\Lambda}\otimes_{\kappa}H_*(M;\kappa))_{k},\,b\in(\hat{\Lambda}\otimes_{\kappa}H_*(M;\kappa))_{2m-k}.\]   So, changing variables to $z=\bar{a}$, our desired conclusion (\ref{needhat}) reduces to the claim that, for $b\in (\hat{\Lambda}\otimes_{\kappa}H_*(M;\kappa))_{2m-k}$ and $z\in (\hat{\Lambda}\otimes_{\kappa}H_*(M;\kappa))_{k}$, we have \begin{equation}\label{needhat2} \underline{L}^{F}_{H,k}(\Phi_{PSS}^{H}b,\Phi_{PSS}^{\hat{H}}z)=\underline{L}^{M}_{f,k}(b,z).\end{equation}  

This follows from the construction \cite{PSS},\cite[Section 20.3]{ohbook} of the chain homotopy inverse $\tilde{\Psi}^{\hat{H}^-,-f}_{PSS}\co CF_*(\hat{H})\to \hat{\Lambda}_{\uparrow}\otimes_{\kappa}\mathbf{CM}_*(-f;\kappa)$ to $\tilde{\Phi}^{-f,\hat{H}^+}_{PSS}$, where $\hat{H}^-\co \left( (\C\setminus\{0\})\cup\{\infty\}\right)\times M\to \R$ satisfies $\hat{H}^-(re^{2\pi it},x)=\hat{H}(t,x)$ for $r<1/2$.  Taking $\hat{H}^-(z,m)=H^+(1/z,m)$,  one has a well-known relation (see \cite[Section 2.6.8]{EP},\cite[Proof of Proposition 3.13(vii)]{U11}, though these references do not address sign issues), for all and $c\in (\hat{\Lambda}_{\uparrow}\otimes_{\kappa}\mathbf{CM}_*(f;\kappa))_{2m-k}$ and $d\in CF_k(\hat{H})$, \begin{equation}\label{pssadj} L_{H,k}^{F}(\tilde{\Phi}_{PSS}^{f,H^+}(c), d)=L_{f,k}^{M}(c,\tilde{\Psi}^{\hat{H}^-,-f}_{PSS}(d)), \end{equation} essentially because the spiked planes enumerated by the left-hand side are in one-to-one correspondence, by reversing the time-variable of the gradient flowline and precomposing the perturbed-holomorphic plane by $z\mapsto 1/z$, with the spiked planes enumerated by the right-hand side.  Corollary \ref{floerdualdiff} shows that, with orientation conventions as described in the Appendix, (\ref{pssadj}) holds as stated with no sign.  Passing to homology and using that $\tilde{\Psi}^{H,f}_{PSS}$ induces on homology the inverse to $\Phi_{PSS}^{H}$, we infer that (\ref{needhat2}) holds for all $b\in QH_{2m-k}(M,\omega)$ and $z\in QH_{k}(M,\omega)$, and in particular for all $b\in (\hat{\Lambda}\otimes_{\kappa}H_*(M;\kappa))_{2m-k}$ and $z\in (\hat{\Lambda}\otimes_{\kappa}H_*(M;\kappa))_{k}$.  This establishes (\ref{needhat}) and hence the proposition.
\end{proof}

\begin{cor}\label{describepn}
The filtered matched pair $\mathcal{P}(\mathcal{N}(M,\omega,H))_{k}$ associated to the Poincar\'e--Novikov structure $\mathcal{N}(M,\omega,H)$ is isomorphic to the filtered matched pair \begin{equation}\label{betterpn} \xymatrix{ & \left(\overline{HF_k(\hat{H})},-\rho_{\uparrow}^{\hat{H}}\right) \\ (\hat{\Lambda}\otimes_{\kappa}H_*(M;\kappa))_{k} \ar[ru]^{\Phi^{\hat{H}}_{PSS}\circ(\mathfrak{c}\otimes 1)} \ar[rd]_{\Phi^{H}_{PSS}} & \\ & (HF_k(H),\rho_{\uparrow}^{H}) } \end{equation}  where $\mathfrak{c}\co \hat{\Lambda}\to\hat{\Lambda}$ is defined by $\mathfrak{c}\left(\sum \lambda_iq^i\right)=\bar{\lambda}_iq^i$ for $\lambda_i\in \Lambda$.
\end{cor}

\begin{proof}
Indeed, the isomorphism with (\ref{hfpn}) is provided by $\lambda_{H,k}^{F}\co \overline{HF_k(\hat{H})}\to {}^{\vee}\!HF_{2m-k}(H)$, under which $-\rho_{\uparrow}^{\hat{H}}$ corresponds to ${}^{\vee}\!\rho_{\uparrow}^{H}$ and $\Phi^{\hat{H}}_{PSS}\circ(\mathfrak{c}\otimes 1)$ corresponds to $\tilde{S}_k$ by the preceding two propositions.
\end{proof}

Note that one can define $\mathcal{N}(M,\omega,H)$ and hence
$\mathcal{P}(\mathcal{N}(M,\omega,H))$ even for degenerate Hamiltonians $H$ for which the Floer complex is not directly defined, by continuity with respect to $H$, as the assignments $H\mapsto \rho_{\uparrow}^{H}\circ\Phi_{PSS}^H$ are continuous, \emph{cf}. \cite[Section 12.4]{MS}.  Corollary \ref{describepn} will likewise still hold for degenerate $H$, with $HF_k(H)$ and $HF_k(\hat{H})$ interpreted as direct limits over approximating generic nondegenerate Hamiltonians.

For any Hamiltonian $H\co S^1\times M\to \R$ let $\tilde{\phi}_H$ be the element of $\mathrm{Ham}(M,\omega)$ represented by the flow $\{\phi_{H}^{t}\}_{t\in [0,1]}$
Recall that if $a\in QH_{*}(M,\omega)$ and $\tilde{\phi}\in \widetilde{\mathrm{Ham}}(M,\omega)$ the \textbf{spectral invariant} $c(a,\tilde{\phi})$ is defined by $c(a,\tilde{\phi})=\rho_{\uparrow}^{H}(\Phi_{PSS}^{H}(a))$ for one and hence every mean-normalized Hamiltonian $H$ such that $\tilde{\phi}_H=\tilde{\phi}$. (If such $H$ are degenerate the right-hand-side is defined by continuity, approximating $H$ by generic nondegenerate Hamiltonians.) From Corollary \ref{describepn} it follows that, for   mean-normalized $H$, the essential barcode of $\mathcal{N}(M,\omega,H)$, if it exists (as it always does if $\Gamma$ is discrete), consists of intervals modulo $\Gamma$-translation $[c(e_j,\tilde{\phi}),-c(\bar{e}_j,\tilde{\phi}^{-1})]^{\Gamma}$ in degree $k$ or $(-c(\bar{e}_j,\tilde{\phi}^{-1}),c(e_j,\tilde{\phi}))^{\Gamma}$ in degree $k-1$ (whichever is nonempty), as $e_j$ varies through a doubly-orthogonal basis for the $\Lambda$-module $
(\Lambda\otimes_{\kappa}H_*(M,\kappa))_{k}$.  Here $\bar{e}_j=(\mathfrak{c}\otimes 1)e_j$.  (In particular, if $\Gamma=\{0\}$, \emph{e.g.} if $(M,\omega)$ is monotone, then $\bar{e}_j=e_j$.)  The $i$th gap $\mathcal{G}_{k,i}(\mathcal{N}(M,\omega,H))$ is, in this situation, equal to the $i$th largest of the values $-\left(c(e_j,\tilde{\phi})+c(\bar{e}_j,\tilde{\phi}^{-1})\right)$ by Proposition \ref{gapchar}.  Regardless of whether the essential barcode is defined (\emph{i.e.}, whether doubly-orthogonal bases exist for (\ref{betterpn}) for all $k$), we have \begin{equation}\label{hfgk} \mathcal{G}_{k,i}(\mathcal{N}(M,\omega,H))=-\inf\left\{\gamma\left|\begin{array}{c}(\exists S \subset (\Lambda\otimes_{\kappa}H_*(M,\kappa))_{k}\mbox{ independent over }\Lambda)\\ (\#S=i\mbox{ and }(\forall x\in S)\left(c(x,\tilde{\phi})+c(\bar{x},\tilde{\phi}^{-1})\leq\gamma\right)\end{array}\right.\right\} \end{equation} When one does have doubly-orthogonal bases, Corollary \ref{pndual} implies the non-obvious fact that, for any $k$, the set of quantities in (\ref{hfgk}) as $i$ varies is reflected across $0$ when $k$ is replaced by $2m-k$.

The appearance of quantities $c(x,\tilde{\phi})+c(\bar{x},\tilde{\phi}^{-1})$ above is obviously reminiscent of the spectral pseudonorm $\gamma(\tilde{\phi})=c([M],\tilde{\phi})+c([M],\tilde{\phi}^{-1})$ on $\widetilde{\mathrm{Ham}}(M,\omega)$ \cite[Section 22.1]{ohbook}\footnote{As we are taking $\gamma$ to be defined on $\widetilde{\mathrm{Ham}}$ rather than $\mathrm{Ham}(M,\omega)$ we can only call it a pseudonorm, as it is conceivable that $\gamma$ might vanish on some nontrivial element of $\pi_1(\mathrm{Ham}(M,\omega))$.}, noting that $\overline{[M]}=[M]$.  Unless $QH_{2m}(M,\omega)$ is one-dimensional (equivalently, either $I_{c_1}=0$ or $N>m$), though, it is not clear whether any of the $\mathcal{G}_{2m,i}(\mathcal{N}(M,\omega,H))$ will equal $\gamma(\tilde{\phi})$.   

As already noted in Remark \ref{htopyinv}, the $\mathcal{G}_{k,i}$ and (when defined) the essential barcode associated to a normalized Hamiltonian $H$ share with $\gamma$ the property of being invariant under conjugation by $\mathrm{Symp}(M,\omega)$, and they depend only on the class $\tilde{\phi}_H$ in $\widetilde{\mathrm{Ham}}(M,\omega)$ represented by the flow of $H$.  It is natural to ask under what circumstances they depend only on the time-one map $\phi_{H}^{1}\in \mathrm{Ham}(M,\omega)$. Certainly this is the case if the spectral invariants descend to $\mathrm{Ham}(M,\omega)$ as in \cite[Proposition 3.1 (i) and (ii)]{M10}, but this is a somewhat stringent condition on $(M,\omega)$. In general, an element $\eta\in\pi_1(\mathrm{Ham}(M,\omega))$, together with a homotopy class $\tilde{\eta}$ of sections of the fiber bundle over $S^2$ associated to $\eta$ by the clutching construction, induces an isomorphism $\mathcal{S}_{\tilde{\eta},H}\co CF_*(H)\to CF_*(\eta^*H)$ which shifts filtration by a uniform amount $\mathbb{A}(\tilde{\eta})$ (see \cite[Exercise 12.5.1]{MS}).  This is well-known to imply that the concise barcodes associated by \cite{UZ} to $\tilde{\phi}_H$ and $\tilde{\phi}_H\circ\eta$ are related by translating each interval by $\mathbb{A}(\tilde{\eta})$ and appropriately shifting the gradings; in particular, after adjusting for the grading shift, the lengths of the intervals do not change when $\tilde{\phi}_H$ is replaced by another element of $\widetilde{\mathrm{Ham}}(M,\omega)$ lifting $\phi_{H}^{1}$.  One might expect then that the gaps $\mathcal{G}_{k,i}$, which also represent lengths of bars in a barcode, would also be invariant modulo a grading shift.  However, since the map induced on homology by $\mathcal{S}_{\tilde{\eta},H}$ corresponds under the PSS map to the Seidel action $\mathcal{S}_{\tilde{\eta}}^{QH}\co QH_*(M,\omega)\to QH_*(M,\omega)$ from \cite{Se} of $\tilde{\eta}$ on quantum homology, such invariance may not hold when the Seidel representation is far from being trivial, as the following example shows.

\begin{ex}\label{cpn}
	Consider $M=\mathbb{C}P^n$  with its standard  symplectic structure $\omega_{FS}$ giving area $\pi$ to a complex projective line.  
	We have $\Gamma=\{0\}$ since $(\mathbb{C}P^n,\omega_{FS})$ is monotone.  As the minimal Chern number is $n+1$, in our conventions $QH_k(\mathbb{C}P^n,\omega_{FS})$ is trivial if $k$ is odd and is isomorphic to $\kappa$ for all even integers $k$, with multiplication by $q$ giving isomorphisms $QH_k(\mathbb{C}P^n,\omega_{FS})\cong QH_{k-2(n+1)}(\mathbb{C}P^n,\omega_{FS})$.   
	
	The identity element $\tilde{1}\in\widetilde{\mathrm{Ham}}(\mathbb{C}P^n,\omega_{FS})$ is generated by the Hamiltonian $0$ and so, for any nonzero $a\in H_k(\mathbb{C}P^n;\kappa)$ and $j\in \Z$, one has \[ c(q^ja,\tilde{1})=-j\pi \] (\emph{cf}. \cite[Theorem 21.3.7(3)]{ohbook}).  So for any even degree $k$, say with $k+2j(n+1)\in \{0,\ldots,2n\}$ so that $QH_k(M,\omega)=q^jH_{k+2j(n+1)}(\C P^n;\kappa)$, either the degree-$k$ essential barcode consists of a single element $[-j\pi,j\pi]$ or the degree $(k-1)$ essential barcode consists of a single element $(j\pi,-j\pi)$, depending on the sign of $j$; in either case the corresponding gap $\mathcal{G}_{k,1}(\mathbb{C}P^n)$ is $2j\pi$.
	
	Consider instead the loop $\eta\co S^1\to \mathrm{Ham}(\mathbb{C}P^n,\omega_{FS})$ given by $\eta(t)([z_0:z_1:\cdots:z_n])=[e^{2\pi it}z_0:z_1:\cdots:z_n]$, which is generated by the normalized Hamiltonian \[ G_{\eta}(t,[z_0:z_1:\cdots:z_n]) = \pi\left(\frac{1}{n+1}-\frac{|z_0|^2}{\sum_{i=0}^{n}|z_i|^2}\right).\]  A special case of the first example in \cite[Section 11]{Se} shows that, for one class of sections $\tilde{\eta}$ of the associated fiber bundle over $S^2$, the Seidel action $\mathcal{S}_{\tilde{\eta}}^{QH}\co QH_{*}(\mathbb{C}P^n,\omega_{FS})\to  QH_{*}(\mathbb{C}P^n,\omega_{FS})$ is given by quantum multiplication by the hyperplane class $h\in H_{2n-2}(\mathbb{C}P^n;\omega_{FS})$.  
	
	Recall (see \emph{e.g.} \cite[Example 11.1.10]{MS}) that if $h^{\cap i}\in H_{2n-2i}(\mathbb{C}P^n;\kappa)$ denotes the $i$th power of $h$ under the standard intersection product, the quantum product $\ast$ on $QH_{*}(\mathbb{C}P^n,\omega_{FS})$ is determined by the facts that it is associative and bilinear over $\kappa[q,q^{-1}]$, and that it  obeys $h\ast h^{\cap i}=h^{\cap(i+1)}$ for $0\leq i\leq n-1$ while $h\ast h^{\cap n}=q$.  By \cite[Lemma 12.5.3]{MS}, 
	the spectral invariants of $\tilde{1}$ and of the homotopy classes $[\eta],[\eta]^{-1}\in \pi_1(\mathrm{Ham}(M,\omega))\subset \widetilde{\mathrm{Ham}}(M,\omega)$ are related by \begin{equation}\label{etaact} c(a,[\eta]^{-1})=c(\mathcal{S}_{\tilde{\eta}}a,\tilde{1})+\mathbb{A}(\tilde{\eta}) \quad\mbox{and}\quad c(a,[\eta])=c(\mathcal{S}_{\tilde{\eta}}^{-1}a,\tilde{1})-\mathbb{A}(\tilde{\eta}) \end{equation}  for a constant $\mathbb{A}(\tilde{\eta})$.  To find this constant easily, we can observe that since $\mathcal{S}_{\tilde{\eta}}^{n+1}$ acts by quantum multiplication by the $(n+1)$st quantum power of $h$, which equals $q$, and since $[\eta]^{n+1}=\tilde{1}$, iterating (\ref{etaact}) gives $c(a,\tilde{1})=c(qa,\tilde{1})+(n+1)\mathbb{A}(\tilde{\eta})$ and hence \[ \mathbb{A}(\tilde{\eta})=\frac{\pi}{n+1}.\]

	Now if $i\in \{1,\ldots,n-1\}$ the generator $h^{\cap i}$ of $QH_{2n-2i}(\mathbb{C}P^n,\omega_{FS})$ has $\mathcal{S}_{\tilde{\eta}}^{\pm 1}h^{\cap i}=h^{\cap(i\pm 1)}$; we hence have $c(q^jh^{\cap i},[\eta])=-j\pi -\frac{\pi}{n+1}$ and $c(q^j h^{\cap i},[\eta]^{-1})=-j\pi +\frac{\pi}{n+1}$.  So for $0<i<n$ the degree-$(2n-2i-2(n+1)j)$ quantum homology of $\mathbb{C}P^n$ contributes to the essential barcode of $\mathcal{N}(\C P^n,\omega_{FS},G_{\eta})$  a single  interval $\left[-j\pi-\frac{\pi}{n+1},j\pi-\frac{\pi}{n+1}\right]$ or $\left(j\pi-\frac{\pi}{n+1},-j\pi-\frac{\pi}{n+1}\right)$, whichever is nonempty.  In particular the gaps in these gradings, namely $2\pi j$, are the same for both Hamiltonians $0$ and $G_{\eta}$.  On the other hand, we have \[ \mathcal{S}_{\tilde{\eta}}h^{\cap n}=q,\quad \mathcal{S}_{\tilde{\eta}}^{-1}h^{\cap n}=h^{\cap(n-1)},\quad \mathcal{S}_{\tilde{\eta}}h^{\cap 0}=h,\quad\mbox{and } \mathcal{S}_{\tilde{\eta}}^{-1}h^{\cap 0}=q^{-1}h^{\cap n}.\]
	Hence \[ c(q^jh^{\cap n},[\eta])=c(q^jh^{\cap(n-1)},\tilde{1})-\frac{\pi}{n+1}=-j\pi-\frac{\pi}{n+1} \] while \[ c(q^jh^{\cap n},[\eta]^{-1})=c(q^{j+1},\tilde{1})+\frac{\pi}{n+1}=-(j+1)\pi+\frac{\pi}{n+1},\] and similarly \[ c(q^jh^{\cap 0},[\eta])=-(j-1)\pi-\frac{\pi}{n+1}\quad\mbox{and} \quad c(q^jh^{\cap 0},[\eta]^{-1})=-j\pi +\frac{\pi}{n+1}.\]
So the contribution to the essential barcode from $QH_{-2(n+1)j}(\C P^n,\omega_{FS})$ consists of a single interval with endpoints $ -\pi\left(j+\frac{1}{n+1}\right)$ and $\pi\left(j+1-\frac{1}{n+1}\right)$ and that from   $QH_{2n-2(n+1)j}(\C P^n,\omega_{FS})$ consists of a single interval with endpoints $-\pi\left(j-1+\frac{1}{n+1}\right)$ and $\pi\left(j-\frac{1}{n+1}\right)$.  In particular, the gaps in these degrees for the Hamiltonian $H_{\eta}$, namely $(2j+1)\pi$ for degree $-2(n+1)j$ and $(2j-1)\pi$ for degree $2n-(2n+1)j$,  are distinct from the gaps (in either these degrees or any others) of the Hamiltonian $0$, even though these Hamiltonians generate the same time-one map.  As mentioned earlier, this contrasts with the situation for standard, sublevel-type Floer-theoretic barcodes as in \cite{PS},\cite{UZ}, in which the lengths of intervals depend only on the time-one map, possibly modulo a grading shift. 
\end{ex}

\section{Chain level constructions and persistence modules}\label{clsec}

In the present section we lift the notions of filtered matched pairs and Poincar\'e--Novikov structures to the level of chain complexes and (homotopy classes of) chain maps between them.  This will allow us to see the concise and essential barcodes as arising from a common algebraic structure, and eventually to relate these barcodes to interlevel persistence for $S^1$-valued Morse functions.

\subsection{Generalities about filtrations and cones}

We collect here some conventions, definitions, and simple results concerning modules that are filtered by partially ordered sets, and constructions associated to the filtered maps between them.

Let $(\mathcal{P},\preceq)$ be a partially ordered set, and $R$ be a commutative ring.  A \textbf{$\mathcal{P}$-filtered module} over $R$ is an $R$-module $M$ equipped with submodules $M^{\preceq p}$ for all $p\in\mathcal{P}$ such that $M^{ \preceq p}\subset M^{\preceq q}$ whenever $p\preceq q$, and $\cup_{p\in P}M^{\preceq p}=M$.  

\begin{ex}\label{orthfiltex} In the language of the preceding sections, if $(V_{\uparrow},\rho_{\uparrow})$ is a normed $\Lambda_{\uparrow}$-space, then $V_{\uparrow}$ obtains the structure of a $\R$-filtered module over $\kappa$\footnote{While $V_{\uparrow}$ is a vector space over $\Lambda_{\uparrow}$, the individual $V_{\uparrow}^{\preceq p}$ are not preserved by scalar multiplication by elements of $\Lambda_{\uparrow}$ so $V_{\uparrow}$ is not a filtered module over $\Lambda_{\uparrow}$.  It is a filtered module over the subring of $\Lambda_{\uparrow}$ consisting of elements with $\nu_{\uparrow}\geq 0$, though we will not use this in a direct way.} by setting, for $p\in \R$, $V_{\uparrow}^{\preceq p}=\{v\in V_{\uparrow}|\rho_{\uparrow}(v)\leq p\}$.  (The fact that $V_{\uparrow}^{\preceq p}$ is a $\kappa$-subspace  follows from the inequality $\rho_{\uparrow}(v+w)\leq \max\{\rho_{\uparrow}(v),\rho_{\uparrow}(w)\}$.)

On the other hand, if $(V^{\downarrow},\rho^{\downarrow})$ is a normed $\Lambda^{\downarrow}$-space, it is the superlevel sets, not the sublevel sets, of $\rho^{\downarrow}$ that are $\kappa$-subspaces of $V^{\downarrow}$.  So we obtain a $\R$-filtered module structure (with the usual partial order on $\R$) by setting \[ V^{\downarrow\preceq p}=\{v\in V^{\downarrow}|\rho^{\downarrow}(v)\geq -p\}.\]
\end{ex}

If $M$ and $N$ are $\mathcal{P}$-filtered modules over $R$, an $R$-module homomorphism $\phi\co M\to N$ is said to be filtered if $\phi(M^{\preceq p})\subset N^{\preceq p}$ for all $p$. If $M_1,\ldots,M_k$ are filtered $\mathcal{P}$-modules, their \textbf{filtered direct sum} is the module $\oplus_{i=1}^{k}M_i$ equipped with the filtration given by \[ \left(\oplus_{i=1}^{k}M_i\right)^{\preceq p}=\oplus_{i=1}^{k}(M_i^{\preceq p}).\]  If $(V_{\uparrow},\rho_{\uparrow})$ is a normed $\Lambda_{\uparrow}$-space, a $\Lambda_{\uparrow}$-basis $\{x_1,\ldots,x_k\}$ is a $\rho_{\uparrow}$-\emph{orthogonal} basis for $V_{\uparrow}$ (in the sense of Definition \ref{orthdef}) if and only if the obvious addition isomorphism \[ \oplus_{i=1}^{k}(\Lambda_{\uparrow}x_i)\to V_{\uparrow} \] is a \emph{filtered} isomorphism (where the $\Lambda_{\uparrow}$-spans $\Lambda_{\uparrow} x_i$ of the $x_i$ are separately given the filtrations that they inherit as submodules of $V_{\uparrow}$). 

 A \textbf{$\mathcal{P}$}-filtered chain complex of $R$-modules is a chain complex $\mathcal{C}=(C_*,\partial)$ of $R$-modules in which the chain modules $C_k$ are $\mathcal{P}$-filtered and the boundary maps $\partial\co C_k\to C_{k-1}$ are filtered homomorphisms.  Thus for each $p\in \mathcal{P}$ we have a subcomplex $\mathcal{C}^{\preceq p}$ of $\mathcal{C}$, with the chain modules of $\mathcal{C}^{\preceq p}$ given by the $C_{k}^{\preceq p}$.  Recalling that a \textbf{persistence module} over $\mathcal{P}$ is a functor to the category of $R$-modules from the category with object set $\mathcal{P}$ and a unique morphism from $p$ to $q$ whenever $p\preceq q$,  for each $k$ the maps $H_{k}^{\preceq p}(\mathcal{C})\to H_k^{\preceq q}(\mathcal{C})$, induced on homology by the inclusions $C_k^{\preceq p}\to C_{k}^{\preceq q}$, form the structure of a persistence module.  Persistence modules form a category, with morphisms given by natural transformations of functors.

If $\mathcal{C}=(C_*,\partial_C)$ and $\mathcal{D}=(D_*,\partial_D)$ are $\mathcal{P}$-filtered chain complexes, two filtered chain maps $f,g\co \mathcal{C}\to\mathcal{D}$ are said to be filtered homotopic if there is a filtered module homomorphism $K\co C_*\to D_{*+1}$ such that $f-g=\partial_D\circ K+K\circ \partial_C$,  A \textbf{filtered homotopy equivalence} from $\mathcal{C}$ to $\mathcal{D}$ is a filtered chain map $f\co \mathcal{C}\to\mathcal{D}$ such that there is another filtered chain map $h\co \mathcal{D}\to\mathcal{C}$ such that $f\circ h$ and $h\circ f$ are filtered homotopic to the respective identities. A \textbf{filtered quasi-isomorphism} from $\mathcal{C}$ to $\mathcal{D}$ is  filtered chain map $f\co\mathcal{C}\to\mathcal{D}$ such that each induced map on homology $f_*\co H_k(\mathcal{C}^{\preceq p})\to H_k(\mathcal{D}^{\preceq p})$ is an isomorphism, \emph{i.e.} such that $f$ induces an isomorphism between the corresponding homology persistence modules. Of course, a filtered homotopy equivalence is  a filtered quasi-isomorphism.  In general, we will use the terms ``homotopy'' and ``chain homotopy'' interchangeably.

Now suppose that $\mathcal{C}=(C_*,\partial_C)$ is a $\mathcal{P}$-filtered chain complex of $R$-modules and $\mathcal{D}=(D,\partial_D)$ is another complex of $R$-modules (which we do not assume to be filtered).  If $f\co \mathcal{C}\to\mathcal{D}$ is a chain map, the \textbf{mapping cone} of $f$, denoted $\mathrm{Cone}(f)$, is the chain complex whose $i$th graded part is $\mathrm{Cone}(f)_i:=C_{i-1}\oplus D_i$ and whose differential is given by $\partial_{\mathrm{Cone}}(c,d)=(-\partial_Cc,fc+\partial_Dd)$.  $\mathrm{Cone}(f)$ inherits the $\mathcal{P}$-filtration from $\mathcal{C}$, setting $\mathrm{Cone}(f)_{i}^{\preceq p}=C_{i-1}^{\preceq p}\oplus D_i$.

For all $p\in\mathcal{P}$ one then has a short exact sequence of chain complexes \[ \xymatrix{0 \ar[r] & \mathcal{D}\ar[r] & \mathrm{Cone}(f)^{\preceq p}\ar[r] & \mathcal{C}^{\preceq p}[-1] \ar[r] & 0 } \] (where the maps are the obvious inclusion and projection, and where we use the usual convention that, for a chain complex $\mathcal{E}=(E_*,\partial_E)$ and for $a\in \Z$, $\mathcal{E}[a]$ denotes the chain complex whose $k$th graded part is $E_{k+a}$ and whose differential is $(-1)^{a}\partial_E)$, and the induced long exact sequence on homology has its connecting homomorphism $H_k(\mathcal{C}^{\preceq p}[-1])=H_{k-1}(\mathcal{C}^{\preceq p})\to H_{k-1}(\mathcal{D})$ equal to the map on homology induced by $f|_{C_{k-1}^{\preceq p}}$.

\begin{prop}\label{filtinv}
Let $\mathcal{C},\hat{\mathcal{C}}$ be two $\mathcal{P}$-filtered chain complexes, and $\mathcal{D},\hat{\mathcal{D}}$ be two chain complexes
and consider a diagram of chain maps \begin{equation}\label{coneseq} \xymatrix{ \mathcal{C}\ar[r]^f  \ar[d]_{\gamma} & \mathcal{D} \ar[d]^{\delta} \\ \hat{\mathcal{C}}\ar[r]^{\hat{f}} & \hat{\mathcal{D}} }\end{equation} where $\gamma$ is a filtered quasi-isomorphism, $\delta$ is a quasi-isomorphism, and the two compositions $\delta\circ f,\hat{f}\circ \gamma\co \mathcal{C}\to \hat{\mathcal{D}}$ are homotopic.  Then there is a \textbf{filtered} quasi-isomorphism $\mathrm{Cone}(f)\to\mathrm{Cone}(\hat{f})$.
\end{prop} 

\begin{proof}
Let $K\co C_{*-1}\to \hat{D}_*$ obey $\partial_{\hat{D}}K+K\partial_{C}=\delta f-\hat{f}\gamma$.  This identity together with the fact that $\gamma$ and $\delta$ are chain maps imply that the map $A\co \mathrm{Cone}(f)\to \mathrm{Cone}(\hat{f})$ defined by, for $c\in C_{k-1}$ and $d\in D_k$, \[ A(c,d)=(\gamma c, Kc+\delta d),\] is a chain map.  Evidently $A$ sends each filtered subcomplex $\mathrm{Cone}(f)^{\preceq p}=\mathcal{C}^{\preceq p}\oplus \mathcal{D}$ to $\mathrm{Cone}(\hat{f})^{\preceq p}$. Moreover $A$ intertwines the exact sequences as in (\ref{coneseq}) into a commutative diagram \[ \xymatrix{ 0 \ar[r] &  \mathcal{D}\ar[r]\ar[d]_{\delta} & \mathrm{Cone}(f)^{\preceq p}\ar[r]\ar[d]_{A} & \mathcal{C}^{\preceq p}[-1]\ar[d]_{\gamma} \ar[r] & 0  \\ 0 \ar[r] & \hat{\mathcal{D}}\ar[r] & \mathrm{Cone}(\hat{f})^{\preceq p}\ar[r] & \hat{\mathcal{C}}^{\preceq p}[-1] \ar[r] & 0 }  \]  The resulting diagram of long exact sequences on homology is therefore also commutative, and so the fact that $A$ induces isomorphisms $H_k(\mathrm{Cone}(f)^{\preceq p})\cong H_k(\mathrm{Cone}(\hat{f})^{\preceq p})$ for all $k$ and $p$ follows from the five lemma and the assumptions that $\delta$ is a quasi-isomorphism and $\gamma$ is a filtered quasi-isomorphism.
\end{proof}

\subsection{From chain-level filtered matched pairs to persistence modules}\label{pmdefsec}

\cite[Definition 4.1]{UZ} defines a \textbf{Floer-type complex} $\mathcal{C}_{\uparrow}=(C_{\uparrow *},\partial_{\mathcal{C}_{\uparrow}},\ell_{\uparrow})$ to be a chain complex $(C_{\uparrow *}=\oplus_kC_{\uparrow k},\partial_{\mathcal{C}_{\uparrow}})$ over the Novikov field $\Lambda_{\uparrow}$ equipped with a function $\ell_{\uparrow}\co C_{\uparrow*}\to\R\cup\{-\infty\}$ such that each $(C_{\uparrow k},\ell_{\uparrow}|_{C_{\uparrow k}})$ is an orthogonalizable $\Lambda_{\uparrow}$-space and, for all $x\in C_{*}$, we have $\ell_{\uparrow}(\partial_{\mathcal{C}_{\uparrow}}x)\leq \ell_{\uparrow}(x)$.  Said differently, endowing each $C_k$ with the $\R$-filtration described in Example \ref{orthfiltex}, the last condition amounts to each $\partial_{\mathcal{C}_{\uparrow}}|_{C_{k}}$ being a filtered homomorphism.

In this paper we will refer to the Floer-type complexes of \cite{UZ} as ``$\Lambda_{\uparrow}$-Floer-type complexes,'' as, just as in Definition \ref{orthdef}, one has the companion notion of a $\Lambda^{\downarrow}$\textbf{-Floer-type complex}: a structure $\mathcal{C}^{\downarrow}=(C^{\downarrow}_{*},\partial_{\mathcal{C}^{\downarrow}},\ell^{\downarrow})$ with each $(C^{\downarrow}_{k},\ell^{\downarrow}|_{C^{\downarrow}_{k}})$ an orthogonalizable $\Lambda^{\downarrow}$-space and, for each $x\in C^{\downarrow}_{k}$, $\ell^{\downarrow}(\partial_{\mathcal{C}^{\downarrow}}x)\geq \ell^{\downarrow}(x)$.  As in Remark \ref{switch}, $(C ^{\downarrow}_{*},\partial_{\mathcal{C}^{\downarrow}},\ell^{\downarrow})$ is a $\Lambda^{\downarrow}$-Floer-type complex if and only if $(\overline{C^{\downarrow}}_{*},\partial_{\mathcal{C}_{\downarrow}},-\ell^{\downarrow})$ is a $\Lambda_{\uparrow}$-Floer-type complex, so general results in \cite{UZ} about Floer-type complexes apply to either type (with appropriate sign adjustments in the $\Lambda^{\downarrow}$ case).

\begin{dfn}\label{cfmpdfn}
A \textbf{chain-level filtered matched pair} $\mathcal{CP}=(\mathcal{C},\mathcal{C}^{\downarrow},\mathcal{C}_{\uparrow},\phi^{\downarrow},\phi_{\uparrow})$ consists of the following data:
\begin{itemize}\item a chain complex $\mathcal{C}=(C_*,\partial)$ of modules over the group algebra $\Lambda=\kappa[\Gamma]$, such that each graded piece $C_k$ is finitely generated and free as a $\Lambda$-module;
\item a $\Lambda^{\downarrow}$-Floer-type complex $\mathcal{C}^{\downarrow}=(C^{\downarrow}_{*},\partial_{\mathcal{C}^{\downarrow}},\ell^{\downarrow})$, and a $\Lambda_{\uparrow}$-Floer-type complex $\mathcal{C}_{\uparrow}=(C_{\uparrow *},\partial_{\mathcal{C}_{\uparrow}},\ell_{\uparrow})$; 
\item chain maps (commuting with the $\Lambda$-actions)  $\phi^{\downarrow}\co C_*\to C^{\downarrow}_{*}$ and $\phi_{\uparrow}\co C_*\to C_{\uparrow *}$ such that the coefficient extensions $1_{\Lambda^{\downarrow}}\otimes\phi^{\downarrow}\co \Lambda^{\downarrow}\otimes_{\Lambda}C_*\to C^{\downarrow}_{*}$ and $1_{\Lambda_{\uparrow}}\otimes \phi_{\uparrow}\co \Lambda_{\uparrow}\otimes_{\Lambda}C_*\to C_{\uparrow *}$ are chain homotopy equivalences.
\end{itemize}
\end{dfn}

\begin{remark}\label{cfmpfmp}
For any $\Lambda_{\uparrow}$-Floer-type complex $\mathcal{C}_{\uparrow}=(C_{\uparrow *},\partial_{\mathcal{C}_{\uparrow}},\ell_{\uparrow})$ and any $k\in \Z$, \cite[Proposition 6.6]{UZ} shows that the $k$th homology $H_k(\mathcal{C}_{\uparrow})$ becomes an orthogonalizable $\Lambda_{\uparrow}$-space with respect to the filtration function $\rho_{\uparrow}\co H_k(\mathcal{C}_{\uparrow})\to\R\cup\{-\infty\}$ defined by \[ \rho_{\uparrow}(\alpha)=\inf\{\ell_{\uparrow}(c)|c\in C_{\uparrow k},\partial_{\mathcal{C}_{\uparrow}}c=0,\,[c]=\alpha\}.\]  Likewise, for a $\Lambda^{\downarrow}$-Floer-type complex $\mathcal{C}^{\downarrow}=(C^{\downarrow}_{*},\partial_{\mathcal{C}^{\downarrow}},\ell^{\downarrow})$ and $k\in \Z$, we have an orthogonalizable $\Lambda$-space $(H_k(\mathcal{C}^{\downarrow}),\rho^{\downarrow})$ with $\rho^{\downarrow}\co H_k(\mathcal{C}^{\downarrow})\to\R\cup\{+\infty\}$ defined by \[  \rho^{\downarrow}(\alpha)=\sup\{\ell^{\downarrow}(c)|c\in C^{\downarrow}_{k},\partial_{\mathcal{C}^{\downarrow}}c=0,\,[c]=\alpha\}.\]  
Hence a chain-level filtered matched pair as in Definition \ref{cfmpdfn} induces, for each $k$, a filtered matched pair  in the sense of Definition \ref{fmpdfn}, namely \[ \xymatrix{ & (H_k(\mathcal{C}^{\downarrow}),\rho^{\downarrow}) \\ H_k(\mathcal{C})\ar[ru]^{\phi^{\downarrow}_{*}}\ar[rd]_{\phi_{\uparrow *}} & \\ & (H_k(\mathcal{C}_{\uparrow}),\rho_{\uparrow})  }\]  We denote this latter filtered matched pair by $\mathcal{H}_k(\mathcal{CP})$.
\end{remark}

Give $\R^2$ the partial order defined by declaring $(s,t)\preceq (s',t')$ iff both $s\leq s'$ and $t\leq t'$.  Given a chain-level filtered matched pair, we now describe a construction of, for each $k\in \Z$, an isomorphism class of persistence modules $\mathbb{H}_k(\mathcal{CP})=\{\mathbb{H}_k(\mathcal{CP})_{s,t}\}_{(s,t)\in \R^2}$  over $\R^2$.  Writing $\mathcal{CP}=  (\mathcal{C},\mathcal{C}^{\downarrow},\mathcal{C}_{\uparrow},\phi^{\downarrow},\phi_{\uparrow})$ as in the definition, let $\psi^{\downarrow}\co C^{\downarrow}_{*}\to \Lambda^{\downarrow}\otimes_{\Lambda}C_*$ and $\psi_{\uparrow}\co C_{\uparrow *}\to \Lambda_{\uparrow}\otimes_{\Lambda}C_*$ be chain homotopy inverses to $\phi^{\downarrow},\phi_{\uparrow}$, respectively.  Note that these are defined uniquely up to homotopy: if $\psi'_{\uparrow}$ is a different choice of homotopy inverse to $\phi_{\uparrow}$ then $\psi_{\uparrow},\psi'_{\uparrow}$ are each homotopic to $\psi'_{\uparrow}\phi_{\uparrow}\psi_{\uparrow}$ and hence are homotopic to each other, and similarly for $\psi^{\downarrow}$.  Now recall the $\Lambda$-module $\Lambda_{\updownarrow}$ from (\ref{updown}), and the inclusions $j^{\downarrow}\co \Lambda^{\downarrow}\to \Lambda_{\updownarrow}$ and $j_{\uparrow}\co \Lambda_{\uparrow}\to\Lambda_{\updownarrow}$.  We  unite the maps $\psi^{\downarrow},\psi_{\uparrow}$ into a single map \begin{equation}\label{bigmap} -j^{\downarrow}\otimes \psi^{\downarrow}+j_{\uparrow}\otimes\psi_{\uparrow}\co C^{\downarrow}_{*}\oplus C_{\uparrow *}\to \Lambda_{\updownarrow}\otimes_{\Lambda}C_*.\end{equation}  This is a chain map (in the category of complexes over $\Lambda$), whose homotopy class depends only on the original data $\mathcal{CP}=  (\mathcal{C},\mathcal{C}^{\downarrow},\mathcal{C}_{\uparrow},\phi^{\downarrow},\phi_{\uparrow})$ since $\psi^{\downarrow},\psi_{\uparrow}$ are unique up to homotopy.  

The domain of (\ref{bigmap}) is naturally filtered by $\mathbb{R}^2$, by setting \[    (C^{\downarrow}_{*}\oplus C_{\uparrow *})^{\preceq (s,t)}=\{(x,y)\in C^{\downarrow}_{*}\oplus C_{\uparrow *}|\ell^{\downarrow}(x)\geq -s\mbox{ and } \ell_{\uparrow}(y)\leq t\} \] (\emph{cf}. Example \ref{orthfiltex}).  We then define \begin{equation}\label{hkdef} \mathbb{H}_k(\mathcal{CP})_{s,t}=H_{k+1}\left(\mathrm{Cone}\left(\xymatrix{\left(C^{\downarrow}_{*}\oplus C_{\uparrow *}\right)^{\preceq (s,t)}\ar[rrr]^{-j^{\downarrow}\otimes \psi^{\downarrow}+j_{\uparrow}\otimes\psi_{\uparrow}}&&&\Lambda_{\updownarrow}\otimes_{\Lambda}C_* }\right)\right).\end{equation}

The inclusions $\left(C^{\downarrow}_{*}\oplus C_{\uparrow *}\right)^{\preceq (s,t)}\hookrightarrow \left(C^{\downarrow}_{*}\oplus C_{\uparrow *}\right)^{\preceq (s',t')}$ induce maps that make $\mathbb{H}_k(\mathcal{CP})=\{\mathbb{H}_k(\mathcal{CP})_{s,t}\}_{(s,t)\in \R^2}$  into a persistence module of $\kappa$-vector spaces over $\R^2$.  In motivating examples, as demonstrated by Theorem \ref{introiso}, if $s+t\geq 0$ (so that the interval $[-s,t]$ is nonempty) then $\mathbb{H}_k(\mathcal{CP})_{s,t}$ is isomorphic to the singular homology $H_k(f^{-1}([-s,t]);\kappa)$ for some Morse function $f$.  However, in our conventions $\mathbb{H}_k(\mathcal{CP})_{s,t}$ is still defined when $s+t<0$ even though it is not meant to be interpreted as the homology of an interlevel set in this case.  

To give some hint of the motivation for (\ref{hkdef}), for a smooth function $f\co X\to \R$ and an interval $I\subset \R$ write $X_{I}=f^{-1}(I)$, and for any space $Y$ denote  the $\kappa$-coefficient singular chain complex of   $Y$ by $\mathbf{S}_*(Y)$.  Then, if $s,t\in \R$ are regular values of $f$ and  $s+t>0$, the standard proof of the Mayer--Vietoris sequence gives rise to an isomorphism \[ H_{k}(X_{[-s,t]};\kappa)\cong H_{k+1}\left(\mathrm{Cone}\left(\xymatrix{ \mathbf{S}_*(X_{[-s,\infty)})\oplus \mathbf{S}_{*}(X_{(-\infty,t]}) \ar[rr]^<<<<<<<<<<{-j_{[-s,\infty)}+j_{(-\infty,t]}} & & \mathbf{S}_*(X) }\right)\right)  \]
where $j_{[-s,\infty)}$ and $j_{(-\infty,t]}$ are the inclusions into $\mathbf{S}_*(X)$.
The definition (\ref{hkdef}) is designed to parallel this, in a way that also works in the context of Novikov homology.  This reasoning is made more precise in the proof of Theorem \ref{introiso} in Section \ref{isosect}.

While (\ref{hkdef}) requires a choice of homotopy inverses $\psi_{\uparrow},\psi^{\downarrow}$,  the fact that these homotopy inverses are unique up to homotopy implies that the persistence module isomorphism type of $\mathbb{H}_k(\mathcal{CP})$ depends only on the original data $\mathcal{CP}$ by  Proposition \ref{filtinv}. We will also see that we have invariance under the following relation:

\begin{dfn}\label{fmhedef}
Two chain-level filtered matched pairs $\mathcal{CP}=(\mathcal{C},\mathcal{C}^{\downarrow},\mathcal{C}_{\uparrow},\phi^{\downarrow},\phi_{\uparrow})$ and  $\widehat{\mathcal{CP}}=(\widehat{\mathcal{C}},\widehat{\mathcal{C}}^{\downarrow},\widehat{\mathcal{C}}_{\uparrow},\widehat{\phi}^{\downarrow},\widehat{\phi}_{\uparrow})$ are said to be \textbf{filtered matched homotopy equivalent} if there is a diagram of chain maps  \begin{equation}\label{sixdiagram} \xymatrix{ & \mathcal{C}^{\downarrow}\ar[rr]^f & & \widehat{\mathcal{C}}^{\downarrow} \\ \mathcal{C}\ar[ur]^<<<<<<{\phi^{\downarrow}}\ar[rr]^{g}\ar[dr]_<<<<<<{\phi_{\uparrow}} & & \widehat{\mathcal{C}}\ar[ur]_>>>>{\widehat{\phi}^{\downarrow}}\ar[dr]^>>>>{\widehat{\phi}_{\uparrow}} &\\ & \mathcal{C}_{\uparrow}\ar[rr]^h &  & \widehat{\mathcal{C}}_{\uparrow} & } \end{equation} where $f$ and $h$ are filtered homotopy equivalences, $g$ is a homotopy equivalence, and the two parallelograms commute up to homotopy.
\end{dfn}

\begin{prop}\label{persinv}
If the chain-level filtered matched pairs $\mathcal{CP}=(\mathcal{C},\mathcal{C}^{\downarrow},\mathcal{C}_{\uparrow},\phi^{\downarrow},\phi_{\uparrow})$ and  $\widehat{\mathcal{CP}}=(\widehat{\mathcal{C}},\widehat{\mathcal{C}}^{\downarrow},\widehat{\mathcal{C}}_{\uparrow},\widehat{\phi}^{\downarrow},\widehat{\phi}_{\uparrow})$ are filtered matched homotopy equivalent then for each $k$ there is an isomorphism of persistence modules $\mathbb{H}_k(\mathcal{CP})\cong \mathbb{H}_k(\widehat{\mathcal{CP}})$. 
\end{prop} 

\begin{proof}
Let $\psi^{\downarrow},\psi_{\uparrow},\widehat{\psi}^{\downarrow},\widehat{\psi}_{\uparrow}$ be homotopy inverses to $1_{\Lambda^{\downarrow}}\otimes \phi^{\downarrow},1_{\Lambda_{\uparrow}}\otimes \phi_{\uparrow},1_{\Lambda^{\downarrow}}\otimes \widehat{\phi}^{\downarrow}$ and $1_{\Lambda_{\uparrow}}\otimes \widehat{\phi}_{\uparrow}$, respectively.  Then, using $\sim$ denote the relation of chain homotopy between chain maps, the fact that the upper parallelogram in (\ref{sixdiagram}) commutes up to homotopy implies (after tensoring with $\Lambda^{\downarrow}$) that \[ (1_{\Lambda^{\downarrow}}\otimes \widehat{\phi}^{\downarrow})\circ (1_{\Lambda^{\downarrow}}\otimes g)\circ \psi^{\downarrow}\sim f\circ(1_{\Lambda^{\downarrow}}\otimes \phi^{\downarrow})\circ\psi^{\downarrow}\sim f \] and hence \[ \widehat{\psi}^{\downarrow}\circ f\sim 
\widehat{\psi}^{\downarrow}\circ(1_{\Lambda^{\downarrow}}\otimes \widehat{\phi}^{\downarrow})\circ (1_{\Lambda^{\downarrow}}\otimes g)\circ \psi^{\downarrow}\sim (1_{\Lambda^{\downarrow}}\otimes g)\circ \psi^{\downarrow}.\]  Similarly $\widehat{\psi}_{\uparrow}\circ h\sim (1_{\Lambda_{\uparrow}}\otimes g)\circ\psi_{\uparrow}$.  So we have commutative-up-to-homotopy diagrams \begin{equation}\label{twosquares} \xymatrix{\mathcal{C}^{\downarrow}\ar[r]^{\psi^{\downarrow}} \ar[d]_{f} & \Lambda^{\downarrow}\otimes_{\Lambda}\mathcal{C} \ar[d]^{1_{\Lambda^{\downarrow}}\otimes g} \\ \widehat{\mathcal{C}}^{\downarrow}\ar[r]_{\widehat{\psi}^{\downarrow}} & \Lambda^{\downarrow}\otimes_{\Lambda}\widehat{\mathcal{C}}}\qquad  \xymatrix{\mathcal{C}_{\uparrow}\ar[r]^{\psi_{\uparrow}} \ar[d]_{h} & \Lambda_{\uparrow}\otimes_{\Lambda}\mathcal{C} \ar[d]^{1_{\Lambda_{\uparrow}}\otimes g} \\ \widehat{\mathcal{C}}_{\uparrow}\ar[r]_{\widehat{\psi}_{\uparrow}} & \Lambda_{\uparrow}\otimes_{\Lambda}\widehat{\mathcal{C}}}   \end{equation}  It follows that the diagram \[ \xymatrix{ \mathcal{C}^{\downarrow}\oplus \mathcal{C}_{\uparrow}\ar[rrr]^{-j^{\downarrow}\otimes \psi^{\downarrow}+j_{\uparrow}\otimes\psi_{\uparrow}} \ar[d]_{f\oplus h} &&& \Lambda_{\updownarrow}\otimes_{\Lambda}\mathcal{C} \ar[d]^{1_{\Lambda_{\updownarrow}}\otimes g} \\ \widehat{\mathcal{C}}^{\downarrow}\oplus \widehat{\mathcal{C}}_{\uparrow}\ar[rrr]^{-j^{\downarrow}\otimes \widehat{\psi}^{\downarrow}+j_{\uparrow}\otimes\widehat{\psi}_{\uparrow}} &&&  \Lambda_{\updownarrow}\otimes_{\Lambda}\widehat{\mathcal{C}}  } \] also commutes up to homotopy, and so the result follows from Proposition \ref{filtinv}.
\end{proof}

\subsection{Building blocks}\label{bbs}

Let us compute the graded persistence modules $\mathbb{H}_*(\mathcal{CP})=\oplus_k\mathbb{H}_k(\mathcal{CP})$ associated by (\ref{hkdef}) to various rather simple examples of chain-level filtered matched pairs $\mathcal{CP}$.  As we will see later, when $\Gamma$ is discrete,  any chain-level filtered matched pair $\mathcal{CP}$ is filtered matched homotopy equivalent to a direct sum of examples such as those considered presently.

\subsubsection{$\mathcal{PE}_{\uparrow}(a,L,k)$}\label{peupsec}  First, borrowing notation from the second case of \cite[Definition 7.2]{UZ}, if $a\in \R$, $L\in[0,\infty)$, and $k\in\Z$, let $\mathcal{E}_{\uparrow}(a,L,k)$ denote the Floer-type complex $(E_*,\partial_E,\ell_E)$ given by setting $E_k$ and $E_{k+1}$ equal to one-dimensional vector spaces over $\Lambda_{\uparrow}$ generated by symbols $x$ and $y$ respectively while $E_j=\{0\}$ for $j\notin\{k,k+1\}$, putting $\partial_Ey=x$ and $\partial_Ex=0$, and defining the filtration function $\ell_E$ by $\ell_E(\lambda x)=a-\nu_{\uparrow}(\lambda)$ and $\ell_E(\lambda y)=a+L-\nu_{\uparrow}(\lambda)$.  The identity map on the chain complex $(E_*,\partial_E)$ is homotopic to the zero map (the chain homotopy is given by sending $x$ to $y$ and $y$ to $0$), in view of which we may define a chain-level filtered matched pair $\mathcal{PE}_{\uparrow}(a,L,k)$ by \[\xymatrix{ & 0 \\ 0 \ar[ru]^{0} \ar[rd]_{0} & \\ & \mathcal{E}_{\uparrow}(a,L,k)}\]  By definition, we have, for any $j\in\Z$ and $(s,t)\in \R^2$, \begin{align*} 
\mathbb{H}_j(\mathcal{PE}_{\uparrow}(a,L,k))_{s,t} & =H_{j+1}\left(\mathrm{Cone}\left((\{0\}\oplus \mathcal{E}_{\uparrow}(a,L,k))^{\preceq(s,t)}\to \{0\}\right)\right)
\\ &= H_j(\mathcal{E}_{\uparrow}(a,L,k)^{\leq t})=\frac{\{v\in E_j|\partial_Ev=0,\,\ell_E(v)\leq t\}}{\partial_E\left(\{w\in E_{j+1}|\ell_E(w)\leq t\}\right)},\end{align*} with structure maps $\mathbb{H}_j(\mathcal{PE}_{\uparrow}(a,L,k))_{s,t}\to \mathbb{H}_j(\mathcal{PE}_{\uparrow}(a,L,k))_{s',t'}$ for $s\leq s'$, $t\leq t'$ induced by inclusion.  Since $\partial_E$ has trivial kernel in degrees other than $k$ we thus have $\mathbb{H}_{j} (\mathcal{PE}_{\uparrow}(a,L,k))_{s,t} =\{0\}$ for $j\neq k$.  In degree $k$ we find \[
\mathbb{H}_k(\mathcal{PE}_{\uparrow}(a,L,k))_{s,t}=\frac{\{\lambda x|\lambda\in\Lambda_{\uparrow},a-\nu_{\uparrow}(\lambda)\leq t\}}{\{\lambda x|\lambda\in\Lambda_{\uparrow},(a+L)-\nu_{\uparrow}(\lambda)\leq t\}}.\]  Now an element $\lambda=\sum_{g\in \Gamma}c_gT^g$ of $\Lambda_{\uparrow}$ obeys $a-\nu_{\uparrow}(\lambda)\leq t$ if and only if all $g$ for which $c_g$ is nonzero obey $t\geq a-g$; the corresponding element $\lambda x$ will vanish in the above quotient if and only if all $g$ for which $c_g$ is nonzero obey $t\geq (a+L)-g$.  So, for any $t$, $\mathbb{H}_k(\mathcal{PE}_{\uparrow}(a,L,k))_{s,t}$ is isomorphic as a $\kappa$-vector space to the $\kappa$-vector space $V_{t,a,L}$ of elements $\sum_gc_gT^g$ of $\Lambda_{\uparrow}$ where the sum ranges over just those elements $g\in\Gamma$ with the property that $t\in [a-g,a+L-g)$.  The persistence module structure maps correspond under these isomorphisms to the maps $V_{t,a,L}\to V_{t',a,L}$ for $t\leq t'$ which truncate a sum $\sum_g c_gT^g\in V_{t,a,L}$ by deleting all terms with $t'\geq a+L-g$.  

Let us rephrase this calculation in terms of the ``block modules'' of \cite{CO}.  If $R\subset \R^2$ is any product of intervals $R=I\times J$ (including the case that $I$ and/or $J$ is all of $\R$), we have a persistence module $\kappa_R$ over $\R^2$ defined by \[ (\kappa_{R})_{s,t}=\left\{\begin{array}{ll} \kappa & \mbox{if }(s,t)\in R \\ \{0\} & \mbox{otherwise} \end{array}\right. \] with structure maps $(\kappa_R)_{s,t}\to(\kappa_R)_{s',t'}$ given by the identity if $(s,t),(s',t')\in R$ and $0$ otherwise.  The above discussion shows that \begin{equation}\label{peup} \mathbb{H}_j(\mathcal{PE}_{\uparrow}(a,L,k))\cong\left\{\begin{array}{ll} \bigoplus_{g\in\Gamma} \kappa_{\R\times[a-g,a+L-g)} & \mbox{if }j=k \\ \{0\} & \mbox{otherwise}\end{array}\right.,\end{equation} with the summand $\kappa_{\R\times[a-g,a+L-g)}$ corresponding to the coefficient $c_g$ in an element $\sum_g c_gT^g$ as in the previous paragraph.  (In particular, $\mathbb{H}_j(\mathcal{PE}_{\uparrow}(a,L,k))$ is trivial if $L=0$.) 

\subsubsection{$\mathcal{PE}^{\downarrow}(b,L,k)$}\label{pedownsec} Dually, if $b\in \R$, $L\in[0,\infty)$, and $k\in \Z$ we have a $\Lambda^{\downarrow}$-Floer-type complex $\mathcal{E}^{\downarrow}(b,L,k)=(E_*,\partial_E,\ell_E)$ given by taking $E_j$ to be $\{0\}$ for $j\notin\{k,k+1\}$ and to be generated in degree $k$ by an element $x$ with $\ell_E(x)=b$ and in degree $k+1$ by an element $y$ with $\ell_E(y)=b-L$ and $\partial_Ey=x$.  (This $\Lambda^{\downarrow}$-Floer type complex is obtained by conjugating the $\Lambda_{\uparrow}$-Floer-type complex $\mathcal{E}_{\uparrow}(-b,L,k)$ along the lines of Remark \ref{switch}.)  $\mathcal{E}^{\downarrow}(b,L,k)$ fits into the following chain-level filtered matched pair, denoted $\mathcal{PE}^{\downarrow}(b,L,k)$: \[  \xymatrix{ & \mathcal{E}^{\downarrow}(b,L,k) \\ 0 \ar[ru]^{0} \ar[rd]_{0} & \\ & 0} \]
Just as in the previous case, one has \begin{align*} \mathbb{H}_j(\mathcal{PE}^{\downarrow}(b,L,k))_{s,t}&=H_{j+1}\left(\mathrm{Cone}\left((\mathcal{E}^{\downarrow}(b,L,k)\oplus\{0\})^{\preceq(s,t)}\to \{0\}\right)\right) \\ &=H_j((\mathcal{E}^{\downarrow}(b,L,k)\oplus\{0\})^{\preceq(s,t)}).\end{align*}  This vanishes for $j\neq k$, and for $j=k$ it equals \[ \frac{\{\mu x|\mu\in \Lambda^{\downarrow},b-\nu^{\downarrow}(\mu)\geq -s\} }{\{\mu x|\mu\in\Lambda^{\downarrow},(b-L)-\nu^{\downarrow}(\mu)\geq -s\}}.\]  

The same analysis as in Section \ref{peupsec} then yields isomorphisms of persistence modules \begin{equation}\label{pedown} \mathbb{H}_j(\mathcal{PE}^{\downarrow}(b,L,k))\cong \left\{\begin{array}{ll} \bigoplus_{g\in\Gamma} \kappa_{[-b-g,-b+L-g)\times \R} & \mbox{if }j=k \\ \{0\} & \mbox{otherwise}\end{array}\right.,\end{equation} 

\subsubsection{$\mathcal{PM}(a,b,k)$} \label{mabsect}  Given $a,b\in \R$ and $k\in \Z$ we define a chain-level filtered matched pair as follows.  The chain complex $\mathcal{C}=(C_*,\partial)$ of $\Lambda$-modules is given by setting $C_k=\Lambda$ and $C_j=\{0\}$ for $j\neq k$, so necessarily $\partial=0$.  The $\Lambda_{\uparrow}$-Floer-type complex $\mathcal{C}_{\uparrow}$ is likewise zero in degrees other than $k$, has zero  differential, and has degree-$k$ chain group given by a copy of $\Lambda_{\uparrow}$ equipped with the filtration function $\ell_{\uparrow a}(\lambda)=a-\nu_{\uparrow}(\lambda)$.  Likewise, the $\Lambda^{\downarrow}$-Floer-type complex $\mathcal{C}^{\downarrow}$ is zero in degrees other than $k$, has zero differential, and has degree-$k$ chain group given by a copy of $\Lambda^{\downarrow}$ equipped with the filtration function $\ell^{\downarrow b}(\mu)=b-\nu^{\downarrow}(\mu)$.  The filtered matched pair $\mathcal{PM}(a,b,k)$ is then defined using these $\mathcal{C},\mathcal{C}_{\uparrow},\mathcal{C}^{\downarrow}$, taking the maps $\phi_{\uparrow}$ and $\phi^{\downarrow}$ to be given in degree $k$ by the standard inclusions of $\Lambda$ into $\Lambda_{\uparrow}$ and $\Lambda^{\downarrow}$.

The maps $1_{\Lambda_{\uparrow}}\otimes \phi_{\uparrow},1_{\Lambda^{\downarrow}}\otimes\phi^{\downarrow}$ are, in degree $k$, just the standard canonical isomorphisms $\Lambda_{\uparrow}\otimes_{\Lambda}\Lambda\cong \Lambda_{\uparrow}$ and $\Lambda^{\downarrow}\otimes_{\Lambda}\Lambda\cong \Lambda^{\downarrow}$, so for the homotopy inverses $\psi_{\uparrow},\psi^{\downarrow}$ we can (indeed must) take the inverses of these isomorphisms.  It then follows that (implicitly passing through the canonical isomorphism $\Lambda_{\updownarrow}\otimes_{\Lambda}\Lambda\cong \Lambda_{\updownarrow}$) the map $-j^{\downarrow}\otimes\psi^{\downarrow}+j_{\downarrow}\otimes\psi_{\uparrow}$ appearing in (\ref{hkdef}) is identified in degree $k$ with the map $\delta\co \Lambda^{\downarrow}\oplus\Lambda_{\uparrow}\to  \Lambda_{\updownarrow}$ defined by $(\mu,\lambda)\mapsto -j^{\downarrow}\mu+j_{\uparrow}\lambda$.  The $(s,t)$-filtered subcomplex $(C^{\downarrow}_{*}\oplus C_{\uparrow}*)^{\preceq(s,t)}$ consists in degree $k$ of the space $(\Lambda^{\downarrow}\oplus\Lambda_{\uparrow})^{\preceq_{a,b}(s,t)}$ of those $(\mu,\lambda)\in\Lambda^{\downarrow}\oplus\Lambda_{\uparrow}$ with $b-\nu^{\downarrow}(\mu)\geq -s$ and $a-\nu_{\uparrow}(\lambda)\leq t$.

The complex $\mathrm{Cone}\left(\xymatrix{\left(C^{\downarrow}_{*}\oplus C_{\uparrow *}\right)^{\preceq (s,t)}\ar[rrr]^{-j^{\downarrow}\otimes \psi^{\downarrow}+j_{\uparrow}\otimes\psi_{\uparrow}}&&&\Lambda_{\updownarrow}\otimes_{\Lambda}C_* }\right)$ of (\ref{hkdef}) is thus given in degree $k+1$ by $(\Lambda^{\downarrow}\oplus\Lambda_{\uparrow})^{\preceq_{a,b}(s,t)}$, in degree $k$ by $\Lambda_{\updownarrow}$, and in all other degrees by $\{0\}$, and the only nontrivial differential is given by the restriction of the subtraction map $\delta\co \Lambda^{\downarrow}\oplus\Lambda_{\uparrow}\to \Lambda_{\updownarrow}$.  Hence, noting the degree shift in (\ref{hkdef}), \[ \mathbb{H}_k(\mathcal{PM}(a,b,k))_{s,t}\cong \ker(\delta|_{(\Lambda^{\downarrow}\oplus\Lambda_{\uparrow})^{\preceq_{a,b}(s,t)}}),\quad \mathbb{H}_{k-1}(\mathcal{PM}(a,b,k))_{s,t}\cong\mathrm{coker}(\delta|_{(\Lambda^{\downarrow}\oplus\Lambda_{\uparrow})^{\preceq_{a,b}(s,t)}}),\] and $\mathbb{H}_{j}(\mathcal{PM}(a,b,k))_{s,t}=\{0\}$ for $j\notin\{k-1,k\}$, with the persistence module structure maps induced by the inclusons $(\Lambda^{\downarrow}\oplus\Lambda_{\uparrow})^{\preceq_{a,b}(s,t)}\hookrightarrow (\Lambda^{\downarrow}\oplus\Lambda_{\uparrow})^{\preceq_{a,b}(s',t')}$.  
 
Let us now compute the above kernel and cokernel.  A pair $(\mu,\lambda)\in \Lambda^{\downarrow}\oplus\Lambda_{\uparrow}$ lies in the kernel of the difference map $\delta\co \Lambda^{\downarrow}\oplus\Lambda_{\uparrow}\to\Lambda_{\updownarrow}$ if and only if $\mu$ and $\lambda$ both lie in the common submodule $\Lambda=\kappa[\Gamma]$ of $\Lambda_{\uparrow}$ and $\Lambda^{\downarrow}$, and are equal to each other.  Since $(\mu,\lambda)\in (\Lambda^{\downarrow}\oplus\Lambda_{\uparrow})^{\preceq_{a,b}(s,t)}$ if and only if $\nu^{\downarrow}(\mu)\leq b+s$ and $\nu_{\uparrow}(\lambda)\geq a-t$, this identifies $\mathbb{H}_k(\mathcal{PM}(a,b,k))_{s,t}$ with the sub-$\kappa$-vector space of $\Lambda$ consisting of sums $\sum_{g\in \Gamma}c_gT^g$ for which each $g$ appearing in the sum obeys $a-t\leq g\leq b+s$ (equivalently, $t\geq a-g$ and $s\geq -b+g$).  The persistence module maps for $s\leq s'$ and $t\leq t'$ are just the inclusions.  This yields an isomorphism of persistence modules \begin{equation}\label{hkm} \mathbb{H}_k(\mathcal{PM}(a,b,k))\cong \bigoplus_{g\in \Gamma}\kappa_{[-b+g,\infty)\times [a-g,\infty)}.\end{equation}

We now turn to the cokernel of $\delta|_{(\Lambda^{\downarrow}\oplus\Lambda_{\uparrow})^{\preceq_{a,b}(s,t)}}$.   Consider a general element $\gamma=\sum_{g\in\Gamma}c_g T^g$ of $\Lambda_{\updownarrow}$; by definition of $\Lambda_{\updownarrow}$, this in general may be an infinite sum, subject to the constraint that in each bounded interval there are only finitely many $g$ with $c_g\neq 0$.  If $a-t\leq b+s$ we may (perhaps non-uniquely) partition $\{g|c_g\neq 0\}$ into sets $S_1,S_2$ such that each $g\in S_1$ has $g\leq b+s$ and each $g\in S_2$ has $g\geq a-t$.  Then letting $\mu=-\sum_{g\in S_1}c_gT^g$ and $\lambda=\sum_{g\in S_2}c_gT^g$, we see that $(\mu,\lambda)\in  (\Lambda^{\downarrow}\oplus\Lambda_{\uparrow})^{\preceq_{a,b}(s,t)}$ and $\delta(\mu,\lambda)=\gamma$.  Thus, when $a-t\leq b+s$, the map $\delta|_{(\Lambda^{\downarrow}\oplus\Lambda_{\uparrow})^{\preceq_{a,b}(s,t)}}$ has trivial cokernel.  Assuming instead that $b+s<a-t$, given a general element $\gamma=\sum_{g\in \Gamma}c_gT^g$ as above, we may uniquely express $\gamma$ in the form \[ \gamma=-\mu+\left(\sum_{b+s<g<a-t}c_gT^g\right) +\lambda \] where $(\mu,\lambda)\in (\Lambda^{\downarrow}\oplus\Lambda_{\uparrow})^{\preceq_{a,b}(s,t)}$. This yields an isomorphism \[ \mathbb{H}_{k-1}(\mathcal{PM}(a,b,k))_{s,t}\cong\bigoplus_{g\in\Gamma\cap(b+s,a-t)}\kappa,\] with the persistence module structure maps associated to increasing both $s$ and $t$ (and hence shrinking the interval $(b+s,a-t)$) corresponding to the obvious projection.  Thus, in terms of block modules, \begin{equation}\label{hk-1m} \mathbb{H}_{k-1}(\mathcal{PM}(a,b,k))\cong \bigoplus_{g\in \Gamma}\kappa_{(-\infty,-b+g)\times (-\infty,a-g)}.\end{equation}

Since the graded persistence modules $\mathbb{H}_*(\mathcal{PE}_{\uparrow}(a,L,k))$, $\mathbb{H}_*(\mathcal{PE}^{\downarrow}(b,L,k))$, $\mathbb{H}_*(\mathcal{PM}(a,b,k))$ are meant to be building blocks for an algebraic model of interlevel persistence, it is not a coincidence that (when $\Gamma=\{0\}$) they resemble the block $MV$-systems of \cite[Definition 2.16]{BGO}, though the fact that \cite{BGO} works with persistence modules defined only over $\{(a,b)\in \R^2|a+b>0\}$ rather than all of $\R^2$ leads to some slight differences.  When $\Gamma$ is nontrivial, we have seen that these building blocks are modified from the $\Gamma=\{0\}$ case simply by making them $\Gamma$-periodic in a natural way.   On the other hand, consistently with \cite{BH17}, an additional type of building block arises when $\Gamma\neq\{0\}$ that is unlike the block $MV$-systems of \cite{BGO}, owing to the fact that, when $\Gamma\neq\{0\}$, modules over $\kappa[\Gamma]$ can have torsion.  We describe this final building block now.

\subsubsection{$\mathcal{PR}(T,k)$}\label{pr}  Let $T$ be a finitely-generated, torsion $\Lambda$-module, and $k\in \Z$.  Let \begin{equation}\label{rest} \cdots\to F_{m}\to F_{m-1}\to\cdots\to F_1\to F_0\to T\to 0 \end{equation} be a resolution of $T$ by finitely-generated, free $\Lambda$-modules, as in \cite[Proof of Lemma 7.19]{Rot}.  Let $\mathcal{C}$ be the chain complex given in degree $j$ by $F_{j-k}$ (interpreted as $\{0\}$ if $j<k$), with differentials given by the maps in the resolution.  Thus $H_k(\mathcal{C})\cong T$ while all other $H_j(\mathcal{C})$ vanish.   By Proposition \ref{kerext}, $\Lambda_{\uparrow}\otimes_{\Lambda}T\cong \Lambda^{\downarrow}\otimes_{\Lambda}T=\{0\}$ since $T$ is torsion.  So since $\Lambda_{\uparrow}$ and $\Lambda^{\downarrow}$ are flat $\Lambda$-modules, the exactness of (\ref{rest}) implies that the complexes $\Lambda_{\uparrow}\otimes_{\Lambda}\mathcal{C}$ and $\Lambda^{\downarrow}\otimes_{\Lambda}\mathcal{C}$ are acyclic.  Since an acyclic chain complex over a field is chain homotopy equivalent to the zero complex, the following diagram defines a chain-level filtered matched pair $\mathcal{PR}(T,k)$: \[ \xymatrix{ & 0 \\ \mathcal{C} \ar[ur]^{0} \ar[dr]_{0} & \\ & 0 } \]  (The notation $\mathcal{PR}(T,k)$ suppresses the choice of free resolution, but since any two free resolutions are homotopy equivalent, the filtered matched homotopy equivalence class of $\mathcal{PR}(T,k)$ depends only on $T$ and $k$.)  

By definition, we have \[ \mathbb{H}_j(\mathcal{PR}(T,k))_{s,t}=H_{j+1}\left(\mathrm{Cone}(0\to \Lambda_{\updownarrow}\otimes_{\Lambda}\mathcal{C})\right) = H_{j+1}(\Lambda_{\updownarrow}\otimes_{\Lambda}\mathcal{C}),\] independently of $s,t$ (and with structure maps $\mathbb{H}_j(\mathcal{PR}(T,k))_{s,t}\to \mathbb{H}_j(\mathcal{PR}(T,k))_{s',t'}$ equal to the identity).  Since $\Lambda_{\updownarrow}\otimes_{\Lambda}\mathcal{C}$ is given in degree $j$ by $\{0\}$ if $j<k$ and by $\Lambda_{\updownarrow}\otimes F_{j-k}$ otherwise, with differentials extended from (\ref{rest}), 
we conclude that $\mathbb{H}_{j}(\mathcal{PR}(T,k))_{s,t}\cong \mathrm{Tor}_{j+1-k}^{\Lambda}(T,\Lambda_{\updownarrow})$ (which trivially vanishes if $j<k-1$).  By Proposition \ref{tor}, since $T$ is torsion, $\mathrm{Tor}_{j+1-k}^{\Lambda}(T,\Lambda_{\updownarrow})$ vanishes if $j>k$ and is isomorphic to $T$ if $j=k$.  In the remaining case that $j=k-1$, we have $\mathrm{Tor}_{0}^{\Lambda}(T,\Lambda_{\updownarrow})\cong \Lambda_{\updownarrow}\otimes_{\Lambda}T$, which vanishes due to the exact sequence $0\to\Lambda\to\Lambda_{\uparrow}\oplus \Lambda^{\downarrow}\to\Lambda_{\updownarrow}\to 0$ and the right exactness of $(\cdot)\otimes_{\Lambda}T$.  So $\mathbb{H}_{j}(\mathcal{PR}(T,k))_{s,t}$ vanishes in all degrees other than $k$, and is isomorphic (independently of $(s,t)$) to $T$ for $j=k$.  So if the dimension of $T$ as a vector space over $\kappa$ is $d$, we have (as persistence modules of $\kappa$-vector spaces over $\mathbb{R}^2$) \[ \mathbb{H}_j(\mathcal{PR}(T,k))\cong \left\{\begin{array}{ll} 0 &\mbox{if }j\neq k \\ \left(\kappa_{(-\infty,\infty)\times(-\infty,\infty)}\right)^{\oplus d} & \mbox{if }j=k\end{array}\right..\]

\subsection{Direct sum decompositions and barcodes}\label{decompsect}

In this section we recall some facts about Floer-type complexes and their barcodes from \cite{UZ}, and we combine these with results from Sections \ref{fmpsec} and \ref{bbs} to show that, when $\Gamma$ is discrete, chain-level filtered matched pairs can be described up to filtered matched homotopy equivalence by a combination of the barcodes of \cite{UZ} and the basis spectra of Definition \ref{hbsdef}.  In \cite{UZ} our notational convention for handling the symmetry of our persistence modules with respect to the action of $\Gamma$ was to define the elements of our barcodes to be pairs $([a],L)\in (\mathbb{R}/\Gamma)\times [0,\infty]$; these were meant to correspond to equivalence classes intervals $[a,a+L)$ with the left endpoint $a$ (but not the length $L$) defined only modulo $\Gamma$.    To avoid confusion with the similar convention in Definition \ref{hbsdef}, and also because in the current context we will need to maintain distinctions between intervals of all four types $[a,b),(a,b],[a,b],(a,b)$, our convention in this paper is to represent elements of the barcode directly as equivalence classes $I^{\Gamma}$ of intervals $I$ modulo $\Gamma$-translation, as in Notation \ref{intnot}.

Let $\mathcal{C}_{\uparrow}=(C_{\uparrow*},\partial_{\mathcal{C}_{\uparrow}},\ell_{\uparrow})$ be a $\Lambda_{\uparrow}$-Floer-type complex. By \cite[Corollaries 2.17 and 2.18]{UZ}, the subspaces $\Img(\partial_{\mathcal{C}_{\uparrow}}|_{C_{\uparrow(k+1)}})$ and $\ker(\partial_{\mathcal{C}_{\uparrow}}|_{C_{\uparrow k}})$ of the orthogonalizable $\Lambda_{\uparrow}$-space $C_{\uparrow k}$ are both orthogonalizable, any $\ell_{\uparrow}$-orthogonal basis for $\Img(\partial_{\mathcal{C}_{\uparrow}}|_{C_{\uparrow(k+1)}})$ can be extended to an $\ell_{\uparrow}$-orthogonal basis for $\ker(\partial_{\mathcal{C}_{\uparrow}}|_{C_{\uparrow k}})$, and any $\ell_{\uparrow}$-orthogonal basis for $\ker(\partial_{\mathcal{C}_{\uparrow}}|_{C_{\uparrow k}})$ can be extended to an $\ell_{\uparrow}$-orthogonal basis for $C_{\uparrow k}$.  \cite[Theorem 3.4]{UZ} shows that each map $\partial_{\mathcal{C}_{\uparrow}}|_{C_{\uparrow(k+1)}}\co C_{\uparrow (k+1)}\to \ker(\partial_k)$ admits a (nonarchimedean) singular value decomposition: there are $\ell_{\uparrow}$-orthogonal bases $\{y_{1}^{k+1},\ldots,y_{r}^{k+1},\ldots,y_{n}^{k+1}\}$ for $C_{\uparrow (k+1)}$ and $\{x_{1}^{k},\ldots,x_{r}^{k},\ldots,x_{m}^{k}\}$ for $\ker(\partial_{\mathcal{C}_{\uparrow}}|_{C_{\uparrow k}})$ with $\{y_{r+1}^{k+1},\ldots,y_{n}^{k+1}\}$ a basis for $\ker(\partial_{\mathcal{C}_{\uparrow}}|_{C_{\uparrow (k+1)}})$  while $\partial_{\mathcal{C}_{\uparrow}} y_{i}^{k+1}=x_{i}^{k}$ for $i=1,\ldots,r$.  
Write $F_{\uparrow(k+1)}=\mathrm{span}_{\Lambda_{\uparrow}}\{y_{1}^{k+1},\ldots,y_{r}^{k+1}\}$.   Then $F_{\uparrow(k+1)}$ and $\ker(\partial_{\mathcal{C}_{\uparrow}}|_{C_{\uparrow (k+1)}})$ are orthogonal complements in $C_{\uparrow(k+1)}$ (\emph{i.e.} their direct sum is $C_{\uparrow(k+1)}$ and $\ell_{\uparrow}(f+z)=\max\{\ell_{\uparrow}(f),\ell_{\uparrow}(z)\}$ for $f\in F_{\uparrow(k+1)}$ and $z\in \ker(\partial_{\mathcal{C}_{\uparrow}}|_{C_{\uparrow (k+1)}})$), so, more generally, the union of any $\ell_{\uparrow}$-orthogonal basis for $F_{\uparrow(k+1)}$ and any $\ell_{\uparrow}$-orthogonal basis for $\ker(\partial_{\mathcal{C}_{\uparrow}}|_{C_{\uparrow (k+1)}})$  gives an orthogonal basis for $C_{\uparrow(k+1)}$.

Similarly, we obtain a singular value decomposition of $\partial_{\mathcal{C}_{\uparrow}}|_{C_{\uparrow k}}$, yielding in particular a subspace $F_{\uparrow k}\subset C_{\uparrow k}$ with an $\ell_{\uparrow}$-orthogonal basis $\{y_{1}^{k},\ldots,y_{s}^{k}\}$ such that $F_{\uparrow k}$ and $\ker(\partial_{\mathcal{C}_{\uparrow}}|_{C_{\uparrow k}})$  are orthogonal complements and such that $\{\partial_{\mathcal{C}_{\uparrow}}y_{1}^{k},\ldots,\partial_{\mathcal{C}_{\uparrow}}y_{s}^{k}\}$ is an $\ell_{\uparrow}$-orthogonal basis for $\Img(\partial_{\mathcal{C}_{\uparrow}}|_{C_{\uparrow k}})$.  Writing $B_{\uparrow k}=\Img(\partial_{\mathcal{C}_{\uparrow}}|_{C_{\uparrow (k+1)}})$ (so that $B_{\uparrow k}=\mathrm{span}_{\Lambda_{\uparrow}}\{x_{1}^{k},\ldots,x_{r}^{k}\}$) and $H_{\uparrow k}=\mathrm{span}_{\Lambda}\{x_{r+1}^{k},\ldots,x_{m}^{k}\}$, we have thus decomposed each $C_{\uparrow k}$ as a direct sum of three orthogonal subspaces $B_{\uparrow k},H_{\uparrow k},F_{\uparrow k}$ such that $\partial_{\mathcal{C}_{\uparrow}}|_{B_{\uparrow k}\oplus H_{\uparrow k}}=0$ and $\partial_{\mathcal{C}_{\uparrow}}$ maps the $\ell_{\uparrow}$-orthogonal basis $\{y_{1}^{k+1},\ldots,y_{r}^{k+1}\}$ for $F_{\uparrow(k+1)}$ to the $\ell_{\uparrow}$-orthogonal basis $\{x_{1}^{k},\ldots,x_{r}^{k}\}$ for $B_{\uparrow k}$.  

The \textbf{degree $k$ verbose barcode} of the $\Lambda_{\uparrow}$-Floer-type complex $\mathcal{C}_{\uparrow}$ is defined in \cite[Definition 6.3]{UZ} as the multiset of elements of $(\R/\Gamma)\times [0,\infty]$ consisting of the pairs $([\ell_{\uparrow}(x_{i}^{k})],\ell_{\uparrow}(y_{i}^{k})-\ell_{\uparrow}(x_{i}^{k}))$ for $i=1,\ldots,r$ as well as the pairs $([\ell_{\uparrow}(x_{i}^{k})],\infty)$ for $i=r+1,\ldots,m$ (where for $a\in \R$ we write $[a]$ for its equivalence class in $\R/\Gamma$).  \cite[Theorem 7.1]{UZ} shows that the verbose barcode is independent of the choice of singular value decomposition used to define it. 
The \emph{concise barcode} of $\mathcal{C}_{\uparrow}$ is, by definition, obtained from the verbose barcode by discarding all intervals of zero length, and is shown in \cite[Theorem B]{UZ} to be a complete invariant of $\mathcal{C}_{\uparrow}$ up to filtered chain homotopy equivalence (whereas the verbose barcode is a complete invariant up to the finer relation of filtered chain isomorphism); in particular, the distinction between the verbose and concise barcodes is immaterial from the standpoint of homology persistence modules.  In keeping with Notation \ref{intnot}, in this paper we represent elements of the concise barcode as equivalence classes $[\ell_{\uparrow}(x_{i}^{k}),\ell_{\uparrow}(y_{i}^{k+1}))^{\Gamma}$ (for $i\leq r$, provided that $\ell_{\uparrow}(x_{i}^{k})<\ell_{\uparrow}(y_{i}^{k+1})$) or $[\ell_{\uparrow}(x_{i}^{k}),\infty)^{\Gamma}$ (for $i>r$) of intervals modulo $\Gamma$-translation.

\begin{prop}\label{iPi}
With notation as above, the inclusion of chain complexes $\iota_{\uparrow}\co \left(\oplus_k H_{\uparrow k},0\right)\to \left(C_{\uparrow *},\partial_{\mathcal{C}_{\uparrow}}\right)$ is a homotopy equivalence, with homotopy inverse $\Pi_{\uparrow}$ restricting to each $H_{\uparrow k}$ as the identity and mapping both $B_{\uparrow k}$ and $F_{\uparrow k}$ to $\{0\}$.
\end{prop}

\begin{proof}
By definition, $\Pi_{\uparrow}\circ \iota_{\uparrow}$ is the identity on $\oplus_k H_{\uparrow k}$.  If we define $L\co C_{\uparrow *}\to C_{\uparrow (*+1)}$ to be zero on each $H_{\uparrow k}\oplus F_{\uparrow k}$ and to send the basis element $x_{i}^{k}\in B_{\uparrow k}$ to the corresponding $y_{i}^{k+1}\in F_{\uparrow(k+1)}$, one readily checks that $L$ defines a homotopy between $\iota_{\uparrow}\circ\Pi_{\uparrow}$ and the identity on $C_{\uparrow *}$.
\end{proof}

Since the homology of a complex with zero differential is just the underlying chain module of the complex, the map $\iota_{\uparrow}$ of the previous proposition induces an isomorphism $\iota_{\uparrow *}$ from the submodule $H_*=\oplus_k H_{\uparrow k}$ of $\oplus_kC_{\uparrow k}$ to the homology $H_*(\mathcal{C}_{\uparrow})$.  The $(H_{\uparrow k},\ell_{\uparrow}|_{H_{\uparrow k}})$ are orthogonalizable $\Lambda_{\uparrow}$-spaces by \cite[Corollary 2.17]{UZ}, so by putting $\rho_{\uparrow k}=\ell_{\uparrow}\circ \iota_{\uparrow *}^{-1}|_{H_k(\mathcal{C}_{\uparrow})}$ we obtain orthogonalizable $\Lambda_{\uparrow}$-spaces $(H_k(\mathcal{C}_{\uparrow}),\rho_{\uparrow k})$.  It follows from \cite[Proposition 6.6]{UZ} that the map $\rho_{\uparrow k}\co H_k(\mathcal{C}_{\uparrow})\to \R\cup\{-\infty\}$ assigns to each $\alpha\in H_k(\mathcal{C}_{\uparrow})$ the ``spectral invariant'' \begin{equation}\label{specup} \rho_{\uparrow k}(\alpha)=\inf\{\ell_{\uparrow}(a)|a\in\ker\partial_{\mathcal{C}_{\uparrow}}|_{C_{\uparrow k}},\,[a]=\alpha\in H_k(\mathcal{C}_{\uparrow})\}.\end{equation}  In particular this function $\rho_{\uparrow k}$ is independent of the choice of the singular value decomposition that we initially used to obtain it.  By \cite[Proposition 5.5]{UZ}, for any $\rho_{\uparrow k}$-orthogonal basis $\{h_1,\ldots,h_{m-r}\}$ for $H_k(\mathcal{C}_{\uparrow})$, the multiset of values $\{\rho_{\uparrow k}(h_1)\mod \Gamma,\ldots,\rho_{\uparrow k}(h_{m-r})\mod\Gamma\}$ is, independently of the choice of orthogonal basis, equal to $\{\ell_{\uparrow}(x_{r+1}^{k})\mod\Gamma,\ldots,\ell_{\uparrow}(x_m^k)
\mod\Gamma\}$, \emph{i.e.}, to the collection of left endpoints (modulo $\Gamma$) of the infinite-length bars in the degree-$k$ (verbose or concise) barcode.

As usual, using Remark \ref{switch}, analogous statements hold for $\Lambda^{\downarrow}$-Floer-type complexes $\mathcal{C}^{\downarrow}=(C^{\downarrow}_{*},\partial_{\mathcal{C}^{\downarrow}},\ell^{\downarrow})$: each $C^{\downarrow}_{k}$ decomposes orthogonally as $C^{\downarrow}_{k}=B^{\downarrow}_{k}\oplus H^{\downarrow}_{k}\oplus F^{\downarrow}_{k}$, with $\partial_{\mathcal{C}^{\downarrow}}$ vanishing on $B_{k}^{\downarrow}\oplus H_{k}^{\downarrow}$ and mapping some orthogonal basis $\{y_{1}^{k+1},\ldots,y_{r}^{k+1}\}$ for $F^{\downarrow}_{k+1}$ bijectively to an orthogonal basis for $B^{\downarrow}_{k}$.  Writing $x_{i}^{k}=\partial_{\mathcal{C}^{\downarrow}}y_{i}^{k+1}$ for $i=1,\ldots,r$ and letting $\{x_{r+1}^{k},\ldots,x_{m}^{k}\}$ be an orthogonal basis for $H^{\downarrow}_{k}$, the degree-$k$ concise barcode of $\mathcal{C}^{\downarrow}$ is taken to consist of half-open intervals modulo $\Gamma$-translation $(\ell^{\downarrow}(y_{i}^{k+1}),\ell^{\downarrow}(x_{i}^{k})]^{\Gamma}$ for those $i\in\{1,\ldots,r\}$ with $\ell^{\downarrow}(y_{i}^{k+1})\neq \ell^{\downarrow}(x_{i}^{k})$, and $(-\infty,\ell^{\downarrow}(x_{i}^{k})]^{\Gamma}$ for $i=r+1,\ldots,m$.  The inclusions $\iota^{\downarrow}$ of $H^{\downarrow}_{k}$ (with zero differential) into $C^{\downarrow}_{k}$ give a chain homotopy equivalence with homotopy inverse given by the projection $\Pi^{\downarrow}\co C^{\downarrow}_{k}\to H^{\downarrow}_{k}$ associated to the direct sum decomposition $C^{\downarrow}_{k}=B^{\downarrow}_{k}\oplus H^{\downarrow}_{k}\oplus F^{\downarrow}_{k}$.  The resulting isomorphism $H^{\downarrow}_{k}\cong H_k(\mathcal{C}^{\downarrow})$ identifies $\ell^{\downarrow}|_{H^{\downarrow}_{k}}$ with the spectral invariant $\rho^{\downarrow}_{k}\co H_k(\mathcal{C}^{\downarrow})\to\R\cup\{\infty\}$ defined by \begin{equation}\label{specdown} \rho^{\downarrow}_{k}(\alpha)=\sup\{\ell^{\downarrow}(a)|a\in\ker(\partial_{\mathcal{C}^{\downarrow}}|_{C^{\downarrow}_{k}}),\,[a]=\alpha\in H_k(\mathcal{C}^{\downarrow})\}.\end{equation} 

Now suppose that we have a chain-level filtered matched pair $\mathcal{CP}=(\mathcal{C},\mathcal{C}^{\downarrow},\mathcal{C}_{\uparrow},\phi^{\downarrow},\phi_{\uparrow})$ as in Definition \ref{cfmpdfn}.  In particular $\mathcal{C}=(C_*=\oplus_kC_k,\partial)$ is a chain complex of free, finitely-generated $\Lambda$-modules.  In order to ensure that we can obtain suitable direct sum decompositions we \textbf{assume at this point that the subgroup $\Gamma$ of $\mathbb{R}$ is discrete}, and hence that $\Lambda$ is a PID. Then each $\ker(\partial|_{C_k})$ is also a finitely generated free $\Lambda$-module, and so by putting the differential $\partial|_{C_{k+1}}\co C_{k+1}\to \ker(\partial|_{C_k})$ in Smith
normal form we obtain bases $\{y_{1}^{k+1},\ldots,y_{r}^{k+1},y_{r+1}^{k+1},\ldots,y_{n}^{k+1}\}$ for $C_{k+1}$ and $\{x_{1}^{k},\ldots,x_{r}^{k},\ldots,x_{m}^{k}\}$ for $\ker(\partial|_{C_k})$, as well as nonzero elements $\alpha_1,\ldots,\alpha_r\in \Lambda$ such that $\alpha_i|\alpha_{i+1}$, with the property that $\partial y_{i}^{k+1}=\alpha_ix_{i}^{k}$ for $i=1,\ldots,r$ and $\partial y_{i}^{k+1}=0$ for $i=r+1,\ldots,n$.  

Analogously to our discussion of $\mathcal{C}_{\uparrow}$ and $\mathcal{C}^{\downarrow}$, but with modest additional complication due to the fact that the $\alpha_i$ may not be invertible, write $F_{k+1}=\mathrm{span}_{\Lambda}\{y_{1}^{k+1},\ldots,y_{r+1}^{k+1}\}$, $\tilde{B}_k=\mathrm{span}_{\Lambda}\{x_{1}^{k},\ldots,x_{r}^{k}\}$, and $H_{k}^{\mathrm{free}}=\mathrm{span}_{\Lambda}\{x_{r+1}^{k},\ldots,x_{m}^{k}\}$.  (Thus $\ker(\partial|_{C_k})=H_{k}^{\mathrm{free}}\oplus \tilde{B}_k$, and the $k$th homology $H_k(\mathcal{C})$ is the direct sum of $H_{k}^{\mathrm{free}}$ with the torsion submodule $\frac{\tilde{B}_k}{\mathrm{span}_{\Lambda}\{\alpha_1x_{1}^{k},\ldots,\alpha_rx_{r}^{k}\}}$.)  In just the same way, the Smith normal form of $\partial|_{C_k}$ yields submodules $F_k$ of $C_k$ and $H_{k-1}^{\mathrm{free}}$ and $\tilde{B}_{k-1}$ of $\ker(\partial|_{C_{k-1}})$.  In particular, we have direct sum decompositions $C_k=F_k\oplus \ker(\partial|_{C_k})=F_k\oplus H_{k}^{\mathrm{free}}\oplus \tilde{B}_k$.

Writing $H_{*}^{\mathrm{free}}=\oplus_k H_{k}^{\mathrm{free}}$, let $\iota\co H_{*}^{\mathrm{free}}\to C_*$ be the inclusion and let $\Pi\co C_*\to H_{*}^{\mathrm{free}}$ be given in each degree $k$ by the projection associated to the direct sum decomposition $C_k=F_k\oplus H_{k}^{\mathrm{free}}\oplus \tilde{B}_k$.  Regarding $H_{*}^{\mathrm{free}}$ as a chain complex with zero differential, we see that $\iota$ is a chain map because $\partial|_{H_{k}^{\mathrm{free}}}=0$, and $\Pi$ is a chain map because $\Img(\partial|_{C_{k+1}})$ is contained in $\tilde{B}_k$ which is annihilated by $\Pi$.  In contrast to the situation of Proposition \ref{iPi}, $\iota$ and $\Pi$ should not be expected to be homotopy inverses since $H_k(\mathcal{C})$ may have torsion in which case it is not isomorphic to $H_{k}^{\mathrm{free}}$.  However, recalling that we are assuming $\mathcal{C}$ to be part of a chain-level filtered matched pair $\mathcal{CP}=(\mathcal{C},\mathcal{C}^{\downarrow},\mathcal{C}_{\uparrow},\phi^{\downarrow},\phi_{\uparrow})$, we have:

\begin{prop}\label{phipi}
The chain map $\phi_{\uparrow}\co \mathcal{C}\to\mathcal{C}_{\uparrow}$ is chain homotopic to $\phi_{\uparrow}\circ\iota\circ\Pi$.  Similarly $\phi^{\downarrow}\co \mathcal{C}\to\mathcal{C}^{\downarrow}$ is chain homotopic to $\phi^{\downarrow}\circ\iota\circ\Pi$.
\end{prop}

\begin{proof}
The two cases are identical, so we just consider $\phi_{\uparrow}$. Recalling that $\mathcal{C}_{\uparrow}$ is, among other properties, a chain complex of vector spaces over the \emph{field} $\Lambda_{\uparrow}\supset \Lambda$, we may define a homomorphism $L\co C_{*}\to C_{\uparrow(*+1)}$ by setting $L$ equal to zero on each $F_k\oplus H_{k}^{\mathrm{free}}$ and by letting $L(x_{i}^{k})=\frac{1}{\alpha_i}\phi_{\uparrow}(y_{i}^{k+1})$ for $i=1,\ldots,r$, where as before $\{y_{1}^{k+1},\ldots,y_r^{k+1}\}$ is a basis for $F_{k+1}$ and $\{x_{1}^{k},\ldots,x_{r}^{k}\}$ is a basis for $\tilde{B}_{k}$, with $\partial y_{i}^{k+1}=\alpha_ix_{i}^{k}$.  It is then straightforward to verify that $\phi_{\uparrow}-\phi_{\uparrow}\circ\iota\circ\Pi=\partial_{\mathcal{C}_{\uparrow}}\circ L+L\circ\partial$.
\end{proof}

Now let us write $\mathcal{A}(\mathcal{C})$, $\mathcal{A}(\mathcal{C}_{\uparrow})$, and $\mathcal{A}(\mathcal{C}^{\downarrow})$ for the subcomplexes of $\mathcal{C},\mathcal{C}_{\uparrow}$, and $\mathcal{C}^{\downarrow}$ whose degree-$k$ parts are, respectively, $F_k\oplus \tilde{B}_k$, $F_{\uparrow k}\oplus B_{\uparrow k}$, and $F^{\downarrow}_{ k}\oplus B^{\downarrow}_{ k}$.  So $\mathcal{A}(\mathcal{C}_{\uparrow})$is a $\Lambda_{\uparrow}$-Floer-type complex with filtration function given by the restriction of $\ell_{\uparrow}$, and likewise restricting $\ell^{\downarrow}$ makes $\mathcal{A}(\mathcal{C}^{\downarrow})$ into a $\Lambda^{\downarrow}$-Floer-type complex; both $\mathcal{A}(\mathcal{C}_{\uparrow})$ and $\mathcal{A}(\mathcal{C}^{\downarrow})$ have trivial homology, while the homology of $\mathcal{A}(\mathcal{C})$ in degree $k$ is the torsion submodule $tH_k(\mathcal{C})$ of $H_k(\mathcal{C})$.  Let $\mathcal{H}^{\mathrm{free}}(\mathcal{C})$ be the graded $\Lambda$-module given in degree $k$ by $\frac{H_k(\mathcal{C})}{tH_{k}(\mathcal{C})}$, regarded as a chain complex with zero differential; the inclusions $\iota\co H_{k}^{\mathrm{free}}\to C_k$ induce on homology a map which, when composed with the quotient projection $H_k(\mathcal{C})\to \frac{H_k(\mathcal{C})}{tH_k(\mathcal{C})}$, is an isomorphism to $\mathcal{H}^{\mathrm{free}}(\mathcal{C})$ from the subcomplex (again with zero differential) of $\mathcal{C}$ given in degree $k$ by $H_{k}^{\mathrm{free}}(\mathcal{C})$.    Finally, let $\mathcal{H}(\mathcal{C}_{\uparrow})$ and $\mathcal{H}(\mathcal{C}^{\downarrow})$ denote the $\Lambda_{\uparrow}$- and $\Lambda^{\downarrow}$-Floer-type complexes, with zero differential, given in degree $k$ by the respective homologies $H_k(\mathcal{C}_{\uparrow})$ and $H_k(\mathcal{C}^{\downarrow})$ and with filtration functions given by the spectral invariants $\rho_{\uparrow k},\rho^{\downarrow}_{k}$.  The maps induced on homology by $\iota_{\uparrow}\co H_{\uparrow k}\to C_{\uparrow k}$ and $\iota^{\downarrow}\co H_{k}^{\downarrow}\to C^{\downarrow}_{k}$ provide isomorphisms to these Floer-type complexes from the sub-Floer-type complexes of $\mathcal{C}_{\uparrow},\mathcal{C}^{\downarrow}$ given in degree $k$ by $H_{\uparrow k}$ and $H^{\downarrow}_{k}$.

The direct sum of normed $\Lambda_{\uparrow}$- or $\Lambda^{\downarrow}$-spaces is defined in the obvious way that makes the summands orthogonal subspaces, leading to  corresponding notions of direct sum of Floer-type complexes an of chain-level filtered matched pairs implicit below.

\begin{prop}\label{splitcfmp}
With notation as above, we have a diagram, commutative up to homotopy,
\begin{equation}\label{splitsix} \xymatrix{ & \mathcal{A}(\mathcal{C}^{\downarrow})\oplus \mathcal{H}(\mathcal{C}^{\downarrow}) \ar[rr]^f & &  \mathcal{C}^{\downarrow} \\ \mathcal{A}(\mathcal{C})\oplus\mathcal{H}^{\mathrm{free}}(\mathcal{C})\ar[ur]^{\Phi^{\downarrow}}\ar[rr]^{g}\ar[dr]_{\Phi_{\uparrow}} & & \mathcal{C}\ar[ur]_<<<<<<<<{\phi^{\downarrow}}\ar[dr]^{\phi_{\uparrow}} &\\ & \mathcal{A}(\mathcal{C}_{\uparrow})\oplus \mathcal{H}(\mathcal{C}_{\uparrow}) \ar[rr]^h & &  \mathcal{C}_{\uparrow} } \end{equation}
where $f$ and $h$ are filtered chain isomorphisms, $g$ is a chain isomorphism, $\Phi^{\downarrow}$ restricts to $\mathcal{A}(\mathcal{C}^{\downarrow})$ as $0$ and to $\mathcal{H}^{\mathrm{free}}(\mathcal{C})$ as the map $\phi^{\downarrow}_{*}\co \mathcal{H}^{\mathrm{free}}(\mathcal{C})\to \mathcal{H}(\mathcal{C}^{\downarrow})$ induced on homology by $\phi^{\downarrow}$, and  $\Phi_{\uparrow}$ restricts to $\mathcal{A}(\mathcal{C}_{\uparrow})$ as $0$ and to $\mathcal{H}^{\mathrm{free}}(\mathcal{C})$ as the map $\phi_{\uparrow *}\co \mathcal{H}^{\mathrm{free}}(\mathcal{C})\to \mathcal{H}(\mathcal{C}_{\uparrow})$ induced on homology by $\phi_{\uparrow}$.\end{prop}

\begin{proof}
As noted above the proposition, in each degree $k$ we have isomorphisms $H_{\uparrow k}\cong H_{k}(\mathcal{C}_{\uparrow})$, $H_k^{\downarrow}\cong H_k(\mathcal{C}^{\downarrow})$, and $H_k^{\mathrm{free}}\cong \frac{H_k(\mathcal{C})}{tH_k(\mathcal{C})}$ induced on homology by $\iota_{\uparrow},\iota^{\downarrow},\iota$ respectively; the chain isomorphisms $f,g,h$ are just the compositions of (the inverses of) these isomorphisms with the maps given in each degree $k$ by obvious addition maps associated to the direct sum decompositions $C_k=(F_k\oplus \tilde{B}_k)\oplus H_{k}^{\mathrm{free}}$, $C_{\uparrow k}=(F_{\uparrow k}\oplus B_{\uparrow k})\oplus H_{\uparrow k}$, $C^{\downarrow}_{k}=(F^{\downarrow}_{k}\oplus B^{\downarrow}_{k})\oplus H_{k}^{\downarrow}$.  (The fact that $f$ and $h$ are \emph{filtered} isomorphisms follows from the corresponding direct sum decompositions being orthogonal.)

It remains to show that the diagram commutes up to homotopy.  Under the same identifications of $\frac{H_k(\mathcal{C})}{tH_k(\mathcal{C})}$ with $H_{k}^{\mathrm{free}}$ and $H_k(\mathcal{C}_{\uparrow})$ with $H_{\uparrow k}$ as in the previous paragraph, the map $\phi_{\uparrow *}\co \frac{H_k(\mathcal{C})}{tH_k(\mathcal{C})}\to H_k(\mathcal{C}_{\uparrow})$ becomes identified with the map $\Pi_{\uparrow}\circ (\phi_{\uparrow}|_{H_k^{\mathrm{free}}})\co H_{k}^{\mathrm{free}}\to H_{\uparrow k}$ between submodules of $C_k$ and $C_{\uparrow k}$.  The map $\Phi_{\uparrow}$ then becomes identified, in degree $k$, with the map $C_k\to C_{\uparrow k}$ that vanishes on $F_k\oplus \tilde{B}_k$ and sends $ H_{k}^{\mathrm{free}}$ to $ H_{\uparrow k}$ via $\Pi_{\uparrow}\circ (\phi_{\uparrow}|_{H_k^{\mathrm{free}}})$; this map $C_k\to C_{\uparrow k}$ can be written as $\iota_{\uparrow}\circ\Pi_{\uparrow}\circ\phi_{\uparrow}\circ\iota\circ\Pi$.    By Propositions \ref{iPi} and \ref{phipi}, this latter map is chain homotopic to $\phi_{\uparrow}$, and hence the bottom half of (\ref{splitsix}) commutes up to homotopy.  That the top half commutes up to homotopy follows by the same argument.
\end{proof}

Thus our original chain-level filtered matched pair $\mathcal{CP}$ is filtered matched homotopy equivalent to the chain-level filtered matched pair
 given by \begin{equation} \xymatrix{ & \mathcal{A}(\mathcal{C}^{\downarrow})\oplus \mathcal{H}(\mathcal{C}^{\downarrow}) \\ \mathcal{A}(\mathcal{C})\oplus\mathcal{H}^{\mathrm{free}}(\mathcal{C})\ar[ur]^{\Phi^{\downarrow}}\ar[dr]_{\Phi_{\uparrow}} & \\ & \mathcal{A}(\mathcal{C}_{\uparrow})\oplus \mathcal{H}(\mathcal{C}_{\uparrow})   } \label{splitversion}\end{equation}  We regard this latter chain-level filtered matched pair as a direct sum of three simpler ones: \[  \xymatrix{ & 0 \\ \mathcal{A}(\mathcal{C})\ar[ur]^{0}\ar[dr]_{0} & \\ & 0 }\qquad \xymatrix{ & \mathcal{A}(\mathcal{C}^{\downarrow}) \\ 0\ar[ur]^{0}\ar[dr]_{0} & \\ & \mathcal{A}(\mathcal{C}_{\uparrow}) }\qquad  \xymatrix{ & \mathcal{H}(\mathcal{C}^{\downarrow}) \\ \mathcal{H}^{\mathrm{free}}(\mathcal{C})\ar[ur]^{\phi^{\downarrow}_{*}}\ar[dr]_{\phi_{\uparrow *}} & \\ &  \mathcal{H}(\mathcal{C}_{\uparrow})   } \] Denoting these, respectively, by $\mathcal{T}(\mathcal{CP}),\mathcal{A}(\mathcal{CP}),\mathcal{H}^{\mathrm{free}}(\mathcal{CP})$ it follows readily from Proposition \ref{splitcfmp} and the definitions that, in the category of persistence modules, \begin{equation}\label{hkdec} \mathbb{H}_k(\mathcal{CP})\cong \mathbb{H}_k(\mathcal{T}(\mathcal{CP}))\oplus \mathbb{H}_k(\mathcal{A}(\mathcal{CP})) \oplus \mathbb{H}_k(\mathcal{H}^{\mathrm{free}}(\mathcal{CP})).\end{equation}  The three summands on the right-hand side can be described in terms of the building blocks in Section \ref{bbs}.

First of all, letting $tH_k(\mathcal{C})$ denote the torsion part of $H_k(\mathcal{C})$, for each $k$ we have a free resolution $0\to F_{k+1}\to \tilde{B}_k\to tH_k(\mathcal{C})\to 0$.  So since $\mathcal{A}(\mathcal{C})$ is the subcomplex $\oplus_j(F_j\oplus \tilde{B}_j)$ of $\mathcal{C}$,  we can identify $\mathcal{A}(\mathcal{C})$ with the direct sum over $k$ of the chain-level filtered matched pairs $\mathcal{PR}(tH_k(\mathcal{C}),k)$ from Section \ref{pr}.  Thus \[ \mathbb{H}_k(\mathcal{T}(CP))\cong (\kappa_{(-\infty,\infty)\times(-\infty,\infty)})^{\oplus (\dim_{\kappa}tH_k(\mathcal{C}))}.\]

As for $\mathcal{A}(\mathcal{CP})$, the subcomplex $\mathcal{A}(\mathcal{C}_{\uparrow})$ of $\mathcal{C}_{\uparrow}$ decomposes as a (filtered) direct sum of simple subcomplexes generated by pairs $\{y_{i}^{k+1},x_{i}^{k}\}$ coming from the singular value decomposition of $\partial_{\mathcal{C}_{\uparrow}}$, with $\partial_{\mathcal{C}_{\uparrow}}y_{i}^{k+1}=x_{i}^{k}$; these simple subcomplexes are isomorphic to the elementary complexes $\mathcal{E}_{\uparrow}(\ell_{\uparrow}(x_{i}^{k}),\ell_{\uparrow}(y_{i}^{k})-\ell_{\uparrow}(x_{i}^{k}),k)$ mentioned in Section \ref{peupsec}, and, except for those $i$ for which $\ell_{\uparrow}(y_{i}^{k+1})=\ell_{\uparrow}(x_{i}^{k})$, they also correspond to the finite-length bars $[\ell_{\uparrow}(x_{i}^{k}),\ell_{\uparrow}(y_{i}^{k+1}))^{\Gamma}$ in the concise barcode of $\mathcal{C}_{\uparrow}$.  Similarly, $\mathcal{A}(\mathcal{C}^{\downarrow})$ decomposes as a filtered direct sum of elementary complexes $\mathcal{E}^{\downarrow}(b,b-a,k)$ for each finite-length bar $(a,b]^{\Gamma}$ in the degree-$k$ concise barcode of $\mathcal{C}^{\downarrow}$, together with any summands of form $\mathcal{E}^{\downarrow}(b,b,k)$ arising from the distinction between the verbose and concise barcodes.  It follows that $\mathcal{A}(\mathcal{C})$ decomposes as a direct sum of the various $\mathcal{PE}_{\uparrow}(a,b-a,k)$ and $\mathcal{PE}^{\downarrow}(b,b-a,k)$ associated respectively to finite-length bars $[a,b)^{\Gamma}$ in the concise barcode of $\mathcal{C}_{\uparrow}$ and to finite-length bars $(a,b]^{\Gamma}$ in the concise barcode of $\mathcal{C}^{\downarrow}$, plus perhaps some $\mathcal{PE}_{\uparrow}(a,a,k)$ or $\mathcal{PE}^{\downarrow}(b,b,k)$.  So by the calculations in Section \ref{bbs}, each finite-length bar in the degree-$k$ concise barcode of $\mathcal{C}_{\uparrow}$ contributes a summand $\oplus_{g\in\Gamma}\kappa_{(-\infty,\infty)\times [a-g,b-g)}$ to $\mathbb{H}_k(\mathcal{A}(\mathcal{CP})$, and likewise each finite-length bar $(a,b]^{\Gamma}$ in the degree-$k$ concise barcode of $\mathcal{C}^{\downarrow}$ contributes a summand $\oplus_{g\in\Gamma}\kappa_{[-b-g,-a-g)\times(-\infty,\infty)}$ to $\mathbb{H}_k(\mathcal{A}(\mathcal{CP}))$.  (Any potential summands  $\mathcal{PE}_{\uparrow}(a,a,k)$ or $\mathcal{PE}^{\downarrow}(b,b,k)$ do not affect the outcome because $\mathbb{H}_k(\mathcal{PE}_{\uparrow}(a,a,k))\cong \mathbb{H}_k(\mathcal{PE}^{\downarrow}(b,b,k))\cong \{0\}$.)

Finally, since all differentials on the complexes involved in the chain-level filtered matched pair $\mathcal{H}^{\mathrm{free}}(\mathcal{CP})$ are zero, for each $k\in \Z$ the degree-$k$ part $\mathcal{H}_{k}^{\mathrm{free}}(\mathcal{CP})$ of $\mathcal{H}^{\mathrm{free}}(\mathcal{CP})$ is a filtered matched pair as in Section \ref{fmpsec} and so, due to our standing assumption that $\Gamma$ is discrete, $\mathcal{H}_{k}^{\mathrm{free}}(\mathcal{CP})$ admits a doubly orthogonal basis $\{[e_1],\ldots,[e_d]\}$ by Theorem \ref{basisconstruct}.  Here each $[e_i]\in\frac{H_k(\mathcal{C})}{tH_k(\mathcal{C})}$, and we may regard $[e_i]$ as obtained via the quotient projection from a doubly orthogonal basis for the filtered matched pair $\mathcal{H}_k(\mathcal{CP})$ discussed in Remark \ref{cfmpfmp}.  It is then easy to see that $\mathcal{H}_{k}^{\mathrm{free}}(\mathcal{CP})$ splits as a (filtered) direct sum of filtered matched pairs, for $i=1,\ldots,d$, \[ \xymatrix{ & \mathrm{span}_{\Lambda^{\downarrow}}\{\phi^{\downarrow}_{*}e_i\} \\ \mathrm{span}_{\Lambda} \{[e_i]\} \ar[rd]_{\phi_{\uparrow *}} \ar[ru]^{\phi^{\downarrow}_{*}} & \\ & \mathrm{span}_{\Lambda_{\uparrow}}\{\phi_{\uparrow *}e_i\}}\] and these (regarded as chain-level filtered matched pairs with zero differential) are respectively isomorphic to the chain-level filtered matched pairs $\mathcal{PM}(\rho_{\uparrow}(\phi_{\uparrow *}e_i),\rho^{\downarrow}(\phi^{\downarrow}_{*}e_i),k)$ of Section \ref{mabsect}.  Recall that the basis spectrum $\Sigma(\mathcal{H}_{k}^{\mathrm{free}}(\mathcal{CP}))$ is by definition the collection of pairs $([\rho_{\uparrow}(\phi_{\uparrow *}e_i)],\rho^{\downarrow}(\phi^{\downarrow}_{*}e_i)-\rho_{\uparrow}(\phi_{\uparrow *}e_i))\in (\R/\Gamma)\times \R$; this is the same as the basis spectrum $\Sigma(\mathcal{H}_k(\mathcal{CP}))$.  It then follows from the calculation in Section \ref{mabsect} that \begin{align*} \mathbb{H}_k(\mathcal{H}(\mathcal{CP}))\cong& \left(\bigoplus_{([a],L)\in \Sigma(\mathcal{H}_k(\mathcal{CP}))}\bigoplus_{g\in\Gamma}\kappa_{[-a-L+g,\infty)\times[a-g,\infty)} \right)\\ & \quad\oplus\left(\bigoplus_{([a],L)\in\Sigma(\mathcal{H}_{k+1}(\mathcal{CP}))}\bigoplus_{g\in\Gamma}\kappa_{(-\infty,-a-L+g)\times(-\infty,a-g)}\right).\end{align*}

We summarize this discussion as follows:

\begin{theorem}\label{bigdecomp}
Assume that $\Gamma$ is discrete and let $\mathcal{CP}=(\mathcal{C},\mathcal{C}^{\downarrow},\mathcal{C}_{\uparrow},\phi^{\downarrow},\phi_{\uparrow})$ be a chain-level filtered matched pair.  For each $k\in\Z$, the persistence module $\mathbb{H}_k(\mathcal{CP})$ splits as a direct sum of the following block modules:
\begin{itemize} \item[(i)] ($\dim_{\kappa}tH_k(\mathcal{C})$)-many copies of $\kappa_{(-\infty,\infty)\times(-\infty,\infty)}$;
\item[(ii)] for each finite-length bar in the degree-$k$ concise barcode of $\mathcal{C}_{\uparrow}$, represented as $[a,b)^{\Gamma}$ for some $a,b\in\R$, and for each $g\in \Gamma$, a copy of $\kappa_{(-\infty,\infty)\times [a-g,b-g)}$;
\item[(iii)] for each finite-length bar in the degree-$k$ concise barcode of $\mathcal{C}^{\downarrow}$, represented as $(a,b]^{\Gamma}$ for some $a,b\in \R$, and for each $g\in\Gamma$, a copy of $\kappa_{[-b-g,-a-g)\times(-\infty,\infty)}$;
\item[(iv)] for each element of $\Sigma(\mathcal{H}_k(\mathcal{CP}))$, represented as $([a],b-a)$ for some $a,b\in \R$, and for each $g\in \Gamma$, a copy of $\kappa_{[-b+g,\infty)\times[a-g,\infty)}$; and 
\item[(v)] for each element of $\Sigma(\mathcal{H}_{k+1}(\mathcal{CP}))$, represented as $([a],b-a)$ for some $a,b\in\R$, and for each $g\in \Gamma$, a copy of $\kappa_{(-\infty,-b+g)\times(-\infty,a-g)}$.  \end{itemize} Here all elements of concise barcodes or basis spectra are counted with multiplicity.
\end{theorem}

Consider in particular the $\kappa$-vector spaces $\mathbb{H}_k(\mathcal{CP})_{-s,s}$, which are meant to correspond to level-set homologies $H_k(f^{-1}(\{s\});\kappa)$. Theorem \ref{bigdecomp} shows that these decompose as a sum of the following contributions:
\begin{itemize} \item[(i)] the $\kappa$-vector space $tH_k(\mathcal{C})$, independently of $s$;
\item[(ii)] for each finite-length bar $[a,b)^{\Gamma}$ in the degree-$k$ concise barcode of $\mathcal{C}_{\uparrow}$, as many copies of $\kappa$ as the number of $g\in\Gamma$ for which $s+g\in [a,b)$; 
\item[(iii)] for each finite-length bar $(a,b]^{\Gamma}$ in the degree-$k$ concise barcode of $\mathcal{C}^{\downarrow}$, as many copies of $\kappa$ as the number of $g\in\Gamma$ for which $s+g\in (a,b]$; 
\item[(iv)] for each element $([a],b-a)\in \Sigma(\mathcal{H}_k(\mathcal{CP}))$ having $a\leq b$, as many copies of $\kappa$ as the number of $g\in\Gamma$ for which $s+g\in [a,b]$; and
\item[(v)] for each element $([a],b-a)\in \Sigma(\mathcal{H}_{k+1}(\mathcal{CP}))$ having $a>b$, as many copies of $\kappa$ as the number of $g\in\Gamma$ for which $s+g\in (b,a)$.\end{itemize}

Moreover, if $I$ is one of the intervals $[a,b),(a,b],[a,b],(a,b)$ listed above, and if $s<s'$ and $g\in \Gamma$ with $s+g,s'+g\in I$, then the corresponding summands of  $\mathbb{H}_k(\mathcal{CP})_{-s,s}$, $\mathbb{H}_k(\mathcal{CP})_{-s',s'}$ are related by the fact that they are sent under the persistence module structure maps to the same summand of  $\mathbb{H}_k(\mathbb{CP})_{-s,s'}$.  Thus the collection of these intervals (modulo $\Gamma$-translation) has the features that would be expected of an interlevel persistence barcode.  We accordingly make the following definition:

\begin{dfn}\label{fullbar}
Assume that $\Gamma$ is discrete, let $k\in\Z$ and let $\mathcal{CP}=(\mathcal{C},\mathcal{C}^{\downarrow},\mathcal{C}_{\uparrow},\phi^{\downarrow},\phi_{\uparrow})$ be a chain-level filtered matched pair.  The \textbf{full degree-$k$ barcode} of $\mathcal{CP}$ consists of the following intervals modulo $\Gamma$-translation:\begin{itemize}
\item[(i)] all finite-length intervals $[a,b)^{\Gamma}$ in the degree-$k$ concise barcode of $\mathcal{C}_{\uparrow}$;
\item[(ii)] all finite-length intervals $(a,b]^{\Gamma}$ in the degree-$k$ concise barcode of $\mathcal{C}^{\downarrow}$;
\item[(iii)] for each $([a],\ell)\in \Sigma(\mathcal{H}_{k}(\mathcal{CP}))$ with $\ell\geq 0$, an interval $[a,a+\ell]^{\Gamma}$; and
\item[(iv)] for each $([a],\ell)\in \Sigma(\mathcal{H}_{k+1}(\mathcal{CP}))$ with $\ell<0$, an interval $(a+\ell,a)^{\Gamma}$.
\end{itemize}
\end{dfn}

Note that a stability theorem for this barcode now follows trivially from stability theorems concerning the various ingredients in the barcode, namely Theorem \ref{stab} for the intervals in (iii) and (iv), and \cite[Theorem 8.17]{UZ} (and a conjugated version thereof) for the intervals in (i) and (ii).

\subsection{Chain-level Poincar\'e-Novikov structures}\label{cpnstr}

If $\mathcal{C}=(C_*,\partial)$ is a chain complex and $a\in\Z$, as before we let $\mathcal{C}[a]$ denote the chain complex with chain modules $C[a]_{k}=C_{k+a}$ and differential $(-1)^a\partial$.  Also we will write ${}^{\vee}\!\mathcal{C}$ for the chain complex whose $k$th chain module is $({}^{\vee}\!C)_{k}:={}^{\vee}\!(C_{-k})$ and whose differential acts on $({}^{\vee}\!C)_{k+1}$ by ${}^{\vee}\!((-1)^{k+1}\partial|_{C_{-k}})\co {}^{\vee}\!(C_{-k-1})\to {}^{\vee}\!(C_{-k})$.   (Thus ${}^{\vee}\!\mathcal{C}$ is obtained from the cochain complex $(C^*,\delta)$ as in Proposition \ref{cohtopdstr} by negating the grading so as to turn it into a chain complex.  Our sign conventions imply that the differential of $({}^{\vee}\!\mathcal{C})[-n]$ acts on $({}^{\vee}\!\mathcal{C})[-n]_{k+1}={}^{\vee}\!(C_{n-k-1})$ by $(-1)^{k+1}$ times the adjoint of $\partial\co C_{n-k}\to C_{n-k-1}$.)    

We define a \textbf{chain-level $n$-Poincar\'e-Novikov structure} $\mathcal{CN}=(\mathcal{C},\tilde{\mathcal{D}},\mathcal{C}_{\uparrow},\mathcal{S})$ to consist of a chain complex $\mathcal{C}$ of free, finitely-generated $\Lambda$-modules, a chain homotopy equivalence $\tilde{\mathcal{D}}\co \mathcal{C}\to ({}^{\vee}\!\mathcal{C})[-n]$, a $\Lambda_{\uparrow}$-Floer-type complex $\mathcal{C}_{\uparrow}$, and a chain map $\mathcal{S}\co \mathcal{C}\to\mathcal{C}_{\uparrow}$ which becomes a chain homotopy equivalence after tensoring with $1_{\Lambda_{\uparrow}}$.
So $\tilde{\mathcal{D}}$ maps $C_k$ to ${}^{\vee}(C_{n-k})$ and, in the notation of Proposition \ref{cohtopdstr}, induces isomorphisms $D_k\co H_k\to H^{n-k}$, so by that proposition one obtains maps $\mathrm{ev}\circ D_k\co H_k\to {}^{\vee}\!H_{n-k}$ which define a strong $n$-PD structure if $\Gamma$ is discrete, and a weak $n$-PD structure for any $\Gamma$.  Combining this PD structure with the map induced on homology by $\mathcal{S}$ yields a (respectively strong or weak) $n$-Poincar\'e-Novikov structure in the sense of Definition \ref{pnstr}.\vspace{.1 in}

\begin{center}
\begin{tikzpicture}
\node (cpn) [box] at (0,0) {chain-level Poincar\'e-Novikov structure};
\node (cfmp) [box] at (6,0) {chain-level filtered matched pair};
\node (pn) [box] at (0,-3) {Poincar\'e-Novikov structure};
\node (fmp) [box] at (6,-3) {graded filtered matched pair};
\draw [arrow] (cpn) -- node[anchor=east] {$H_*$} (pn);
\draw [arrow] (cfmp) -- node[anchor=west] {$\mathcal{H}_*$ (Remark \ref{cfmpfmp})} (fmp);
\draw [arrow] (pn) -- node[anchor=north] {(\ref{pnkdef})} (fmp);
\draw [dashed, arrow] (cpn) -- (cfmp);
\end{tikzpicture}
\end{center}

Let us complete the commutative diagram above, associating a chain-level filtered matched pair $\mathcal{CP}(\mathcal{CN})$ to a chain-level Poincar\'e-Novikov structure $\mathcal{CN}=(\mathcal{C},\tilde{\mathcal{D}},\mathcal{C}_{\uparrow},\mathcal{S})$  in a manner which, upon passing to  homology, induces in every degree $k$ the operation $\mathcal{N}\mapsto\mathcal{P}(\mathcal{N})_k$ from (\ref{pnkdef}). Thus we must specify data $(\mathcal{C},\mathcal{C}^{\downarrow},\mathcal{C}_{\uparrow},\phi^{\downarrow},\phi_{\uparrow})$ as in Definition \ref{cfmpdfn}.  For $\mathcal{C},\mathcal{C}_{\uparrow}$ we use the corresponding data from $\mathcal{CN}$, and for $\phi_{\uparrow}$ we use the map $\mathcal{S}$ from $\mathcal{CN}$.  For the $\Lambda^{\downarrow}$-Floer-type complex $\mathcal{C}^{\downarrow}$ we use the shifted dual $({}^{\vee}\!\mathcal{C}_{\uparrow})[-n]$, with $k$th chain module equal to ${}^{\vee}(C_{\uparrow (n-k)})$, equipped with the filtration function ${}^{\vee}\!(\ell_{\uparrow}|_{C_{\uparrow (n-k)}})$, and with differential acting on ${}^{\vee}\!(C_{\uparrow (n-k-1)})$  by $(-1)^{k+1}$ times the adjoint of the differential on $\mathcal{C}$,

It remains to specify a chain map $\phi^{\downarrow}\co \mathcal{C}\to ({}^{\vee}\!\mathcal{C}_{\uparrow})[-n]$ which becomes a chain homotopy equivalence after tensoring with $1_{\Lambda^{\downarrow}}$.  This is done just as in the construction of $\tilde{S}_{n-k}=\delta^{\downarrow}(S_{n-k})\circ \mathcal{D}$ in (\ref{pnkdef}).  Namely, letting $\mathcal{T}\co \mathcal{C}_{\uparrow}\to \Lambda_{\uparrow}\otimes_{\Lambda}\mathcal{C}$ denote a homotopy inverse to $1_{\Lambda_{\uparrow}}\otimes\mathcal{S}$, we take $\phi^{\downarrow}$ to be the composition \begin{equation}\label{phidowncpn} \xymatrix{ \mathcal{C}\ar[r]^-{\tilde{\mathcal{D}}} & ({}^{\vee}\!\mathcal{C})[-n]\ar[r] & \Lambda^{\downarrow}\otimes  ({}^{\vee}\!\mathcal{C})[-n]\ar[r]^-{\beta} & \left({}^{\vee}\!\Lambda_{\uparrow}\otimes_{\Lambda}\mathcal{C}\right)[-n]\ar[r]^-{{}^{\vee}\!\mathcal{T}} & ({}^{\vee}\!\mathcal{C}_{\uparrow})[-n]} \end{equation} where the second map is coefficient extension and $\beta$ is (in every grading) as in Proposition \ref{pullout}.  Since $\beta$ is an isomorphism and $\tilde{\mathcal{D}}$ and $\mathcal{T}$ are  homotopy equivalences, we see that $1_{\Lambda^{\downarrow}}\otimes\phi^{\downarrow}$ is likewise a homotopy equivalence.  This completes the definition of the chain-level filtered matched pair $\mathcal{CP}(\mathcal{CN})$ associated to a chain-level Poincar\'e-Novikov structure $\mathcal{CN}$.  Since the map on homology induced by $\mathcal{T}$ is inverse to that induced by $1_{\Lambda_{\uparrow}}\otimes \mathcal{S}$, it is easy to see that, given a chain-level Poincar\'e-Novikov structure $\mathcal{CN}$, one obtains the same filtered matched pair  by first constructing $\mathcal{CP}(\mathcal{CN})$ and then passing to homology as one does by first passing to homology to obtain a Poincar\'e-Novikov complex in the sense of Definition \ref{pnstr} and then taking the associated filtered matched pair as in (\ref{pnkdef}).  

From here one can, for any $k$, construct the two-parameter persistence module $\mathbb{H}_k(\mathcal{CP}(\mathcal{CN}))$ defined in (\ref{hkdef}).  This definition, in general, involves the choice of homotopy inverses $\psi_{\uparrow},\psi^{\downarrow}$ to $1_{\Lambda_{\uparrow}}\otimes\phi_{\uparrow}$ and $1_{\Lambda^{\downarrow}}\otimes \phi^{\downarrow}$ (with the resulting persistence module independent of this choice up to isomorphism, as noted above Definition \ref{fmhedef}).  Since $\phi_{\uparrow}=\mathcal{S}$, in the role of $\psi^{\downarrow}$ we may use the map $\mathcal{T}$ from the previous paragraph.  Since $\phi^{\downarrow}$ is given by (\ref{phidowncpn}) and since ${}^{\vee}\!\mathcal{T}$ has homotopy inverse ${}^{\vee}\!(1_{\Lambda_{\uparrow}}\otimes\mathcal{S})$, the role of $\psi^{\downarrow}$ is fulfilled by the composition \begin{equation}\label{psidowncpn}  \xymatrix{
({}^{\vee}\!\mathcal{C}_{\uparrow})[-n] \ar[r]^-{{}^{\vee}\!(1_{\Lambda_{\uparrow}}\otimes\mathcal{S})} & \left({}^{\vee}\!\Lambda_{\uparrow}\otimes_{\Lambda}\mathcal{C}\right)[-n] \ar[r]^-{\beta^{-1}} & \Lambda^{\downarrow}\otimes  ({}^{\vee}\!\mathcal{C})[-n]\ar[r]^-{1_{\Lambda^{\downarrow}}\otimes\tilde{\mathcal{D}}'} & \Lambda^{\downarrow}\otimes_{\Lambda}\mathcal{C}}   
\end{equation} for any choice of homotopy inverse $\tilde{\mathcal{D}}'$ to the chain-level Poincar\'e duality map $\mathcal{D}$.

Assume now that $\Gamma$ is discrete.  We define the \textbf{full barcode} $\cup_k\mathcal{B}_k(\mathcal{CN})$ of a chain-level $n$-Poincar\'e--Novikov structure $\mathcal{CN}$  to be the full barcode, as given by Definition \ref{fullbar}, of the corresponding chain-level filtered matched pair $\mathcal{CP}(\mathcal{CN})$.  The bars appearing in items (iii) and (iv) of Definition \ref{fullbar} evidently comprise the essential barcode (as defined in Definition \ref{pnbar}) of the $n$-Poincar\'e--Novikov structure $\mathcal{N}$ that is obtained from $\mathcal{CN}$ by passing to homology.  These satisfy the stability property from Corollary \ref{pnstab}, and the duality result of Corollary \ref{pndual}.  (For the latter we use that$\mathcal{N}$ is a \emph{strong} $n$-Poincar\'e--Novikov structure by Proposition \ref{cohtopdstr} and the fact that $\Gamma$ is discrete.) The bars in (i) and (ii) of Definition \ref{fullbar} (in the context of $\mathcal{CP}(\mathcal{CN})$ for our chain-level $n$-Poincar\'e--Novikov structure $\mathcal{CN}=(\mathcal{C},\tilde{\mathcal{D}},\mathcal{C}_{\uparrow},\mathcal{S})$) comprise the finite-length bars $[a,b)^{\Gamma}$ of the $\Lambda_{\uparrow}$-Floer-type complex $\mathcal{C}$, and the finite-length bars $(a,b]^{\Gamma}$ of the $\Lambda^{\downarrow}$-Floer-type complex $({}^{\vee}\!\mathcal{C})[-n]$.  These also satisfy a stability property based on \cite[Theorem 8.17]{UZ}.  After adjusting for conjugation and the degree shift by $n$, it follows from \cite[Proposition 6.7]{UZ} that, for any $k$, bars $[a,b)^{\Gamma}\in \mathcal{B}_k(\mathcal{CN})$ are in one-to-one correspondence with bars $(a,b]^{\Gamma}\in \mathcal{B}_{n-k-1}(\mathcal{CN})$.  

Combining this with Corollary \ref{pndual}, we see that \emph{all} types of bars in $\cup_k\mathcal{B}_k(\mathcal{CN})$ except for those of the form $[a,a]^{\Gamma}$ are subject to a symmetry which matches bars in degree $k$ with bars in degree $n-k-1$ with identical interiors but opposite endpoint conditions (closed intervals $[a,b]^{\Gamma}$ are matched to open intervals $(a,b)^{\Gamma}$, and intervals $[a,b)^{\Gamma}$ are matched to intervals $(a,b]^{\Gamma}$).  This symmetry is consistent with Poincar\'e duality in the regular level sets of a Morse function, and it is also analogous to the symmetry theorem in \cite{ext}.

\section{The chain-level Poincar\'e--Novikov structures of Morse and Novikov theory}\label{concretemorse}
In this section we review basic features of Morse and Novikov complexes in sufficient detail as to allow us to construct and study their associated chain-level Poincar\'e--Novikov structures. This culminates in the proof of Theorem \ref{introiso} which relates, in the case of an $\R$- or $S^1$-valued Morse function, the two-parameter persistence module from (\ref{hkdef}) to interlevel persistence.

\subsection{The Morse complex, continuation maps, and Poincar\'e--Novikov structures with $\Gamma=\{0\}$}\label{morseintro}
Let $X$ be a smooth, $n$-dimensional manifold, and let $f\co X\to \R$ be a Morse function on $X$.  Recall that a vector field $v$ on $X$ is said to be \emph{gradient-like} for $f$ provided that both $df_p(v)>0$ at all $p$ outside of the set $\mathrm{Crit}(f)$ of critical points of $f$, and, around each $p\in \mathrm{Crit}(f)$, there is a coordinate chart in which $v$ is represented in the standard form $\left(-\sum_{i=1}^{k}x_i\partial_{x_i}+\sum_{i=k+1}^{n}x_i\partial_{x_i}\right)$.  It is not hard to construct such vector fields: take $v$ equal to the gradient $\nabla f$ with respect to a Riemannian metric which is standard in Morse charts around the points of $\mathrm{Crit}(f)$; conversely, any gradient-like vector field for $f$ is the gradient of $f$ with respect to some Riemannian metric (\cite[Proposition II.2.11]{Pa}\footnote{Here and below we use Roman numerals to denote chapter numbers in \cite{Pa}, so this reference is to Proposition 2.11 in Chapter 2; the theorem numbering in \cite{Pa} does not include chapters}).  

Assuming that our manifold $X$ is compact and without boundary and that $v$ is gradient-like for $f$, write $\{\phi_{-v}^{t}\}_{t\in \R}$ for the flow of the vector field $-v$.   One then has, for each critical point $p$ of $f$, descending and ascending manifolds \[ \mathbf{D}^v(p)=\{q\in X|\lim_{t\to-\infty}\phi_{-v}^{t}(q)=p\},\qquad \mathbf{A}^v(p)=\{q\in X|\lim_{t\to+\infty}\phi_{-v}^{t}(q)=p\}. \]  These are smoothly embedded disks in $X$ of dimensions, respectively $\mathrm{ind}_f(p)$ and $n-\mathrm{ind}_f(p)$ where $\mathrm{ind}_f(p)$ is the Morse index of $p$ (equal to the integer $k$ appearing in the local representation for $v$ in the previous paragraph).  Obviously the vector field $-v$ is gradient-like for the Morse function $-f$, and we have $\mathbf{D}^{-v}(p)=\mathbf{A}^{v}(p)$.  For choices of $v$ that are generic amongst gradient-like vector fields for $f$,  the flow of $-v$ will be Morse-Smale in the sense that each intersection $\mathbf{W}^v(p,q):=\mathbf{D}^v(p)\cap \mathbf{A}^v(q)$ is transverse, and hence smooth of dimension $\mathrm{ind}_f(p)-\mathrm{ind}_f(q)$ (see, \emph{e.g.}, \cite[Theorem IV.2.13]{Pa}).  

Since the various $\mathbf{D}^v(p),\mathbf{A}^v(q)$ are invariant under the flow of $-v$, this flow induces an $\R$-action on $\mathbf{W}^v(p,q)$, which is free provided that $p\neq q$. So, assuming the Morse-Smale condition, the space $\mathbf{M}^v(p,q)=\frac{\mathbf{W}^v(p,q)}{\R}$ is a smooth manifold of dimension $\mathrm{ind}_f(p)-\mathrm{ind}_f(q)-1$.  Under the correspondence which sends a trajectory $\gamma\co \mathbb{R}\to X$ of the negative gradient-like vector field $-v$ to its initial condition $\gamma(0)$, we may regard $\mathbf{W}^v(p,q)$ as the space of trajectories of $-v$ that are asymptotic as $t\to-\infty$ to $p$ and as $t\to+\infty$ to $q$, and $\mathbf{M}^v(p,q)$  as the space of such trajectories modulo reparametrization. 

Continuing to assume the Morse-Smale condition, if $\mathrm{ind}_f(p)-\mathrm{ind}_f(q)=1$ then $\mathbf{M}_f(p,q)$ is a zero-dimensional manifold, which moreover is compact (\cite[Lemma VI.1.1]{Pa}), and so consists of finitely many points.  We recall in Section \ref{appmorse} how to orient the various $\mathbf{M}^v(p,q)$ (based on arbitrary orientations $\mathfrak{o}_{v,p}$ of the disks $\mathbf{D}^v(p)$); an orientation of a zero-dimensional manifold amounts to a choice of sign for each of its points, and so we obtain from these orientations counts $n_{f,v}(p,q)\in\Z$ of the elements of $\mathbf{M}^v(p,q)$, each point contributing $+1$ or $-1$ according to its orientation sign. 

For $k\in \N$ let $\mathrm{Crit}_k(f)$ denote the set of index-$k$ critical points of $f$. The integer-coefficient Morse chain complex $\mathbf{CM}_*(f;\Z)$ (or $\mathbf{CM}_{*}(f,v;\Z)$ if we wish to emphasize the dependence on the gradient-like vector field $v$) has its degree $k$ part equal to the free abelian group with one generator for each element of $\mathrm{Crit}_{k}(f)$, with the Morse boundary operator $\partial_{k+1}^{f}\co \mathbf{CM}_{k+1}(f;\Z)\to \mathbf{CM}_{k}(f;Z)$ defined to be the homomorphism given by, for each $p\in \mathrm{Crit}_{k+1}(f)$, \[ \partial_{k+1}^{f}p=\sum_{q\in\mathrm{Crit}_k(f)}n_{f,v}(p,q)q.\]
See, \emph{e.g.}, \cite[Section 4.1]{Sc} or \cite[Chapter 3]{AD} for self-contained proofs (based on spaces of broken gradient trajectories) that one has $\partial_k^f\circ\partial_{k+1}^f=0$, so that $\mathbf{CM}_{*}(f;\Z)$ is indeed a chain complex; its homology will be denoted by $\mathbf{HM}_{*}(f;\Z)$, and is isomorphic to the singular homology of $X$.  Indeed, in \cite[Section VI.3.3]{Pa}, Pajitnov constructs a chain homotopy equivalence $\mathcal{E}(f,v)$ from  $\mathbf{CM}_{*}(f;\Z)$ to the singular chain complex $\mathbf{S}_*(X)$.

A related classical (\emph{e.g.} \cite[Theorem 3.5]{Mil},\cite[Theorem 1]{Kal}) perspective on the Morse complex identifies it as the cellular chain complex of a CW complex.  \cite[Theorems 3.8 and 3.9]{Qin} show that our manifold $X$ admits a CW decomposition with open $k$-cells equal to the $\mathbf{D}_f(p)$ for $p\in \mathrm{Crit}_k(f)$, and moreover that, using the orientations $\mathfrak{o}_{v,p}$ of the $\mathbf{D}^v(p)$ to identify the $k$th cellular chain group of this CW complex with the free abelian group on $\mathrm{Crit}_k(f)$ (\emph{i.e.}, with $\mathbf{CM}_k(f;\Z)$), the cellular differential for the  CW complex agrees with the Morse differential defined above.   See also \cite[Section 4.9]{AD} for a different proof and \cite{Lau},
\cite{BH01} for earlier related results.

For any ring $R$ we may form the chain complex of $R$-modules $\mathbf{CM}_*(f,v;R):=\mathbf{CM}_{*}(f,v;\Z)\otimes_{\Z}R$.  For a field $\kappa$, $\mathbf{CM}_{*}(f,v;\kappa)$, can naturally be endowed with the structure of (using the language of Section \ref{clsec}, with the group $\Gamma$ equal to $\{0\}$) a $\Lambda_{\uparrow}$-Floer-type complex, by defining the filtration function $\ell_{\uparrow}^{f}\co C_k\to\R\cup\{-\infty\}$ by \[ \ell_{\uparrow}^{f}\left(\sum_{p\in \mathrm{Crit}_k(f)}n_pp\right)=\max\{f(p)|n_p\neq 0\},\] the maximum of the empty set being $-\infty$.  Indeed $\mathbf{CM}_k(f,v;\kappa)$ has basis given by $\mathrm{Crit}_k(f)$ which, as is clear from the above formula, is $\ell_{\uparrow}^{f}$-orthogonal, and the inequality $\ell_{\uparrow}^{f}(\partial_kx)\leq \ell_{\uparrow}^{f}(x)$ follows from the fact that $f$ decreases along the flow of the negative gradient-like field $-v$.  We write $\mathbf{CM}(f,v;\kappa)_{\uparrow}$ (or just $\mathbf{CM}(f)_{\uparrow}$ if $v$ and $\kappa$ are understood) for $\Lambda_{\uparrow}$-Floer-type complex consisting of the chain complex $\mathbf{CM}_*(f,v;\kappa)$ together with the filtration function $\ell_{\uparrow}^{f}$.

\begin{remark}
If one is only interested in the case that, as here, $\Gamma=\{0\}$, then our algebraic language is unnecessarily baroque: the fields $\Lambda_{\uparrow}$ and $\Lambda^{\downarrow}$, the ring $\Lambda$, and the module $\Lambda_{\updownarrow}$ are all equal to the field $\kappa$.  Even so, in this context  a ``$\Lambda_{\uparrow}$-Floer-type complex'' is still a slightly different object than a ``$\Lambda^{\downarrow}$-Floer-type complex'' as the filtration function on the former satisfies $\ell_{\uparrow}\circ\partial\leq \ell_{\uparrow}$ while that on the latter satisfies $\ell^{\downarrow}\circ\partial\geq \ell^{\downarrow}$.
\end{remark}

Now suppose that $f_-,f_+$ are two Morse functions on $X$, with respective gradient-like vector fields $v_-,v_+$ each having Morse-Smale flows.  As in, \emph{e.g.}, \cite[Section 4.1.3]{Sc}, one defines a \textbf{continuation map} $\Phi_{f_-f_+}\co \mathbf{CM}_*(f_-,v_-;\Z)\to \mathbf{CM}_{*}(f_+,v_+;\Z)$ as follows.  Choose a smooth, $\R$-parametrized family of vector fields $\mathbb{V}=\{v_s\}_{s\in \R}$ having the property that $v_s=v_-$ for $s\leq -1$ and $v_s=v_+$ for $s\geq 1$, and for $p_-\in \mathrm{Crit}(f_-)$ and $p_+\in \mathrm{Crit}(f_+)$ let $\mathbf{N}^{\mathbb{V}}(p_-,p_+)$ be the set of solutions $\gamma\co \R\to X$ to non-autonomous equation \begin{equation}\label{conteq} \gamma'(s)+v_s(\gamma(s))=0 \end{equation} such that $\gamma(s)\to p_-$ as $s\to-\infty$ and $\gamma(s)\to p_+$ as $s\to\infty$. If we denote by $\Psi^{\mathbb{V}}\co X\to X$ the diffeomorphism which sends $x\in X$ to the value $\eta_x(1)$ where $\eta_x\co [-1,1]\to X$ is the unique solution to $\eta'_{x}(s)=v_s(\eta_x(s))$ subject to the initial condition $\eta_x(-1)=x$, then any $\gamma\in \mathbf{N}^{\mathbb{V}}(p_-,p_+)$ has $\gamma(-1)\in \mathbf{D}^{v_-}(p_-)$, $\gamma(1)\in\mathbf{A}^{v_+}(p_+)$, and $\Psi^{\mathbb{V}}(\gamma(-1))=\gamma(1)$.  In this way we see that $\mathbf{N}^{\mathbb{V}}(p_-,p_+)$ is in bijection with the intersection $\mathbf{D}^{v_-}(p_-)\cap (\Psi^{\mathbb{V}})^{-1}(\mathbf{A}^{v_+}(p_+))$.  For generic choices of the interpolating family $\mathbb{V}$, this intersection will be transverse for all choices of $p_-$ and $p_+$ (\cite[Proposition VI.4.7]{Pa}), and hence a manifold of dimension $\mathrm{ind}_{f_-}(p_-)-\mathrm{ind}_{f_+}(p_+)$. When this dimension is zero,  $\mathbf{N}^{\mathbb{V}}(p_-,p_+)$ is a finite set (\cite[Lemma VI.4.5]{Pa}), and so   for generators $p_-$ of $\mathbf{CM}_k(f_-,v_-;\Z)$ we put $\Phi_{f_-f_+}p_-=\sum_{p_+\in \mathrm{Crit}_k(f_+)}\#\mathbf{N}^{\mathbb{V}}(p_-,p_+)p_+$ where $\#\mathbf{N}^{\mathbb{V}}(p_-,p_+)$ is the signed count of points in the compact zero-manifold $\mathbf{N}^{\mathbb{V}}(p_-,p_+)$, oriented as in Section \ref{appmorse}.


The continuation map $\Phi_{f_-f_+}$ is a chain homotopy equivalence between the chain complexes  $\mathbf{CM}_*(f_-,v_-;\Z)$ and $\mathbf{CM}_{*}(f_+,v_+;\Z)$, whose chain homotopy class is independent of the choice of interpolating family $\mathbb{V}$. Indeed, $\Phi_{f_-f_+}$ is evidently equal to the map which would be written in the notation of \cite[p. 221]{Pa} as $(\Psi^{\mathbb{V}})_{\flat}$, and so by \cite[Proposition VI.4.3 and Theorem VI.4.6]{Pa} and the fact that $\Psi^{\mathbb{V}}\co X\to X$ is homotopic to the identity, we have a homotopy-commutative diagram \begin{equation}\label{contcomm} \xymatrix{ \mathbf{CM}_*(f_-,v_-;\Z)\ar[rr]^{\Phi_{f_-f_+}} \ar[dr]_{\mathcal{E}(f_-,v_-)} & & \mathbf{CM}_*(f_-,v_-;\Z) \ar[ld]^{\mathcal{E}(f_+,v_+)} \\ & \mathbf{S}_{*}(X)  & } \end{equation} where $\mathcal{E}(f_{\pm},v_{\pm})$ are the chain homotopy equivalences from \cite[Section VI.3.3]{Pa}.
 Of course, we can then tensor with any ring and obtain the corresponding statement for the Morse and singular complexes with coefficients in that ring.  

The behavior of $\Phi_{f_-f_+}$ with respect to the filtration functions $\ell_{\uparrow}^{f_{\pm}}$ is also of interest.  For this purpose we take a family of functions $f_s$ to be of the form $f_s=f_-+\beta(s)(f_+-f_-)$ where $\beta$ is a smooth function such that $\beta'(s)\geq 0$ for all $s$, $\beta(s)=0$ for $s\leq -1$, and $\beta(s)=1$ for $s\geq 1$, and for the family of vector fields $\mathbb{V}=\{v_s\}_{s\in \R}$ we require that $(df_s)_x(v_s)\geq 0$ for all $s\in\R,p\in X$.  (For instance $v_s$ could be the gradient of $f_s$ with respect to a Riemannian metric that varies smoothly with $s$.)  Then if $\gamma$ is a solution to (\ref{conteq}) with $\gamma(s)\to p_{\pm}$ as $s\to\pm\infty$, we will have 
\begin{align} f_+(p_+)-f_-(p_-)&=\int_{-\infty}^{\infty}\frac{d}{ds}(f_s(\gamma(s)))ds=\int_{-\infty}^{\infty}\left((df_s)_{\gamma(s)}(\gamma'(s))+\frac{\partial f_s}{\partial s}(\gamma(s))\right)ds   \nonumber \\ &= -\int_{-\infty}^{\infty}(df_s)_{\gamma(s)}(v_s)ds+\int_{-\infty}^{\infty}\beta'(s)\left(f_+(\gamma(s))-f_-(\gamma(s))\right)ds\nonumber \\& \leq \max_X(f_+-f_-).\label{filtchange}\end{align}

With coefficients in the field $\kappa$ (so that we have $\Lambda_{\uparrow}$-Floer-type complexes $\mathbf{CM}(f_{\pm},v_{\pm};\kappa)_{\uparrow}$ with filtration functions $\ell_{\uparrow}^{f_{\pm}}$) it follows that $\ell_{\uparrow}^{f_+}(\Phi_{f_-f_+}x)\leq \ell_{\uparrow}^{f_-}(x)+\max_X(f_+-f_-)$ for all $x\in \mathbf{CM}_{*}(f_-,v_-;\kappa)$ provided that the family of vector fields $\mathbb{V}$ used to define $\Phi_{f_-f_+}$ is chosen as above.   In the special case that $f_-$ and $f_+$ are both equal to the same Morse function $f$ (but the gradient-like vector fields $v_{\pm}$ may differ) it follows that $\Phi_{ff}$ is a morphism of filtered complexes.  In fact, in this situation $\Phi_{ff}$ is a filtered homotopy equivalence, as can be seen by inspecting the construction of the appropriate chain homotopies from \cite[Proposition 4.6]{Sc}. In what follows we will again often suppress the relevant gradient-like vector fields and the field $\kappa$ from the notation for the Morse complex.

We now have almost all of the ingredients needed define the chain-level Poincar\'e--Novikov structure $\mathcal{CP}(f)$ for a Morse function $f$ on a compact manifold $X$ (with respect to the field $\kappa$, and with the group $\Gamma$ equal to $\{0\}$). If the characteristic of $\kappa$ is not $2$, then we also require that $X$ be oriented. To describe $\mathcal{CP}(f)$ we must specify data $(\mathcal{C},\tilde{\mathcal{D}},\mathcal{C}_{\uparrow},\mathcal{S})$ as at the start of Section \ref{cpnstr}. The role of the $\Lambda_{\uparrow}$-Floer-type complex $\mathcal{C}_{\uparrow}$ is played by the filtered Morse complex $\mathbf{CM}(f)_{\uparrow}$ of $f$ with coefficients in the field $\kappa$.  For the chain complex $\mathcal{C}$ we use the (unfiltered) Morse complex $\mathbf{CM}_{*}(h_0)$ of another Morse function $h_0$; this function $h_0$ should be regarded as fixed for normalization purposes at the same time that we fix the manifold $X$ (so we would by default use the same $h_0$ to construct $\mathcal{CP}(f)$ for different functions $f$ on the same manifold, though this is not entirely necessary as a different choice of $h_0$ would not change the barcode).  Since in the present context $\Gamma=\{0\}$ (so $\Lambda_{\uparrow}=\Lambda=\kappa$), the map $\mathcal{S}$ should then just be a chain homotopy equivalence from $\mathbf{CM}_{*}(h_0)$ to $\mathbf{CM}_{*}(f)$, and for this purpose we may use the continuation map $\Phi_{h_0f}$.

It remains to describe the ``chain-level Poincar\'e duality map'' $\tilde{\mathcal{D}}\co \mathbf{CM}_{*}(h_0)\to ({}^{\vee}\!(\mathbf{CM}_*(h_0))[-n]$ where $\dim X=n$.  For this purpose we establish the following proposition, whose proof in the case that $\mathrm{char}(\kappa)\neq 2$ depends in part on the sign analysis that we defer to Appendix \ref{orsect}.  Recall that if $h\co X\to \R$ is a Morse function we write $\mathrm{Crit}_k(f)$ for the set of index-$k$ critical points of $h$, and note that $\mathrm{Crit}_{k}(-h)=\mathrm{Crit}_{n-k}(h)$.  For the statement and proof of the proposition we restore the vector field $v$ and field $\kappa$ to the notation for $\mathbf{CM}_{*}(\pm h_0)$ in order to make explicit the various dependencies.

\begin{prop}\label{negdual}  Given an element $c=\sum_{p\in\mathrm{Crit}_{k}(-h_0)}a_pp$ define $\epsilon_c\co \mathbf{CM}_{n-k}(h_0,v;\kappa)\to \kappa$ by \[ \epsilon_c\left(\sum_{q\in\mathrm{Crit}_{n-k}(h_0)}b_qq\right)=\sum_{p\in \mathrm{Crit}_k(h_0)}a_pb_p.\]  Then, assuming that either $\mathrm{char}(\kappa)=2$ or that $X$ is oriented and that we choose orientations for the descending manifolds $\mathbf{D}^v(p),\mathbf{D}^{-v}(q)$ as in Section \ref{appmorse}, the assignment $c\mapsto \epsilon_c$ defines an isomorphism of chain complexes $\epsilon\co \mathbf{CM}_{*}(-h_0,-v;\kappa)\to ({}^{\vee}\!\mathbf{CM}_*(h_0,v;\kappa))[-n]$.
\end{prop}

\begin{proof}
Note that the degree $k$-part of $({}^{\vee}\!\mathbf{CM}_*(h_0,v;\kappa))[-n]$ is, by definition, the dual of $\mathbf{CM}_{n-k}(h_0,v;\kappa)$.  (Since we are in the case that $\Gamma=\{0\}$ there is no distinction between the ``conjugated dual'' ${}^{\vee}$ and the usual vector space dual.)  Now, as $\kappa$-vector spaces, $\mathbf{CM}_{k}(-h_0,-v;\kappa)$ and $\mathbf{CM}_{n-k}(h_0,v;\kappa)$ are equal to each other, both having $\mathrm{Crit}_k(-h_0)$ as a basis.  In degree $k$, the map $\epsilon$ is the linear map that sends each element of the basis $\mathrm{Crit}_k(-h_0)$ for $\mathbf{CM}_{k}(-h_0,-v;\kappa)$ to the corresponding dual basis element for the dual of  $\mathbf{CM}_{n-k}(h_0,v;\kappa)$. Thus $\epsilon$ is an isomorphism of graded vector spaces; it remains to check that it is a chain map.

This latter statement amounts to the claim that, for any $q\in \mathrm{Crit}_{k+1}(-h_0)=\mathrm{Crit}_{n-k-1}(h_0)$ and $p\in \mathrm{Crit}_{k}(-h_0)=\mathrm{Crit}_{n-k}(h_0)$, we have $\left(\epsilon(\partial^{-h_0}q)\right)(p)= (-1)^{k+1}(({}^{\vee}\!\partial^{h_0})(\epsilon q))(p)$.  (The $(-1)^{k+1}$ on the right hand side results from the conventions mentioned at the start of Section \ref{cpnstr}: a factor $(-1)^{k+1-n}$ arises from the sign convention for the differentials on dual complexes from Proposition \ref{cohtopdstr}, and a factor $(-1)^n$ arises from the convention that shifting the grading of a complex by $n$ is accompanied by multiplying the differential by $(-1)^n$.)  But $\left(\epsilon(\partial^{-h_0}q)\right)(p)$ is the signed count $n_{-h_0,-v}(q,p)$ of elements of the zero-dimensional  space $\mathbf{M}^{-v}(q,p)$, while $(({}^{\vee}\!\partial^{h_0})(\epsilon q))(p)=(\epsilon q)(\partial^{h_0}p)$ is the corresponding signed count $n_{h_0,v}(p,q)$ of elements of $\mathbf{M}^{v}(p,q)$.  If $\mathrm{char}(\kappa)=2$ then there is no distinction between positive and negative contributions to the counts $n_{h_0,v}(p,q),n_{-h_0,-v}(q,p)$, and also $(-1)^{k+1}=1$ in $\kappa$, so the desired equality follows from the obvious bijection between $\mathbf{M}_{-h_0}(q,p)$ and $\mathbf{M}_{h_0}(p,q)$.  If instead we assume that $X$ is oriented then the equality $n_{-h_0,-v}(q,p)=(-1)^{k+1}n_{h_0,v}(p,q)$ is proven in Corollary \ref{mor}.
\end{proof}

We are thus justified in making the following definition, which will be generalized in Section \ref{novsect}:

\begin{dfn}\label{cmdef} Let $X$ be a smooth compact manifold  and if our ground field $\kappa$ has $\mathrm{char}(\kappa)\neq 2$ assume that $X$ is oriented.  Fix a Morse function $h_0\co X\to\R$.  Given any other Morse function $f\co X\to \R$ we take $\mathcal{CM}(f)$ to be the chain-level Poincar\'e-Novikov complex with the following data  $(\mathcal{C},\tilde{\mathcal{D}},\mathcal{C}_{\uparrow},\mathcal{S})$ as in Section \ref{cpnstr}:
\begin{itemize}
\item $\mathcal{C}=\mathbf{CM}_{*}(h_0)$ is the Morse complex of $h_0$;
\item $\tilde{\mathcal{D}}\co \mathbf{CM}_{*}(h_0)\to ({}^{\vee}\!\mathbf{CM}_{*}(h_0))[-n]$ is the composition of the continuation map $\Phi_{h_0,-h_0}\co \mathbf{CM}_*(h_0)\to \mathbf{CM}_{*}(-h_0)$ with the chain isomorphism $\epsilon\co \mathbf{CM}_{*}(-h_0)\to ({}^{\vee}\!\mathbf{CM}_*(h_0))[-n]$ from Proposition \ref{negdual}.
\item $\mathcal{C}_{\uparrow}=\mathbf{CM}(f)_{\uparrow}$ (\emph{i.e.}, the $\Lambda_{\uparrow}$-Floer-type complex with chain complex $\mathbf{CM}_{*}(f)$ and the filtration function  $\ell_{\uparrow}^{f}$).
\item $\mathcal{S}\co \mathbf{CM}_{*}(h_0)\to \mathbf{CM}_{*}(f)$ is the continuation map $\Phi_{h_0f}$.
\end{itemize}
\end{dfn}

\subsection{Lifted Morse complexes on covering spaces}

To generalize beyond Section \ref{morseintro}, we consider Morse theory for functions $\tilde{f}\co \tilde{X}\to\R$ defined on the domains of regular covering spaces $\pi\co \tilde{X}\to X$, where $X$ is a compact smooth $n$-dimensional manifold.  Let $\tilde{\Gamma}$ be the deck transformation group of $\pi$; we will generally be concerned with the case that $\tilde{\Gamma}$ is a free abelian group, isomorphic to the subgroup $\Gamma<\R$ considered throughout the paper, though for the moment we do not assume this. 

Unless $\tilde{\Gamma}$ is a finite group, the space $\tilde{X}$ will be noncompact, and so the algebraic structure (if any) that one obtains from Morse theory for such a function $\tilde{f}$ depends on how that function behaves at infinity.  The simplest case is that $\tilde{f}=\pi^{*}h$ for some Morse function $h\co X\to \R$.  In this case the index-$k$ critical points $\tilde{p}$ of $\tilde{f}$ are precisely the elements of $\tilde{X}$ which map by $\pi$ to index-$k$ critical points of $h$.  Let us use a vector field $\tilde{v}$ on $\tilde{X}$ that is the lift to $\tilde{X}$ of a gradient-like vector field $v$ for $h$ on $X$ whose flow is Morse-Smale; then the flow of $-\tilde{v}$ on $\tilde{X}$ will just be the lift of that of $-v$ on $X$ and in particular will also be Morse-Smale.  For $\tilde{p}\in \mathrm{Crit}_{k}(\pi^*h)$, the descending and ascending manifolds $\mathbf{D}^{\tilde{v}}(\tilde{p})$ and $\mathbf{A}^{\tilde{v}}(\tilde{p})$ will be the unique lifts passing through $\tilde{p}$ of, respectively, $\mathbf{D}^v(p)$ and $\mathbf{A}^v(p)$.

Let  $\tilde{p},\tilde{q}$ be two critical points of $\pi^*h$, with $\pi(\tilde{p})=p,\pi(\tilde{q})=q$.   
Identify as usual the one-dimensional manifold $\mathbf{W}^v(p,q)=\mathbf{D}^v(p)\cap \mathbf{A}^v(q)$ with the space of trajectories $\eta\co \R\to X$ for $-v$ having $\eta(t)\to p$ as $t\to-\infty$ and $\eta(t)\to q$ as $t\to +\infty$ (under the identification sending $\eta$ to $\eta(0)$).  Any such $\eta$ lifts uniquely to a negative gradient trajectory $\tilde{\eta}$ of $-\tilde{v}$ having $\tilde{\eta}(t)\to \tilde{p}$ as $t\to -\infty$; the limit $\lim_{t\to\infty}\tilde{\eta}(t)$ will project to $q$ and hence, since our covering space is regular, will be of the form $\gamma\tilde{q}$ for some $\gamma\in \tilde{\Gamma}$.  In this way we see that $\pi$ restricts to a diffeomorphism \[ \coprod_{\gamma\in\tilde{\Gamma}}\mathbf{W}^{\tilde{v}}(\tilde{p},\gamma\tilde{q})\cong \mathbf{W}^v(p,q)\] which is equivariant with respect to the $\R$-actions given by the flows of $-\tilde{v}$ and $-v$ and hence induces a diffeomorphism of the quotients by these actions,  \begin{equation}\label{projmiso} \coprod_{\gamma\in\tilde{\Gamma}}\mathbf{M}^{\tilde{v}}(\tilde{p},\gamma\tilde{q})\cong \mathbf{M}^{v}(p,q).\end{equation} In particular, if $\mathrm{ind}_{\pi^*h}(\tilde{p})-\mathrm{ind}_{\pi^*h}(\tilde{q})=1$, then since $\mathbf{M}^{v}(p,q)$ is a finite set there are only finitely many $\gamma$ such that $\mathbf{M}^{\tilde{v}}(\tilde{p},\gamma\tilde{q})$ is nonempty.  

We may orient the various $\mathbf{D}^{\tilde{v}}(\tilde{p})$ and hence (via Section \ref{appmorse}) the various $\mathbf{M}^{\tilde{v}}(\tilde{p},\tilde{q})$ by first choosing orientations of the each $\mathbf{D}^v(p)$ for $p\in\mathrm{Crit}(h)$ and then requiring that the projection $\pi$ map the orientation on $\mathbf{D}^{\tilde{v}}(\tilde{p})$ to the chosen one on  $\mathbf{D}^{v}(\pi(\tilde{p}))$.  Then (\ref{projmiso}) will be an orientation-preserving diffeomorphism, and moreover any deck transformation $\gamma\in\tilde{\Gamma}$ will induce an orientation-preserving diffeomorphism $\mathbf{M}^{\tilde{v}}(\tilde{p},\tilde{q})\cong \mathbf{M}^{\tilde{v}}(\gamma\tilde{p},\gamma\tilde{q})$.  So the signed counts of points $n_{\pi^*h,\tilde{v}},n_{h,v}$ obey, whenever $\tilde{p}\in \mathrm{Crit}_{k+1}(\pi^*h),\tilde{q}\in\mathrm{Crit}_{k}(\pi^*h)$, \[  n_{h,v}(\pi(\tilde{p}),\pi(\tilde{q}))=\sum_{\gamma\in\tilde{\Gamma}}n_{\pi^*h,\tilde{v}}(\tilde{p},\gamma\tilde{q})  \] and \begin{equation}\label{npi}n_{\pi^*h,\tilde{v}}(\gamma\tilde{p},\gamma\tilde{q})=n_{\pi^*h,\tilde{v}}(\tilde{p},\tilde{q})\,\,(\forall \gamma\in\tilde{\Gamma}).\end{equation}

Given a field $\kappa$, we may then, just as before, define the Morse complex $\widetilde{\mathbf{CM}}_{*}(\pi^*h;\kappa[\tilde{\Gamma}])$ to have chain groups $\widetilde{\mathbf{CM}}_{k}(\pi^*h;\kappa[\tilde{\Gamma}])$ equal to the $\kappa$-vector space with basis $\mathrm{Crit}_k(\pi^*h)$ and with differential $\partial_{k+1}^{\pi^*h}\co \widetilde{\mathbf{CM}}_{k+1}(\pi^*h;\kappa[\tilde{\Gamma}])\to \widetilde{\mathbf{CM}}_{k}(\pi^*h;\kappa[\tilde{\Gamma}])$ given by, for each $\tilde{p}\in \mathrm{Crit}_{k+1}(\pi^*h)$, \[ \partial_{k+1}^{\pi^*h}\tilde{p}=\sum_{\tilde{q}\in\mathrm{Crit}_{k}(\pi^*h)}n_{\pi^*h,\tilde{v}}(\tilde{p},\tilde{q})\tilde{q}.\]  Now $\tilde{\Gamma}$ acts on each set $\mathrm{Crit}_{k}(\pi^*h)$, inducing the structure of a left $\kappa[\tilde{\Gamma}]$-module on $\widetilde{\mathbf{CM}}_{k}(\pi^*h;\kappa[\tilde{\Gamma}])$, and then (\ref{npi}) shows that $\partial_{k+1}^{\pi^*h}$ is a homomorphism of $\kappa[\tilde{\Gamma}]$-modules.  The $\widetilde{\mathbf{CM}}_k(\pi^*h;\kappa[\tilde{\Gamma}])$ are finitely-generated and free as $\kappa[\tilde{\Gamma}]$-modules, a basis being given by selecting, for each $p\in \mathrm{Crit}_k(h)$, a preimage $\tilde{p}$ under $\pi$.

The usual argument that the Morse boundary operator squares to zero is readily seen to show that, likewise, $\partial_{k}^{\pi^*h}\circ\partial_{k+1}^{\pi^*h}=0$, and thus $\widetilde{\mathbf{CM}}_*(\pi^*h;\kappa[\tilde{\Gamma}])$ is a chain complex of free, finitely-generated left $\kappa[\tilde{\Gamma}]$-modules.  By \cite[Theorem A.5]{Pa95}, there is a chain homotopy equivalence of $\kappa[\tilde{\Gamma}]$-modules $\tilde{\mathcal{E}}(h)\co \widetilde{\mathbf{CM}}_*(\pi^*h;\kappa[\tilde{\Gamma}])\to \mathbf{S}_{*}(\tilde{X})$ (where the $\kappa$-coefficient singular chain complex $\mathbf{S}_{*}(\tilde{X})$ is made into a complex of $\kappa[\tilde{\Gamma}]$-modules in the obvious way using the action of $\tilde{\Gamma}$ on $\tilde{X}$).


\begin{remark} \label{actionkey} 
Because $\widetilde{\mathbf{CM}}_{*}(\pi^*h_0;\kappa[\tilde{\Gamma}])$  denotes a complex of modules over $\kappa[\tilde{\Gamma}]$, its structure depends in particular on the action of $\tilde{\Gamma}$ on $\tilde{X}$ and not merely on the topology of $\tilde{X}$ and the abstract group structure on $\tilde{\Gamma}$.  This distinction becomes relevant in Section \ref{assocpair} when the same subgroup $\Gamma<\R$ acts on $\tilde{X}$ in two different ways.  We will resolve this notationally by referring to $\Gamma$ alternately as $\Gamma_{\xi}$ or $\Gamma_{-\xi}$ depending upon which action we are considering.
\end{remark}

On the group algebra $\kappa[\tilde{\Gamma}]$, just as in Section \ref{conjsect}\footnote{Since we are regarding $\tilde{\Gamma}$ as an abstract group which in principle need not be abelian we will write the group operation multiplicatively here, though in the cases that are ultimately of interest $\tilde{\Gamma}$ will be isomorphic to the subgroup $\Gamma<\R$ considered in the rest of the paper.  Correspondingly elements of the group algebra are written here as $\sum_{\gamma}a_{\gamma}\gamma$ instead of $\sum_{g}a_gT^g$.} we have a conjugation operation given by, for each $\lambda=\sum_{\gamma\in\tilde{\Gamma}}a_{\gamma}\gamma$, setting $\bar{\gamma}=\sum_{\gamma\in\Gamma}a_{\gamma}\gamma^{-1}$.  
If $M$ is a left (resp. right) $\kappa[\tilde{\Gamma}]$-module we define its conjugate $\bar{M}$ to be the right (resp. left) module with the same abelian group structure and with the scalar multiplications on $M$ and $\bar{M}$ related by $\lambda m=m\bar{\lambda}$ for $m\in M$ and $\lambda\in\kappa[\tilde{\Gamma}]$.  If $M$ is a left $\kappa[\tilde{\Gamma}]$-module then as in Section \ref{conjsect} we write ${}^{\vee}\!M$ for the left $\kappa[\tilde{\Gamma}]$-module obtained as the conjugate of the right $\kappa[\tilde{\Gamma}]$-module $\mathrm{Hom}_{\kappa[\tilde{\Gamma}]}(M,\kappa[\tilde{\Gamma}])$. (Thus for $\phi\in {}^{\vee}\!M$, $\lambda\in\kappa[\tilde{\Gamma}]$, and $m\in M$ we have $(\lambda\phi)(m)=\phi(m)\bar{\lambda}$.)

In particular, for a Morse function $h$ on the base $X$ of our regular covering space $\pi\co \tilde{X}\to X$ (and for a suitably generic gradient-like vector field for $h$ on $X$ which we pull back to $\tilde{X}$), we have a chain complex of left $\kappa[\tilde{\Gamma}]$-modules $\widetilde{\mathbf{CM}}_{*}(\pi^*h;\kappa[\tilde{\Gamma}])$ from which we can form the conjugated dual chain complex (with grading shift) $({}^{\vee}\!\widetilde{\mathbf{CM}}_{*}(\pi^*h;\kappa[\tilde{\Gamma}]))[-n]$.  This latter complex of left $\kappa[\tilde{\Gamma}]$-modules has grading $k$ part equal to ${}^{\vee}\!(\widetilde{\mathbf{CM}}_{n-k}(\pi^*h;\kappa[\tilde{\Gamma}]))$, and the differential $({}^{\vee}\!\widetilde{\mathbf{CM}}_{*}(\pi^*h;\kappa[\tilde{\Gamma}]))[-n]_{k+1}\to ({}^{\vee}\!\widetilde{\mathbf{CM}}_{*}(\pi^*h;\kappa[\tilde{\Gamma}]))[-n]_k$ acts by $(-1)^{k+1}$ times the adjoint of the differential $\partial^{\pi^*h}_{n-k}$.  Just as in Proposition \ref{negdual}, we have:

\begin{prop}\label{negdualcover} 
 Given an element $c=\sum_{\tilde{p}\in\mathrm{Crit}_{k}(-\pi^*h)}a_{\tilde{p}}\tilde{p}\in \widetilde{\mathbf{CM}}_{k}(-\pi^*h;\kappa[\tilde{\Gamma}])$ define $\epsilon_c\co \widetilde{\mathbf{CM}}_{n-k}(\pi^*h;\kappa[\tilde{\Gamma}])\to \kappa[\tilde{\Gamma}]$ by \[ \epsilon_c\left(\sum_{\tilde{q}\in\mathrm{Crit}_{n-k}(\pi^*h)}b_{\tilde{q}}\tilde{q}\right)=\sum_{\tilde{p}\in \mathrm{Crit}_k(-\pi^*h)}\sum_{\gamma\in\tilde{\Gamma}}a_{\tilde{p}}b_{\gamma\tilde{p}}\gamma.\]  Then, assuming that either $\mathrm{char}(\kappa)=2$ or that $X$ is oriented and that we choose orientations for the descending manifolds $\mathbf{D}^{\tilde{v}}(p),\mathbf{D}^{-\tilde{v}}(q)$ as in Section \ref{appmorse}, the assignment $c\mapsto \epsilon_c$ defines an isomorphism of chain complexes $\epsilon\co \widetilde{\mathbf{CM}}_{*}(-\pi^*h;\kappa[\tilde{\Gamma}])\to ({}^{\vee}\!\widetilde{\mathbf{CM}}_{*}(\pi^*h;\kappa[\tilde{\Gamma}]))[-n]$.  
\end{prop}

\begin{proof}
The maps $\epsilon_c$ are clearly $\kappa$-linear, as is the map $c\mapsto \epsilon_c$.  If $c=\sum_{\tilde{p}}a_{\tilde{p}}\tilde{p}\in \widetilde{\mathbf{CM}}_{k}(-\pi^*h;\kappa[\tilde{\Gamma}])$ and $d=\sum_{\tilde{q}}b_{\tilde{q}}\tilde{q}\in \widetilde{\mathbf{CM}}_{n-k}(\pi^*h;\kappa[\tilde{\Gamma}])$ then for $\xi,\eta\in\tilde{\Gamma}$ we have $\xi c=\sum_{\tilde{p}}a_{\xi^{-1}\tilde{p}}\tilde{p}$ and $\eta d=\sum_{\tilde{q}}b_{\eta^{-1}\tilde{q}}\tilde{q}$ and hence (making the substitutions $\tilde{q}=\xi^{-1}\tilde{p}$ in the sum in the second equality and $\mu=\eta^{-1}\gamma\xi$ in the third) \begin{align*} \epsilon_{\xi c}(\eta d) &=\sum_{\tilde{p}\in\mathrm{Crit}_{k}(-\pi^*h)}\sum_{\gamma\in\tilde{\Gamma}}a_{\xi^{-1}\tilde{p}}b_{\eta^{-1}\gamma\tilde{p}}\gamma \\ &= \sum_{\tilde{q}\in\mathrm{Crit}_{k}(-\pi^*h)}\sum_{\gamma\in\tilde{\Gamma}}a_{\tilde{q}}b_{\eta^{-1}\gamma\xi\tilde{q}}\gamma = \sum_{\tilde{q}\in\mathrm{Crit}_{k}(-\pi^*h)}\sum_{\mu\in\tilde{\Gamma}}a_{\tilde{q}}b_{\mu\tilde{q}}\eta\mu\xi^{-1} \\ &= \eta\epsilon_c(d)\xi^{-1}.\end{align*}

With $\xi$ equal to the identity this shows that $\epsilon_c\in\mathrm{Hom}_{\kappa[\tilde{\Gamma}]}(\widetilde{\mathbf{CM}}_{n-k}(\pi^*h;\kappa[\tilde{\Gamma}]),\kappa[\tilde{\Gamma}])$.  With, instead, $\eta$ equal to the identity we see that for any $\lambda\in \kappa[\tilde{\Gamma}]$, $\epsilon_{\lambda c}(d)=\epsilon_c(d)\bar{\lambda}$ and so $\epsilon\co c\mapsto \epsilon_c$ defines a morphism of $\kappa[\tilde{\Gamma}]$-modules $\widetilde{\mathbf{CM}}_{k}(-\pi^*h;\kappa)\to {}^{\vee}\!\left(\widetilde{\mathbf{CM}}_{n-k}(\pi^*h;\kappa[\tilde{\Gamma}])\right)$.  Recalling that $\widetilde{\mathbf{CM}}_{k}(-\pi^*h;\kappa)$ has basis given by lifts $\tilde{p}_i$ to $\tilde{X}$ of the various index-$k$ critical points of $-\pi^*h$, it is easy to see that $\epsilon$ maps this basis to the corresponding dual basis of ${}^{\vee}\!\left(\widetilde{\mathbf{CM}}_{n-k}(\pi^*h;\kappa[\tilde{\Gamma}])\right)$ and so is an isomorphism of modules.

The fact that $\epsilon$ is an isomorphism of chain complexes then follows from the identity $\epsilon_{\partial_{k+1}^{-\pi^*h}\tilde{q}}(\tilde{p})=(-1)^{k+1}\epsilon_{\tilde{q}}(\partial_{n-k}^{\pi^*h}\tilde{p})$ for $\tilde{q}\in\mathrm{Crit}_{k+1}(-\pi^*h)$ and $\tilde{p}\in\mathrm{Crit}_k(-\pi^*h)$, and results from the same calculation as in Proposition \ref{negdual}.
\end{proof}

If now $h_-,h_+$ are two Morse functions on $X$, one obtains a continuation map $\tilde{\Phi}_{h_-h_+}\co \widetilde{\mathbf{CM}}_{*}(h_-;\kappa[\tilde{\Gamma}])\to\widetilde{\mathbf{CM}}_*(h_+;\kappa[\tilde{\Gamma}])$ by choosing a family of vector fields $\mathbb{V}=\{v_s\}_{s\in \R}$ on $X$ interpolating between gradient-like vector fields for $h_-$ and $h_+$, lifting these vector fields  to $\tilde{X}$, and counting solutions in $\tilde{X}$ of the obvious analogue of (\ref{conteq}) according to their asymptotics in the usual way.  (Alternatively, $\tilde{\Phi}_{h_-h_+}$ could be constructed directly from the spaces $\mathbf{N}^{\mathbb{V}}(p_-,p_+)$ underlying the construction of the original continuation map $\Phi_{h_-h_+}$ by keeping track of the asymptotics of the lifts to $\tilde{X}$ of the solutions to (\ref{conteq}).)  This yields a chain homotopy equivalence of complexes of $\kappa[\tilde{\Gamma}]$-modules (with chain homotopy inverse given by $\tilde{\Phi}_{h_+h_-}$).   

In particular, one obtains a chain-level Poincar\'e duality map $\tilde{\mathcal{D}}\co \widetilde{\mathbf{CM}}_{*}(h;\kappa[\tilde{\Gamma}])\to ({}^{\vee}\!\widetilde{\mathbf{CM}}_{*}(h;\kappa[\tilde{\Gamma}]))[-n]$ by composing the continuation map $\tilde{\Phi}_{h,-h}$ with the isomorphism $\epsilon$ from Proposition \ref{negdualcover}.

\subsection{Novikov complexes}\label{novsect}

We now consider a compact connected smooth manifold $X$ and a de Rham cohomology class $\xi\in H^1(X;\R)$.  Let $\pi \co\tilde{X}_{\xi}\to X$ be the covering space associated to the kernel of the evaluation map $\langle \xi,\cdot\rangle\co \pi_1(X)\to \R$.  Then $\pi^{*}\xi=0\in H^1(X_{\xi};\R)$, and we may and do identify the deck transformation group of $\tilde{X}_{\xi}$ with the subgroup $\Gamma_{\xi}=\Img(\langle \xi\,\cdot\rangle\co \pi_1(X)\to \R)$  of $\R$ in such a way that, for $\tilde{p}\in \tilde{X}_{\xi}$ and $g\in\Gamma_{\xi}$, $\xi$ evaluates as $-g$ on the image under $\pi$ of a path from $\tilde{x}$ to $g\tilde{x}$.  Note that, while $\tilde{X}_{\xi}=\tilde{X}_{-\xi}$ and $\Gamma_{\xi}=\Gamma_{-\xi}$ as sets, our conventions lead to the action of $\Gamma_{-\xi}$ on $\tilde{X}_{-\xi}$ being opposite to the action of $\Gamma_{\xi}$ on $\tilde{X}_{\xi}$.

If $\theta\in \Omega^1(X)$ is a closed $1$-form representing $\xi$, then we will have $\pi^*\theta=d\tilde{f}$ for a smooth function $\tilde{f}\co \tilde{X}_{\xi}\to \R$ that satisfies $\tilde{f}(g\tilde{p})=\tilde{f}(\tilde{p})-g$ for all $g\in\Gamma_{\xi}$ and $\tilde{p}\in\tilde{X}_{\xi}$.  The $1$-form $\theta$ is said to be Morse if, considered as a section of the cotangent bundle $T^*X$, it is transverse to the zero section; this is equivalent to $\tilde{f}$ being a Morse function.  In this case, the Novikov complex $\mathbf{CN}_{*}(\tilde{f};\xi)$ is defined using the flow of the negative of a gradient-like vector field $\tilde{v}$ for $\tilde{f}$ which is the lift of a Morse-Smale vector field $v$ on $X$. Each zero $p$ of $\theta$ will then have injectively immersed descending and ascending manifolds $\mathbf{D}^v(p),\mathbf{A}^v(p)$ for the flow of $v$, and by the Morse-Smale condition (which can be arranged to hold by taking $v$ equal to the metric dual of $\theta$ with respect to a suitably generic metric by \cite[Proposition 1]{BH01}) the various $\mathbf{D}^v(p)$ and $\mathbf{A}^v(q)$ will be transverse.  The critical points $\tilde{p}$ of $\tilde{f}$ are the preimages under $\pi$ of the zeros of $\theta$, and their descending and ascending manifolds $\mathbf{D}^{\tilde{v}}(\tilde{p}),\mathbf{A}^{\tilde{v}}(\tilde{p})$ are embedded disks which are lifts of the $\mathbf{D}^v(p),\mathbf{A}^{v}(p)$.  We orient the $\mathbf{D}^{\tilde{v}}(\tilde{p})$ by lifting chosen orientations on the $\mathbf{D}^{v}(p)$; thus for distinct critical points $\tilde{p},g\tilde{p}$ of $\tilde{f}$ in the same fiber of $\pi$ the orientations of $\mathbf{D}^{\tilde{v}}(\tilde{p})$ and $\mathbf{D}^{\tilde{v}}(g\tilde{p})$ correspond under the deck transformation $g$.

In this situation, as before, for distinct critical points $\tilde{p},\tilde{q}$ of $\tilde{f}$ we obtain oriented manifolds $\mathbf{W}^{\tilde{v}}(\tilde{p},\tilde{q})=\mathbf{D}^{\tilde{v}}(\tilde{p})\cap\mathbf{A}^{\tilde{v}}(\tilde{q})$ and $\mathbf{M}^{\tilde{v}}(\tilde{p},\tilde{q})=\frac{\mathbf{W}^{\tilde{v}}(\tilde{p},\tilde{q})}{\R}$, and, if $\mathrm{ind}_{\tilde{f}}(\tilde{p})-\mathrm{ind}_{\tilde{f}}(q)=1$, a signed count $n_{\tilde{f},\tilde{v}}(\tilde{p},\tilde{q})$ of the points of $\mathbf{M}^{\tilde{v}}(\tilde{p},\tilde{q})$.  In contrast to the situation of the previous section, though, given $\tilde{p}\in\mathrm{Crit}_{k+1}(\tilde{f})$ there may be infinitely many $\tilde{q}\in\mathrm{Crit}_k(\tilde{f})$ for which $n_{\tilde{f},\tilde{v}}(\tilde{p},\tilde{q})\neq 0$.  There will however be only finitely many such $\tilde{q}$ obeying any given lower bound $\tilde{f}(\tilde{q})\geq c$, and so one defines, following \cite{Nov}, the degree-$k$ part of the Novikov chain complex of $\tilde{f}$ to be consist of possibly-infinite sums as follows: \[ \mathbf{CN}_{k}(\tilde{f};\xi)=\left\{\left.\sum_{\tilde{p}\in\mathrm{Crit}_{k}(\tilde{f})}a_{\tilde{p}}\tilde{p}\right|a_{\tilde{p}}\in\kappa,(\forall c\in \R)(\#\{\tilde{p}|a_{\tilde{p}}\neq 0,\,\tilde{f}(\tilde{p})>c\}<\infty)\right\}.\]
This is  a vector space over the field $\Lambda_{\uparrow}$ from Section \ref{basic} (using $\Gamma_{\xi}$ in the role of $\Gamma$) with the obvious action extending the action of the group ring, bearing in mind our sign convention that $\tilde{f}(g\tilde{p})=\tilde{f}(\tilde{p})-g$ for $g\in\Gamma_{\xi}$.
Moreover, defining a boundary operator $\partial_{k+1}^{\tilde{f}}\co \mathbf{CN}_{k+1}(\tilde{f};\xi)\to  \mathbf{CN}_{k}(\tilde{f};\xi)$ by the usual prescription $\partial_{k+1}^{\tilde{f}}\tilde{p}=\sum_{\tilde{q}\in\mathrm{Crit}_{k}(\tilde{f})}n_{\tilde{f},\tilde{v}}(\tilde{p},\tilde{q})\tilde{q}$ makes $(\mathbf{CN}_{*}(\tilde{f});\xi),\partial^{\tilde{f}})$ into a chain complex of $\Lambda_{\uparrow}$-vector spaces.  If the zeros of $\theta$ having index $k$ are $p_1,\ldots,p_d$, and then any choice $\{\tilde{p}_1,\ldots,\tilde{p}_d\}$ of lifts to $\tilde{X}_{\xi}$ of the $p_i$ gives a basis for $\mathbf{CN}_k(\tilde{f};\xi)$.  See \cite{Lat},\cite{BH01} for more details regarding the construction of the Novikov complex.

The Novikov complex becomes a Floer-type complex by setting \[ \ell_{\uparrow}^{\tilde{f}}\left(\sum_{\tilde{p}}a_{\tilde{p}}\tilde{p}\right)=\max\{\tilde{f}(\tilde{p})|a_{\tilde{p}}\neq 0\}.\] A basis $\{\tilde{p}_1,\ldots,\tilde{p}_d\}$ as in the preceding paragraph will  then be $\ell_{\uparrow}^{\tilde{f}}$-orthogonal, with $\ell_{\uparrow}^{\tilde{f}}\left(\sum_{i}\lambda_i\tilde{p}_i\right)=\max_i\{\tilde{f}(\tilde{p}_i)-\nu_{\uparrow}(\lambda_i)\}$ for $\lambda_i\in\Lambda_{\uparrow}$.  We write $\mathbf{CN}(\tilde{f};\xi)_{\uparrow}$ for the Floer-type complex consisting of the Novikov chain complex $(\mathbf{CN}_{*}(\tilde{f};\xi),\partial^{\tilde{f}})$ and the filtration funtion $\ell_{\uparrow}^{\tilde{f}}$.  

Suppose that $\theta_-,\theta_+$ are closed Morse one-forms on $X$ that both represent the de Rham cohomology class $\xi$, and let $\tilde{f}_-,\tilde{f}_+\co \tilde{X}_{\xi}\to \R$ be (necessarily Morse) functions such that $d\tilde{f}_{\pm}=\pi^{*}\theta_{\pm}$.  Using the lift to $\tilde{X}$ of a suitable one-parameter family of vector fields on $X$, we then obtain continuation maps $\Phi_{\tilde{f}_-\tilde{f}_+}\co \mathbf{CN}_{*}(\tilde{f}_-;\xi)\to\mathbf{CN}_{*}(\tilde{f}_+;\xi)$ and $\Phi_{\tilde{f}_+\tilde{f}_-}\co \mathbf{CN}_{*}(\tilde{f}_+;\xi)\to\mathbf{CN}_{*}(\tilde{f}_-;\xi)$, by the same procedure as in Section \ref{morseintro}; these are chain maps defined over $\Lambda_{\uparrow}$ that are homotopy inverses to each other, and so the chain homotopy type of $\mathbf{CN}_{*}(\tilde{f};\xi)$ depends only on the cohomology class $\xi$.  It is important to note that the existence of these continuation maps depends on $\tilde{f}_{\pm}$ being pullbacks of \emph{cohomologous} $1$-forms; were this not the case, the difference $\tilde{f}_+-\tilde{f}_-$ would be unbounded, so (\ref{conteq}) would no longer provide bounds that are needed in the proof that $\Phi_{\tilde{f}_-\tilde{f}_+}$ is well-defined and respects the finiteness condition in the definition of the Novikov chain complexes.  In particular, if $\xi\neq 0$ there is no natural continuation map from $\mathbf{CN}_{*}(\tilde{f};\xi)$ to $\mathbf{CN}_{*}(-\tilde{f};-\xi)$ even though their homologies are isomorphic as graded vector spaces over the same field $\Lambda_{\uparrow}$; if there were such a map, much of the complexity of this paper might have been avoided.

Adapting part of the proof of \cite[Th\'eor\`eme 2.18]{Lat}, the homotopy type of $\mathbf{CN}_{*}(\tilde{f};\xi)$ can be identified as follows.  Begin with a Morse function $h_0\co X\to \R$ and a gradient-like vector field $v_0$ for $h_0$ whose flow is Morse-Smale. Then choose a closed $1$-form $\theta_0$ on $X$ which represents the cohomology class $\xi$ and has support contained in $X\setminus \Omega$ for some contractible neighborhood $\Omega$ of the zero locus of $dh_0$.  Then if $N\gg 1$, the closed one-form $\theta_N:=Ndh_0+\theta_0$ on $X$ will also represent the class $\xi$, will coincide with $Ndh_0$ on a neighborhood of $\mathrm{Crit}(h_0)$, and will have the property that $(\theta_N)_x(v_0)>0$ for all $x\in X\setminus \mathrm{Crit}(h_0)$.   Consequently, letting $\tilde{h}_N\co \tilde{X}_{\xi}\to \R$ be such that $d\tilde{h}_N=\pi^*(Ndh_0+\theta_0)$, the lift $\tilde{v}_0$ of $v_0$ to $\tilde{X}_{\xi}$ will be gradient-like both for $\tilde{h}_N$ and for $\pi^*h_0$.  Thus we may use the descending and ascending manifolds of this same vector field $\tilde{v}_0$ in  the  formation both of the lifted Morse complex $\widetilde{\mathbf{CM}}_{*}(\pi^*h_0;\kappa[\Gamma_{\xi}])$ and of the Novikov complex $\mathbf{CN}_{*}(\tilde{h}_N;\xi)$.  It follows readily from this that we have identical chain complexes \begin{equation}\label{cxid} \mathbf{CN}_{*}(\tilde{h}_N;\xi)=\Lambda_{\uparrow}\otimes_{\kappa[\Gamma_{\xi}]}\widetilde{\mathbf{CM}}_{*}(\pi^*h_0;\kappa[\Gamma_{\xi}]). \end{equation}  We define the \textbf{Latour map} $\mathcal{L}_{h_0\tilde{f}}\co \widetilde{\mathbf{CM}}_{*}(\pi^*h_0;\kappa[\Gamma_{\xi}])\to \mathbf{CN}_{*}(\tilde{f};\xi)$ to be the composition of the coefficient extension $\widetilde{\mathbf{CM}}_{*}(\pi^*h_0;\kappa[\Gamma_{\xi}])\hookrightarrow \Lambda_{\uparrow}\otimes_{\kappa[\Gamma_{\xi}]}\widetilde{\mathbf{CM}}_{*}(\pi^*h_0;\kappa[\Gamma_{\xi}])$ and the continuation map $\mathbf{CN}_{*}(\tilde{h}_N;\xi)\to \mathbf{CN}_{*}(\tilde{f};\xi)$, as (\ref{cxid}) identifies the codomain of the former with the domain of the latter.
 Then $1_{\Lambda_{\uparrow}}\otimes \mathcal{L}_{h_0\tilde{f}}$ is just the continuation map   $\mathbf{CN}_{*}(\tilde{h}_N;\xi)\to \mathbf{CN}_{*}(\tilde{f};\xi)$ and so is a chain homotopy equivalence.  We now have the ingredients needed to generalize Definition \ref{cmdef} to the Novikov context:

\begin{dfn}\label{cndef} Let $X$ be a smooth compact manifold   and if our ground field $\kappa$ has $\mathrm{char}(\kappa)\neq 2$ assume that $X$ is oriented.  Fix a Morse function $h_0\co X\to\R$. Let $\xi\in H^1(X;\R)$ with associated covering space $\pi\co \tilde{X}_{\xi}\to \R$ and let $\Gamma=\Gamma_{\xi}=\Img(\langle\xi,\cdot\rangle\co \pi_1(X)\to\R)$.   Given any Morse function $\tilde{f}\co \tilde{X}\to \R$ such that $d\tilde{f}$ is the pullback of a form in class $\xi$, we take $\mathcal{CN}(\tilde{f};\xi)$ to be the chain-level Poincar\'e-Novikov complex with the following data  $(\mathcal{C},\tilde{\mathcal{D}},\mathcal{C}_{\uparrow},\mathcal{S})$ as in Section \ref{cpnstr}:
\begin{itemize}
\item $\mathcal{C}=\widetilde{\mathbf{CM}}_{*}(\pi^*h_0;\kappa[\Gamma_{\xi}])$ is the lifted Morse complex of $h_0$;
\item $\tilde{\mathcal{D}}\co \mathbf{CM}_{*}(\pi^*h_0;\kappa[\Gamma_{\xi}])\to ({}^{\vee}\!\mathbf{CM}_{*}(\pi^*h_0;\kappa[\Gamma_{\xi}]))[-n]$ is the composition of the continuation map $\Phi_{\pi^*h_0,-\pi^*h_0}\co \widetilde{\mathbf{CM}}_*(\pi^*h_0;\kappa[\Gamma_{\xi}])\to \widetilde{\mathbf{CM}}_{*}(-\pi^*h_0;\kappa[\Gamma_{\xi}])$ with the chain isomorphism $\epsilon\co \widetilde{\mathbf{CM}}_{*}(-\pi^*h_0;\kappa[\Gamma_{\xi}])\to ({}^{\vee}\!\widetilde{\mathbf{CM}}_*(\pi^*h_0;\kappa[\Gamma_{\xi}]))[-n]$ from Proposition \ref{negdualcover}.
\item $\mathcal{C}_{\uparrow}=\mathbf{CN}(\tilde{f};\xi)_{\uparrow}$ (\emph{i.e.}, the $\Lambda_{\uparrow}$-Floer-type complex with chain complex $\mathbf{CN}_{*}(\tilde{f};\xi)$ and the filtration function  $\ell_{\uparrow}^{\tilde{f}}$).
\item $\mathcal{S}$ is the Latour map $\mathcal{L}_{h_0\tilde{f}}\co \widetilde{\mathbf{CM}}_{*}(\pi^*h_0;\kappa[\Gamma_{\xi}])\to \mathbf{CN}_{*}(\tilde{f};\xi)$  described above.
\end{itemize}
\end{dfn} 


\subsection{The associated chain-level filtered matched pairs}\label{assocpair}

Having defined the chain-level Poincar\'e-Novikov structures $\mathcal{CN}(\tilde{f};\xi)$ associated to Morse functions $\tilde{f}\co \tilde{X}_{\xi}\to \R$ on covers $\tilde{X}_{\xi}$ such that $d\tilde{f}$ is the pullback of a closed one-form in the cohomology class $\xi$, let us now discuss the chain-level filtered matched pairs $\mathcal{CP}(\mathcal{CN}(\tilde{f};\xi))$ associated to these by the construction in Section \ref{cpnstr}, as these are in turn the inputs for the rest of the constructions in Section \ref{clsec}.

By definition, the data $(\mathcal{C},\mathcal{C}^{\downarrow},\mathcal{C}_{\uparrow},\phi^{\downarrow},\phi_{\uparrow})$ of $\mathcal{CP}(\mathcal{CN}(\tilde{f};\xi))$ will include $\mathcal{C}=\widetilde{\mathbf{CM}}_{*}(\pi^*h_0;\kappa[\Gamma_{\xi}])$, $\mathcal{C}_{\uparrow}=\mathbf{CN}(\tilde{f};\xi)_{\uparrow}$, and $\phi_{\uparrow}=\mathcal{L}_{h_0\tilde{f}}$.  In the role of the $\Lambda^{\downarrow}$-Floer-type complex $\mathcal{C}^{\downarrow}$ is $\left({}^{\vee}\!\mathbf{CN}(\tilde{f};\xi)_{\uparrow}\right)[-n]$, but it is simpler to describe this using an analogue of Proposition \ref{negdualcover}.  Note first that, in our conventions, the covering spaces $\tilde{X}_{\xi}$ and $\tilde{X}_{-\xi}$ are the same, and we have well-defined $\Lambda_{\uparrow}$-Floer-type complexes $\mathbf{CN}_*(f;\xi)$ and $\mathbf{CN}_{*}(-f;-\xi)$ using the same subgroup $\Gamma=\Gamma_{\xi}=\Gamma_{-\xi}$ of $\R$ for the Novikov fields in both cases; however the ways in which this subgroup of $\R$ is identified with the deck transformation group of $\tilde{X}_{\xi}=\tilde{X}_{-\xi}$ are different due to the general rule that the action of an element $g\in \Gamma$ should decrease the value of the corresponding function (respectively $\tilde{f}$ or $-\tilde{f}$) by $g$. This leads to the need to conjugate the $\Lambda_{\uparrow}$-Floer-type complex $\mathbf{CN}(-\tilde{f};-\xi)_{\uparrow}$ to obtain a $\Lambda^{\downarrow}$-Floer-type complex in the following.  (Though we suppress the notation for the gradient-like vector fields in the statement of the proposition, it should be understood that if $\tilde{v}$ is the vector field used to define $\mathbf{CN}_{*}(f;\xi)$ then we use $-\tilde{v}$ to define $\mathbf{CN}_{*}(-f;-\xi)$.)

\begin{prop}\label{negdualnov}
 Given an element $c=\sum_{\tilde{p}\in\mathrm{Crit}_{k}(-\tilde{f})}a_{\tilde{p}}\tilde{p}\in \mathbf{CN}_{k}(-\tilde{f};-\xi)$ define $\epsilon_c\co \mathbf{CN}_{n-k}(\tilde{f};\xi)\to \Lambda_{\uparrow}$ by \begin{equation}\label{novdual} \epsilon_c\left(\sum_{\tilde{q}\in\mathrm{Crit}_{n-k}(\tilde{f})}b_{\tilde{q}}\tilde{q}\right)=\sum_{\tilde{p}\in \mathrm{Crit}_k(-\tilde{f})}\sum_{g\in\Gamma}a_{\tilde{p}}b_{g\tilde{p}}T^g \end{equation} where the notation $g\tilde{p}$ refers to the deck transformation action on $\tilde{X}_{\xi}$ associated to $\xi$ (not $-\xi$).  Then, assuming that either $\mathrm{char}(\kappa)=2$ or that $X$ is oriented and that we choose orientations for the descending manifolds $\mathbf{D}_{\tilde{f}}(p),\mathbf{D}_{-\tilde{f}}(q)$ as in Section \ref{appmorse}, the assignment $c\mapsto \epsilon_c$ defines an isomorphism of $\Lambda^{\downarrow}$-Floer-type complexes $\epsilon\co \overline{\mathbf{CN}(-\tilde{f};-\xi)_{\uparrow}  }\to ({}^{\vee}\!\mathbf{CN}(\tilde{f};\xi)_{\uparrow})[-n]$. 
\end{prop}

\begin{proof}
Let $\mathrm{Crit}_{k}(-f)=\{p_1,\ldots,p_d\}$ and choose lifts $\tilde{p}_1,\ldots,\tilde{p}_d$ of the $p_i$ to the covering space $\tilde{X}_{-\xi}=\tilde{X}_{\xi}$.  Thus $\{\tilde{p}_1,\ldots,\tilde{p}_d\}$ gives an orthogonal basis for both orthogonalizable $\Lambda_{\uparrow}$-spaces $\mathbf{CN}_{k}(-f;-\xi)$ and $\mathbf{CN}_{n-k}(f;\xi)$. 

Any $\tilde{p}\in\mathrm{Crit}_{k}(-\tilde{f})$ can be expressed as $h\tilde{p}_i$ for unique $i\in\{1,\ldots,d\}$ and $h\in \Gamma_{\xi}$, where as in the statement of the proposition we use in this notation the deck transformation action associated to $\xi$ rather than $-\xi$.  Thus the corresponding elements $\tilde{p}$ in the respective Novikov complexes will equal $T^{-h}\tilde{p}_i$ in $\mathbf{CN}_{k}(-\tilde{f};-\xi)$ and $T^h\tilde{p}_i$ in $\mathbf{CN}_{n-k}(\tilde{f};\xi)$.  So the pairing $(c,x)\mapsto \epsilon_c(x)$ can be rewritten in terms of the basis as \[ \left(\sum_{i=1}^{d}\sum_{s\in\Gamma}a_{i,s}T^s\tilde{p}_i, \sum_{i=1}^{d}\sum_{t\in\Gamma}b_{i,t}T^t\tilde{p}_i\right)\mapsto \sum_{i=1}^{d}\sum_{s,t\in\Gamma}a_{i,s}b_{i,t}T^{s+t},\] or equivalently as, for $\lambda_i,\mu_i\in\Lambda_{\uparrow}$, \[ \left(\sum_{i}\lambda_i\tilde{p}_i,\sum_i\mu_i\tilde{p}_i\right)\mapsto  \sum_i\lambda_i\mu_i.\]  This makes clear that (\ref{novdual}) gives a well-defined element of $\Lambda_{\uparrow}$, and that $\epsilon$ defines an isomorphism from $\mathbf{CN}_{k}(-f;-\xi)$ to the (unconjugated) dual of $\mathbf{CN}_{n-k}(f;\xi)$, and hence, after conjugation, from $\overline{\mathbf{CN}_{k}(-f;-\xi)}$ to ${}^{\vee}\!(\mathbf{CN}_{n-k}(f;\xi))$.  That $\epsilon$ is an isomorphism of chain complexes follows from the identity $n_{-\tilde{f},-\tilde{v}}(\tilde{q},\tilde{p})=(-1)^{k+1}n_{\tilde{f},\tilde{v}}(\tilde{p},\tilde{q})$ just as in Proposition \ref{negdualcover}.

To conclude that $\epsilon$ is an isomorphism of $\Lambda^{\downarrow}$-Floer-type complexes we need to check that it intertwines the filtration functions, namely $-\ell_{\uparrow}^{-\tilde{f}}$ on $\overline{\mathbf{CN}(-\tilde{f};-\xi)_{\uparrow}  }$ (see Remark \ref{switch}) and ${}^{\vee}\!\ell_{\uparrow}^{\tilde{f}}$ on $({}^{\vee}\!\mathbf{CN}(\tilde{f};\xi)_{\uparrow})[-n]$ (see Section \ref{dualsec}).  In grading $k$, it follows from the preceding paragraph that the $(-\ell_{\uparrow}^{-\tilde{f}})$-orthogonal basis $\{\tilde{p}_1,\ldots,\tilde{p}_{d}\}$ for $\overline{\mathbf{CN}_{k}(-\tilde{f};-\xi)}$ is mapped by $\epsilon$ to the dual basis   $\{\tilde{p}_{1}^{*},\ldots,\tilde{p}_{d}^{*}\}$ to the $\ell_{\uparrow}^{\tilde{f}}$-orthogonal basis $\{\tilde{p}_1,\ldots,\tilde{p}_{d}\}$ for $\mathbf{CN}_{n-k}(\tilde{f};\xi)$.  By Remark \ref{dualbasisrem}, this dual basis is an ${}^{\vee}\!\ell_{\uparrow}^{\tilde{f}}$-orthogonal basis for ${}^{\vee}\!\mathbf{CN}_{n-k}(f;\xi)$, and moreover we see that $(-\ell_{\uparrow}^{-\tilde{f}})(\tilde{p}_i)={}^{\vee}\!\ell_{\uparrow}^{\tilde{f}}(\tilde{p}_{i}^{*})=\tilde{f}(\tilde{p}_i)$.  Thus, in each grading, $\epsilon$ maps a $(-\ell_{\uparrow}^{-\tilde{f}})$-orthogonal basis to a ${}^{\vee}\!\ell_{\uparrow}^{\tilde{f}}$-orthogonal basis, preserving the filtration level of each basis element, from which it follows readily that ${}^{\vee}\!\ell_{\uparrow}^{\tilde{f}}=(-\ell_{\uparrow}^{-\tilde{f}})\circ \epsilon$, as desired.
\end{proof}

\begin{notation}
We write $\mathbf{CN}(\tilde{f};\xi)^{\downarrow}$ for the $\Lambda^{\downarrow}$-Floer-type complex $\overline{\mathbf{CN}(-\tilde{f};-\xi)_{\uparrow}}$.
\end{notation}

This notation is designed to emphasize that $\overline{\mathbf{CN}(-\tilde{f};-\xi)_{\uparrow}}$ can be interpreted directly in terms of $\tilde{f}$ instead of $-\tilde{f}$: it is a version of the filtered Novikov chain complex for $\tilde{f}$ in which the differential is defined in terms of the positive gradient flow of $\tilde{f}$ instead of the negative gradient flow as in the usual Novikov complex, and correspondingly the filtration is descending rather than ascending.  The general degree-$k$ chain in $\mathbf{CN}(\tilde{f};\xi)^{\downarrow}$ takes the form $\sum_{\tilde{p}\in\mathrm{Crit}_{n-k}(\tilde{f})}a_{\tilde{p}}\tilde{p}$ where for any $C\in\R$ only finitely many $\tilde{p}$ have both $a_{\tilde{p}}\neq 0$ and $\tilde{f}(\tilde{p})<C$, and the filtration function $\ell_{\tilde{f}}^{\downarrow}:=-\ell^{-\tilde{f}}_{\uparrow}$ is given by \[ \ell_{\tilde{f}}^{\downarrow}\left(\sum_{\tilde{p}\in\mathrm{Crit}_{n-k}(\tilde{f})}a_{\tilde{p}}\tilde{p}\right)=\min\{\tilde{f}(\tilde{p})|a_{\tilde{p}}\neq 0\}.\]

There is a version of the Latour map, which we will denote by $\bar{\mathcal{L}}_{h_0\tilde{f}}$, from the lifted Morse complex $\widetilde{\mathbf{CM}}_{*}(\pi^*h_0;\kappa[\Gamma_{\xi}])$ to the underlying chain complex $\overline{\mathbf{CN}_{*}(-\tilde{f};-\xi)}$ of $\mathbf{CN}(f;\xi)^{\downarrow}$.  Namely, as alluded to in Remark \ref{actionkey},  $\Gamma_{\xi}$ and $\Gamma_{-\xi}$ refer to the same group $\Gamma$ but with actions on $\tilde{X}_{\xi}=\tilde{X}_{-\xi}$ that are inverse to each other, so that we have an equality of complexes of $\kappa[\Gamma]$-modules \[ \widetilde{\mathbf{CM}}_{*}(\pi^*h_0;\kappa[\Gamma_{\xi}])=\overline{\widetilde{\mathbf{CM}}_{*}(\pi^*h_0;\kappa[\Gamma_{-\xi}])}.\]  We then have a Latour map $\mathcal{L}_{h_0,-\tilde{f}}\co   \widetilde{\mathbf{CM}}_{*}(\pi^*h_0;\kappa[\Gamma_{-\xi}])\to \mathbf{CN}_{*}(-\tilde{f};-\xi)$; applying the conjugation functor together with the above identification of $\kappa[\Gamma]$-modules yields the promised conjugate Latour map \[ 
\bar{\mathcal{L}}_{h_0\tilde{f}}\co \widetilde{\mathbf{CM}}_{*}(\pi^*h_0;\kappa[\Gamma_{\xi}]) \to \overline{\mathbf{CN}_{*}(-\tilde{f};-\xi)}.\]

\begin{prop}\label{downup}
The chain-level filtered matched pair $\mathcal{CP}(\mathcal{CN}(\tilde{f};\xi))$ associated to $\mathcal{CN}(\tilde{f};\xi)$ is filtered matched homotopy equivalent to the chain-level filtered matched pair \begin{equation}\label{eqcmp} \xymatrix{ & \mathbf{CN}(\tilde{f};\xi)^{\downarrow} \\ \widetilde{\mathbf{CM}}_{*}(\pi^*h_0;\kappa[\Gamma_{\xi}]) \ar[ur]^{\bar{\mathcal{L}}_{h_0\tilde{f}}} \ar[dr]_{\mathcal{L}_{h_0\tilde{f}}} & \\ & \mathbf{CN}(\tilde{f};\xi)_{\uparrow}   }\end{equation}
\end{prop}

\begin{proof} In the general notation of Definition \ref{cfmpdfn}, the two chain-level filtered matched pairs in question have the same data $\mathcal{C}$ and $\mathcal{C}_{\uparrow}$, while the role of $\mathcal{C}^{\downarrow}$ is played by $({}^{\vee}\!\mathbf{CN}(f;\xi)_{\uparrow})[-n]$ in one case and $\mathbf{CN}(f;\xi)^{\downarrow}$ in the other.  So we may use the isomorphism $\epsilon$ from Proposition \ref{negdualnov} in the role of the top horizontal map in Definition \ref{fmhedef}, and the respective identities in the roles of the other horizontal maps; it remains only to check that the diagram \begin{equation}\label{topsquare} \xymatrix{  \overline{\mathbf{CN}_{*}(-\tilde{f};-\xi)} \ar[rr]^{\epsilon} & & ({}^{\vee}\!\mathbf{CN}(f;\xi)_{\uparrow})[-n] \\ & \widetilde{\mathbf{CM}}_{*}(\pi^*h_0;\kappa[\Gamma_{\xi}]) \ar[ul]^{\bar{\mathcal{L}}_{h_0\tilde{f}}} \ar[ur]_{\phi^{\downarrow}} & } \end{equation} commutes up to homotopy, where $\phi^{\downarrow}$ is the map obtained as in (\ref{phidowncpn}), with the map $\mathcal{T}$ therein equal to a homotopy inverse to $1_{\Lambda_{\uparrow}}\otimes \mathcal{L}_{h_0\tilde{f}}$.  (The desired conclusion is independent of which homotopy inverse we use.)

As in the definition of the Latour map, let $\theta_0$ be a closed one-form representing the cohomology class $\xi$ which vanishes on a neighborhood of $\mathrm{Crit}(h_0)$, and choose $N$ sufficiently large that the lift $\tilde{v}$ to $\tilde{X}_{\xi}$ of a gradient-like vector field $v$ for $h_0$ is still gradient-like for a primitive $\tilde{h}_N$ for the one-form $\pi^*(Ndh_0+\theta_0)$.  Increasing $N$ if necessary, assume that $\tilde{v}$ is also gradient-like for a primitive $\tilde{g}_N$ for the one-form $\pi^*(Ndh_0-\theta_0)$.  Using $\tilde{v}$ in the definition of the differentials on the Morse and Novikov complexes for $\pi^*h_0$, $\tilde{h}_N$, and $\tilde{g}_N$, we then have identical complexes $\Lambda_{\uparrow}\otimes_{\kappa[\Gamma_{\xi}]}\widetilde{\mathbf{CM}}_{*}(\pi^*h_0;\kappa[\Gamma_{\xi}])=\mathbf{CN}_{*}(\tilde{h}_N;\xi)$, and also identical complexes $\Lambda_{\uparrow}\otimes_{\kappa[\Gamma_{-\xi}]}\widetilde{\mathbf{CM}}_{*}(\pi^*h_0;\kappa[\Gamma_{-\xi}])=\mathbf{CN}_*(\tilde{g}_N;-\xi)$. Using the first of these identifications,    $1_{\Lambda_{\uparrow}}\otimes \mathcal{L}_{h_0\tilde{f}}$ is equal to the continuation map $\Phi_{\tilde{h}_N\tilde{f}}$.  So for its homotopy inverse $\mathcal{T}$ as in (\ref{phidowncpn}) we may use the continuation map $\Phi_{\tilde{f}\tilde{h}_N}\co\mathbf{CN}_{*}(\tilde{f};\xi)\to \mathbf{CN}_{*}(\tilde{h}_N;\xi)$.  

We consider the diagram \begin{equation}\label{8term} \xymatrix{  \overline{\mathbf{CN}_{*}(-\tilde{f};-\xi)} \ar[rr]^{\epsilon} &   & ({}^{\vee}\!\mathbf{CN}(f;\xi)_{\uparrow})[-n]  \\ 
   \overline{\mathbf{CN}_{*}(\tilde{g}_N;-\xi)} \ar[r]_{\Phi_{\tilde{g}_N,-\tilde{h}_N}} \ar[u]^{\Phi_{\tilde{g}_N,-\tilde{f}}} & \overline{\mathbf{CN}_{*}(-\tilde{h}_N;-\xi)} \ar[r]_{\epsilon} & ({}^{\vee}\!\mathbf{CN}_*(\tilde{h}_N;\xi))[-n] \ar[u]_{{}^{\vee}\!\Phi_{\tilde{f}\tilde{h}_N}}
\\
\widetilde{\mathbf{CM}}_{*}(\pi^*h_0;\kappa[\Gamma_{\xi}]) \ar[r]_{\Phi_{\pi^*h_0,-\pi^*h_0}} \ar[u]^{f_1} & \widetilde{\mathbf{CM}}_{*}(-\pi^*h_0;\kappa[\Gamma_{\xi}]) \ar[r]_{\epsilon}\ar[u]^{f_2} & ({}^{\vee}\!\widetilde{\mathbf{CM}}_{*}(\pi^*h_0;\kappa[\Gamma_{\xi}]))[-n] \ar[u]_{f_3} 
}\end{equation}
Here the  maps denoted $\Phi$ are continuation maps, the maps denoted $\epsilon$ are appropriate versions of the maps from Propositions \ref{negdualcover} and \ref{negdualnov}, and the maps $f_1,f_2,f_3$ are as follows:
\begin{itemize}
\item $f_1\co \widetilde{\mathbf{CM}}_{*}(\pi^*h_0;\kappa[\Gamma_{\xi}])\to \overline{\mathbf{CN}_{*}(\tilde{g}_N;-\xi)}$ is the composition of the trivial identification of $\widetilde{\mathbf{CM}}_{*}(\pi^*h_0;\kappa[\Gamma_{\xi}])$ with $\overline{\widetilde{\mathbf{CM}}_{*}(\pi^*h_0;\kappa[\Gamma_{-\xi}])}$, the coefficient extension of the latter to $\overline{\Lambda_{\uparrow}\otimes_{\kappa[\Gamma_{-\xi}]}\widetilde{\mathbf{CM}}_{*}(\pi^*h_0;\kappa[\Gamma_{-\xi}])}$, and the identification of $\overline{\Lambda_{\uparrow}\otimes_{\kappa[\Gamma_{-\xi}]}\widetilde{\mathbf{CM}}_{*}(\pi^*h_0;\kappa[\Gamma_{-\xi}])}$ with $\overline{\mathbf{CN}_{*}(\tilde{g}_N;-\xi)}$ that results from $\tilde{v}$ being gradient-like for both $\pi^*h_0$ and $\tilde{g}_N$.
\item Similarly, $f_2$ is the composition \begin{align*} & \widetilde{\mathbf{CM}}_{*}(-\pi^*h_0;\kappa[\Gamma_{\xi}])=\overline{\widetilde{\mathbf{CM}}_{*}(-\pi^*h_0;\kappa[\Gamma_{-\xi}])}\\ & \,\,\to \overline{\Lambda_{\uparrow}\otimes_{\kappa[\Gamma_{-\xi}]}\widetilde{\mathbf{CM}}_{*}(-\pi^*h_0;\kappa[\Gamma_{-\xi}])}=\overline{\mathbf{CN}_{*}(-\tilde{h}_N;-\xi)} \end{align*} where we use the vector field $-\tilde{v}$ (which is gradient-like for both $-\pi^*h_0$ and $-\tilde{h}_N$) to define $\widetilde{\mathbf{CM}}_{*}(-\pi^*h_0;\kappa[\Gamma_{\xi}])$ and $\overline{\mathbf{CN}_{*}(-\tilde{h}_N;-\xi)}$.  
\item $f_3\co ({}^{\vee}\!\widetilde{\mathbf{CM}}_{*}(\pi^*h_0;\kappa[\Gamma_{\xi}]))[-n]\to ({}^{\vee}\!\mathbf{CN}_{*}(\tilde{h}_N;\xi))[-n]$ is the composition of the second and third maps in (\ref{phidowncpn}) (applied with $\mathcal{C}=\widetilde{\mathbf{CM}}_{*}(\pi^*h_0;\kappa[\Gamma_{\xi}])$) followed by the dual of the identification of $\Lambda_{\uparrow}\otimes_{\kappa[\Gamma_{\xi}]}\widetilde{\mathbf{CM}}_{*}(\pi^*h_0;\kappa[\Gamma_{\xi}])$ with $\mathbf{CN}_{*}(\tilde{h}_N;\xi)$.  
\end{itemize}

By definition, the composition on the left side of (\ref{8term}) is $\bar{\mathcal{L}}_{h_0\tilde{f}}$.  Also, the composition on the bottom of (\ref{8term}) is the Poincar\'e duality map $\tilde{\mathcal{D}}$, so the bottom and right sides together give the map $\phi^{\downarrow}$ in (\ref{topsquare}).  So to complete the proof it suffices to show that (\ref{8term}) is homotopy-commutative.\footnote{Strictly speaking, (\ref{8term}) is underspecified in that we have not said what time-dependent vector fields $\mathbb{V}=\{v_s\}_{s\in\R}$ are to be used in the construction of the various continuation maps; however since the homotopy classes of these continuation maps are independent of such choices we may make whatever choices are convenient.}   

The bottom left square of (\ref{8term}) is (homotopy)-commutative because we may use the same time-dependent vector fields $\mathbb{V}$ in the construction of $\Phi_{\tilde{g}_N,-\tilde{h}_N}$ as we use in the construction of $\Phi_{\pi^*h_0,-\pi^*h_0}$; this results in all relevant trajectory spaces being identical and hence in  $\Phi_{\tilde{g}_N,-\tilde{h}_N}$ being the coefficient extension to $\Lambda_{\uparrow}$ of the $\kappa[\Gamma_{\xi}]$-module homomorphism $\Phi_{\pi^*h_0,-\pi^*h_0}$.  

The bottom right square of (\ref{8term}) is also commutative.  Indeed, if $\tilde{S}\subset \mathrm{Crit}(\pi^*h_0)$ contains one point in the fiber over each critical point of the original Morse function $h_0\co X\to \R$, then $\tilde{S}$ serves as a basis simultaneously for the graded $\kappa[\Gamma_{\xi}]$-module $\widetilde{\mathbf{CM}}_{*}(-\pi^*h_0;\kappa[\Gamma_{\xi}])$ and for the graded $\Lambda_{\uparrow}$-module $\mathbf{CN}_{*}(\tilde{h}_N;\xi)[n]$.  It is easy to check that both the compositions $\epsilon\circ f_2$ and $f_3\circ \epsilon$ send the elements of this basis to their corresponding dual basis elements in ${}^{\vee}\!(\mathbf{CN}_{*}(\tilde{h}_N;\xi)[n])$ (which, as a graded vector space, coincides with $({}^{\vee}\!\mathbf{CN}_{*}(\tilde{h}_N;\xi))[-n]$).

Finally, we consider the top rectangle in (\ref{8term}).  Now $\Phi_{\tilde{g}_N,-\tilde{h}_N}$ has homotopy inverse $\Phi_{-\tilde{h}_N,\tilde{g}_N}$, and  $\Phi_{\tilde{g}_N,-\tilde{f}}\circ \Phi_{-\tilde{h}_N,\tilde{g}_N}$ is homotopic to $\Phi_{-\tilde{h}_N,-\tilde{f}}$.  So the top rectangle in (\ref{8term}) is commutative up to homotopy if and only if \begin{equation}\label{lastsquare} \xymatrix{  \overline{\mathbf{CN}_{*}(-\tilde{f};-\xi)} \ar[r]^{\epsilon} & ({}^{\vee}\!\mathbf{CN}_{*}(\tilde{f};\xi))[-n] \\ \overline{\mathbf{CN}_{*}(-\tilde{h}_N;-\xi)}\ar[u]^{\Phi_{-\tilde{h}_N,-\tilde{f}}}\ar[r]_{\epsilon} & ({}^{\vee}\!\mathbf{CN}_{*}(\tilde{h}_N;\xi))[-n] \ar[u]_{{}^{\vee}\!\Phi_{\tilde{f}\tilde{h}_N}} } \end{equation} is commutative up to homotopy.  Assume that $\Phi_{\tilde{f}\tilde{h}_N}$ is defined using (the lift to $\tilde{X}_{\xi}$ of) a time-dependent vector field $\mathbb{V}=\{v_s\}_{s\in \R}$; then in constructing $\Phi_{-\tilde{h}_N,-\tilde{f}}$ we may use the time-dependent vector field $\bar{\mathbb{V}}=\{-v_{-s}\}_{s\in \R}$.  With this choice we shall show that the above diagram commutes.  Indeed, it suffices to show that for each $\tilde{p}\in \mathrm{Crit}_k(-\tilde{h}_N)$ and each $\tilde{q}\in \mathrm{Crit}_{k}(-\tilde{f})=\mathrm{Crit}_{n-k}(\tilde{f})$ we have \begin{equation}\label{contadj} \left(\epsilon(\Phi_{-\tilde{h}_N,-\tilde{f}}\tilde{p})\right)(\tilde{q})=(\epsilon(\tilde{p}))(\Phi_{\tilde{f}\tilde{h}_N}\tilde{q}).\end{equation}  But these quantities are equal to the counts of points in the (oriented, if $\mathrm{char}(\kappa)\neq 2$) zero-dimensional continuation trajectory spaces $\mathbf{N}^{\bar{\mathbb{V}}}(\tilde{p},\tilde{q})$ and $\mathbf{N}^{\mathbb{V}}(\tilde{q},\tilde{p})$, respectively.   These continuation trajectory spaces are in bijection under the map which sends a path $\gamma$ to its time-reversal $\bar{\gamma}(s)=\gamma(-s)$; in the oriented case this bijection is orientation-preserving by Proposition \ref{wor}. (The sign in Proposition \ref{wor} is $1$ because $\mathrm{ind}_{\tilde{h}_N}(\tilde{p})=\mathrm{ind}_{\tilde{f}}(\tilde{q})$.) This completes the proof that (\ref{lastsquare}) commutes, and hence that (\ref{8term}) and (\ref{topsquare}) commute up to homotopy.
\end{proof}


\subsection{The isomorphism with interlevel persistence}\label{isosect}

As in Section \ref{novsect}, let $\xi$ be a class in the de Rham cohomology $H^1(X;\R)$ of a compact connected smooth manifold $X$.  This section is concerned with the special case in which the image $\Gamma=\Gamma_{\xi}$ of the integration homomorphism $\langle\xi,\cdot\rangle\co \pi_1(X)\to \R$ is discrete, and thus equal to $\lambda_0\Z$ for some $\lambda_0\geq 0$.  If $\lambda_0=0$ then $\xi=0$ and the integration cover $\tilde{X}_{\xi}$ is equal to $X$, so that the functions $\tilde{f}$ as in Section \ref{novsect} are just real-valued Morse functions on $X$ and their Novikov complexes are the usual Morse complexes.  Otherwise $\lambda_0>0$, $\tilde{X}_{\xi}$ is an infinite cyclic cover of $X$,  and a Morse function $\tilde{f}\co \tilde{X}_{\xi}\to \R$ as in Section \ref{novsect} fits into a commutative diagram  \[ \xymatrix{ \tilde{X}_{\xi}\ar[r]^{\tilde{f}} \ar[d] & \R \ar[d] \\ X \ar[r]^{f} & \R/\lambda_0\Z} \] for a circle-valued Morse function $f\co X\to \R/\lambda_0\Z$.  Note that $\tilde{f}$ is proper, and its critical values form a discrete subset of $\R$ (more specifically, a finite union of cosets of the discrete group $\Gamma_{\xi}$).

If $X$ is $\kappa$-oriented, associated to $\tilde{f}\co \tilde{X}_{\xi}\to \R$ we have a chain-level Poincar\'e--Novikov structure $\mathcal{CN}(\tilde{f};\xi)$ as in Definition \ref{cndef} (which specializes to Definition \ref{cmdef} in case $\xi=0$, as the one-form $\theta_0$ can then be taken to be zero).   By the general algebraic theory in Section \ref{clsec}, this yields a chain-level filtered matched pair $\mathcal{CP}(\mathcal{CN}(\tilde{f};\xi))$ and then, for each $k\in \Z$, a persistence module $\mathbb{H}_k(\mathcal{CP}(\mathcal{CN}(\tilde{f};\xi)))$ of $\kappa$-vector spaces over $\R^2$ given by (\ref{hkdef}). 

Our main comparison result with $\kappa$-coefficient singular homology, Theorem \ref{introiso}, asserts that the restriction of the persistence module $\mathbb{H}_k(\mathcal{CP}(\mathcal{CN}(\tilde{f};\xi)))$ to the index set $\{s+t\geq 0\}$ coincides with the persistence module given by singular homologies of interlevel sets (under the correspondence associating a pair $(s,t)$ with $-s\leq t$ to the interval $[-s,t]$).

In view of Propositions \ref{persinv} and \ref{downup}, the persistence module $\mathbb{H}_k(\mathcal{CP}(\mathcal{CN}(\tilde{f};\xi)))$ is isomorphic to the one given at $(s,t)\in \R^2$ by the $(k+1)$th homology of the mapping cone \begin{equation}\label{proofcone} \mathrm{Cone}\left(\xymatrix{ \mathbf{CN}(\tilde{f};\xi)^{\downarrow}_{\geq -s}\oplus \mathbf{CN}(\tilde{f};\xi)_{\uparrow}^{\leq t} \ar[rrr]^{-j^{\downarrow}\otimes\psi^{\downarrow}+j_{\uparrow}\otimes\psi_{\uparrow}} & & & \Lambda_{\updownarrow}\otimes_{\kappa[\Gamma_{\xi}]}\widetilde{\mathbf{CM}}_{*}(\pi^*h_0;\kappa[\Gamma_{\xi}]) }\right) \end{equation} where \[ \mathbf{CN}(\tilde{f};\xi)^{\downarrow}_{\geq -s}=\{c\in \mathbf{CN}(\tilde{f};\xi)^{\downarrow}|\ell_{\tilde{f}}^{\downarrow}(c)\geq -s\},\,\,\mathbf{CN}(\tilde{f};\xi)_{\uparrow}^{\leq t} =\{c\in \mathbf{CN}(\tilde{f};\xi)_{\uparrow}|\ell_{\uparrow}^{\tilde{f}}(c)\leq t\} \]
and the rest of the notation is as in Section \ref{pmdefsec}, with persistence module structure maps induced by inclusion of subcomplexes. So to prove Theorem \ref{introiso} it suffices to prove the isomorphism with this persistence module (which we abbreviate $\mathbb{H}_k(\tilde{f})$) in the role of $\mathbb{H}_k(\mathcal{CP}(\mathcal{CN}(\tilde{f};\xi)))$.  To state the result explicitly:  

\begin{theorem}\label{bigiso}
Assume that $\Gamma_{\xi}$ is discrete and let $\tilde{f}\co \tilde{X}_{\xi}\to \R$ be a Morse function such that $d\tilde{f}$ is the pullback to $\tilde{X}_{\xi}$ of a one-form representing the class $\xi\in H^1(X;\R)$. For $(s,t)\in\R^2$, let $\mathbb{H}_k(\tilde{f})_{s,t}$ denote  the $(k+1)$th homology of (\ref{proofcone}). Then, for $s+t\geq 0$, there are isomorphisms \[ \sigma_{s,t}\co \mathbb{H}_k(\tilde{f})_{s,t}\to H_k(\tilde{f}^{-1}([-s,t]);\kappa) \] such that, when $s\leq s'$ and $t\leq t'$, we have a commutative diagram \[ \xymatrix{ \mathbb{H}_k(\tilde{f})_{s,t} \ar[r] \ar[d]_{\sigma_{s,t}} & \mathbb{H}_k(\tilde{f})_{s',t'}  \ar[d]^{\sigma_{s',t'}} \\  H_k(\tilde{f}^{-1}([-s,t]);\kappa) \ar[r] & H_k(\tilde{f}^{-1}([-s',t']);\kappa) }  \]  where the horizontal maps are induced by inclusion.
\end{theorem}

 The rest of this section is devoted to the proof of this theorem, which is in fact slightly more general than Theorem \ref{introiso} since $\mathbb{H}_k(\tilde{f})$ is defined even when $X$ is not $\kappa$-oriented.

We first argue that it suffices to construct the maps $\sigma_{s,t}$ of Theorem \ref{bigiso} in the case that $s+t>0$ and both $-s$ and $t$ are regular values of $\tilde{f}$.  In this direction, note that due to the discreteness of the set of critical values of $\tilde{f}$, for any $(s,t)\in \R^2$ there is $\epsilon_{s,t}>0$ such that there are no critical values in the union of open intervals $(t,t+\epsilon_{s,t})\cup(-(s+\epsilon_{s,t}),-s)$.  Then if $\delta,\epsilon\in [0,\epsilon_{s,t})$,  the inclusions of subcomplexes $\mathbf{CN}(\tilde{f};\xi)^{\downarrow}_{\geq -s}\subset \mathbf{CN}(\tilde{f};\xi)^{\downarrow}_{\geq -(s+\delta)}$ and $\mathbf{CN}(\tilde{f};\xi)_{\uparrow}^{\leq t}\subset \mathbf{CN}(\tilde{f};\xi)_{\uparrow}^{\leq t+\epsilon}$ are equalities.  Thus, for $0\leq \delta,\epsilon<\epsilon_{s,t}$, we have $\mathbb{H}_k(\tilde{f})_{s,t}=\mathbb{H}_k(\tilde{f})_{s+\delta,t+\epsilon}$, and the structure map $\mathbb{H}_k(\tilde{f})_{s,t}\to \mathbb{H}_k(\tilde{f})_{s+\delta,t+\epsilon}$ is the identity.  

Similarly, again assuming that $\delta,\epsilon\in [0,\epsilon_{s,t})$, the fact that the only critical values of $\tilde{f}$ lying in the interval $[-s-\delta,t+\epsilon]$ in fact lie in $[-s,t]$ implies that the gradient flow of $\tilde{f}$ may be used to construct a deformation retraction of $\tilde{f}^{-1}([-s-\delta,t+\epsilon])$ to $\tilde{f}^{-1}([-s,t])$.  (If neither $-s$ nor $t$ is a critical value of $\tilde{f}$ this follows straightforwardly as in \cite[Theorem 3.1]{Mil} applied first to $\tilde{f}$ and then to $-\tilde{f}$; if $-s$ and/or $t$ is a critical value one can instead use the argument in \cite[Remark 3.4]{Mil}.)  Hence if $0\leq \delta,\epsilon<\epsilon_{s,t}$ the inclusion-induced map $H_k(\tilde{f}^{-1}([-s,t]);\kappa)\to H_k(\tilde{f}^{-1}([-(s+\delta),t+\epsilon]);\kappa)$ is an isomorphism.  So supposing the maps $\sigma_{s,t}$ as in Theorem \ref{bigiso} have been constructed for all $(s,t)$ belong to the set $\mathcal{U}=\{(s,t)\in \R^2|s+t>0,s,t\notin\mathrm{Crit}(\tilde{f})\}$, we may uniquely extend the construction to all of $\{(s,t)|s+t\geq 0\}$ by, whenever $s+t\geq 0$, choosing arbitrary $\delta,\epsilon\in [0,\epsilon_{s,t})$ such that $(s+\delta,t+\epsilon)\in \mathcal{U}$ and requiring $\sigma_{s,t}$ to coincide with $\sigma_{s+\delta,t+\epsilon}$ under the equality $\mathbb{H}_k(\tilde{f})_{s,t}=\mathbb{H}_k(\tilde{f})_{s+\delta,t+\epsilon}$ and the inclusion-induced isomorphism $H_k(\tilde{f}^{-1}([-s,t]);\kappa)\cong H_k(\tilde{f}^{-1}([-(s+\delta),t+\epsilon]);\kappa)$; this prescription is clearly independent of the choice of $\delta,\epsilon$ and preserves the commutation of the relevant diagrams.  This justifies restricting attention to those $(s,t)$ that belong to the subset $\mathcal{U}$ in what follows.

\subsubsection{The classical Morse case}\label{easyiso}
Let us first prove Theorem \ref{bigiso} in the simpler case that $\xi=0$.  Thus $\Gamma_{\xi}=0$ and $\kappa[\Gamma_{\xi}],\Lambda_{\uparrow},\Lambda^{\downarrow}$, and $\Lambda_{\updownarrow}$ are all equal to $\kappa$, and the Novikov complexes simplify to Morse complexes of functions on the original compact manifold $X$. (In particular, the underlying chain complex of $\mathbf{CN}(\tilde{f};0)^{\downarrow}$ is the Morse complex $\mathbf{CM}_{*}(-\tilde{f})$, with $\mathbf{CN}(\tilde{f};0)^{\downarrow}_{\geq -s}$ equal to the subcomplex $\mathbf{CM}_{*}(-\tilde{f})^{\leq s}$ generated by critical points $p$ with $-\tilde{f}(p)\leq s$.) Referring to (\ref{proofcone}), the maps $j_{\uparrow},j^{\downarrow}$ are in this case just the identity on $\kappa$; the map $\psi_{\uparrow}\co \mathbf{CM}_{*}(\tilde{f})\to \mathbf{CM}_{*}(h_0)$ is a homotopy inverse to the continuation map $\Phi_{h_0\tilde{f}}$ and hence can be taken equal to the continuation map $\Phi_{\tilde{f}h_0}$; likewise, $\psi^{\downarrow}$ can be taken equal to $\Phi_{-\tilde{f},h_0}$.

For any space $Y$ let $\mathbf{S}_{*}(Y)$ denote the singular chain complex of $Y$ with coefficients in $\kappa$; for $Z\subset Y$ we let $\mathbf{S}_*(Y,Z)=\frac{\mathbf{S}_*(Y)}{\mathbf{S}_*(Z)}$ denote the relative singular chain complex.  If $g\co W\to [a,b]$ is a Morse function on a compact manifold with boundary $W$ such that $\partial W$ is a disjoint union of regular level sets $\partial_0 W=g^{-1}(\{a\})$ and $\partial_1 W=g^{-1}(\{b\})$, one has (suppressing notation for gradient-like vector fields) a Morse complex $\mathbf{CM}_{*}(g)$ and a chain homotopy equivalence $\mathcal{E}(g)\co \mathbf{CM}_*(g)\to \mathbf{S}_*(W,\partial W_0)$ from \cite[p. 218]{Pa}.  In particular, we can set $g$ equal to our Morse function $\tilde{f}\co X\to \R$ to obtain a chain homotopy equivalence $\mathcal{E}(f)\co \mathbf{CM}_*(f)\to \mathbf{S}_*(X)$; also, for any regular value $t$ of $f$, denoting $X^{\leq t}=\tilde{f}^{-1}((-\infty,t])$ and $\tilde{f}^t=\tilde{f}|_{X^{\leq t}}$ we can observe that $\mathbf{CM}_{*}(\tilde{f}^{t})=\mathbf{CM}_{*}(\tilde{f})^{\leq t}$ to obtain a chain homotopy equivalence $\mathcal{E}(\tilde{f}^t)\co \mathbf{CM}_{*}(\tilde{f})^{\leq t}\to \mathbf{S}_*(X^{\leq t})$.  It follows from arguments in \cite{Pa} that one has a homotopy-commutative diagram \begin{equation}\label{pajcomp} \xymatrix{ \mathbf{CM}_{*}(\tilde{f})^{\leq t} \ar[r] \ar[d]_{\mathcal{E}(\tilde{f}^t)} & \mathbf{CM}_{*}(\tilde{f}) \ar[d]^{\mathcal{E}(\tilde{f})} \\ \mathbf{S}_*(X^{\leq t}) \ar[r] & \mathbf{S}_*(X) } \end{equation} where the horizontal maps are inclusions.  
Indeed, as in the proof of \cite[Propositon VI.2.4]{Pa}, the cellular filtration used to construct $\mathcal{E}(\tilde{f}^t)$ can be taken to be contained in that used to construct $\mathcal{E}(f)$; the top arrow above can be interpreted as being induced by this inclusion of filtrations, and then the homotopy commutativity of the diagram follows from the functoriality statement in \cite[Corollary VI.3.8]{Pa}.  Likewise, if $-s$ is a regular value for $\tilde{f}$, we set $X_{\geq -s}=\tilde{f}^{-1}([-s,\infty))$  and $\tilde{f}_{-s}=\tilde{f}|_{X_{\geq -s}}$, and then we have chain homotopy equivalences $\mathcal{E}(-\tilde{f})\co \mathbf{CM}_{*}(-\tilde{f})\to \mathbf{S}_*(X)$ and $\mathcal{E}(-\tilde{f}_{-s})\co \mathbf{CM}_{*}(-\tilde{f})^{\leq s}\to \mathbf{S}_{*}(X_{\geq -s})$, compatible up to homotopy with the inclusions.  In view of the homotopy commutativity of diagrams such as 
(\ref{contcomm}) (applied to $\psi_{\uparrow}=\Phi_{\tilde{f}h_0}$ and $\psi^{\downarrow}=\Phi_{-\tilde{f},h_0}$ defined on the full Morse complexes $\mathbf{CM}_{*}(\pm \tilde{f})$) it then follows that we have a homotopy-commutative diagram 
\begin{equation}\label{tosing} \xymatrix{ \mathbf{CM}_{*}(-\tilde{f})^{\leq s}\oplus \mathbf{CM}_{*}(\tilde{f})^{\leq t} \ar[d]_{\mathcal{E}(-\tilde{f}_{-s})\oplus \mathcal{E}(\tilde{f}^{t})} \ar[rrr]^{-j^{\downarrow}\otimes\psi^{\downarrow}+j_{\uparrow}\otimes\psi_{\uparrow}} & & & \mathbf{CM}_{*}(h_0)  \ar[d]^{\mathcal{E}(h_0)} 
\\ \mathbf{S}_*(X_{\geq -s})\oplus \mathbf{S}_{*}(X^{\leq t}) \ar[rrr]^{-j_{-s}+j^{t}} & & & \mathbf{S}_*(X) } \end{equation} where $j_{-s},j^t$ are the maps on singular chains induced by inclusion of $X_{\geq -s},X^{\leq t}$ into $X$.  Since our persistence module $\mathbb{H}_{k}(\tilde{f})$ is given at $(s,t)$ by the $(k+1)$th homology of the cone of the top horizontal map above, it follows from Proposition \ref{filtinv} that $\mathbb{H}_k(\tilde{f})_{s,t}$ can equally well be computed as the $(k+1)$th homology of the cone of the bottom horizontal map.  This correspondence respects the inclusion induced maps associated to pairs $(s,t),(s',t')$ with $s\leq s',t\leq t'$, due to the fact that the Pajitnov chain homotopy equivalences $\mathcal{E}$ are likewise compatible up to homotopy with inclusions $X^{\leq t}\subset X^{\leq t'}$, $X_{\geq -s}\subset X_{\geq -s'}$.  (This follows by the same argument that was used above to yield  homotopy commutative diagrams such as (\ref{pajcomp}).)

So we now consider the cone of the map $-j_{-s}+j^{t}$ at the bottom of (\ref{tosing}).  As noted just before the start of this subsection, for the purposes of the proof of the theorem it suffices to confine attention to pairs $(s,t)$ where $-s$ and $t$ are regular values of $\tilde{f}$ and $s+t>0$.  In this case the interiors $\{x|\tilde{f}(x)<t\}$ and $\{x|\tilde{f}(x)>-s\}$ of $X^{\leq t}$ and $X_{\geq -s}$ cover $X$.  So if we let $\mathbf{S}^{\circ}_{*}(X)$ denote the subcomplex of $\mathbf{S}_*(X)$ generated by $\mathbf{S}_*(X^{\leq t})$ and $\mathbf{S}_{*}(X_{\geq -s})$, the theorem of small chains (\cite[Proposition 2.21]{Ha}) shows that the inclusion $\mathbf{S}_{*}^{\circ}(X)\hookrightarrow \mathbf{S}_*(X)$ is a chain homotopy equivalence. Of course the map $-j_{-s}+j^{t}$ factors through $\mathbf{S}_{*}^{\circ}(X)$, so the inclusion gives a chain homotopy equivalence \begin{equation}\label{coneq} \mathrm{Cone}(-j_{-s}+j^t\co \mathbf{S}_*(X_{\geq -s})\oplus \mathbf{S}_{*}(X^{\leq t})\to \mathbf{S}_{*}^{\circ}(X)) \sim \mathrm{Cone}(-j_{-s}+j^t\co \mathbf{S}_*(X_{\geq -s})\oplus \mathbf{S}_{*}(X^{\leq t})\to \mathbf{S}_{*}(X)).\end{equation}   Moreover, since $X^{\leq t}\cap X_{\geq -s}=\tilde{f}^{-1}([-s,t])$ we have a short exact sequence of chain complexes \begin{equation}\label{mvses} \xymatrix{  0 \ar[r] & \mathbf{S}_*(\tilde{f}^{-1}([-s,t]))\ar[r]^{(i_{-s},i^{t})} &  \mathbf{S}_*(X_{\geq -s})\oplus \mathbf{S}_{*}(X^{\leq t}) \ar[r]^<<<<{-j_{-s}+j^{t}} &  \mathbf{S}_{*}^{\circ}(X)\ar[r] & 0 }\end{equation} where $i_{-s},i^t$ are  inclusions.

Now it quite generally holds that if $\xymatrix{0 \ar[r]& A_* \ar[r]^{\alpha} & B_* \ar[r]^{\beta} & C_*\ar[r] & 0}$ is a short exact sequence of chain complexes then the map $\hat{\alpha}\co A_k\to B_k\oplus C_{k+1}=(\mathrm{Cone}(\beta)[1])_k$ defined by $\hat{\alpha}(a)=(\alpha a,0)$ is a quasi-isomorphism from $A_*$ to $\mathrm{Cone}(\beta)[1]$.  Applying  this (together with (\ref{coneq})) to (\ref{mvses}) gives an isomorphism \[ \mathfrak{p}_{s,t}\co H_k(\tilde{f}^{-1}([-s,t]);\kappa)\to H_{k+1}\left(\mathrm{Cone}\left(\xymatrix{ \mathbf{S}_*(X_{\geq -s})\oplus \mathbf{S}_{*}(X^{\leq t})\ar[r]^<<<<<{-j_{-s}+j^t} & \mathbf{S}_*(X)} \right)\right). \]  As $s,t$ vary, the $\mathfrak{p}_{s,t}$ are clearly compatible with the obvious inclusion-induced maps for $s\leq s',t\leq t'$.  Composing $\mathfrak{p}_{s,t}^{-1}$ with the isomorphism of cones induced by (\ref{tosing}) gives the isomorphisms $\sigma_{s,t}$ promised in Theorem \ref{bigiso}.  This completes the proof of that theorem in the case that $\xi=0$.

\subsubsection{The $S^1$-valued case}
We now turn to the somewhat more complicated case of Theorem \ref{bigiso} in which the image of $\langle\xi,\cdot\rangle\co \pi_1(X)\to \R$ is a nontrivial discrete subgroup $\Gamma_{\xi}=\lambda_0\Z$ of $\R$, where $\lambda_0>0$.  The Novikov field is then \[ \Lambda_{\uparrow}=\left\{\left.\sum_{m=m_0}^{\infty}a_mT^{\lambda_0m}\right|m_0\in \Z, a_m\in \kappa\right\}; \] substituting $x=T^{\lambda_0}$, we can identify $\Lambda_{\uparrow}$ with the Laurent series field $\kappa[x^{-1},x]]$; similarly $\Lambda$ is the Laurent polynomial ring $\kappa[x^{-1},x]$ and $\Lambda^{\downarrow}$ is $\kappa[[x^{-1},x]$  (the field of Laurent series in the variable $x^{-1}$).

For $t\in \R$ and $m\in \N$, the action of $x^{m}=T^{\lambda_0m}$ on the Novikov chain complex $\mathbf{CN}_{*}(\tilde{f};\xi)$ sends the filtered subcomplex $\mathbf{CN}(\tilde{f};\xi)^{\leq t}_{\uparrow}$ isomorphically to $\mathbf{CN}(\tilde{f};\xi)^{\leq (t-\lambda_0m)}_{\uparrow}$.  If $t$ is a regular value of $\tilde{f}$ and $m>0$, the quotient complex $\mathcal{N}_{*}(t,m):=\frac{\mathbf{CN}(\tilde{f};\xi)^{\leq t}_{\uparrow}}{x^{m}\mathbf{CN}(\tilde{f};\xi)^{\leq t}_{\uparrow}}$ is just the same as the (relative) Morse complex of the restriction of $\tilde{f}$ to $\tilde{f}^{-1}([t-\lambda_0m,t])$, and so the construction from \cite[p. 218]{Pa} mentioned in the previous subsection gives a chain homotopy equivalence $\mathcal{N}_{*}(t,m)\to \mathbf{S}_{*}(\tilde{f}^{-1}([t-\lambda_0m,t]),\tilde{f}^{-1}(\{t-\lambda_0m\}))$ and hence, using excision, a chain homotopy equivalence $\mathcal{E}(\tilde{f},t,m)\co\mathcal{N}_{*}(t,m)\to \mathbf{S}_*(X^{\leq t},X^{\leq(t-\lambda_0m)})$.  Here as in the previous subsection we in general write $X^{\leq t}=\tilde{f}^{-1}((-\infty,t])$ and (later) $X_{\geq -s}=\tilde{f}^{-1}([-s,\infty))$.

Now there is a straightforward identification of the filtered subcomplex $\mathbf{CN}(\tilde{f};\xi)^{\leq t}_{\uparrow}$ with the inverse limit \[ \varprojlim_{m\to\infty}\mathcal{N}_{*}(t,m)=\varprojlim_{m\to\infty}\frac{\mathbf{CN}(\tilde{f};\xi)^{\leq t}_{\uparrow}}{x^{m}\mathbf{CN}(\tilde{f};\xi)^{\leq t}_{\uparrow}} \] While $\mathbf{CN}(\tilde{f};\xi)^{\leq t}_{\uparrow}$ is not invariant under the action of the full Novikov field $\Lambda_{\uparrow}=\kappa[x^{-1},x]]$, it is preserved by the subring $\kappa[[x]]$ (\emph{i.e.} by the elements having $\nu_{\uparrow}\geq 0$). Likewise, since the action of $x$ on $\tilde{X}_{\xi}$ decreases the value of $\tilde{f}$, the singular chain complex $\mathbf{S}_{*}(X^{\leq t})$ is a complex of $\kappa[x]$-modules.  The discussion leading up to \cite[Definition XI.4.5]{Pa} explains how to refine the construction of the map $\mathcal{E}(\tilde{f},t,m)$ to make it a chain homotopy equivalence of complexes of $\frac{\kappa[x]}{x^m\kappa[x]}$-modules (not just $\kappa$-modules), canonical up to homotopy.  We have $\mathbf{S}_*(X^{\leq t},X^{\leq(t-\lambda_0m)})=\frac{\mathbf{S}_{*}(X^{\leq t})}{x^m\mathbf{S}_{*}(X^{\leq t})}$, and then $\kappa[[x]]\otimes_{\kappa[x]}\mathbf{S}_{*}(X^{\leq t})$ can be identified with $\varprojlim_{m\to\infty}\mathbf{S}_*(X^{\leq t},X^{\leq(t-\lambda_0m)})$.

In \cite[Proposition XI.6.3]{Pa}, Pajitnov constructs a chain homotopy equivalence of complexes of $\kappa[[x]]$-modules $\Psi^{t}_{\tilde{f}}\co \mathbf{CN}(\tilde{f};\xi)^{\leq t}\to \kappa[[x]]\otimes_{\kappa[x]}\mathbf{S}_{*}(X^{\leq t})$
with the property that each diagram \[ \xymatrix{ \mathbf{CN}(\tilde{f};\xi)^{\leq t}_{\uparrow}\ar[r]^<<<<<{\Psi^{t}_{\tilde{f}}} \ar[d] & \kappa[[x]]\otimes_{\kappa[x]} \mathbf{S}_{*}(X^{\leq t}) \ar[d] \\ \mathcal{N}_{*}(t,m)\ar[r]^<<<<<{\mathcal{E}(\tilde{f},t,m)} & \mathbf{S}_*(X^{\leq t},X^{\leq(t-\lambda_0m)}) } \] commutes up to homotopy, where the vertical maps are the projections from the respective inverse limits.  Moreover the full Novikov complex $\mathbf{CN}_{*}(\tilde{f};\xi)$ can be recovered from $\mathbf{CN}(\tilde{f};\xi)^{\leq t}_{\uparrow}$ as $\kappa[x^{-1},x]]\otimes_{\kappa[[x]]}\mathbf{CN}(\tilde{f};\xi)^{\leq t}_{\uparrow}$, and \cite[Theorem XI.6.2]{Pa} gives a chain homotopy equivalence $\mathbb{P}_{\tilde{f}}\co \mathbf{CN}_{*}(\tilde{f};\xi)\to \kappa[x^{-1},x]]\otimes_{\kappa[[x]]}\mathbf{S}_{*}(\tilde{X}_{\xi})$ characterized (in view of the construction in \cite[Section XI.6.1]{Pa}) uniquely up to homotopy by the property that, for each regular value $t$, $\mathbb{P}_{\tilde{f}}$ is chain homotopic to the coefficient extension of $\Psi^{t}_{\tilde{f}}$ from $\kappa[[x]]$ to $\kappa[x^{-1},x]]$. (Note that in the notation of the rest of the paper the codomain of $\mathbb{P}_{\tilde{f}}$ would be written as $\Lambda_{\uparrow}\otimes_{\Lambda}\mathbf{S}_{*}(\tilde{X}_{\xi})$.) With these chain homotopy equivalences in hand, the proof of Theorem \ref{bigiso} in the present case follows a strategy similar to the one in the case considered in Section \ref{easyiso}. 

Part of the proof will require identifying the map $-j^{\downarrow}\otimes\psi^{\downarrow}+j_{\uparrow}\otimes\psi_{\uparrow}$ in (\ref{proofcone}), up to homotopy equivalence, with a map between singular chain complexes.  Let $h_0,\tilde{h}_N,\tilde{g}_N$ be as in Section \ref{assocpair}, with $v$ a Morse-Smale gradient-like vector field on $X$ whose lift $\tilde{v}$ to $\tilde{X}_{\xi}=\tilde{X}_{-\xi}$ is gradient-like for both $\tilde{g}_N$ and $\tilde{h}_N$.  Thus, using $\tilde{v}$ to define the various Morse or Novikov complexes of $\pi^*h_0,\tilde{h}_N,\tilde{g}_N$, we have $\mathbf{CN}_{*}(\tilde{h}_N;\xi)=\Lambda_{\uparrow}\otimes_{\Lambda}\widetilde{\mathbf{CM}}_{*}(\pi^*h_0;\kappa[\Gamma_{\xi}])$ and $\mathbf{CN}_{*}(\tilde{g}_N;-\xi)=\Lambda_{\uparrow}\otimes_{\Lambda}\widetilde{\mathbf{CM}}_*(\pi^*h_0;\kappa[\Gamma_{-\xi}])$.  We also recall from \cite[Theorem A.5]{Pa95} the chain homotopy equivalence of complexes of $\kappa[\Gamma_{\xi}]$-modules $\tilde{\mathcal{E}}(h_0)\co\widetilde{\mathbf{CM}}_{*}(\pi^*h_0;\kappa[\Gamma_{\xi}])\to \mathbf{S}_{*}(\tilde{X}_{\xi})$.  Applying the conjugation functor of Section \ref{conjsect} (which does not change the underlying set-theoretic map), $\tilde{\mathcal{E}}(h_0)$ can equally well be regarded as a chain homotopy equivalence $\widetilde{\mathbf{CM}}_{*}(\pi^*h_0;\kappa[\Gamma_{-\xi}])\to \mathbf{S}_{*}(\tilde{X}_{-\xi})$ of complexes of $\kappa[\Gamma_{-\xi}]$-modules.  (Here as before, while $\tilde{X}_{-\xi}$ refers to the same covering space as $\tilde{X}_{\xi}$, and $\Gamma_{-\xi}$ is the same group as $\Gamma_{\xi}$ with $\kappa[\Gamma_{\xi}]=\kappa[\Gamma_{-\xi}]=\Lambda$, we use different notation for them to reflect that the action of $\Gamma_{-\xi}$ on $\tilde{X}_{-\xi}$ is inverse to that of $\Gamma_{\xi}$ on $\tilde{X}_{\xi}$.)

\begin{lemma}\label{lmtonov}
The chain homotopy equivalence \[ \mathbb{P}_{\tilde{h}_N}\co \mathbf{CN}_{*}(\tilde{h}_N;\xi)=\Lambda_{\uparrow}\otimes_{\Lambda}\widetilde{\mathbf{CM}}(\pi^*h_0;\kappa[\Gamma_{\xi}])\to\Lambda_{\uparrow}\otimes_{\Lambda}\mathbf{S}_{*}(\tilde{X}_{\xi}) \] is  homotopic to the coefficient extension to $\Lambda_{\uparrow}$ of $\tilde{\mathcal{E}}(h_0)\co\widetilde{\mathbf{CM}}_{*}(\pi^*h_0;\kappa[\Gamma_{\xi}])\to \mathbf{S}_{*}(\tilde{X}_{\xi})$.  Similarly, $\mathbb{P}_{\tilde{g}_N}\co \mathbf{CN}_{*}(\tilde{g}_N;-\xi)\to \Lambda_{\uparrow}\otimes_{\Lambda}\mathbf{S}_{*}(\tilde{X}_{-\xi})$ is  homotopic to the coefficient extension to $\Lambda_{\uparrow}$ of the  homotopy equivalence $\tilde{\mathcal{E}}(h_0)\co\widetilde{\mathbf{CM}}_{*}(\pi^*h_0;\kappa[\Gamma_{-\xi}])\to \mathbf{S}_{*}(\tilde{X}_{-\xi})$.
\end{lemma}

\begin{proof}
	The two cases are identical so we just prove the first statement, involving $h_0$ and $\tilde{h}_N$.  This is perhaps most easily done by using the characterization of a homotopy inverse to $\mathbb{P}_{\tilde{h}_N}$ in \cite[Section XI.7.2]{Pa}, based on \cite{HL}.  Namely, letting $\mathbf{S}_{*}^{\pitchfork}(\tilde{X}_{\xi})$ denote the subcomplex of $\mathbf{S}_{*}(\tilde{X}_{\xi})$ generated by $C^{\infty}$ simplices $\sigma\co \Delta^k\to \tilde{X}_{\xi}$ that are transverse to all ascending manifolds $\mathbf{A}^{\tilde{v}}(\tilde{p})$, one defines $\phi\co \Lambda_{\uparrow}\otimes_{\Lambda}\mathbf{S}_{k}^{\pitchfork}(\tilde{X}_{\xi})\to \mathbf{CN}_{k}(\tilde{h}_N;\xi)$ by setting $\phi(\sigma)=\sum_{\tilde{p}\in\mathrm{Crit}_{k}(\tilde{h}_N)}[\sigma:\tilde{p}]\tilde{p}$ where $[\sigma:\tilde{p}]$ is the intersection number of $\sigma$ with $\mathbf{A}^{\tilde{v}}(\tilde{p})$. This yields a chain map, and \cite[Theorem XI.7.5]{Pa} asserts that $\mathbb{P}_{\tilde{h}_N}\circ \phi$ is homotopic to the inclusion of the coefficient extension of   $\mathbf{S}_{*}^{\pitchfork}(\tilde{X}_{\xi})$ into $\mathbf{S}_{*}(\tilde{X}_{\xi})$, which is itself a chain homotopy equivalence.  

But in the present context where $\tilde{v}$ is the lift to $\tilde{X}_{\xi}$ of a Morse-Smale gradient-like vector field for $h_0$, intersections of $\sigma$ as above with ascending manifolds $\mathbf{A}^{\tilde{v}}(\tilde{p})$ in $\tilde{X}_{\xi}$ are in bijection with intersections of $\pi\circ\sigma$ with ascending manifolds $\mathbf{A}^{v}(p)$ in $X$, and there are only finitely many such intersections by compactness considerations.  Thus the sums defining $\phi$ are finite, and so $\phi$ is the coefficient extension
to $\Lambda_{\uparrow}$ of a similarly defined map $\phi_0\co \mathbf{S}_{*}^{\pitchfork}(\tilde{X}_{\xi})\to \widetilde{\mathbf{CM}}(\pi^*h_0;\kappa[\Gamma_{\xi}])$.  Then $\tilde{\mathcal{E}}(h_0)\circ \phi_0$ is homotopic to the inclusion of $\mathbf{S}_{*}^{\pitchfork}(\tilde{X}_{\xi})$ into $\mathbf{S}_{*}(\tilde{X}_{\xi})$ by a straightforward lifting of \cite[Theorem VI.4.12]{Pa} to covering spaces.

It follows that $1_{\Lambda_{\uparrow}}\otimes\tilde{\mathcal{E}}(h_0)$ and $\mathbb{P}_{\tilde{h}_N}$ are chain homotopy inverses to the same chain map (namely the coefficient extension to $\Lambda_{\uparrow}$ of the composition of $\phi_0$ with a chain homotopy inverse to the inclusion $\mathbf{S}_{*}^{\pitchfork}(\tilde{X}_{\xi})\hookrightarrow\mathbf{S}_{*}(\tilde{X}_{\xi})$); hence  $1_{\Lambda_{\uparrow}}\otimes\tilde{\mathcal{E}}(h_0)$ and $\mathbb{P}_{\tilde{h}_N}$ are  homotopic.  
\end{proof}

For the rest of the proof let $-s,t$ be regular values of $\tilde{f}$ with $s+t>0$.  In (\ref{proofcone}), the map $\psi_{\uparrow}$ is the composition of the inclusion $\mathbf{CN}(\tilde{f};\xi)^{\leq t}_{\uparrow}\hookrightarrow \mathbf{CN}_{*}(\tilde{f};\xi)$ with a chain homotopy inverse to the continuation map $\Phi_{\tilde{h}_N\tilde{f}}\co \Lambda_{\uparrow}\otimes_{\Lambda}\widetilde{\mathbf{CM}}_{*}(\pi^*h_0;\kappa[\Gamma_{\xi}])\to \mathbf{CN}_{*}(\tilde{f};\xi)$; for this chain homotopy inverse we may use the continuation map $\Phi_{\tilde{f}\tilde{h}_N}$.  By the same argument as in (\ref{contcomm}), $\mathbb{P}_{\tilde{h}_N}\circ\Phi_{\tilde{f}\tilde{h}_N}$ is chain homotopic to $\mathbb{P}_{\tilde{f}}$, so by Lemma \ref{lmtonov} we have a homotopy-commutative diagram of complexes of $\kappa[[x]]$-modules \[ \xymatrix{ \mathbf{CN}(\tilde{f};\xi)^{\leq t}_{\uparrow} \ar[rr]^{\psi_{\uparrow}} \ar[dd]_{\Psi^{t}_{\tilde{f}}} \ar@{^{(}->}[rd] & & \Lambda_{\uparrow}\otimes_{\Lambda}\widetilde{\mathbf{CM}}_{*}(\pi^*h_0;\kappa[\Gamma_{\xi}]) \ar[dd]^{1_{\Lambda_{\uparrow}}\otimes \tilde{\mathcal{E}}(h_0)}\\ & \mathbf{CN}_{*}(\tilde{f};\xi) \ar[ru]^<<<<<{\Phi_{\tilde{f}\tilde{h}_N}} \ar[rd]_{\mathbb{P}_{\tilde{f}}} & \\ \kappa[[x]]\otimes_{\kappa[x]}\mathbf{S}_{*}(X^{\leq t}) \ar@{^{(}->}[rr] & & \Lambda_{\uparrow}\otimes_{\Lambda}\mathbf{S}_{*}(\tilde{X}_{\xi}) } \] where the vertical arrows are chain homotopy equivalences.  Similarly, using the same reasoning with $\xi,\tilde{f},\tilde{h}_N$ replaced by $-\xi,-\tilde{f},\tilde{g}_N$, and using the conjugation functor and Proposition \ref{flipconj} to convert $\Lambda_{\uparrow}$-vector spaces to $\Lambda^{\downarrow}$-vector spaces, we have a homotopy-commutative diagram of complexes of $\kappa[[x^{-1}]]$-modules \[ \xymatrix{  \mathbf{CN}(\tilde{f};\xi)_{\geq -s}^{\downarrow} \ar[r]^<<<<<<<{\psi^{\downarrow}} \ar[d]_{\bar{\Psi}^{-s}_{-\tilde{f}}} & \Lambda^{\downarrow}\otimes_{\Lambda}\widetilde{\mathbf{CM}}_{*}(\pi^*h_0;\kappa[\Gamma_{\xi}]) \ar[d]^{1_{\Lambda^{\downarrow}}\otimes \tilde{\mathcal{E}}(h_0)} \\ \kappa[[x^{-1}]]\otimes_{\kappa[x^{-1}]}\mathbf{S}_{*}(X_{\geq -s}) \ar@{^{(}->}[r] & \Lambda^{\downarrow}\otimes_{\Lambda}\mathbf{S}_{*}(\tilde{X}_{\xi}) }  \]

Taking the direct sum of the two diagrams above, and combining with the map $-j^{\downarrow}+j_{\uparrow}\co \Lambda^{\downarrow}\oplus\Lambda_{\uparrow}\to\Lambda_{\updownarrow}$, we see that we have a homotopy-commutative diagram (at the level of complexes of $\kappa$-vector spaces) \begin{equation}\label{novtosing} \xymatrix{ \mathbf{CN}(\tilde{f};\xi)_{\geq -s}^{\downarrow}\oplus \mathbf{CN}(\tilde{f};\xi)_{\uparrow}^{\leq t}\ar[d]_{\Psi^{t}_{\tilde{f}}\oplus \bar{\Psi}^{-s}_{-\tilde{f}}} \ar[rrr]^{-j^{\downarrow}\otimes\psi^{\downarrow}+j_{\uparrow}\otimes\psi_{\uparrow}}   & & & \Lambda_{\updownarrow}\otimes_{\Lambda}\widetilde{\mathbf{CM}}_{*}(\pi^*h_0;\kappa[\Gamma_{\xi}]) \ar[d]^{1_{\Lambda_{\updownarrow}}\otimes \tilde{\mathcal{E}(h_0)}} \\ \left(\kappa[[x^{-1}]]\otimes_{\kappa[x^{-1}]}\mathbf{S}_{*}(X_{\geq -s})\right)\oplus\left(\kappa[[x]]\otimes_{\kappa[x]}\mathbf{S}_{*}(X^{\leq t})\right) \ar[rrr]^>>>>>>>>>>>>>>>{-j^{\downarrow}\otimes\iota_{X_{\geq -s}}+j_{\uparrow}+\otimes\iota_{X^{\leq t}}} & & & \Lambda_{\updownarrow}\otimes_{\Lambda}\mathbf{S}_{*}(X) } \end{equation}
where $\iota_{X_{\geq -s}},\iota_{X^{\leq t}}$ are the inclusions and the vertical maps are chain homotopy equivalences.  This induces isomorphisms between the homologies of the cones of the horizontal maps; as the pair $(s,t)$ varies these homology isomorphisms are compatible with the inclusion-induced maps for $s\leq s',t\leq t'$ due to \cite[Proposition XI.4.6]{Pa}.  So it remains only to identify $H_{k+1}(\mathrm{Cone}(-j^{\downarrow}\otimes\iota_{X_{\geq -s}}+j_{\uparrow}\otimes\iota_{X^{\leq t}}))$ with $H_k(\tilde{f}^{-1}([-s,t]);\kappa)$, compatibly with inclusions $[-s,t]\subset [-s',t']$, continuing to assume that $s+t>0$. Choose any number $\delta$ with $0<\delta<s+t$ and for any subset $Y$ of $\tilde{X}_{\xi}$ let $\mathbf{S}_{*}^{\delta}(Y)$ denote the subcomplex of $\mathbf{S}_{*}(Y)$ generated by simplices $\sigma\co \Delta^k\to Y$ with $\max(\tilde{f}\circ\sigma)-\min(\tilde{f}\circ \sigma)\leq \delta$.  By \cite[Proposition 2.21]{Ha} applied to the cover $\{\tilde{f}^{-1}([a,a+\delta])|a\in \R\}$ of $\tilde{X}_{\xi}$, the inclusion $\mathbf{S}_{*}^{\delta}(Y)\to \mathbf{S}_{*}(Y)$ is a homotopy equivalence, so we may instead consider the cone of \begin{equation}\label{smallcone} 
-j^{\downarrow}\otimes\iota_{X_{\geq -s}}+j_{\uparrow}\otimes\iota_{X^{\leq t}}\co \left(\kappa[[x^{-1}]]\otimes_{\kappa[x^{-1}]}\mathbf{S}_{*}^{\delta}(X_{\geq -s})\right)\oplus \left(\kappa[[x]]\otimes_{\kappa[x]}\mathbf{S}_{*}^{\delta}(X^{\leq t})\right)\to \Lambda_{\updownarrow}\otimes_{\Lambda}\mathbf{S}^{\delta}_{*}(\tilde{X}_{\xi}).\end{equation} Suppose that $\lambda=\sum_{m\in \Z}a_mx^m$ is a general element of $\Lambda_{\updownarrow}$ (with $a_m\in\kappa$) and  $\sigma\co \Delta^k\to \tilde{X}_{\xi}$ is a generator of $ \mathbf{S}_{*}^{\delta}(\tilde{X}_{\xi})$.  If $m_0\in \Z$ is the minimal integer such that $\max(\tilde{f}\circ \sigma)-m\lambda_0\leq t$, then $x^{m_0}\sigma$ has image contained in $X^{\leq t}$, while
the choice of $\delta$ implies that $x^{m_0-1}\sigma$ has image contained in $X_{\geq -s}$.  It follows readily from this that $\lambda\otimes\sigma$ lies in the image of (\ref{smallcone}); thus this map is surjective, and it is easy to see that its kernel is equal to the image of $\mathbf{S}_{*}^{\delta}(\tilde{f}^{-1}([-s,t]))$ under the diagonal embedding.  So just as in Section \ref{easyiso} we obtain an isomorphism between the $(k+1)$th homology of the cone of (\ref{smallcone}) and $H_k(\mathbf{S}_{*}^{\delta}(\tilde{f}^{-1}([-s,t])))$, which in turn is isomorphic to $H_k(\tilde{f}^{-1}([-s,t]);\kappa)$ by another appeal to the theorem of small chains.  If $s'\geq s$ and $t'\geq t$ the resulting isomorphisms are compatible with the obvious inclusions, since we can use the same value of $\delta$ for both pairs $(s,t),(s',t')$.  This completes the proof of Theorem \ref{bigiso}.

\appendix

\section{Orientations}
\label{orsect} 

In this appendix we describe our conventions for the orientations of various spaces that arise in Morse and Hamiltonian Floer theory and work out some signs that result from these.  Some authors use different conventions than ours (for instance in \cite{Lat}, among other references, Morse trajectory spaces are oriented by means of a choice of orientations of the ascending manifolds, not of the descending manifolds), and in general this can affect the signs that appear in some formulas.  A reader who is willing to work in characteristic $2$ can safely skip this material; in characteristics other than $2$, the calculations here are needed to justify the assertions that (with our conventions) the chain-level Poincar\'e duality map $\tilde{\mathcal{D}}$ appearing in Definitions \ref{cmdef} and \ref{cndef} is a chain map, and  that Proposition \ref{floerpair} and the identities (\ref{pssadj}) and (\ref{contadj}) hold without any sign correction.  

\subsection{Initial conventions}
First we establish some general rules that we will use repeatedly.

\subsubsection{Short exact sequences}\label{ses}
Given a short exact sequence of real vector spaces \begin{equation}\label{sesex} \xymatrix{0 \ar[r]& A\ar[r]^{\alpha} & B \ar[r]^{\beta} & C \ar[r] & 0 }\end{equation} in which two of the three vector spaces $A,B,C$ carry specified orientations,  the remaining vector space may be oriented by requiring that, if $(a_1,\ldots,a_k)$ is a positively-oriented ordered basis for $A$, and if $\hat{c}_1,\ldots,\hat{c}_{\ell}\in B$ have the property that $\left(\beta(\hat{c}_1),\ldots,\beta(\hat{c}_{\ell})\right)$ is a positively-oriented ordered basis for $C$, then $\left(\alpha(a_1),\ldots,\alpha(a_k),\hat{c}_1,\ldots,\hat{c}_{\ell}\right)$ is a positively-oriented ordered basis for $B$. (In the case that $\dim A=0$, so that $\beta$ is a vector space isomorphism  and an orientation of $A$ is by definition a choice of sign, this criterion should be interpreted as meaning that the orientation sign of $A$ is positive if $\beta\co B\to C$ is orientation-preserving and negative otherwise.) If, \emph{e.g.}, $B$ and $C$ already carry orientations, the orientation on $A$ determined by this prescription will be referred to as the one ``induced by the short exact sequence (\ref{sesex}).''  Rephrasing slightly, for any real vector space $V$ let $\det V=\Lambda^{\dim V}V$, so that $\det V$ is one-dimensional and an orientation of $V$ is equivalent to an orientation of $\det V$ (\emph{i.e.}, to a choice of connected component of $\det V\setminus\{0\}$). The short exact sequence (\ref{sesex}) yields an isomorphism $\gamma\co (\det A)\otimes (\det C)\to \det B$ which sends $(a_1\wedge\cdots\wedge a_k)\otimes(c_1\wedge\cdots\wedge c_{\ell})$ to $\alpha(a_1)\wedge\cdots\wedge \alpha(a_k)\wedge \hat{c}_1\wedge\cdots\wedge \hat{c}_{\ell}$ where the $\hat{c}_i$ are arbitrary subject to the condition that $\beta(\hat{c}_i)=c_i$, and our orientation prescription amounts to the condition that $\gamma$ be orientation-preserving.

\subsubsection{Preimages}\label{preim}
Suppose that $A$ and $Z$ are smooth manifolds and $f\co A\to Z$ is a smooth map that is transverse to some submanifold $B\subset Z$.  Assume further that $A$ comes equipped with an orientation, and that $B$ comes equipped with a \emph{co}orientation (\emph{i.e.}, an orientation of its normal bundle $NB$).  The preimage $f^{-1}(B)$ is then given the orientation induced at every point $x\in f^{-1}(B)$ by the short exact sequence \[ \xymatrix{ 0 \ar[r]& T_xf^{-1}(B)\ar[r] & T_xA\ar[r]^{\pi_{NB}\circ f_*} & N_{f(x)}B\ar[r]& 0},\] where $f_*\co T_xA\to T_{f(x)}Z$ is the derivative and $\pi_{NB}$ is the projection from $TZ|_B$ to the normal bundle $NB$. 

\subsubsection{Manifolds with boundary}\label{bdryor} If $A$ is a smooth oriented manifold with boundary, we orient $\partial A$ by the usual ``outer normal first'' convention: for each $x\in \partial A$, if $\mathsf{n}_x\in T_xA$ points away from $\mathrm{int}(A)$, a basis $\{v_1,\ldots,v_{a-1}\}$ for $T_x\partial A$ is positively-oriented iff $\{\mathsf{n}_x,v_1,\ldots,v_{a-1}\}$ is positively-oriented as a basis for $T_xA$.  Of course, if $\dim A=1$ this is interpreted as meaning that $x$ is a positively-oriented point of $\partial A$ if a positive generator for $T_xA$ points away from $\mathrm{int}(A)$ and a negatively-oriented point otherwise.

We likewise use an outer normal first convention for coorientations.  Namely, suppose that $B$ is a cooriented submanifold with boundary of $Z$, that $y\in \partial B$ and that $\mathsf{n}_y\in T_y B$ points away from $\mathrm{int}(B)$.  We then coorient $\partial B$ by the condition that, if $\{w_1,\ldots,w_{c}\}$ is a positively oriented basis for $N_yB$ then $\{\mathsf{n}_y,w_1,\ldots,w_c\}$ is a positively oriented basis for $N_y\partial B$.

To provide an initial consequence of these conventions, suppose that $A$ is a smooth manifold with boundary with a smooth map $f\co A\to Z$, and that $B$ is a smooth cooriented submanifold with boundary of $Z$.  Suppose further that $f^{-1}(\partial B)\cap \partial A=\varnothing$, that $f$ and $f|_{\partial A}$ are both transverse to $B$, and that $f$ is transverse to $\partial B$.   Then $f^{-1}(B)$, $(f|_{\partial A})^{-1}(B)$, and $f^{-1}(\partial B)$ each obtain orientations induced by (\ref{preim}), in the latter two cases using the orientation on $\partial A$ and the coorientation on $\partial B$ that we have just described.  At the same time, $f^{-1}(A)$ is a smooth manifold with boundary, with $\partial f^{-1}(A)$ consisting of the disjoint union of $(f|_{\partial A})^{-1}(B)$  and $f^{-1}(\partial B)$, so the latter two manifolds also obtain orientations as parts of $\partial f^{-1}(A)$.  One may easily check that, as oriented manifolds, \begin{equation}\label{prebdry} \partial f^{-1}(A)=(f|_{\partial A})^{-1}(B)\sqcup (-1)^{\dim f^{-1}(A)-1}f^{-1}(\partial B) \end{equation} where the manifolds on the right hand side are oriented via (\ref{preim}).   In particular, in the case that $\dim f^{-1}(A)=1$, the boundary orientations on both $(f|_{\partial A})^{-1}(B)$ and $f^{-1}(\partial B)$ are the same as their preimage orientations.  

\subsection{Orientations in Morse theory}\label{appmorse}

For critical points $p,q$ of a Morse function $f\co X\to \R$, let us recall how to orient the manifolds $\mathbf{W}^v(p,q)$ and $\mathbf{M}^v(p,q)$ (using the same notation as in Section \ref{morseintro}, and continuing to assume that $v$ is a gradient-like vector field for $f$ and that the flow of $v$ is Morse-Smale so that these are manifolds of the expected dimensions); one can check that our conventions below are equivalent to those in \cite[Section 6]{Qin} and \cite[Chapter 6]{Pa}.  One chooses first of all an orientation $\mathfrak{o}_{v,p}$ for each descending manifold $\mathbf{D}^v(p)$. These orientations yield orientations of the normal bundles $N\mathbf{A}^v(p)$ to the ascending disks $\mathbf{A}^v(p)$, as the fiber $N_p\mathbf{A}^v(p)$ is naturally identified with  $T_p\mathbf{D}^v(p)$, so that $\mathfrak{o}_{v,p}$ orients the fiber $N_p\mathbf{A}^v(p)$ and this orientation extends uniquely to an orientation of the bundle $N\mathbf{A}^v(p)$. For each $x\in\mathbf{W}^v(p,q)$, due to the transversality of $\mathbf{D}^v(p)$ and $\mathbf{A}^v(q)$, we obtain an orientation of the intersection $\mathbf{W}^v(p,q)=\mathbf{D}^v(p)\cap \mathbf{A}^v(q)$ by applying (\ref{preim}) to the inclusion $\iota\co \mathbf{D}^{v}(p)\hookrightarrow X$, as $\mathbf{W}^v(p,q)=\iota^{-1}(\mathbf{A}^v(q))$. Thus the orientation on each $T_x\mathbf{W}^v(p,q)$ is the one induced (in the sense of (\ref{ses})) by the short exact sequence \[ 0\to T_x\mathbf{W}^v(p,q)\to T_x\mathbf{D}^v(p)\to N_x\mathbf{A}^v(q)\to 0. \]  

Having oriented $\mathbf{W}^v(p,q)$, the unparametrized trajectory spaces $\mathbf{M}^v(p,q)=\frac{\mathbf{W}^v(p,q)}{\R}$ for $p\neq q$ are oriented as follows.  If $x\in \mathbf{W}^v(p,q)$ has image under the quotient projection equal to $[x]\in\mathbf{M}^v(p,q)$, the derivative of the quotient projection gives a surjection $T_x\mathbf{W}^v(p,q)\to T_{[x]}\mathbf{M}^v(p,q)$ with kernel spanned by $-v(x)$.  If $(-v(x),e_2,\ldots,e_{d})$ is a positively-oriented basis for $T_x\mathbf{W}^v(p,q)$ then we declare the image under the above surjection of $(e_2,\ldots,e_{d})$ to be a positively-oriented basis for $T_{[x]}\mathbf{M}^v(p,q)$.  (In other words, the orientation on $T_{[x]}\mathbf{M}^{v}(p,q)$ is the one induced by the short exact sequence \[ 0 \to \R\to T_x\mathbf{W}^v(p,q)\to T_{[x]}\mathbf{M}^v(p,q) \] where the first map sends $t$ to $(-tv(x))$ and the second is the derivative of the quotient projection.)

Thus, in the special case that $\mathrm{ind}_{f}(p)=\mathrm{ind}_f(q)+1$, $\mathbf{M}^v(p,q)$ is an oriented $0$-manifold, and for each $x\in\mathbf{W}^v(p,q)$ the sign associated to the equivalence class $[x]\in\mathbf{M}^{v}(p,q)$ is $+1$ if the flow direction $-v$ is positive for the orientation on $T_x\mathbf{W}^v(p,q)$ and $-1$ otherwise.

In general, the descending manifolds $\mathbf{D}^v(p)$ admit compactifications as described \emph{e.g.} in \cite{Qin}, whose codimension-one faces are identified with the various $\mathbf{M}^{v}(p,q)\times\mathbf{D}^v(q)$ for $\mathrm{ind}_f(q)<\mathrm{ind}_f(p)$; the inclusion of $\mathbf{D}^{v}(p)$ into $X$ extends to each such face as the inclusion of $\mathbf{D}^v(q)$ (see \cite[Theorem 3.4]{Qin}).  Similarly, the ascending manifolds $\mathbf{A}^v(q)$ have compactifications $\overline{\mathbf{A}^v(q)}$ with codimension-one strata $\mathbf{A}^v(p)\times \mathbf{M}^v(p,q)$.  

For both $\mathbf{D}^v(p)$ and $\mathbf{A}^v(q)$, let us restrict attention just to those strata corresponding to the case that $\mathrm{ind}_f(p)=\mathrm{ind}_f(q)+1$, as these are the only strata whose images under the extension of the inclusion continue to have codimension one.  Write $\hat{\mathbf{D}}^v(p)$ and $\hat{\mathbf{A}}^v(p)$ for the unions of $\mathbf{D}^v(p)$ (respectively, $\mathbf{A}^v(p)$) with these codimension-one strata; then $\hat{\mathbf{D}}^v(p)$ and $\hat{\mathbf{A}}^v(p)$ are smooth manifolds with boundary.  These codimension-one strata then decompose further as disjoint unions of the various $\{[x]\}\times \mathbf{D}^v(q)$ or $\mathbf{A}^v(p)\times\{[x]\}$ as $[x]$ varies through oriented zero-manifold $\mathbf{M}^{v}(p,q)$.  Given orientations of all of the descending manifolds of $v$, each $\{[x]\}\times\mathbf{D}^v(q)$ then obtains in principle two orientations, one directly from the orientation of $\mathbf{D}^v(q)$ and the other from its status as a boundary component of the partial compactification of $\mathbf{D}^{v}(p)$.  Similarly each $\mathbf{A}^v(p)\times \{[x]\}$ is \emph{co}oriented in two ways, one of which uses the boundary coorientation prescription in (\ref{bdryor}).\footnote{To be slightly more precise, the inclusion $i\co \mathbf{A}^v(q)\to X$ extends smoothly to a map $\hat{i}\co \hat{\mathbf{A}}^v(q)\to M$ which is an immersion, though it will no longer be injective if some $\mathbf{M}^v(p,q)$ has more than one element.  In this case a coorientation of $\hat{\mathbf{A}}^v(q)$ should be understood as an orientation of the bundle $\frac{\hat{i}^*TX}{T\hat{\mathbf{A}}^v(q)}$, and similarly for a coorientation of a stratum $\{[x]\}\times\mathbf{A}^v(p)$.}  For use in Section \ref{pssor}, we record how these orientations and coorientations  are related:

\begin{prop}[\cite{Qin}] \label{boundaryor} 
Given $[x]\in\mathbf{M}^v(p,q)$ with $\mathrm{ind}_f(p)=\mathrm{ind}_f(q)+1$, let $\ep([x])\in\{-1,1\}$ denote	
the orientation sign of $[x]$.  Then: \begin{itemize} 
	\item The orientation of $\{[x]\}\times \mathbf{D}^v(q)$ as a boundary component of $\hat{\mathbf{D}}^{v}(p)$ is $\ep([x])$ times the orientation of $\mathbf{D}^v(q)$. 
	\item  The coorientation of $\mathbf{A}^v(p)\times\{[x]\}$ as a boundary component of $\hat{\mathbf{A}}^v(q)$ is $-\ep([x])$ times the coorientation of $\mathbf{A}^v(p)$. 
	\end{itemize}
\end{prop}
\begin{proof}
	The statement about $\mathbf{D}^v(p)$ is explicitly contained in \cite[Theorem 3.6(2)]{Qin}. There is not an assertion in \cite{Qin} that immediately corresponds to our statement about $\mathbf{A}^v(q)$, but we can infer the desired result from \cite[Theorem 3.6(3)]{Qin} which concerns the boundaries of the compactifications of the spaces $\mathbf{W}^{v}(y,z)$.  Consider $a\in \mathrm{Crit}(f)$ having $\mathrm{ind}_f(a)=n=\dim X$, so that $\mathbf{D}^v(a)$ is an open subset of $X$; given our Morse-Smale assumptions the union of the $\mathbf{W}^v(a,p)$ over such $a$ is open and dense in $\mathbf{A}^v(p)$, and similarly for the union of the $\mathbf{W}^v(a,q)$.  So it suffices to show that, for any such $a$, if $y\in \mathbf{W}^v(a,p)\subset \mathbf{A}^v(p)$, then the boundary coorientation of $\mathbf{A}^v(p)\times\{[x]\}\subset \partial \hat{\mathbf{A}}^v(q)$ at $(y,[x])$ is $-\ep([x])$ times the coorientation of $\mathbf{A}^v(p)$.
	  
	The open subsets $\mathbf{W}^v(a,p)\subset \mathbf{A}^v(p)$ and $\mathbf{W}^v(a,q)\subset\mathbf{A}^v(q)$ are oriented using the originally-chosen orientation of $\mathbf{D}^v(a)$, in such a way that, if $\{w_1,\ldots,w_k\}$ is a lift to $T_y \mathbf{D}^v(a)=T_yX$ of a positively-oriented basis for $N_y\mathbf{A}^v(p)$, and if $\{v_1,\ldots,v_{n-k}\}$ is a positively-oriented basis for 	$\mathbf{W}^v(a,p)$, then $\{v_1,\ldots,v_{n-k},w_1,\ldots,w_k\}$ is a positively-oriented basis for $T_y\mathbf{D}^v(a)$, and similarly for $\mathbf{W}^v(a,q)$.  Now \cite[Theorem 3.6(3)]{Qin} implies that the boundary orientation of $\mathbf{W}^v(a,p)\times\{x\}\subset\partial \overline{\mathbf{W}^v(a,q)}$ is $(-1)^{n-k+1}\ep([x])$ times the orientation of $\mathbf{W}^f(a,p)$ that was described in the previous sentence.  So if $y\in\mathbf{W}^v(p)$ and  $\mathsf{n}_y\in T_{(y,[x])} \overline{\mathbf{W}^v(a,q)}$ points in the outer normal direction from the open subset $\mathbf{W}^v(a,p)\times \{[x]\}$ of $\mathbf{A}^v(p)\times\{[x]\}$, and if $\{w'_1,\ldots,w'_{k-1}\}$  is a lift to $T_y\mathbf{D}^v(a)$ of a positively oriented basis for $N_{(y,[x])}\hat{\mathbf{A}}^f(q)$, then $\{\mathsf{n}_y,v_1,\ldots,v_{n-k},w'_1,\ldots,w'_{k-1}\}$ has sign $(-1)^{n-k+1}\ep([x])$ with respect to the orientation on $T_y\mathbf{D}^v(a)$.  (Here we trivially identify tangent vectors to $X$ with tangent vectors to $X\times\{[x]\}$.)  Therefore $\{v_1,\ldots,v_{n-k},\mathsf{n}_y,w'_1,\ldots,w'_{k-1}\}$ has sign $-\ep([x])$ with respect to the orientation on $T_y\mathbf{D}^v(a)$.  Hence the original coorientation of $\mathbf{A}^v(p)$ 
	is given by $(-\ep([x]))$ times the coorientation of $\mathbf{A}^v(p)\times\{[x]\}$ that is induced by the outer normal first convention.
\end{proof}

The spaces $\mathbf{N}^{\mathbb{V}}(p_-,p_+)$ underlying the continuation maps $\Phi_{f_-f_+}\co \mathbf{CM}_*(f_-,v_-;\Z)\to\mathbf{CM}_*(f_+,v_+;\Z)$ are oriented in a similar way to the $\mathbf{W}^{v}(p,q)$.  We assume as in Section \ref{morseintro} that $\mathbb{V}=\{v_s\}_{s\in\R}$ is a smooth family of vector fields such that $v_s=v_{\pm}$ for $\pm s\geq 1$ where $v_+$ and $v_-$ are gradient-like vector fields for the Morse functions $f_+$ and $f_-$, respectively.  We also choose orientations of the descending manifolds $\mathbf{D}^{v_-}(p_-),\mathbf{D}^{v_+}(p_+)$ for $p_-\in\mathrm{Crit}(f_-)$ and $p_+\in\mathrm{Crit}(f_+)$.  Thus with $\Psi^{\mathbb{V}}$ the diffeomorphism of $X$ that sends the value at time $-1$ of a general integral curve of the time-dependent vector field $\mathbb{V}$ to the curve's value at time $1$, 
the continuation trajectory space $\mathbf{N}^{\mathbb{V}}(p_-,p_+)$ is identified with the intersection $\mathbf{D}^{v_-}(p_-)\cap (\Psi^{\mathbb{V}})^{-1}(\mathbf{A}^{v_+}(p_+))$, which we assume to be transverse through a generic choice of $\mathbb{V}$.  As before, the orientation of $\mathbf{D}^{v_+}(p_+)$ is equivalent to a coorientation of $\mathbf{A}^{v_+}(p_+)$.  Accordingly, we orient $\mathbf{N}^{\mathbb{V}}(p_-,p_+)$ by regarding it as the preimage under the map $\Psi^{\mathbb{V}}|_{\mathbf{D}^{v_-}(p_-)}\co \mathbf{D}^{v_-}(p_-)\to X$ of the cooriented submanifold $\mathbf{A}^{v_+}(p_+)$. Thus  the orientation on $\mathbf{N}^{\mathbb{V}}(p_-,p_+)$  is induced by the short exact sequence 
\[ 0  \to T_x\mathbf{N}^{\mathbf{V}}(p_-,p_+)\to T_x\mathbf{D}^{v_-}(p_-)\to N_{\Psi^{\mathbb{V}}(x)}\mathbf{A}^{v_+}(p_+)\to 0 \] where the the second nontrivial map is the composition of the derivative of $\Psi^{\mathbb{V}}$ with the quotient projection.


Our earlier trajectory spaces $\mathbf{W}^{v}(p,q)$ can be regarded as a special case of the continuation trajectory spaces $\mathbf{N}^{\mathbb{V}}(p,q)$ in which the Morse functions $f_-,f_+$ are both equal to $f$ and $\mathbb{V}=\{v_s\}_{s\in \R}$ has each $v_s$ equal to the same vector field $v$.  In this case, the diffeomorphism $\Psi^{\mathbb{V}}$ is the time-$2$ map of the vector field $-v$; since the flow of $-v$ preserves the submanifolds $\mathbf{A}^{v}(q)$ (and hence also their coorientations) the orientation prescription for $\mathbf{N}^{\mathbb{V}}(p,q)$ from the previous paragraph evidently agrees with our earlier orientation of $\mathbf{W}^v(p,q)$.

We now consider the effect of ``turning $X$ upside down'' by replacing Morse functions and vector fields by their negatives.  For a time-dependent vector field $\mathbb{V}=\{v_s\}_{s\in \R}$ as above let $\bar{\mathbb{V}}=\{-v_{-s}\}_{s\in \R}$; then for $p_-\in\mathrm{Crit}(f_-)$ and $p_+\in\mathrm{Crit}(f_+)$ we obtain  spaces $\mathbf{N}^{\bar{\mathbb{V}}}(p_+,p_-)$ which may be used to define continuation maps $\Phi_{-f_+,-f_-}$; regarding these spaces as consisting of solutions $\gamma\co \R\to X$ to the appropriate version of (\ref{conteq}), $\mathbf{N}^{\bar{\mathbb{V}}}(p_+,p_-)$ is in bijection with $\mathbf{N}^{\mathbb{V}}(p_-,p_+)$ via the map that sends $\gamma\co\R\to X$ to its reversal $\bar{\gamma}(s)=\gamma(-s)$.  Our prescriptions above give orientations of $\mathbf{N}^{\bar{\mathbb{V}}}(p_+,p_-)$ dependent upon choices of orientations $\mathfrak{o}_{-v_-,p_-},\mathfrak{o}_{-v_+,p_+}$ of  $\mathbf{D}^{-v_{\pm}}(p_{\pm})$, (\emph{i.e.}, of the ascending manifolds $\mathbf{A}^{v_{\pm}}(p_{\pm})$). Let us compare the orientations of $\mathbf{N}^{\mathbb{V}}(p_-,p_+)$ and $\mathbf{N}^{\bar{\mathbb{V}}}(p_+,p_-)$, where these spaces are implicitly identified using the bijection $\gamma\mapsto \bar{\gamma}$.  Note that if we use the assignment $\gamma\mapsto\gamma(-1)$ to identify $\mathbf{N}^{\mathbb{V}}(p_-,p_+)$ with $\mathbf{D}^{v_-}(p_-)\cap (\Psi^{\mathbb{V}})^{-1}(\mathbf{A}^{v_+}(p_+))$, and likewise for $\mathbf{N}^{\bar{\mathbb{V}}}(p_+,p_-)$, then the bijection $\gamma\mapsto\bar{\gamma}$ from  $\mathbf{N}^{\mathbb{V}}(p_-,p_+)$ to $\mathbf{N}^{\bar{\mathbb{V}}}(p_+,p_-)$ corresponds to the diffeomorphism 
 \[ \Psi^{\mathbb{V}}\co \mathbf{D}^{v_-}(p_-)\cap (\Psi^{\mathbb{V}})^{-1}(\mathbf{A}^{v_+}(p_+))\to \mathbf{A}^{v_+}(p_+)\cap\Psi^{\mathbb{V}}(\mathbb{D}^{v_-}(p_-)).\]

\begin{prop}\label{wor}
	Assume that the $n$-dimensional smooth manifold $X$ is oriented and that, for each $p\in\mathrm{Crit}(f_{\pm})$, the orientations $\mathfrak{o}_{v_{\pm},p}$ on each $\mathbf{D}^{v_{\pm}}(p)$ and $\mathfrak{o}_{-v_{\pm},p}$ on $\mathbf{A}^{v_{\pm}}(p)$ are chosen in such a way that, as oriented vector spaces, $T_{p}\mathbf{A}^{v_{\pm}}(p)\oplus T_p\mathbf{D}^{v_{\pm}}(p)=T_pX$ for each critical point $p$ of $f_+$ or $f_-$.  Then, for $p_-\in\mathrm{Crit}(f_-)$ and $p_+\in\mathrm{Crit}(f_+)$, the resulting orientations on $\mathbf{N}^{\mathbb{V}}(p_-,p_+)$ and $\mathbf{N}^{\bar{\mathbb{V}}}(p_+,p_-)$ are related by the sign $(-1)^{(n-\mathrm{ind}_{f_-}(p_-))(\mathrm{ind}_{f_-}(p_-)-\mathrm{ind}_{f_+}(p_+))}$.
\end{prop}

\begin{proof}
	As explained before the proposition, for $p$ a critical point either of $f_-$ or of $f_+$,  the choice of $\mathfrak{o}_{v_{\pm},p}$ induces orientations both of the tangent bundle $T\mathbf{D}^{v_{\pm}}(p)$ and of the normal bundle $N\mathbf{A}^{v_{\pm}}(p)$. Likewise, then, $\mathfrak{o}_{-v_{\pm},p}$ induces orientations both of the tangent bundle $T\mathbf{A}^{v_{\pm}}(p)$ and of the normal bundle $N\mathbf{D}^{v_{\pm}}(p)$. From now on we identify normal bundles to submanifolds $S$ as subbundles (as opposed to quotient bundles) of $TX|_S$ using a Riemannian metric.  Our prescription on the relation between $\mathfrak{o}_{v_{\pm},p}$ and $\mathfrak{o}_{-v_{\pm},p}$ yields that $N_{p}\mathbf{D}^{v_{\pm}}(p)\oplus T_{p}\mathbf{D}^{v_{\pm}}(p)=T_{p}X$ as oriented vector spaces, and hence (since $\mathbf{D}^{v_{\pm}}(p)$ is connected) that $N_x\mathbf{D}^{v_{\pm}}(p)\oplus T_x\mathbf{D}^{v_{\pm}}(p)=T_xX$ as oriented vector spaces for all $x\in \mathbf{D}^{v_{\pm}}(p)$ and all critical points $p$ of $f_{\pm}$.  Since $N_{p}\mathbf{A}^{v_{\pm}}(p)=T_{p}\mathbf{D}^{v_{\pm}}(p)$ and $T_{p}\mathbf{A}^{v_{\pm}}(p)=N_{p}\mathbf{D}^{v_{\pm}}(p)$ and since $\mathbf{A}^{v_{\pm}}(p)$ is connected, it likewise follows that $T_x\mathbf{A}^{v_{\pm}}(p)\oplus N_x\mathbf{A}^{v_{\pm}}(p)=T_xX$ as oriented vector spaces for all $x\in \mathbf{A}^{v_{\pm}}(p)$ and all critical points $p$ of $f_{\pm}$. 
	
	
	Let us abbreviate $\Psi=\Psi^{\mathbb{V}}$, and note that then $\Psi^{\bar{\mathbb{V}}}=\Psi^{-1}$.  We then have identifications $\mathbf{N}^{\mathbb{V}}(p_-,p_+)\cong (\Psi|_{\mathbf{D}^{v_-}(p_-)})^{-1}( \mathbf{A}^{v_+}(p_+))$ and $\mathbf{N}^{\bar{\mathbf{V}}}(p_+,p_-)\cong (\Psi^{-1}|_{\mathbf{A}^{v_+}(p_+)})^{-1}( \mathbf{D}^{v_-}(p_-))$ and our task is to see how the preimage orientations on these spaces are related under the diffeomorphism $\Psi$ between them. We write $\Psi_*$ for the derivative of $\Psi$.
	
	Given $x\in \mathbf{N}^{\mathbb{V}}(p_-,p_+)$, consider an ordered basis for $T_xX$ of the form \[ \mathcal{B}_x=\left(u_1,\ldots,u_k,v_1,\ldots,v_{\ell},w_1,\ldots,w_m\right)\] where $(u_1,\ldots,u_k)$ is a positively-oriented basis for $T_x\mathbf{N}^{\mathbb{V}}(p_-,p_+)$, $(u_1,\ldots,u_k,w_1,\ldots,w_{m})$ is a positively-oriented basis for $T_x\mathbf{D}^{v_-}(p_-)$, and $(\Psi_*u_1,\ldots,\Psi_*u_k,\Psi_*v_1,\ldots,\Psi_*v_{\ell})$ is a positively-oriented basis for $T_{\Psi(x)}\mathbf{A}^{v_+}(p_+)$.  Thus $k+m=|p_-|$,  $k+l=n-|p_+|$, and $k+\ell+m=n$, so $\ell=n-|p_-|$ and $k=|p_-|-|p_+|$, where we abbreviate $\mathrm{ind}_{f_{\pm}}(p_{\pm})$ as $|p_{\pm}|$. The definition of the orientation on $T_x\mathbf{N}^{\mathbb{V}}(p_-,p_+)$ implies that $(w_1,\ldots,w_{m})$ maps by $\Psi_*$ to a positively-oriented basis for $N_{\Psi(x)}\mathbf{A}^{v_+}(p_+)$.   So since $T_{\Psi(x)}\mathbf{A}^{v_+}(p_+)\oplus N_{\Psi(x)}\mathbf{A}^{v_+}(p_+)=T_{\Psi(x)}X$ as oriented vector spaces, it follows that $\Psi_*\mathcal{B}_x$ is a positively-oriented basis for $T_{\Psi(x)}X$.
	
	Let $\epsilon$ denote the sign of $(\Psi_*u_1,\ldots,\Psi_*u_k)$ relative to the orientation on $\mathbf{N}^{\bar{\mathbf{V}}}(p_+,p_-)$; the conclusion of the proposition is that $\epsilon=(-1)^{k\ell}$.  Since $(\Psi_*u_1,\ldots,\Psi_*u_k,\Psi_*v_1,\ldots,\Psi_*v_{\ell})$ is a positively-oriented basis for $T_{\Psi(x)}\mathbf{A}^{v_+}(p_+)$, it follows that $\epsilon$ is also the sign of $(v_1,\ldots,v_{\ell})$ relative to the orientation of $N_x\mathbf{D}^{v_-}(p_-)$. Hence since $N_x\mathbf{D}^{v_-}(p_-)\oplus T_x\mathbf{D}^{v_-}(p_-)=T_xX$ as oriented vector spaces and since $(u_1,\ldots,u_k,w_1,\ldots,w_{m})$ is a positively-oriented basis for   $T_x\mathbf{D}^{v_-}(p_-)$, the ordered basis $(v_1,\ldots,v_{\ell},u_1,\ldots,u_k,w_1,\ldots,w_{m})$ has sign $\epsilon$ with respect to the orientation on $T_xX$. Comparing this to our original ordered basis $\mathcal{B}_x$, which we previously showed to be positively-oriented, we see that indeed $\epsilon=(-1)^{k\ell}$.
\end{proof} 

As noted earlier, in the special case that $f_-=f_+=f$ and $v_s=v$ independently of $s$, for $p,q\in\mathrm{Crit}(f)$ the spaces $\mathbf{N}^{\mathbb{V}}(p,q)$ and $\mathbf{N}^{\bar{\mathbb{V}}}(q,p)$ coincide as oriented manifolds with the respective spaces $\mathbf{W}^{v}(p,q),\mathbf{W}^{-v}(q,p)$.  Taking the quotients of these spaces by $\R$-translation to obtain $\mathbf{M}^{v}(p,q),\mathbf{M}^{-v}(q,p)$, we have, consistently with \cite[Lemma VI.1.17]{Pa}:

\begin{cor}\label{mor}
	Assume that $X$ is oriented and that the orientations $\mathfrak{o}_{v,p},\mathfrak{o}_{-v,p}$ are related as in Proposition \ref{wor}.  Let $p,q$ be critical points of $f$ with $\mathrm{ind}_f(p)=n-k,\,\mathrm{ind}_{f}(q)=n-k-1$ (and hence $\mathrm{ind}_{-f}(p)=k,\,\mathrm{ind}_{-f}(q)=k+1$).  Then the signed counts of points $n_{f,v}(p,q),n_{-f,-v}(q,p)$ in $\mathbf{M}^v(p,q)$ and $\mathbf{M}^{-v}(q,p)$ are related by \[ n_{-f,-v}(q,p)=(-1)^{k+1}n_{f,v}(p,q).\]
\end{cor}

\begin{proof}
	Proposition \ref{wor} shows that the (set-theoretically equal) oriented $1$-manifolds $\mathbf{W}^v(p,q)$ and $\mathbf{W}^{-v}(q,p)$ have their orientations related by the sign $(-1)^{k}$.  By our definition of the orientation on the $0$-manifold $\mathbf{M}^{v}(p,q)$, for any $x\in \mathbf{W}^{v}(p,q)$ the projection of $x$ to $\mathbf{M}^{v}(p,q)$ is positively-oriented iff $-v(x)$ agrees with the orientation of $\mathbf{W}^v(p,q)$.  On the other hand the projection of $x$ to $\mathbf{M}^{-v}(q,p)$ is positively-oriented iff $-(-v)(x)=+v(x)$ agrees with the orientation of $\mathbf{W}^{-v}(p,q)$.  So the contributions of (the respective equivalence classes of) $x$ to $n_{f,v}(p,q)$ and to $n_{-f,-v}(q,p)$ are related by $-(-1)^{k}=(-1)^{k+1}$.
\end{proof}

\begin{remark}
	The conclusion of Corollary \ref{mor} is dependent on our convention in Proposition \ref{wor} about the relation of $\mathfrak{o}_{v,p}$ and $\mathfrak{o}_{-v,p}$; if we had instead assumed that $T_p\mathbf{D}^{v}(p)\oplus T_p\mathbf{A}^{v}(p)=T_pX$ as oriented vector spaces then the sign relating $n_{-f,-v}(q,p)$ and $n_{f,v}(p,q)$ would have been $(-1)^{n-k}$.
\end{remark}

\begin{remark} \label{morsepairsign}
	In terms of the bilinear maps $L^{M,\kappa}_{f,k}\co \mathbf{CM}_{n-k}(f;\kappa)\times \mathbf{CM}_{k}(-f;\kappa)\to\kappa$ defined as in the proof of Proposition \ref{floerpair} by $(\sum_ia_ip_i,\sum_ib_jp_j)\mapsto \sum_ia_ib_i$, Corollary \ref{mor} shows that one has, for $a\in\mathbf{CM}_{n-k}(f;\kappa)$ and $y\in \mathbf{CM}_{k+1}(-f;\kappa)$, \[ L_{f,k}^{M,\kappa}(a,\partial^{-f}b)=(-1)^{k+1}L_{f,k+1}^{M,\kappa}(\partial^fa,b).\]	
\end{remark}

\subsection{Orientations for Hamiltonian Floer theory and PSS maps} \label{pssor}

We now turn to the matter of orienting moduli spaces relevant to Hamiltonian Floer homology.  The general technique for doing this goes back to \cite{FHor}, in which it was shown that Floer trajectory spaces can be oriented in a way that is coherent with respect to the gluing maps that are foundational to the theory.   Coherent orientations are not unique; while the chain complexes associated to different choices of coherent orientations for a given Hamiltonian are isomorphic, a statement such as Proposition \ref{floerpair} implicitly requires some relationship between the coherent orientations used for the Floer complex of a Hamiltonian $H$ and that of its time reversal $\hat{H}$, as well as relationships between the orientations used in the constructions of the PSS isomorphism with quantum homology for the two Hamiltonians $H$ and $\hat{H}$.  In much of the literature on the PSS map in Hamiltonian Floer theory, the orientations used are not described specifically enough to identify such relationships, and our goal here is to remedy this. Some related issues are discussed in detail in the context of the (more complicated) Lagrangian Floer-theoretic version of the PSS map in \cite{KM} and especially \cite{Zap}; the latter briefly addresses the Hamiltonian version as well, though not in a way that makes it straightforward to extract the signs that we need in the proof of Proposition \ref{floerpair}.

First we recall some relevant background; we will gloss over many standard analytical subtleties addressed in, \emph{e.g.}, \cite{FHor},\cite{ScT}, and \cite{MS}, and we will frequently assume without explanation that data have been chosen generically so that various moduli spaces are cut out transversely.

\subsubsection{$\delbar$-operators and determinant lines}\label{dets}
Let $(M,\omega)$ be a closed semipositive $2m$-dimensional symplectic manifold and $H\co S^1\times M\to \R$ a generic nondegenerate Hamiltonian.  The orientations needed to define the Floer differential $\partial_H$ on the Hamiltonian Floer complex $CF_*(H)$ in general characteristic arise from consideration of Fredholm operators $D_B\co W^{1,p}(Z,\mathbb{C}^m)\to L^p(Z,\mathbb{C}^m)$ taking the form \begin{equation}\label{gendb} (D_B\xi)(s,t)=\frac{\partial \xi}{\partial s}+J_0\left(\frac{\partial \xi}{\partial t}-B(s,t)\xi(s,t)\right) . \end{equation}  Here $p>2$, $Z=\R\times S^1$ with coordinates $(s,t)$, $J_0$ is the standard complex structure on $\mathbb{C}^m=\R^{2m}$, and $B\co Z\to \mathfrak{gl}_{2m}$ has the property that, for some $s_0>0$ and all $s$ with $\pm s\geq s_0$, we have $B(s,t)=B_{\pm}(t)$ for fixed $B_+(t)$ and $B_-(t)$ which lie in the Lie algebra $\mathfrak{sp}_{2m}$ of the symplectic linear group.  Moreover, the paths $t\mapsto B_{\pm}(t)$ are required to have the property that, for the symplectic linear maps $\Psi_{\pm}(t)$ given by $\Psi_{\pm}(0)=\mathbf{1}_{\R^{2m}}$ and $\frac{d\Psi_{\pm}(t)}{dt}=B_{\pm}(t)\Psi_{\pm}(t)$, the maps $\Psi_{\pm}(1)$ do not have $1$ as an eigenvalue.  One then has Conley--Zehnder indices $\mu_{CZ}(\Psi_{\pm})$, and the operator $D_B$ is Fredholm with index $\mu_{CZ}(\Psi_+)-\mu_{CZ}(\Psi_-)$ (\cite[Theorem 2.2]{Sal}).

The Floer boundary operator is constructed from spaces $\tilde{\mathcal{M}}_{J,H}([\gamma^-,w^-],[\gamma^+,w^+])$ of finite-energy solutions $u\co Z\to M$ to the Floer equation $\frac{\partial u}{\partial s}+J(u(s,t))\left(\frac{\partial u}{\partial t}-X_{H_t}(u,s,t) \right)=0$ such that $u(s,\cdot)\to \gamma^{\pm}$ as $s\to\pm\infty$ and such that the capping disk for $\gamma^+$ formed by gluing $w^-$ to $u$ represents the same element of $\widetilde{\mathcal{L}_0M}$ as $w^+$. (We use the same notation as in Section \ref{floersect}.)  
The linearization of the left-hand side of the Floer equation at $u$ then takes the form (\ref{gendb}) when expressed in terms of a suitable trivialization of $u^*TM$; thus an orientation prescription for $\tilde{\mathcal{M}}_{J,H}([\gamma^-,w^-],[\gamma^+,w^+])$ (assuming that it is cut out transversely) is induced by a continuous prescription for orienting the determinant lines of Fredholm operators such as $D_B$.  Recall here that, in general, the determinant line of a Fredholm operator $D\co E\to F$ between two Banach spaces is the one-dimensional vector space $\mathrm{Det}(D)=\det(\ker(D))\otimes \det(\mathrm{coker}(D))^*$ (where as in Section \ref{ses} $\det$ denotes top exterior power).  As is reviewed in the appendix to \cite{FHor}, if $\alpha\co Y\to \mathrm{Fred}(E,F)$ is a continuous map from a topological space $Y$ to the space of Fredholm operators from $E$ to $F$, the various $\mathrm{Det}(\alpha(y))$ for $y\in Y$ fit together to give a real line bundle $\alpha^*\mathrm{Det}$ over $Y$.  So if $Y$ is simply connected, choosing an orientation of $\mathrm{Det}(\alpha(y_0))$ for a single $y_0\in Y$ canonically induces orientations of every $\mathrm{Det}(\alpha(y))$, as there will be a unique-up-to-homotopy nowhere-vanishing section of $\alpha^*\mathrm{Det}$ that agrees with the chosen orientation of $\mathrm{Det}(\alpha(y_0))$.  In this case we say that the orientation on $\mathrm{Det}(\alpha(y))$ is obtained by continuation from that on $\mathrm{Det}(\alpha(y_0))$.  For example, an orientation of $\mathrm{Det}(D_B)$ for a single choice of $B$ induces by continuation orientations for all other $B$ having the same asymptotic operators $B_{+}$ and $B_-$, because the space of such $D_B$ is a convex set of Fredholm operators.

More generally, as in \cite[Chapter 3]{ScT}, one may replace the cylinder $Z$ by a complex one-manifold $\Sigma$ having (a possibly empty set of) cylindrical ends $C_i$ each modeled on either $(-\infty,0)\times S^1$ or $(0,\infty)\times S^1$, such that $\Sigma\setminus \cup_iC_i$ is compact, and consider a complex-rank-$m$ Hermitian vector bundle $E$ over $\Sigma$ together with an operator $D\co W^{1,p}(\Sigma,E)\to L^p(\Sigma,T^{0,1}\Sigma\otimes_{\mathbb{C}} E)$ which, in suitable local trivializations, is given as a zeroth-order perturbation of the standard $\bar{\partial}$ operator, coinciding on the cylindrical ends with operators $D_{B_i}$ as in (\ref{gendb}). Such operators, following \cite[Definition 3.1.6]{ScT} to which we refer for a more precise formulation, will be called admissible $\delbar$-operators over $\Sigma$. For our purposes it will suffice to restrict to the cases in which $\Sigma$ has either $0$ or $1$ positive cylindrical ends and either $0$ or $1$ negative cylindrical ends, so that, up to biholomorphism, $\Sigma$ is either the cylinder, the standard complex plane $\mathbb{C}$ (which has a positive cylindrical end), the plane $\mathbb{C}^-$ with the negative of the standard complex structure (and hence a negative cylindrical end), or $\mathbb{C}P^1$.  We will identify $\mathbb{C}P^1$ with $\mathbb{C}\cup\{\infty\}$ in the standard way, and $\mathbb{C}^-$ with $(\mathbb{C}\setminus\{0\})\cup\{\infty\})$ (with its standard complex structure).  If $\Sigma=Z$ then $E$ has a global trivialization in which an admissible $\delbar$-operator $D$ is given by an operator $D_B$ as above.  If $\Sigma=\mathbb{C}$ or $\Sigma=\mathbb{C}^-\subset \mathbb{C}\cup\{\infty\}$ then its cylindrical end is parametrized by $(s,t)\mapsto e^{2\pi(s+it)}$ and the associated matrices $B(s,t)$ from the expression of $D$ on the cylindrical end in terms of a trivialization that extends over $\Sigma$ have limiting values $B_+(t)$ (in the case $\Sigma=\mathbb{C}$) or $B_-(t)$ (in the case $\Sigma=\mathbb{C}^-$) generating paths of symplectic matrices $\Psi_+(t)$ or $\Psi_-(t)$ as before.   \cite[Proposition 3.3.8]{ScT} gives the Fredholm index of $D$ in either of these cases as $m\pm \mu_{CZ}(\Psi^{\pm})$.  Finally, in the case that $\Sigma=\mathbb{C}P^1$, the Fredholm index of $D$ is $2(m+\langle c_1(E),[\mathbb{C}P^1]\rangle)$ by the Riemann-Roch theorem.
 
If $D$ is an admissible $\delbar$-operator over a closed complex one-manifold such as $\mathbb{C}P^1$, then $\mathrm{Det}(D)$ has a canonical orientation: in case $D$ is complex-linear this is the orientation induced by the complex orientations on $\ker(D)$ and $\mathrm{coker}(D)$, while more generally the complex-linear part $D^{\mathbb{C}}$ of $D$ will also be an admissible $\delbar$-operator, as will convex combinations of $D^{\mathbb{C}}$ and $D$, so that an orientation of $\mathrm{Det}(D)$ is induced from the complex orientation of $\mathrm{Det}(D^{\mathbb{C}})$ by continuation along a line segment connecting $D$ and $D^{\mathbb{C}}$.  

When $\Sigma$ has nonempty boundary, the determinant lines of admissible $\delbar$-operators do not generally have canonical orientations; our plan will be to choose orientations for $\mathrm{Det}(D)$ for some admissible $\delbar$-operators $D$ (specifically, ones defined over $\mathbb{C}^{-}$) arbitrarily, and then obtain orientations for other $D$ by requiring compatibility under the gluing operation described in \cite[Section 3]{FHor}, \cite[Section 3.2]{ScT}.  In terms of the underlying geometry, this is very similar to the approach in \cite[Section 1.4]{Ab}, though Abouzaid translates this into Floer theory in a more manifestly choice-independent way.

The gluing procedure takes as input admissible $\delbar$-operators $D^j\co W^{1,p}(\Sigma^j,E^j)\to L^p(\Sigma,T^{0,1}\Sigma\otimes_{\mathbb{C}}E^j)$ for $j=1,2$ such that $\Sigma^1$ has a positive cylindrical end $C^1$ and $\Sigma^2$ has a negative end $C^2$, with trivializations of $E^j$ over $C^j$ in terms of which the $D^j$ agree with $D_{B^j}$ as in (\ref{gendb}) where, for some $S_0$ and all $s>S_0$, $B^1(s,\cdot)=B^2(-s,\cdot)=B_0$ for a fixed loop $B_0\co S^1\to \mathfrak{sp}_{2m}$.   For $S>S_0$, one may then form a complex one-manifold $\Sigma^1\#_S\Sigma^2$  from the disjoint union of $\Sigma^1$ and $\Sigma^2$ and a tube $\mathcal{T}=(-2S,2S)\times S^1$ by removing the loci in $C^1$ where $s\geq 2S$ and in $C^2$ where $s\leq -2S$ and gluing $(-2S,-S)\times S^1\subset \mathcal{T}$ to $(S,2S)\times S^1\subset C^1$ and $(S,2S)\times S^1\subset \mathcal{T}$ to $(-2S,-S)\times S^1\subset C^2$ by the maps $(s,t)\mapsto (s\pm 3S,t)$.  The vector bundles $E^1\to \Sigma^1$ and $E^2\to \Sigma^2$ likewise glue in an obvious way to a vector bundle over  $\Sigma^1\#_S\Sigma^2$ which is trivial over $\mathcal{T}$, and our assumptions about the asymptotics of $B^1$ and $B^2$ imply that we have a glued admissible $\delbar$-operator $D^1\#_S D^2$ agreeing with $D^j$ over $\Sigma^j\setminus C^j$ and equal on $\mathcal{T}$ to the version of $D_B$ as in (\ref{gendb}) with $B(s,t)=B_0(t)$.  
The key fact that we require is that gluing elements of stabilized versions of the $D^j$ as in \cite[Corollary 3.2.11]{ScT} gives rise to an isomorphism $\mathrm{Det}(D^1)\otimes \mathrm{Det}(D^2)\cong \mathrm{Det}(D^1\#_S D^2)$, see \cite[p. 25]{FHor}, \cite[Lemma 1.4.5]{Ab}.  Via this isomorphism, orientations of two of $\mathrm{Det}(D^1),\mathrm{Det}(D^2),\mathrm{Det}(D^1\#_S D^2)$ canonically determine an orientation of the third.  Moreover this is natural with respect continuation of orientations along paths, and is associative in the appropriate sense (\cite[Proposition 10]{FHor}).

Now given a generic nondegenerate Hamiltonian $H\co S^1\times M\to \R$ where $(M,\omega)$ is a semipositive symplectic manifold and $S^1=\R/\Z$,  and given a Morse function $f\co M\to \R$ with Morse-Smale gradient-like vector field $v$, we shall describe a prescription for determining signs for the matrix elements associated to the Floer boundary operator $\partial_H\co CF_*(H)\to CF_{*-1}(H)$ and the PSS maps $\tilde{\Phi}_{PSS}^{f,H^+}\co \hat{\Lambda}_{\uparrow}\otimes_{\kappa}\mathbf{CM}_*(f)\to CF_*(H)$ and $\tilde{\Psi}_{PSS}^{H^-,f}\co CF_*(H)\to \hat{\Lambda}_{\uparrow}\otimes_{\kappa}\mathbf{CM}_*(f)$. Let $P(H)$ denote the set of contractible one-periodic orbits $\gamma\co S^1\to M$.  Denote by $\mathbb{D}$ the closed unit disk in $\mathbb{C}$, and set \[ \mathbb{D}^-=\{|z|\in\mathbb{C}||z|\geq 1\}\cup\{\infty\}\subset \mathbb{C}\cup\{\infty\}. \]  For any $\gamma\in P(H)$, a \textbf{left-cap} for $\gamma$ is a smooth map $w\co \mathbb{D}\to M$ such that $w(e^{2\pi it})=\gamma(t)$ for all $t\in S^1$, and a \textbf{right-cap} for $\gamma$ is a smooth
map $y\co \mathbb{D}^-\to M$ such that $y(e^{2\pi it})=\gamma(t)$ for all $t\in S^1$.  Defining $\hat{H}\co S^1\times M\to \R$ by $\hat{H}(t,x)=-H(1-t,x)$, the time-reversal $\bar{\gamma}\co S^1\to M$ (given by $\bar{\gamma}(t)=\gamma(1-t)$) belongs to $P(\hat{H})$.  If $w$ is a left-cap for $\gamma\in P(H)$, we use the following notation:
\begin{itemize} \item $w^R$ is the right-cap for $\gamma$ defined by $w^R(z)=w(1/\bar{z})$ for $z\in \mathbb{D}^-\subset \C\cup\{\infty\}$;
	\item $\bar{w}$ is the left-cap for $\bar{\gamma}$ defined by $\bar{w}(z)=w(\bar{z})$ for $z\in  \mathbb{D}$.
	\item $w^I$ is the right-cap for $\bar{\gamma}$ defined by $w^I(z)=w(1/z)$ for $z\in \mathbb{D}^-$. (So $w^I=\bar{w}^R$.) 
\end{itemize}

Let $\{\phi_{H}^{t}\}_{t\in [0,1]}$ denote the Hamiltonian flow of $H$ and choose for each $\gamma\in P(H)$ a left-cap $w_{\gamma}$, with corresponding right-cap $w_{\gamma}^{R}$.  A symplectic trivialization of $\gamma^{*}TM$ that extends to a trivialization of $w_{\gamma}^{*}TM$ yields a path $t\mapsto \Psi_{\gamma}(t)$ in the symplectic linear group, by taking $\Psi_{\gamma}(t)$ to be the expression of $\phi_{H*}^{t}\co T_{\gamma(0)}M\to T_{\gamma(t)}M$ in terms of the trivialization.  Let $B_{\gamma}\co S^1\to \mathfrak{sp}_{2m}$ be the loop given by $\frac{d\Psi_{\gamma}}{dt}=B_{\gamma}(t)\Psi_{\gamma}(t)$.  Now let $B_{\gamma}^{-}\co \mathbb{C}^{-}\to \mathfrak{gl}_{2m}$ be smooth with $B_{\gamma}^{-}(e^{2\pi(s+it)})=B_{\gamma}(t)$ for $s\ll 0$ and $B_{\gamma}^{-}(e^{2\pi(s+it)})=0$ for $s\gg 0$ (and hence $B_{\gamma}^{-}(\infty)=0$).  There is then an admissible $\delbar$-operator $D_{w_{\gamma}}^{-}$ on sections of the trivial bundle over $\mathbb{C}^-$ that agrees with $2\pi$ times the standard $\delbar$-operator near $\infty$ and with $2\pi\left(r\frac{\partial}{\partial r}+J_0\frac{\partial}{\partial\theta}\right)-J_0B_{\gamma}^{-}(re^{i\theta})$ in standard polar coordinates $(r,\theta)$ on $\C^-\setminus\{\infty\}=\C\setminus\{0\}$.  (Thus in terms of the  coordinates $s,t$ parametrizing $\C^{-}\setminus\{\infty\}$ via $(s,t)\mapsto e^{2\pi(s+it)}$, $D_{w_{\gamma}}^{-}$ restricts to the cylindrical end as (\ref{gendb}) with $B=B_{\gamma}^{-}$.)  

As input for the rest of our Floer-theoretic orientations we now choose, arbitrarily, orientations for $\mathrm{Det}(D_{w_{\gamma}}^{-})$ for each $\gamma\in P(H)$.  These orientations play a role analogous to the orientations of descending manifolds that are chosen in the construction of the Morse complex.  As both the set of possible trivializations of $w_{\gamma}^{*}TM$ and the set of possible maps $B_{\gamma}^{-}\co \C^-\to \mathrm{gl}_{2m}$ restricting as above to $B_{\gamma}\co S^1\to \mathfrak{sp}_{2m}$ on small circles around the origin are contractible, our chosen orientation for $\mathrm{Det}(D_{w_{\gamma}}^{-})$ canonically induces orientations of the operators that can be obtained from different choices of the trivialization or of $B_{\gamma}^{-}$ by continuation.  
 In fact, the choice of orientation of $\mathrm{Det}(D_{w_{\gamma}}^{-})$ canonically induces orientations of the determinant lines of the analogous operators $D_{w'_{\gamma}}^{-}$ associated to different left-caps $w'_{\gamma}$ for $\gamma$ because, up to homotopy, $D_{w'_{\gamma}}^{-}$ can be identified with an operator obtained by gluing $D_{w_{\gamma}}^{-}$ to a complex linear $\delbar$-operator $D^{\C}$ over $\mathbb{C}P^1$, and so $\mathrm{Det}(D_{w'_{\gamma}}^{-})$ obtains an orientation by gluing the orientation of $\mathrm{Det}(D_{w_{\gamma}}^{-})$ to the complex orientation of $\mathrm{Det}(D^{\C})$.  (See \cite[Lemma 1.4.10]{Ab}.)  From now on we implicitly use these orientations of the determinant lines on operators that are related by gluing and/or continuation to $D_{w_{\gamma}}^{-}$.

\subsubsection{Floer caps and the PSS maps} \label{caps}
Choose an $\omega$-compatible almost complex structure $J$ on $M$ and a smooth map $H^-\co \C^-\times M\to \R$ such that $H^-(z,x)=0$ for all $x\in M$ whenever $|z|>2$, and $H^-(re^{2\pi it},x)=H(t,x)$ for all $x\in M$ whenever $r<\frac{1}{2}$.  An \textbf{outward Floer cap} for an orbit $\gamma\in \mathcal{P}(H)$  is a finite-energy solution $u\co \C^-\times M\to \R$ to the equation (expressed in polar coordinates $(r,\theta)$ on the open dense subset $\C^-\setminus\{\infty\}$ of $\C^-$) \begin{equation}\label{rfc} 2\pi r\frac{\partial u}{\partial r}+J(u)\left(2\pi \frac{\partial u}{\partial \theta}-X_{H^-(re^{i\theta},\cdot)}(u)\right)=0 \end{equation}
	such that $u(re^{2\pi it})\to \gamma(t)$ as $t\to 0$.  For any such $u$, the map $re^{i\theta}\mapsto u((r-1)e^{i\theta})$ (defined on the interior of $\mathbb{D}^-$) extends continuously to the boundary of $\mathbb{D}^-$ to give a right-cap for $\gamma$; composing this latter map with $z\mapsto \frac{1}{\bar{z}}$ then gives a left-cap for $\gamma$ which we will denote by $u^L$.  With notation as at the start of Section \ref{floersect}, given a left-cap $w\co \mathbb{D}\to M$ for $\gamma$, let \[ \mathcal{M}_{J,H^-}^{\mathrm{out}}([\gamma,w])\] denote the space of outward Floer caps $u$ for $\gamma$ (with auxiliary data $J,H^-$) such that $[\gamma,u^L]=[\gamma,w]\in \widetilde{\mathcal{L}_0M}$. For generic data $(J,H^-)$ the linearization $D_{u,J,H^-}$ of the left-hand side of (\ref{rfc}) at each $u\in  
	\mathcal{M}_{J,H^-}^{\mathrm{out}}([\gamma,w])$ is surjective, and so $\mathcal{M}_{J,H^-}^{\mathrm{out}}([\gamma,w])$ is a smooth manifold of dimension equal to the index of $D_{u,J,H^-}$, which is the quantity denoted $\mu_{H}([\gamma,w])$ at the start of Section \ref{floersect}.  Indeed, for a suitable choice of trivialization  and of the function $B_{\gamma}^{-}$, $D_{u,J,H}$ agrees with an operator $D_{w}^{-}$ as in the preceding paragraph; hence our original choice of orientation for $\mathrm{Det}(D_{w_{\gamma}}^{-})$ induces, by continuation and gluing, orientations of the $\mathrm{Det}(D_{u,J,H})$ and hence of the smooth manifold  $\mathcal{M}_{J,H^-}^{\mathrm{out}}([\gamma,w])$.  Furthermore, we have a smooth map $ev_{\infty,[\gamma,w]}^{J,H^-}\co \mathcal{M}_{J,H^-}^{\mathrm{out}}([\gamma,w])\to M$ defined by $ev_{\infty,[\gamma,w]}^{J,H^-}(u)=u(\infty)$.
	
	Symmetrically, we consider \textbf{inward Floer caps} for orbits $\gamma\in P(H)$ with respect to a compatible almost complex structure $J$ and a smooth function $H^+\co \C\times M\to \R$ such that $H^+(z,\cdot)\equiv 0$ for $|z|<\frac{1}{2}$ and $H^+(re^{2\pi i t},x)=H(t,x)$ for all $r>2$ and $x\in M$.  These are finite-energy solutions $u\co \C\to M$ to the equation that restricts to $\C\setminus\{0\}$ as the same equation (\ref{rfc}) with $H^-$ replaced by $H^+$, and with $u(re^{2\pi it})\to \gamma(t)$ as $r\to\infty$.  (Such solutions restrict to the disk of radius $\frac{1}{2}$ around $0$ as $J$-holomorphic curves, just as outward Floer caps are $J$-holomorphic in a neighborhood of $\infty$.)  After reparametrizing by a diffeomorphism $\mathrm{int}(D)\cong \C$ (say for definiteness $re^{i\theta}\mapsto \frac{e^{i\theta}}{1-r}$), any inward Floer cap extends to $\mathbb{D}$ as a left-cap $u^L$ for $\gamma$; given a left-cap $w\co \mathbb{D}\to M$ for $\gamma$ we let $\mathcal{M}_{J,H^+}^{\mathrm{in}}([\gamma,w])$ denote the space of inward Floer caps for $\gamma$ with $[\gamma,u^L]=[\gamma,w]\in \widetilde{\mathcal{L}_0M}$.
	
	The linearization $D_{u,J,H^+}$ of the nonlinear operator whose zero locus is $\mathcal{M}^{\mathrm{in}}_{J,H^+}([\gamma,w])$ at any inward Floer cap $u$ is an admissible $\delbar$-operator over $\C$ with Fredholm index equal to $2m-\mu_H([\gamma,w])$.  For generic $(J,H^+)$, $\mathcal{M}^{J,H^+}([\gamma,w])$ will be cut out transversely and so will be a smooth manifold of dimension 
	$2m-\mu_H([\gamma,w])$.  To orient this smooth manifold, just as with $\mathcal{M}_{J,H^-}^{\mathrm{out}}([\gamma,w])$, it suffices to orient the determinant line of a single admissible $\delbar$-operator $D_{w_{\gamma}}^{+}$ over $\C$ for a left-cap $w_{\gamma}$ for $\gamma$, constructed analogously to our earlier $D_{w_{\gamma}}^{-}$, as this allows one to orient  the various $\mathrm{Det}(D_{u,J,H^+})$ by continuation and gluing. 
	For definiteness, if $D_{w_{\gamma}}^{-}$ is given on $\C^-\setminus\{\infty\}$ by $2\pi(r\frac{\partial}{\partial r}+J_0\frac{\partial}{\partial\theta})-J_0B_{\gamma}^{-}(z)$ where $B_{\gamma}^{-}\co \C^-\to \mathfrak{gl}_{2m}$ vanishes near $\infty$ and has $B_{\gamma}^{-}(re^{2\pi it})=B_{\gamma}(t)$ for small $r$, we can take $D_{w_{\gamma}}^{+}$ equal on $\C\setminus\{0\}$ to 
	$2\pi(r\frac{\partial}{\partial r}+J_0\frac{\partial}{\partial\theta})-J_0B_{\gamma}^{-}(1/\bar{z})$ (and to $2\pi$ times the standard $\delbar$ operator on a neighborhood of $0$).  
	
	The needed orientation on $\mathrm{Det}(D_{w_{\gamma}}^{+})$ is then induced by requiring that the gluing isomorphism (defined for $S\gg 0$) $\mathrm{Det}(D_{w_{\gamma}}^{+})\otimes\mathrm{Det}(D_{w_{\gamma}}^{-})\cong \mathrm{Det}(D_{w_{\gamma}}^{+}\#_SD_{w_{\gamma}}^{-})$ be orientation-preserving.  Here the orientation of $\mathrm{Det}(D_{w_{\gamma}}^{-})$ is the one that we chose arbitrarily earlier, and the orientation of the determinant line of the admissible $\delbar$-operator $D_{w_{\gamma}}^{+}\#_SD_{w_{\gamma}}^{-}$ over $\mathbb{C}P^1$  is the one induced by continuation to the determinant line of a complex-linear operator.  This suffices to prescribe orientations on the smooth manifolds $\mathcal{M}_{J,H^+}^{\mathrm{in}}([\gamma,w])$ for left-caps $w$ of orbits $\gamma\in P(H)$.  Write $ev^{J,H^+}_{0,[\gamma,w]}\co \mathcal{M}_{J,H^+}^{\mathrm{in}}([\gamma,w])\to M$ for the smooth map defined by $u\mapsto u(0)$.
	
Given $\gamma^-,\gamma^+\in\mathcal{P}(H)$, with left-caps $w^-,w^+$, we orient the space $\tilde{\mathcal{M}}_{J,H}([\gamma^-,w^-],[\gamma^+,w^+])$ of Floer trajectories from $[\gamma^-,w^-]$ to $[\gamma^+,w^+]$ as follows (\emph{cf}. \cite[Theorem 1.5.1]{Ab}).  By our initial choices and continuation and gluing, we have orientations of outward-cap operators $D_{w^-}^{-}$ and $D_{w^+}^{-}$ for $\gamma^-$ and $\gamma^+$.  For $u\in \tilde{\mathcal{M}}_{J,H}([\gamma^-,w^-],[\gamma^+,w^+])$ (so in particular $u(s,\cdot)\to \gamma^{\pm}$ as $s\to\pm\infty$), the linearization $D_u$ of the left-hand side of the Floer equation at $u$ is an admissible $\delbar$-operator over $Z$, and for suitable choices of the $D_{w^{\pm}}^{-}$ we may form the glued operator $D_u\#_S D_{w_+}^{-}$, whose determinant line is also oriented by continuation and gluing from our initial choice of orientation of $\mathrm{Det}(D_{w_-}^{-})$.   So we orient $\mathrm{Det}(D_u)$ by requiring the gluing isomorphism $\mathrm{Det}(D_u)\otimes \mathrm{Det}(D_{w_+}^{-})\cong \mathrm{Det}(D_u\#_S D_{w_+}^{-})$ to be orientation-preserving.  This gives orientations on the smooth manifolds $\tilde{\mathcal{M}}_{J,H}([\gamma^-,w^-],[\gamma^+,w^+])$ (assuming these to be cut out transversely).  For $[\gamma^-,w_-]\neq [\gamma^+,w^+]$ one has a free $\R$-action under which $s_0\in \R$ acts by precomposition with $(s,t)\mapsto (s+s_0,t)$; thus the infinitesimal generator at $u\in \tilde{\mathcal{M}}_{J,H}([\gamma^-,w^-],[\gamma^+,w^+])$ is $\partial_su$ and the quotient $\mathcal{M}_{J,H}([\gamma^-,w^-],[\gamma^+,w^+])=\tilde{\mathcal{M}}_{J,H}([\gamma^-,w^-],[\gamma^+,w^+])/\R$ is oriented by means of the exact sequence \[ 0 \to \R\to T_u\tilde{\mathcal{M}}_{J,H}([\gamma^-,w^-],[\gamma^+,w^+])\to T_{[u]}\mathcal{M}_{J,H}([\gamma^-,w^-],[\gamma^+,w^+])\to 0 \] where the first nontrivial map sends $1$ to $\partial_su$ and the second nontrivial map is the derivative of the quotient projection.  

The Floer boundary operator $\partial_{H}\co CF_*(H)\to CF_{*-1}(H)$ is of course defined by $\partial[\gamma^-,w^-]=\sum \#(\mathcal{M}_{J,H}([\gamma^-,w^-],[\gamma^+,w^+]))[\gamma^+,w^+] $ where the sum is over $[\gamma^+,w^+]$ such that $\mu_H([\gamma^-,w^-])-\mu_H([\gamma^+,w^+])=1$ and ``$\#$'' refers to the signed count of points (according to the orientation we have just specified, and valued in the field $\kappa$ used throughout the paper) of the zero-manifold $\mathcal{M}_{J,H}([\gamma^-,w^-],[\gamma^+,w^+])$.  It follows as in \cite{FHor},\cite{Ab} that our orientations are coherent and hence that $\partial_H\circ\partial_H=0$.

We turn now to the PSS maps $\tilde{\Phi}_{PSS}^{f,H^+}\co\hat{\Lambda}_{\uparrow}\otimes_{\kappa}\mathbf{CM}_*(f)\to CF_*(H)$ and $\tilde{\Psi}_{PSS}^{H^-,f}\co CF_*(H)\to \hat{\Lambda}_{\uparrow}\otimes_{\kappa}\mathbf{CM}_*(f)$.  We assume $\mathbf{CM}_*(f)$ to be constructed, as in Sections \ref{morseintro} and \ref{appmorse}, from a gradient-like vector field $v$ for $f$ with associated ascending and descending manifolds $\mathbf{A}^v(p),\mathbf{D}^v(p)$ for $p\in\mathrm{Crit}(f)$, and using (arbitrarily-chosen but fixed) orientations of the various $\mathbf{D}^v(p)$.  We also assume that our data are chosen suitably generically so as to ensure that the various evaluation maps $ev_{0,[\gamma,w]}^{J,H^+}\co \mathcal{M}^{\mathrm{in}}_{J,H^+}([\gamma,w])\to M$ and $ev_{\infty,[\gamma,w]}^{J,H^-}\co \mathcal{M}_{J,H^-}^{\mathrm{out}}([\gamma,w])\to M$ are transverse to each of the $\mathbf{D}^v(p)$ and $\mathbf{A}^v(p)$.
 
As in Section \ref{appmorse}, our orientations of the $\mathbf{D}^v(p)$ induce coorientations of the $\mathbf{A}^v(p)$ by extending the coorientation induced by the obvious identification $N_p\mathbf{A}^v(p)=T_p\mathbf{D}^v(p)$ over the rest of $\mathbf{A}^v(p)$.  For what follows we also use the coorientation on $\mathbf{D}^v(p)$ induced by the convention in Proposition \ref{wor} and its proof: thus we orient $\mathbf{A}^v(p)$ by requiring that $T_p\mathbf{A}^v(p)\oplus T_p\mathbf{D}^v(p)=T_pM$ as oriented vector spaces and then coorient $\mathbf{D}^v(p)$ using the identification $N_p\mathbf{D}^v(p)=T_p\mathbf{A}^v(p)$.  (More concisely, we have $N\mathbf{D}^v(p)\oplus T\mathbf{D}^v(p)=TM|_{\mathbf{D}^v(p)}$ as oriented bundles; of course here we are using the canonical orientation on $M$ associated to its symplectic structure.)

Now for $p\in\mathrm{Crit}(f)$ and $[\gamma,w]\in\widetilde{\mathcal{L}_0M}$ with $\gamma\in P(H)$ we define \[ \mathcal{M}_{J,H^+}^{\mathrm{in}}(p;[\gamma,w])=(ev_{0,[\gamma,w]}^{J,H^+})^{-1}(\mathbf{D}^v(p))\subset \mathcal{M}^{\mathrm{in}}_{J,H^+}([\gamma,w])\] and \[ \mathcal{M}_{J,H^-}^{\mathrm{out}}([\gamma,w];p)=(ev_{\infty,[\gamma,w]}^{J,H^-})^{-1}(\mathbf{A}^v(p))\subset \mathcal{M}^{\mathrm{in}}_{J,H^-}([\gamma,w]).  \]  Given our transversality assumptions, these are smooth manifolds, with $\dim \mathcal{M}_{J,H^+}^{\mathrm{in}}(p;[\gamma,w])=\mathrm{ind}_f(p)-\mu_H([\gamma,w])$ and $\dim\mathcal{M}_{J,H^-}^{\mathrm{out}}([\gamma,w];p)=\mu_H([\gamma,w])-\mathrm{ind}_f(p)$.  Moreover they obtain orientations by the prescription in Section \ref{preim}, using as input the orientations that we have prescribed on $\mathcal{M}^{\mathrm{in}}_{J,H^+}([\gamma,w])$ and $\mathcal{M}^{\mathrm{out}}_{J,H^-}([\gamma,w])$ and the coorientations that we have prescribed on $\mathbf{D}^v(p)$ and $\mathbf{A}^v(p)$.  In the case that $\mathrm{ind}_f(p)=\mu_H([\gamma,w])$, our transversality assumptions together with standard results about the compactifications of $\mathcal{M}^{\mathrm{in}}_{J,H^+}([\gamma,w])$, $\mathcal{M}^{\mathrm{out}}_{J,H^-}([\gamma,w])$, $\mathbf{D}^v(p)$, and $\mathbf{A}^v(p)$ imply that the oriented zero-manifolds $\mathcal{M}_{J,H^+}^{\mathrm{in}}(p;[\gamma,w])$ and 
$\mathcal{M}_{J,H^-}^{\mathrm{out}}([\gamma,w];p)$ are compact.

With this said, we can define the PSS map $\tilde{\Phi}_{PSS}^{f,H^+}\co \hat{\Lambda}_{\uparrow}\otimes_{\kappa}\mathbf{CM}_*(f)\to CF_*(H)$ by extending $\hat{\Lambda}_{\uparrow}$-linearly from, for each index-$k$ critical point $p$ of $f$: \[ \tilde{\Phi}_{PSS}^{f,H^+}p = \sum_{[\gamma,w]\in \tilde{P}_k(H)} \#\left(\mathcal{M}_{J,H^+}^{\mathrm{in}}(p;[\gamma,w])\right)[\gamma,w]  \] where as usual ``$\#$'' refers to the $\kappa$-valued, signed count of points in a compact oriented zero-manifold. Similarly, define $\tilde{\Psi}_{PSS}^{H^-,f}\co CF_*(H)\to\hat{\Lambda}_{\uparrow}\otimes_{\kappa}\mathbf{CM}_*(f)$ by extending $\hat{\Lambda}_{\uparrow}$-linearly from, for each $[\gamma,w]\in\tilde{P}_k(H)$, \[ \tilde{\Psi}_{PSS}^{H^-,f}[\gamma,w]=\sum_{\substack{\tiny p\in\mathrm{Crit}(f),\,A\in\hat{\Gamma}\\ \mathrm{ind}_f(p)=k+2I_{c_1}(A)}}\#\left(\mathcal{M}_{J,H^-}^{\mathrm{out}}([\gamma,w\#(-A)];p)\right)T^Ap.  \] 
(Recall that $T^A\in\hat{\Lambda}_{\uparrow}$ has grading $-2I_{c_1}(A)$.  The geometric reason for the negative sign in $[w\#(-A)]$ is that gluing a sphere in class $-A$ to a left-cap $w$ results in gluing a sphere in class $A$ to the corresponding right-cap $w^R$.)  At least modulo two, these are consistent with the standard definitions as in \cite{PSS},\cite[Section 20.3]{ohbook}, but for reassurance that our orientation prescriptions for $\mathcal{M}_{J,H^+}^{\mathrm{in}}(p;[\gamma,w])$ and 
$\mathcal{M}_{J,H^-}^{\mathrm{out}}([\gamma,w];p)$ are appropriate we should check that the signs work out properly in the following:

\begin{prop}\label{chainmaps}
	$\tilde{\Phi}_{PSS}^{f,H^+}\co \hat{\Lambda}_{\uparrow}\otimes_{\kappa}\mathbf{CM}_*(f)\to CF_*(H)$  and $\tilde{\Psi}_{PSS}^{H^-,f}\co CF_*(H)\to \hat{\Lambda}_{\uparrow}\otimes_{\kappa}\mathbf{CM}_*(f)$ are chain maps.
	\end{prop}
	
	\begin{prop} \label{psscompose} $\tilde{\Psi}_{PSS}^{f,H^+}\circ \tilde{\Phi}_{PSS}^{H^-,f}\co \hat{\Lambda}_{\uparrow}\otimes_{\kappa}\mathbf{CM}_*(f)\to \hat{\Lambda}_{\uparrow}\otimes_{\kappa}\mathbf{CM}_*(f)$ induces the identity on homology.
\end{prop} 

(We do not check separately that $\tilde{\Phi}_{PSS}^{H^+,f}\circ\tilde{\Psi}_{PSS}^{f,H^-}$ induces the identity on $HF_*(H)$, since this follows formally from the corresponding fact about the composition in the opposite order together with the fact that $HF_k(H)$ has the same dimension over $\Lambda_{\uparrow}$ as the grading-$k$
part of the homology of $\hat{\Lambda}_{\uparrow}\otimes_{\kappa}\mathbf{CM}_*(f)$.)

\begin{proof}[Sketch of proof of Proposition \ref{chainmaps}] For the proofs of both propositions we focus on sign issues, taking as given facts that result from standard arguments related to transversality, compactness, and gluing.
	
Write $\partial^f$ for the differential on $\hat{\Lambda}_{\uparrow}\otimes_{\kappa}\mathbf{CM}_*(f)$ and $\partial_H$ for the differential on $CF_*(H)$.  We first consider $\tilde{\Psi}_{PSS}^{H^-,f}$.  Let $[\gamma,w]\in\tilde{P}_k(H)$, $A\in \hat{\Gamma}$, and $q\in \mathrm{Crit}(f)$ with $\mathrm{ind}_f(q)=k+2I_{c_1}(A)-1$.  The coefficients on $T^Aq$ in $\tilde{\Psi}_{PSS}^{H^-,f}\partial_H[\gamma,w]$ and in $\partial^f\tilde{\Psi}_{PSS}^{H^-,f}[\gamma,w]$ enumerate the two types of boundary points of the Gromov-Floer compactification of the one-manifold $\mathcal{M}_{J,H^-}^{\mathrm{out}}([\gamma,w\#(-A)];q)$.  To see that $\tilde{\Psi}_{PSS}^{H^-,f}$ is a chain map, \emph{i.e.} that $\tilde{\Psi}_{PSS}^{H^-,f}\partial_H-\partial^f\tilde{\Psi}_{PSS}^{H^-,f}=0$, we must show that these two types of boundary points contribute to the total signed count of points of $\partial \overline{\mathcal{M}_{J,H^-}^{\mathrm{out}}([\gamma,w\#(-A)];q)}$ with opposite sign.

The evaluation map $ev_{\infty,[\gamma,w\#(-A)]}^{J,H^-}$ extends continuously to a map $\overline{ev}_{\infty,[\gamma,w\#(-A)]}^{J,H^-}$ on the Gromov-Floer compactification of the $(k+2I_{c_1(A)})$-dimensional manifold of outward Floer caps $\mathcal{M}_{J,H^-}^{\mathrm{out}}([\gamma,w\#(-A)])$.  Under our transversality assumptions, the compact set\\ $(\overline{ev}_{\infty,[\gamma,w\#(-A)]}^{J,H^-})^{-1}(\overline{\mathbf{A}^v(q)})\setminus \mathcal{M}_{J,H^-}^{\mathrm{out}}([\gamma,w\#(-A)];q)$ will consist only of elements $u\in \mathcal{M}_{J,H^-}^{\mathrm{out}}([\gamma,w\#(-A)]$ that are mapped to some $\mathbf{A}^v(p)$ with $\mathrm{ind}_f(p)=\mathrm{ind}_f(q)+1$, together with elements $([u_0],u_1)\in \mathcal{M}_{J,H}([\gamma,w\#(-A)],[\gamma',w'\#(-A)])\times \mathcal{M}_{J,H^-}^{\mathrm{out}}([\gamma',w'\#(-A)])$ where  $u_1(\infty)\in \mathbf{A}^v(q)$ and $\mu_H([\gamma',w'\#(-A)])=\mu_H([\gamma,w\#(-A)])-1$.

Write $\hat{\mathbf{A}}^v(q)$ for the union of $\mathbf{A}^v(q)$ together with those strata of its compactification (as recalled before Proposition \ref{boundaryor}) taking the form $\mathbf{A}^v(p)\times\mathbf{M}^v(p,q)$ where $\mathrm{ind}_f(p)-\mathrm{ind}_f(q)=1$.  Thus $\hat{\mathbf{A}}^v(q)$ is a smooth manifold with boundary, whose boundary components are the various $\mathbf{A}^v(p)\times\{[x]\}$ for $[x]\in\mathbf{M}^v(p,q)$, each of whose boundary coorientations is given, according to Proposition \ref{boundaryor}, by $-\ep([x])$ times the coorientation of $\mathbf{A}^v(p)$ that is used to define the orientation of the zero-manifold $\mathcal{M}_{J,H^-}^{\mathrm{out}}([\gamma,w\#(-A)];p)$ and hence the coefficient of $T^Ap$ in $\tilde{\Psi}_{PSS}^{H^-,f}([\gamma,w])$. Here $\epsilon([x])$ is the sign with which $[x]$ contributes to the coefficient of $q$ in $\partial^fp$.  Therefore, using (\ref{prebdry}), the coefficient of $T^Aq$ in $\partial^f\tilde{\Psi}_{PSS}^{H^-,f}[\gamma,w]$ is $-1$ times the signed number of ends of the one-manifold $\mathcal{M}_{J,H^-}^{\mathrm{out}}([\gamma,w\#(-A)];q)$ whose images under $\mathrm{ev}^{J,H^-}_{\infty,[\gamma,w\#(-A)]}$ limit to some $\mathbf{A}^f(p)$ with $\mathrm{ind}_f(p)=\mathrm{ind}_f(q)+1$.  (Here the sign of an end of an oriented one-manifold  is positive if the orientation agrees with the direction pointing out towards the end, and negative otherwise.  Note that we get a different such end for each $[x]\in\mathbf{M}^v(p,q)$.)

Let $\hat{\mathcal{M}}^{\mathrm{out}}_{J,H^-}([\gamma,w\#(-A)])$ denote the union of $\mathcal{M}^{\mathrm{out}}_{J,H^-}([\gamma,w\#(-A)])$ with those strata of its Gromov-Floer compactification taking the form $\mathcal{M}_{J,H}([\gamma,w\#(-A)],[\gamma',w'\#(-A)])\times \mathcal{M}_{J,H^-}^{\mathrm{out}}([\gamma',w'\#(-A)])$ where $\mu_H([\gamma',w'\#(-A)])=\mu_H([\gamma,w\#(-A)])-1$.  This is a smooth manifold with boundary, and the remaining ends of $\mathcal{M}_{J,H^-}^{\mathrm{out}}([\gamma,w\#(-A)];q)$ (besides those discussed in the previous paragraph) correspond to the subset $(\overline{ev}_{\infty,[\gamma,w\#(-A)]}^{J,H^-}|_{\partial \hat{\mathcal{M}}_{J,H^-}^{\mathrm{out}}([\gamma,w\#(-A)])})^{-1}(\mathbf{A}^v(q))$ of $\partial \overline{\mathcal{M}_{J,H^-}^{\mathrm{out}}([\gamma,w\#(-A)];q)}$.  By (\ref{prebdry}), the orientation of  $(\overline{ev}_{\infty,[\gamma,w\#(-A)]}^{J,H^-}|_{\partial \hat{\mathcal{M}}_{J,H^-}^{\mathrm{out}}([\gamma,w\#(-A)])})^{-1}(\mathbf{A}^v(q))$ as a subset of $\partial \overline{\mathcal{M}_{J,H^-}^{\mathrm{out}}([\gamma,w\#(-A)];q)}$ coincides with its orientation as a preimage. Now $\partial
\hat{\mathcal{M}}_{J,H^-}^{\mathrm{out}}([\gamma,w\#(-A)])$ is the disjoint union of the manifolds  $V_{[u_0],[\gamma',w']}:=\{[u_0]\}\times \mathcal{M}_{J,H^-}^{\mathrm{out}}([\gamma',w'\#(-A)])$ as $[\gamma',w'\#(-A)]$ varies through elements of $\tilde{P}_{k+2I_{c_1}(A)-1}(H)$ and $[u_0]$ varies through elements of the zero-manifold $\mathcal{M}_{J,H}([\gamma,w\#(-A)],[\gamma',w'\#(-A)])$ (which in turn is the same as $\mathcal{M}_{J,H}([\gamma,w],[\gamma',w']))$.  The signed number of points in $(\overline{ev}^{J,H^-}_{\infty,[\gamma,w\#(-A)]}|_{V_{[u_0],[\gamma',w']}})^{-1}(\mathbf{A}^v(q))$ is then $\delta([u_0])$ times the coefficient of $T^Aq$ in $\tilde{\Psi}^{H^-,f}_{PSS}[\gamma',w'\#(-A)]$, where $\delta([u_0])$ is the sign relating the orientation of $V_{[u_0],[\gamma',w']}=\{[u_0]\}\times \mathcal{M}_{J,H^-}^{\mathrm{out}}([\gamma',w'\#(-A)])$ as a codimension-zero subset of $\partial\hat{\mathcal{M}}_{J,H^-}^{\mathrm{out}}([\gamma,w\#(-A)])$ to the orientation of $\mathcal{M}_{J,H^-}^{\mathrm{out}}([\gamma',w'])$.  

We claim that $\delta([u_0])$ coincides with the sign $\ep([u_0])$ with which $[u_0]$ contributes to the coefficient of $[\gamma',w']$ in $\partial_H[\gamma,w]$ (equivalently, to the coefficient of $[\gamma',w'\#(-A)]$ in $\partial_H[\gamma,w\#(-A)]$).  For $S\in \R$ let $u_S\in\tilde{\mathcal{M}}_{J,H}([\gamma,w],[\gamma',w'])$ denote the result of precomposing some representative $u_0$ of $[u_0]$ with the translation $(s,t)\mapsto (s+S,t)$.  A collar for $V_{[u_0],[\gamma',w']}$ is formed, for some $S_0\gg 0$, by means of a gluing map $\{u_S|S>S_0\}\times \mathcal{M}_{J,H^-}^{\mathrm{out}}([\gamma',w'\#(-A)])\hookrightarrow \mathcal{M}_{J,H^-}^{\mathrm{out}}([\gamma,w\#(-A)])$.  Our construction of orientations ensures that this map is orientation-preserving (using the orientation on $\{u_S|S>S_0\}$ as an open subset of $\tilde{\mathcal{M}}_{J,H}([\gamma,w],[\gamma',w'])$), and an outer normal vector for the collar is given by the image of $\frac{du_S}{dS}$.  Since $\frac{du_S}{dS}$ agrees with the orientation of $\tilde{\mathcal{M}}_{J,H}([\gamma,w],[\gamma',w'])$ if $\ep([u_0])=1$ and disagrees otherwise, it follows that $\delta([u_0])=\ep([u_0])$.  

Thus the signed count of points in $(\overline{ev}^{J,H^-}_{\infty,[\gamma,w\#(-A)]}|_{V_{[u_0],[\gamma',w']}})^{-1}(\mathbf{A}^v(q))$ is the contribution of $[u_0]$ to the coefficient of $[\gamma',w']$ in $\partial_H[\gamma,w]$ times the coefficient of $T^Aq$ in $\tilde{\Psi}_{PSS}^{H^-,f}[\gamma',w']$.  Summing over $[u_0],[\gamma',w']$ then shows that the coefficient of $T^Aq$ in $\tilde{\Psi}_{PSS}^{H^-,f}\partial_H[\gamma,w]$ enumerates with sign the elements of $(\overline{ev}_{\infty,[\gamma,w\#(-A)]}^{J,H^-}|_{\partial \hat{\mathcal{M}}_{J,H^-}^{\mathrm{out}}([\gamma,w\#(-A)])})^{-1}(\mathbf{A}^v(q))$.  Combining this with our earlier discussion of the ends of $\mathcal{M}_{J,H^-}^{\mathrm{out}}([\gamma,w\#(-A)];q)$ whose images under $ev_{\infty,[\gamma,w\#(-A)]}^{J,H^-}$ limit to $\mathbf{A}^v(p)$ with $\mathrm{ind}_f(p)=\mathrm{ind}_f(q)+1$, we conclude that the full set of boundary points of $\overline{\mathcal{M}_{J,H^-}^{\mathrm{out}}([\gamma,w\#(-A)];q)}$ is enumerated with sign by the coefficient of $T^Aq$ in $(\tilde{\Psi}_{PSS}^{H^-,f}\partial_H-\partial^f\tilde{\Psi}_{PSS}^{H^-,f})[\gamma,w]$.  Since the signed count of all the boundary points of $\overline{\mathcal{M}_{J,H^-}^{\mathrm{out}}([\gamma,w\#(-A)];q)}$ is zero for all $[\gamma,w],A$, and $q$, we conclude that $\tilde{\Psi}^{H^-,f}_{PSS}$ is a chain map.

The analysis of $\tilde{\Phi}_{PSS}^{f,H^+}$ is similar though the signs are slightly more complicated.  Given $p\in \mathrm{Crit}(f)$ with $\mathrm{ind}_f(p)=k$, and $[\gamma,w]\in \tilde{P}_{k-1}(H)$, we shall show that the boundary points of the compact oriented one-manifold with boundary $\overline{\mathcal{M}_{J,H^+}^{\mathrm{in}}(p;[\gamma,w])}$ are enumerated by the coefficient of $[\gamma,w]$ in $(-1)^k(\tilde{\Phi}_{PSS}^{f,H^+}\partial^f-\partial_H\tilde{\Phi}_{PSS}^{f,H^+})p$.

The first type of boundary point of $\overline{\mathcal{M}_{J,H^+}^{\mathrm{in}}(p;[\gamma,w])}$ arises from breaking of a sequence of Morse trajectories connecting the critical point $p$ to the image of $ev_{0,[\gamma,w]}^{J,H^+}$, and corresponds to a component of the boundary of the partial compactification $\hat{\mathbf{D}}^v(p)$ of $\mathbf{D}^v(p)$ having the form $\{[x]\}\times \mathbf{D}^v(q)$ where $\mathrm{ind}_f(q)=k-1$ and $[x]\in\mathbf{M}^v(p,q)$.  Such boundary points are enumerated with sign by $\eta([x])$ times the coefficient of $[\gamma,w]$ in $\tilde{\Phi}_{PSS}^{f,H^+}q$, where $\eta([x])$ is the sign of the boundary coorientation of $\{[x]\}\times \mathbf{D}^v(q)$ with respect to the coorientation of $\mathbf{D}^v(q)$ given by our original prescription.  Let us show that $\eta([x])=(-1)^k\ep([x])$ where $\ep([x])$ is the sign with which $[x]$ contributes to the coefficient of $q$ in $\partial^fp$.  Now by Proposition \ref{bdryor}, the \emph{orientation} of $\{x\}\times \mathbf{D}^v(q)$ as a boundary component of $\hat{\mathbf{D}}^v(p)$ is $\ep([x])$ times its original orientation, so if $\mathsf{n}_x$ denotes an outer normal vector field to $ \{x\}\times \hat{\mathbf{D}}^v(q)$ in $\hat{\mathbf{D}}^v(q)$ we have $\langle \mathsf{n}_x\rangle\oplus T\mathbf{D}^v(q)=\ep([x])T\hat{\mathbf{D}}^v(p)|_{\{[x]\}\times \mathbf{D}^v(q)}$ as oriented bundles.  Since, for general critical points $y$, we have $N\mathbf{D}^v(y)\oplus T\mathbf{D}^v(y)=TM|_{\mathbf{D}^v(y)}$, it follows that $N\mathbf{D}^v(q)=\ep([x])N\hat{\mathbf{D}}^v(p)|_{\{[x]\}\times \mathbf{D}^v(q)}\oplus\langle n_x\rangle$ as oriented bundles.  Since by definition the sign $\eta([x])$ is given by $N\mathbf{D}^v(q)=\eta([x])\langle n_x\rangle\oplus N\hat{\mathbf{D}}^v(p)|_{\{[x]\}\times \mathbf{D}^v(q)}$, it follows that \[\eta([x])=(-1)^{\mathrm{rank}(N\mathbf{D}^v(p))}\ep([x])=(-1)^{2m-k}(\ep([x]))=(-1)^k\ep([x]).\]
Consequently, the boundary points of $\overline{\mathcal{M}_{J,H^+}^{\mathrm{in}}(p;[\gamma,w])}$ that arise from Morse trajectory breaking are enumerated by the coefficient of $[\gamma,w]$ in $(-1)^k\tilde{\Phi}^{f,H^+}_{PSS}\partial^fp$.

The other points on the boundary of $\overline{\mathcal{M}_{J,H^+}^{\mathrm{in}}(p;[\gamma,w])}$ comprise the preimages under $ev_{0,[\gamma,w]}^{J,H^+}$ of the strata $\mathcal{M}_{J,H^+}^{\mathrm{in}}(p;[\gamma',w'])\times\{[u_0]\}$ of the compactification of $\mathcal{M}_{J,H^+}^{\mathrm{in}}(p;[\gamma,w])$, where $\mu_H([\gamma',w'])=k$ and $[u_0]\in\mathcal{M}_{J,H}([\gamma',w'],[\gamma,w])$.  Using (\ref{prebdry}), these are enumerated with sign by $\delta([u_0])$ times the coefficient of $[\gamma',w']$ in $\tilde{\Phi}_{PSS}^{f,H^+}p$, where $\delta([u_0])$ is the sign relating the orientation of 
$\mathcal{M}_{J,H^+}^{\mathrm{in}}([\gamma',w'])\times\{[u_0]\}$ as part of the boundary of the compactification of $\mathcal{M}_{J,H^+}^{\mathrm{in}}([\gamma,w])$ to the original orientation of $\mathcal{M}_{J,H^+}^{\mathrm{in}}([\gamma',w'])$.  Letting $u_0$ be a representative in $\tilde{\mathcal{M}}_{J,H}([\gamma',w'],[\gamma,w])$ for $[u_0]$ and letting $u_S$ denote the precomposition of $u_0$ with $(s,t)\mapsto (s+S,t)$, a collar for $\mathcal{M}_{J,H^+}^{\mathrm{in}}([\gamma',w'])\times\{[u_0]\}$ is given, for $S_0\gg 0$, by an orientation-preserving gluing map $\mathcal{M}_{J,H^+}^{\mathrm{in}}([\gamma',w'])\times \{u_S|S<-S_0\}\hookrightarrow \mathcal{M}_{J,H^+}^{\mathrm{in}}([\gamma,w])$, with the outer normal vector corresponding to $-\frac{du_S}{dS}$.  Denoting by $\epsilon([u_0])$ the sign with which $[u_0]$ contributes to the coefficient of $[\gamma,w]$ in $\partial_H[\gamma',w']$, this outer normal agrees with the orientation of $\tilde{\mathcal{M}}_{J,H}([\gamma',w'],[\gamma,w])$ iff $\epsilon([u_0])=-1$.  So after incorporating a sign $(-1)^{\dim \mathcal{M}_{J,H^+}^{\mathrm{in}}([\gamma',w'])}=(-1)^{2m-k}$ from moving $-\frac{d u_S}{dS}$ from the last to the first slot of an oriented basis for a tangent space to $\mathcal{M}_{J,H^+}^{\mathrm{in}}([\gamma',w'])\times \{u_S|S<-S_0\}$, we find that $\delta([u_0])=(-1)^{k+1}\ep([u_0])$.  Thus the boundary points of $\overline{\mathcal{M}_{J,H^+}^{in}(p;[\gamma,w])}$ that arise from breaking of Floer trajectories are enumerated with sign by the coefficient of $[\gamma,w]$ in $(-1)^{k+1}\partial_H\tilde{\Phi}_{PSS}^{f,H^+}p$.

By combining this with our earlier count of the boundary points associated to Morse trajectory breaking, and using that the signed count of points on the boundary of a compact oriented one-manifold is zero, we deduce that $(-1)^k\tilde{\Phi}_{PSS}^{f,H^+}\partial^f+(-1)^{k+1}\partial_H\tilde{\Phi}_{PSS}^{f,H^+}=0$ and hence that $\tilde{\Phi}_{PSS}^{f,H^+}$ is a chain map.
\end{proof}

\begin{proof}[Sketch of proof of Proposition \ref{psscompose}]
Let $p,q$ be critical points of $f$ and $A\in \hat{\Gamma}$ with $\mathrm{ind}_f(p)=k$ and $\mathrm{ind}_f(q)=k+2I_{c_1}(A)$. The coefficient of $T^Aq$ in $\tilde{\Psi}^{H^-,f}_{PSS}\tilde{\Phi}^{f,H^+}_{PSS}p$ is the sum, over $[\gamma,w]\in \tilde{P}_k(H)$, of the signed counts of points in the oriented zero-manifolds \begin{align}\label{prodpreim} & \mathcal{M}_{J,H^+}^{\mathrm{in}}(p;[\gamma,w])\times \mathcal{M}_{J,H^-}^{\mathrm{out}}([\gamma,w\#(-A)];q)=(ev_{0,[\gamma,w]}^{J,H^+})^{-1}(\mathbf{D}^v(p))\times (ev_{\infty,[\gamma,w\#(-A)]}^{J,H^-})^{-1}(\mathbf{A}^v(q)) \\ & \qquad \qquad \subset \mathcal{M}_{J,H^+}^{\mathrm{in}}([\gamma,w])\times \mathcal{M}_{J,H^-}^{\mathrm{out}}([\gamma,w\#(-A)]). \nonumber
\end{align}	
We observe that (\ref{prodpreim}) is equal, as an \emph{oriented} zero-manifold, to the preimage under the product map $ev_{0,[\gamma,w]}^{J,H^+}\times  ev_{\infty,[\gamma,w\#(-A)]}^{J,H^-}\co \mathcal{M}_{J,H^+}^{\mathrm{in}}([\gamma,w])\times \mathcal{M}_{J,H^-}^{\mathrm{out}}([\gamma,w\#(-A)])\to M\times M$ of the submanifold $\mathbf{D}^v(p)\times \mathbf{A}^v(q)$ of $M\times M$, cooriented using the obvious identification $N\left(\mathbf{D}^v(p)\times \mathbf{A}^v(q) \right)\cong N\mathbf{D}^v(p)\oplus N\mathbf{A}^v(q)$. Indeed, by definition, a point $(u_0,u_{\infty})\in (ev_{0,[\gamma,w]}^{J,H^+}\times  ev_{\infty,[\gamma,w\#(-A)]}^{J,H^-})^{-1}(\mathbf{D}^v(p)\times\mathbf{A}^v(q))$ is positively-oriented iff the derivative of $ev_{0,[\gamma,w]}^{J,H^+}\times  ev_{\infty,[\gamma,w\#(-A)]}^{J,H^-}$ at $(u_0,u_{\infty})$ maps $T_{u_0}\mathcal{M}_{J,H^+}^{\mathrm{in}}([\gamma,w])\times T_{u_{\infty}}\mathcal{M}_{J,H^-}^{\mathrm{out}}([\gamma,w\#(-A)])$ orientation-preservingly to $N_{u_0(0)}\mathbf{D}^v(p)\oplus N_{u_{\infty}(\infty)}\mathbf{A}^v(q)$, which holds iff $u_0$ and $u_{\infty}$ have the same orientation signs as elements of $\mathcal{M}_{J,H^+}^{\mathrm{in}}([\gamma,w])$ and $\mathcal{M}_{J,H^-}^{\mathrm{out}}([\gamma,w\#(-A)])$, \emph{i.e.} iff $(u_0,u_1)$ is positively-oriented as an element of $\mathcal{M}_{J,H^+}^{\mathrm{in}}([\gamma,w])\times \mathcal{M}_{J,H^-}^{\mathrm{out}}([\gamma,w\#(-A)])$.

Identify $\C P^1$ with $\C\cup\{\infty\}$, which contains both $\C$ and $\C^-$.  Given a suitably generic smooth map $\mathcal{H}\co \C P^1\times M\to\R$ having support contained in $(\C\setminus\{0\})\times M$, and given $B\in\hat{\Gamma}$, let $\mathcal{M}_{J,\mathcal{H}}(B)$ denote the space of smooth maps $u\co \C P^1\to M$ that represent $B$ in $\hat{\Gamma}$, are $J$-holomorphic at $0$ and $\infty$, and, on $\C\setminus\{0\}$ with its usual polar coordinates $r,\theta$, obey \[ 2\pi r\frac{\partial u}{\partial r}+J(u)\left(2\pi\frac{\partial u}{\partial \theta}-X_{\mathcal{H}(re^{i\theta},\cdot)}(u)\right)=0.\]  
For generic $J,\mathcal{H}$ this is a smooth manifold of dimension $2m+2I_{c_1}(B)$, carrying an orientation induced by our prescription for orienting determinant lines of admissible $\delbar$-operators over $\C P^1$ by continuation from complex-linear operators.  We write $ev_{0,B}^{J,\mathcal{H}},ev_{\infty,B}^{J,\mathcal{H}}\co \mathcal{M}_{J,\mathcal{H}}(B)\to M$ for the maps given by $u\mapsto u(0)$ and $u\mapsto u(\infty)$; generically, these will be transverse to all ascending and descending manifolds of critical points of $f$.  Due to the semipositivity of $(M,\omega)$, all boundary strata of the Gromov compactification of $\mathcal{M}_{J,\mathcal{H}}(B)$ will generically have codimension at least two.

For $J,\mathcal{H}$ satisfying these transversality properties, we may define a map $\Upsilon^{J,\mathcal{H}}\co \hat{\Lambda}_{\uparrow}\otimes_{\kappa}\mathbf{CM}_*(f)\to \hat{\Lambda}_{\uparrow}\otimes_{\kappa}\mathbf{CM}_*(f)$ 
by extending linearly from, for $p\in\mathrm{Crit}(f)$ with index $k$, \[ \Upsilon^{J,\mathcal{H}}p=\sum_{\substack{\tiny q\in \mathrm{Crit}(f),B\in\hat{\Gamma}\\ \mathrm{ind}_f(q)=k+2I_{c_1}(B)}}\#\left(\left(ev_{0,B}^{J,\mathcal{H}}\times ev_{\infty,B}^{J,\mathcal{H}}\right)^{-1}(\mathbf{D}^v(p)\times \mathbf{A}^v(q))\right)T^Bq. \]  This can be verified to be a chain map by standard arguments together with a sign analysis like those appearing in parts of the proof of the preceding proposition; moreover, a standard argument involving one-parameter families of choices of $(J,\mathcal{H})$ shows that different such choices yield  homotopic chain maps $\Upsilon^{J,\mathcal{H}}$.

Returning to our consideration of $\tilde{\Psi}^{H^-,f}_{PSS}\tilde{\Phi}^{f,H^+}_{PSS}$, recall that the functions $H^-\co \C^-\times M\to \R$ and $H^+\co \C\times M\to \R$ that are used to define the PSS maps satisfy $H^-(re^{2\pi it},x)=H(t,x)$ for $r<\frac{1}{2}$ and $H^+(re^{2\pi it},x)=H(t,x)$ for $r<2$.  So if $S>2$, we may define $\mathcal{H}_S\co \C P^1\times M\to \R$ by setting $\mathcal{H}_S(z,x)=H^+(Sz,x)$ for $|z|\leq 1$ and $\mathcal{H}_S(z,x)=H^-(z/S,x)$ for $|z|\geq 1$, so that $\mathcal{H}_S$ restricts to $\{\frac{2}{S}<|z|<\frac{S}{2}\}\times M$ as $(re^{2\pi it},x)\mapsto H(t,x)$.  

For $A\in\hat{\Gamma}$, a standard gluing construction gives rise, for all sufficiently large $S$, to a diffeomorphism \[ \bigsqcup_{\substack{{\tiny k\in\Z},\\{\tiny [\gamma,w]\in\mathcal{P}_k(H)}}}\mathcal{M}_{J,H^+}^{\mathrm{in}}([\gamma,w])\times \mathcal{M}_{J,H^-}^{\mathrm{out}}([\gamma,w\#(-A)])\cong \mathcal{M}_{J,\mathcal{H}_S}(A).\]  Moreover the orientation on $\mathcal{M}_{J,H^+}^{\mathrm{in}}([\gamma,w])$ was constructed in just such a way as to make this diffeomorphism orientation-preserving.  In the limit as $S\to\infty$, the evaluation maps $ev_{0,[\gamma,w]}^{J,H^+}\times  ev_{\infty,[\gamma,w\#(-A)]}^{J,H^-}\co \mathcal{M}_{J,H^+}^{\mathrm{in}}([\gamma,w])\times \mathcal{M}_{J,H^-}^{\mathrm{out}}([\gamma,w\#(-A)])\to M\times M$ agree under these diffeomorphisms with $ev_{0,A}^{J,\mathcal{H}}\times ev_{\infty,A}^{J,\mathcal{H}}\co \mathcal{M}_{J,\mathcal{H}_S}(A)\to M\times M$.  Hence, for any $p,q,A$ as at the start of the proof, if $S$ is sufficiently large the coefficient of $T^Aq$ in $\tilde{\Psi}^{H^-,f}_{PSS}\tilde{\Phi}^{f,H^+}_{PSS}p$ agrees with the corresponding coefficient in $\Upsilon^{J,\mathcal{H}_S}p$.  

Thus  $\tilde{\Psi}^{H^-,f}_{PSS}\tilde{\Phi}^{f,H^+}_{PSS}$ induces the same map on homology as  $\Upsilon^{J,\mathcal{H}_S}$ for large $S$.  But the map induced on homology by $\Upsilon^{J,\mathcal{H}}$ is independent of $\mathcal{H}$, and so may be determined by setting $\mathcal{H}=0$. We will now show that $\Upsilon^{J,0}$ is the identity, which will complete the proof.  Assuming that $J$ is suitably generic, since $\mathcal{M}_{J,0}(A)$ is acted on locally-freely by the two-dimensional group of automorphisms of $\mathbb{C}P^1$ which fix $0$ and $\infty$, the image $(ev_{0,A}^{J,0}\times ev_{\infty,A}^{J,0})(\mathcal{M}_{J,0}(A))$ will be contained in the image of a pseudocycle of dimension at most $2m+2I_{c_1}(A)-2$ which misses the codimension-$(2m+2I_{c_1}(A))$ submanifolds $\mathbf{D}^f(p)\times\mathbf{A}^f(q)$ whenever $\mathrm{ind}_f(q)-\mathrm{ind}_f(p)=2I_{c_1}(A)$.  Thus \[ \Upsilon^{J,0}p=\sum_{\mathrm{ind}_f(q)=\mathrm{ind}_f(p)}\#\left(\left(ev_{0,0}^{J,0}\times ev_{\infty,0}^{J,0}\right)^{-1}(\mathbf{D}^v(p)\times\mathbf{A}^v(q))\right).\]  
The domain $\mathcal{M}_{J,0}(0)$ of $ev_{0,0}^{J,0}\times ev_{\infty,0}^{J,0}$ is just the space of constant maps $\mathbb{C}P^1\to M$, which we identify with $M$, and then $ev_{0,0}^{J,0}\times ev_{\infty,0}^{J,0}$ is identified with the diagonal embedding  $\Delta\co M\to M\times M$. The space $\mathcal{M}_{J',0}(0)$ is cut out transversely for any compatible almost complex structure $J'$ by \cite[Lemma 6.7.6]{MS}; using this fact and a continuation argument involving a choice of $J'$ that is integrable on some neighborhood, it is not hard to check that the orientation of $\mathcal{M}_{J,0}(0)$ (defined using continuation to complex-linear operators) agrees with the orientation of $M$ induced by its symplectic structure.
	
	By the Morse-Smale condition, the preimage under the diagonal embedding of $\mathbf{D}^v(p)\times\mathbf{A}^v(q)$ is empty when $p,q$ are distinct critical points of equal index, and it consists only of $p$ if $p=q$.  So to confirm that $\Upsilon^{J,0}$ is the identity we just need to check that the orientation sign of each critical point $p$ as an element of $\Delta^{-1}(\mathbf{D}^v(p)\times \mathbf{A}^v(p))$ is $+1$ rather than $-1$.  This is equivalent to the statement that the composition \[ T_pM \to T_pM\oplus T_pM\to N_p\mathbf{D}^v(p)\oplus N_p\mathbf{A}^v(p)\] is orientation-preserving, where the first map is the diagonal inclusion and the second map is the direct sum of the quotient projections.   This holds because, by our conventions, the orientations of $N_p\mathbf{D}^v(p)$ and $N_p\mathbf{A}^v(p)$ are given by their identifications with (respectively) $T_p\mathbf{A}^v(p)$ and $T_p\mathbf{D}^v(p)$, and the orientation of $T_p\mathbf{A}^v(p)$ is chosen so that $T_pM=T_p\mathbf{A}^v(p)\oplus T_p\mathbf{D}^v(p)$ as oriented vector spaces.
\end{proof}

\subsubsection{Time reversal}  Given a generic nondegenerate Hamiltonian $H\co S^1\times M\to \R$ and Morse function $f\co M\to \R$, the preceding subsection describes a prescription for orienting the moduli spaces underlying the boundary operator $\partial_H$ and the PSS maps $\tilde{\Phi}_{PSS}^{f,H^+}$ and $\tilde{\Psi}_{PSS}^{H^-,f}$.  The prescription takes as input orientations of determinant lines of admissible $\delbar$-operators over $\C^-$ associated to arbitrarily-chosen right-caps $w_{\gamma}^{R}$ of the various one-periodic orbits $\gamma\in P(H)$; for the PSS maps we also require the orientations of the descending manifolds $\mathbf{D}^v(p)$ that are used in the construction of the Morse complex of $f$.

Consider now replacing $H$ by its time-reversal $\hat{H}\co S^1\times M\to \R$ given by $\hat{H}(t,x)=-H(1-t,x)$, so that the time-one maps of $H$ and $\hat{H}$ are inverse to each other, and the one-periodic orbits of the Hamiltonian flow of $\hat{H}$ are given by $\bar{\gamma}(t)=\gamma(1-t)$ as $\gamma$ varies through $P(H)$.  We also replace $f$ by $-f$.  We could apply the prescription of the previous subsection to the data $(\hat{H},-f)$, with arbitrary orientations of the appropriate determinant lines and descending manifolds as input, to construct the Floer boundary operator $\partial_{\hat{H}}$ and the PSS maps $\tilde{\Phi}_{PSS}^{-f,\hat{H}^+}$ and $\tilde{\Psi}_{PSS}^{\hat{H}^-,-f}$.  However, in order to obtain the relationships in Section \ref{floersect} between $CF_*(H)$ and $CF_*(\hat{H})$ we shall insist that the input orientations used for $\hat{H}$ and $-f$ be related to those used for $H$ and $f$ in particular ways.  (If the input orientations are not related in these ways,   additional signs might be needed in Section \ref{floersect}.)

First of all, consistently with Proposition \ref{wor}, if $v$ is the gradient-like vector field for $f$ used to construct $\mathbf{CM}_*(f)$ we use $-v$ as the gradient-like vector  field for $-f$  in the construction of $\mathbf{CM}_*(-f)$ and we orient the manifolds $\mathbf{D}^{-v}(p)=\mathbf{A}^v(p)$ by the condition that $T_p\mathbf{A}^v(p)\oplus T_p\mathbf{D}^v(p)=T_pM$ as oriented vector spaces.

As for the relevant $\delbar$-operators over $\C^-$, first note that if $\bar{\gamma}\in P(\hat{H})$, then a right-cap $w^I\co \mathbb{D}^-\to M$ for $\bar{\gamma}$ may be obtained by starting with a left-cap $w\co \mathbb{D}\to M$ for $\gamma\in P(H)$ (where $\bar{\gamma}(t)=\gamma(1-t)$) and then setting $w^I(z)=w(1/z)$. (Thus $w^I=\bar{w}^R$ where $\bar{w}(z)=w(\bar{z})$.)  A symplectic trivialization of $\mathcal{T}_w\co w^*TM\xrightarrow{\sim}\C\times \mathbb{R}^{2m}$ induces a trivialization $\mathcal{T}_{w^I}\co w^{I*}TM\xrightarrow{\sim}\C^-\times \mathbb{R}^{2m}$ by composing with $(z,\vec{v})\mapsto (1/z,\vec{v})$.  If $B_{\gamma}\co S^1\mapsto \mathfrak{sp}_{2m}$ generates the linearized Hamiltonian flow of $H$ along $\gamma$ in terms of the trivialization  $\mathcal{T}_w$, then the linearized Hamiltonian flow of $\hat{H}$ along $\bar{\Gamma}$ is generated in terms of $\mathcal{T}_{w^I}$ by $B_{\bar{\gamma}}(t)=-B_{\gamma}(1-t)$. We can then form an admissible $\delbar$-operator $D_{\bar{w}}^{-}$ over $\C^-$ as in Section \ref{dets} by choosing an arbitrary $B_{\bar{\gamma}}^-\co \C^-\to\mathfrak{gl}_{2m}$ with $B_{\bar{\gamma}}^-(e^{2\pi(s+it)})=B_{\bar{\gamma}}(t)$ for $s\ll 0$ and  $B_{\bar{\gamma}}^-(e^{2\pi(s+it)})=0$ for $s\gg 0$, and taking $D_{\bar{w}}^{-}=2\pi\left(r\frac{\partial}{\partial r}+J_0\frac{\partial}{\partial \theta}\right)-B_{\bar{\gamma}}^-(re^{i\theta})$.  Setting $B_{\gamma}^{+}(z)=-B_{\bar{\gamma}}^-(1/z)$, we obtain a similar operator $D_{w}^{+}=2\pi\left(r\frac{\partial}{\partial r}+J_0\frac{\partial}{\partial \theta}\right)-B_{\gamma}^+(re^{i\theta})$ over $\C$.  

Now, assuming that our orientation prescriptions for $CF_*(H)$ have already been chosen, these determine an orientation for $\mathrm{Det}(D_{w}^{+})$.  Namely, our orientation prescription directly gives an orientation for the determinant line of the admissible $\delbar$-operator $D_{w}^{-}$  over $\C^-$ given by replacing $B_{\gamma}^{+}(z)$ in the formula for $D_{w}^{+}$ by $B_{\gamma}^{+}(1/\bar{z})$, and then $\mathrm{Det}(D_{w}^{+})$ is oriented by requiring  the gluing map $\mathrm{Det}(D_{w}^{+})\otimes \mathrm{Det}(D_{w}^{-})\to \mathrm{Det}(D_w^+\#_S D_w^-)$ to be orientation-preserving for large $S$, where the admissible $\delbar$-operator $D_w^+\#_S D_w^-$ over $\C P^1$ has its determinant line oriented by continuation to a complex-linear operator.  But the biholomorphism $z\mapsto 1/z$ between $\C$ and $\C^-$ is easily seen to induce an isomorphism $\mathrm{Det}(D_{\bar{w}}^-)\cong \mathrm{Det}(D_w^+)$.  We hereby orient the determinant lines $\mathrm{Det}(D_{\bar{w}}^-)$ for cappings $\bar{w}$ of the various elements $\bar{w}$ of $P(\hat{H})$ by requiring these isomorphisms  $\mathrm{Det}(D_{\bar{w}}^-)\cong \mathrm{Det}(D_w^+)$ to be orientation-preserving.  

This suffices to determine the orientations of the spaces involved in the Floer differential and PSS maps for $CF_*(\hat{H})$, via the procedures described in Section \ref{caps}.  Namely, spaces of outward Floer caps $\mathcal{M}_{J,\hat{H}^-}^{\mathrm{out}}([\bar{\gamma},\bar{w}])$ are oriented using continuation from the orientations of $\mathrm{Det}(D_{\bar{w}}^-)$ that we have just chosen.  For elements $u$ of a Floer trajectory space $\tilde{\mathcal{M}}_{J,\hat{H}}([\bar{\gamma},\bar{w}],[\bar{\gamma}',\bar{w}']))$, the linearization of the Floer equation at $u$ is an admissible $\delbar$-operator $D_u$ over $\R\times S^1$ whose determinant line is oriented by the condition that the gluing map $\mathrm{Det}(D_u)\otimes\mathrm{Det}(D_{\bar{w}'}^-)\to \mathrm{Det}(D_u\#_S D_{\bar{w}'}^-)$ is orientation-preserving with respect to the orientation on the codomain induced by continuation from the chosen orientation on $\mathrm{Det}(D_{\bar{w}}^-)$.  Finally, inward Floer cap spaces $\mathcal{M}_{J,H^+}^{\mathrm{in}}([\bar{\gamma},\bar{w}])$ are oriented using continuation from the orientations on determinant lines of admissible $\delbar$-operators $D_{\bar{w}}^+$ over $\C$, which are determined via the condition that the gluing map $\mathrm{Det}(D_{\bar{w}}^+)\otimes 
\mathrm{Det}(D_{\bar{w}}^-)\to \mathrm{Det}(D_{\bar{w}}^+\#_SD_{\bar{w}}^-)$ is orientation-preserving, using the previously-chosen orientation of 
$\mathrm{Det}(D_{\bar{w}}^-)$ and the standard orientation on 
$\mathrm{Det}(D_{\bar{w}}^+\#_SD_{\bar{w}}^-)$ induced by continuation from a complex-linear operator.

For $\gamma,\gamma'\in P(H)$ with left-caps $w,w'$, we have diffeomorphisms of Floer trajectory spaces $\mathbb{I}_{\mathrm{cyl}}\co \tilde{\mathcal{M}}_{J,H}([\gamma,w],[\gamma',w'])\to  \tilde{\mathcal{M}}_{J,\hat{H}}([\bar{\gamma}',\bar{w}'],[\bar{\gamma},\bar{w}])$ defined by $(\mathbb{I}_{\mathrm{cyl}}u)(s,t)=u(-s,1-t)$ for $(s,t)\in \R\times S^1$.  Similarly, if $H^+\co \C\times M\to \R$, $H^-\co\C^-\times M\to \R$, $\hat{H}^+\co \C\times M\to \R$, and $\hat{H}^-\co \C^-\times M\to \R$ are related by $\hat{H}^{\pm}(z,x)=H^{\mp}(1/z,x)$ we have diffeomorphisms $\mathbb{I}_{\pm}\co \mathcal{M}_{J,H^+}^{\mathrm{in}}([\gamma,w])\to \mathcal{M}_{J,\hat{H}^-}^{\mathrm{out}}([\bar{\gamma},\bar{w}])$ and $\mathbb{I}_{\mp}\co \mathcal{M}_{J,H^-}^{\mathrm{out}}([\gamma,w])\to \mathcal{M}_{J,\hat{H}^+}^{\mathrm{in}}([\bar{\gamma},\bar{w}])$, each of which sends a map $u$ to its precomposition with $z\mapsto 1/z$.

\begin{prop}\label{floerflip}
	The diffeomorphisms $\mathbb{I}_{\mathrm{cyl}},\mathbb{I}_{\pm},\mathbb{I}_{\mp}$ behave as follows with respect to orientations:
	\begin{itemize}
	\item[(i)] $\mathbb{I}_{\pm}\co \mathcal{M}_{J,H^+}^{\mathrm{in}}([\gamma,w])\to \mathcal{M}_{J,\hat{H}^-}^{\mathrm{out}}([\bar{\gamma},\bar{w}])$ is orientation-preserving.
	\item[(ii)] $\mathbb{I}_{\mp}\co \mathcal{M}_{J,H^-}^{\mathrm{out}}([\gamma,w])\to \mathcal{M}_{J,\hat{H}^+}^{\mathrm{in}}([\bar{\gamma},\bar{w}])$ affects orientations by the sign $(-1)^{\mu_H([\gamma,w])}$.
	\item[(iii)] $\mathbb{I}_{\mathrm{cyl}}\co \tilde{\mathcal{M}}_{J,H}([\gamma,w],[\gamma',w'])\to  \tilde{\mathcal{M}}_{J,\hat{H}}([\bar{\gamma}',\bar{w}'],[\bar{\gamma},\bar{w}])$ affects orientations by the sign $(-1)^{\mu_H([\gamma,w])(\mu_H([\gamma,w])-\mu_{H}([\gamma',w']))}$.	
	\end{itemize}
\end{prop}

\begin{remark}
	The sign in (iii) should be regarded as directly analogous to that in Proposition \ref{wor}, noting that in the present context the number $n=\dim M$ appearing in Proposition \ref{wor} is even.
\end{remark}

\begin{proof}
	Each of the statements corresponds to relations between the determinant lines of operators arising as linearizations of the equations that cut out the various moduli spaces.  Part (i) is immediate from our convention that $z\mapsto 1/z$ induces an orientation-preserving isomorphism between the determinant lines of the operators denoted earlier as $D_{w}^{+}$ and $D_{\bar{w}}^{-}$.  
	
	The other signs are related to the following general point (see, \emph{e.g.}, \cite[p. 150]{SeBook}).  Suppose that $A_1\co X_1\to Y_1$ and $A_2\co X_2\to Y_2$ are two Fredholm operators, giving rise to direct sum operators $A_1\oplus A_2\co X_1\oplus X_2\to Y_1\oplus Y_2$ and $A_2\oplus A_1\co X_2\oplus X_1\to Y_2\oplus Y_1$.  One has natural isomorphisms $\mathrm{Det}(A_1)\otimes \mathrm{Det}(A_2)\cong \mathrm{Det}(A_1\oplus A_2)$, and similarly with $A_1$ and $A_2$ reversed. On the other hand there is a straightforward identification of $\mathrm{Det}(A_1\oplus A_2)$ with $\mathrm{Det}(A_2\oplus A_1)$, and in terms of the isomorphisms mentioned in the previous sentence this identification is given by the map $\Sigma\co \mathrm{Det}(A_1)\otimes \mathrm{Det}(A_2)\to \mathrm{Det}(A_2)\otimes \mathrm{Det}(A_1)$ defined by $\Sigma(a_1\otimes a_2)=(-1)^{\mathrm{ind}(A_1)\mathrm{ind}(A_2)}a_2\otimes a_1$.
	
	With this said, the sign in (ii) is the sign associated to the map $\mathrm{Det}(D_w^-)\to \mathrm{Det}(D_{\bar{w}}^+)$ induced by $z\mapsto 1/z$. Now $z\mapsto 1/z$ also induces maps $\mathrm{Det}(D_w^+)\to \mathrm{Det}(D_{\bar{w}}^-)$ and $\mathrm{Det}(D_w^+\#_SD_w^-)\to \mathrm{Det}(D_{\bar{w}}^+\#_SD_{\bar{w}}^-)$, and these maps are both orientation-preserving. (In the first case this just follows from our orientation convention; in the second, by using continuation to appropriate complex-linear operators, it results from $z\mapsto 1/z$ being holomorphic as a map $\C P^1\to \C P^1$.)  We have a commutative diagram \[ \xymatrix{ \mathrm{Det}(D_w^+)\otimes \mathrm{Det}(D_w^-) \ar[d]_{\mathbb{I}_*\otimes\mathbb{I}_*} \ar[r] & \mathrm{Det}(D_w^+\#_SD_w^-)
	 \ar[dd]^{\mathbb{I}_*} \\ \mathrm{Det}(D_{\bar{w}}^-)\otimes \mathrm{Det}(D_{\bar{w}}^+) \ar[d]_{\Sigma} & \\ \mathrm{Det}(D_{\bar{w}}^+)\otimes \mathrm{Det}(D_{\bar{w}}^-) \ar[r] &  \mathrm{Det}(D_{\bar{w}}^+\#_SD_{\bar{w}}^-)
	}\] where the horizontal maps are the gluing maps, which are orientation-preserving, and all maps labeled $\mathbb{I}_*$ are induced (on appropriate domains) by $z\mapsto 1/z$.   Since the right arrow is orientation-preserving, so is the composition on the left. Since $\mathbb{I}_*\co \mathrm{Det}(D_w^+)\to \mathrm{Det}(D_{\bar{w}}^-)$ preserves orientation, it then follows that  $\mathbb{I}_*\co \mathrm{Det}(D_w^-)\to \mathrm{Det}(D_{\bar{w}}^+)$ affects orientation by the sign apprearing in the formula for $\Sigma$, namely $(-1)^{\mathrm{ind}(D_w^+)\mathrm{ind}(D_w^-)}=(-1)^{(2m-\mu_H([\gamma,w]))\mu_H([\gamma,w])}=(-1)^{\mu_H([\gamma,w])}$.
	
	(iii) follows similarly: for $u\in \tilde{\mathcal{M}}([\gamma,w],[\gamma',w'])$ with $D_u$ the linearization of the left-hand-side of the Floer equation at $u$, reversal of the $s$ and $t$ variables induces an isomorphism $\mathbb{J}_*\co \mathrm{Det}(D_u)\to \mathrm{Det}(D_{\mathbb{I}_{\mathrm{cyl}}u})$, while the glued operators $D_w^+\#_SD_{u}$ and $D_{\mathbb{I}_{\mathrm{cyl}}u}\#_SD_{\bar{w}}^-$ are intertwined by $z\mapsto 1/z$. The induced map $\mathbb{I}_*\co \mathrm{Det}(D_w^+\#_SD_{u})\to \mathrm{Det}(D_{\mathbb{I}_{\mathrm{cyl}}u}\#_SD_{\bar{w}}^-)$ preserves orientation, as the orientations on domain and range are related by continuation to the orientations on $\mathrm{Det}(D_{w'}^+)$ and $\mathrm{Det}(D_{\bar{w}'}^-)$ and our orientation conventions dictate that $z\mapsto 1/z$ induces an orientation-preserving map between $\mathrm{Det}(D_{w'}^+)$ and $\mathrm{Det}(D_{\bar{w}'}^-)$.
	We then get a commutative diagram \[ \xymatrix{ \mathrm{Det}(D_w^+)\otimes \mathrm{Det}(D_u) \ar[d]_{\mathbb{I}_*\otimes\mathbb{J}_*} \ar[r] & \mathrm{Det}(D_w^+\#_SD_u)
		\ar[dd]^{\mathbb{I}_*} \\ \mathrm{Det}(D_{\bar{w}}^-)\otimes \mathrm{Det}(D_{\mathbb{I}_{\mathrm{cyl}}u}) \ar[d]_{\Sigma} & \\ \mathrm{Det}(D_{\mathbb{I}_{\mathrm{cyl}}u})\otimes \mathrm{Det}(D_{\bar{w}}^-) \ar[r] &  \mathrm{Det}(D_{\mathbb{I}_{\mathrm{cyl}}u}\#_SD_{\bar{w}}^-)
	}\] where the horizontal gluing maps and both maps denoted $\mathbb{I}_*$ are orientation-preserving; hence $\mathbb{J}_*$ affects orientation by the sign $(-1)^{\mathrm{ind}(D_{w}^+)\mathrm{ind}(D_u)}=(-1)^{(2m-\mu_{H}([\gamma,w]))(\mu_H([\gamma,w])-\mu_H([\gamma',w']))}=(-1)^{\mu_{H}([\gamma,w])(\mu_H([\gamma,w])-\mu_H([\gamma',w']))}$.
\end{proof}

Recall from (\ref{lkfdef}) and (\ref{lkmdef}) the $\Lambda_{\uparrow}$-bilinear pairings $L^{F}_{H,k}\co CF_{2m-k}(H)\times CF_k(\hat{H})\to \Lambda$ and $L^{M}_{f,k}\co (\hat{\Lambda}_{\uparrow}\otimes_{\kappa}\mathbf{CM}_*(f;\kappa))_{2m-k}\times (\hat{\Lambda}_{\uparrow}\otimes_{\kappa}\mathbf{CM}_*(-f;\kappa))_k\to\Lambda_{\uparrow}$ induced in both cases by the obvious bijections between the standard bases for the two factors.  The foregoing allows us to determine the signs arising from the interaction of various maps with these pairings. (Compare Remark \ref{morsepairsign}.)

\begin{cor}\label{floerdualdiff}
	For $c\in CF_{2m-k}(H)$ and $d\in CF_{k+1}(\hat{H})$, we have \[  L_{H,k}^{F}(c,\partial_{\hat{H}}d)=(-1)^{k+1}L_{H,k+1}^{F}(\partial_H c,d). \]
\end{cor}

\begin{proof}
	By bilinearity it suffices to check the identity on generators for the appropriate graded pieces of the Floer complexes.    Let $[\gamma,w]\in \tilde{P}_{2m-k}(H)$ and $[\gamma',w']\in \tilde{P}_{2m-k-1}(H)$, so that $[\bar{\gamma},\bar{w}]\in \tilde{P}_k(\hat{H})$ and $[\bar{\gamma}',\bar{w}']\in\tilde{P}_{k+1}(H)$.  If $B\in \Gamma'$, the coefficient of $T^{I_{\omega}(B)}$ in $L_{H,k}^{F}([\gamma,w],\partial_{\hat{H}}[\bar{\gamma}',\bar{w}'])$ is the signed count of elements of the zero-manifold $\mathcal{M}_{J,\hat{H}}([\bar{\gamma}',\bar{w}'],[\bar{\gamma},B\#\bar{w}])=\frac{\tilde{\mathcal{M}}_{J,\hat{H}}([\bar{\gamma}',\bar{w}'],[\bar{\gamma},B\#\bar{w}])}{\R}$, while the coefficient of $T^{I_{\omega}(B)}$ in $L_{H,k+1}^{F}(\partial_H[\gamma,w],[\bar{\gamma}',\bar{w}'])$ is the signed count of elements of $\frac{\tilde{\mathcal{M}}_{J,H}([\gamma,w],[\gamma',B\#w'])}{\R}$.  
	
	Now $\tilde{\mathcal{M}}_{J,\hat{H}}([\bar{\gamma}',\bar{w}'],[\bar{\gamma},B\#\bar{w}])$ is equal, as an oriented manifold, to $\tilde{\mathcal{M}}_{J,\hat{H}}([\bar{\gamma}',(-B)\#\bar{w}'],[\bar{\gamma},\bar{w}])$ which (since $[\bar{\gamma}',(-B)\#\bar{w}']=[\bar{\gamma}',\overline{B\#w'}]$) is diffeomorphic via $\mathbb{I}_{\mathrm{cyl}}$ to $\tilde{\mathcal{M}}_{J,H}([\gamma,w],[\gamma',B\#w'])$ with orientation sign $(-1)^{\mu_H([\gamma,w])}=(-1)^k$ according to Proposition \ref{floerflip}(iii). Since the map $\mathbb{I}_{\mathrm{cyl}}$ is anti-equivariant with respect to the $\R$-actions (given in both cases by translation in the positive $s$ direction), it descends to a diffeomorphism on the quotients that affects orientation by $(-1)^{k+1}$.  Thus the coefficients on all powers of $T$ in $L_{H,k}^{F}([\gamma,w],\partial_{\hat{H}}[\bar{\gamma}',\bar{w}'])$ and $L_{H,k+1}^{F}(\partial_H[\gamma,w],[\bar{\gamma}',\bar{w}'])$ are related by the sign $(-1)^{k+1}$, from which the result follows.\end{proof}
	
\begin{cor}\label{orpssadj}
	For $c\in (\hat{\Lambda}_{\uparrow}\otimes_{\kappa}\mathbf{CM}_*(f;\kappa))_{2m-k}$ and $d\in CF_{k}(\hat{H})$, we have   \begin{equation}\label{psspair}  L_{H,k}^{F}(\tilde{\Phi}_{PSS}^{f,H^+}(c), d)=L_{f,k}^{M}(c,\tilde{\Psi}^{\hat{H}^-,-f}_{PSS}(d)).\end{equation}
	\end{cor}	
	
	Here, as will be apparent in the proof, we use the orientation rules for the Morse complexes of $f$ and $-f$ (relating the descending manifolds for $\pm v$) dictated in the statement of Proposition \ref{wor}.
\begin{proof}	
	Again it suffices to prove the result when $c$ and $d$ are standard generators for the relevant vector spaces over $\Lambda_{\uparrow}$. Let $p\in \mathrm{Crit}(f)$ and $A\in \hat{\Gamma}$ with $\mathrm{ind}_{f}(p)-2I_{c_1}(A)=k$, and let $[\gamma,w]\in \tilde{P}_{2m-k}(H)$ (so $[\bar{\gamma},\bar{w}]\in \tilde{P}_k(\hat{H})$).  The coefficient (valued in $\kappa$) of $T^0$ in $L_{H,k}^{F}(\tilde{\Phi}_{PSS}^{f,H^+}(T^Ap),[\gamma,w])$ is then the signed count of points in \begin{equation}\label{preim1} \mathcal{M}_{J,H^+}^{\mathrm{in}}(p;[\gamma,w\#(-A)])=(ev_{0,[\gamma,w\#(-A)]}^{J,H^+})^{-1}(\mathbf{D}^v(p)),\end{equation} while the corresponding coefficient in $L_{f,k}^{M}(T^Ap,\tilde{\Psi}^{\hat{H}^-,-f}_{PSS}[\bar{\gamma},\bar{w}])$ is the signed count of points in \begin{equation}\label{preim2}\mathcal{M}_{J,\hat{H}^-}^{\mathrm{out}}([\bar{\gamma},\bar{w}\#A]) = (ev_{\infty,[\bar{\gamma},\bar{w}\#A]}^{J,\hat{H}^-})^{-1}(\mathbf{A}^{-v}(p)).\end{equation}  Since $[\bar{\gamma},\bar{w}\#A]=[\bar{\gamma},\overline{w\#(-A)}]$, Proposition \ref{floerflip}(i) shows that the domains $\mathcal{M}_{J,H^+}^{\mathrm{in}}([\gamma,w\#(-A)])$ and $\mathcal{M}_{J,\hat{H}^-}^{\mathrm{out}}([\bar{\gamma},\bar{w}\#A])$ of $ev_{0,[\gamma,w\#(-A)]}^{J,H^+}$ and $ev_{\infty,[\bar{\gamma},\bar{w}\#A]}^{J,\hat{H}^-}$  are identified orientation-preservingly by the diffeomorphism $\mathbb{I}_{\pm}$; moreover one has $ev_{0,[\gamma,w\#(-A)]}^{J,H^+}= ev_{\infty,[\bar{\gamma},\bar{w}\#A]}^{J,\hat{H}^-}\circ\mathbb{I}_{\pm}$. Of course the submanifolds $\mathbf{D}^v(p)$ and $\mathbf{A}^{-v}(p)$ appearing in (\ref{preim1}) and (\ref{preim2}) are equal to each other as sets.  Their coorientations, as defined by our conventions and used in the definitions of the orientations on (\ref{preim1}) and (\ref{preim2}) are also equal to each other: $\mathbf{D}^v(p)$ was cooriented so that $N_p\mathbf{D}^v(p)\oplus T_p\mathbf{D}^v(p)=T_pM$ as oriented vector spaces; meanwhile, by general rule, we coorient $\mathbf{A}^{-v}(p)$ by identifying $N_p\mathbf{A}^{-v}(p)$ with $T_p\mathbf{D}^{-v}(p)$, and the prescription in Proposition \ref{wor} dictates that $T_p\mathbf{D}^{-v}(p)\oplus T_p\mathbf{D}^v(p)=T_pM$ as oriented vector spaces.  Given that these coorientations agree, $\mathbb{I}_{\pm}$ restricts as an orientation-preserving diffeomorphism between the zero-manifolds in (\ref{preim1}) and (\ref{preim2}). Thus for any $\hat{A}\in \hat{\Gamma}$ with $\mathrm{ind}_f(p)-2I_{c_1}(A)=\mu_H([\gamma,w])=2m-k$, the coefficients of $T^0$ in  $L_{H,k}^{F}(\tilde{\Phi}_{PSS}^{f,H^+}(T^Ap),[\gamma,w])$ and $L_{f,k}^{M}(T^Ap,\tilde{\Psi}^{\hat{H}^-,-f}_{PSS}[\bar{\gamma},\bar{w}])$ agree.  Applying this with $A$ replaced by $A-B$ for arbitrary $B$ having $I_{c_1}(B)=0$ shows the respective coefficients of $T^{\omega(B)}$ also all agree, implying the result.
 \end{proof}

\begin{remark}
	On both the Morse and Floer sides, one could remove some of the arbitrariness in our definitions by taking the approach in \cite[(12f)]{SeBook},\cite[Section 1.5]{Ab}: instead of taking the summand in the Morse complex $\mathbf{CM}_*(f;\kappa)$ corresponding to a critical point $p$ to be an explicit copy of the field $\kappa$ (with a distingushed generator $1$), one would take it to be a one-dimensional $\kappa$-vector space $o_{p,f}$ generated by the two possible orientations of $\mathbf{D}^v(p)$ subject to the relation that their sum is zero.  Similarly each $\gamma\in P(H)$ for which some left-cap $w$ has $[\gamma,w]\in\tilde{P}_k(H)$ would contribute to $CF_k(H)$ a one-dimensional $\Lambda_{\uparrow}$-vector space $o_{\gamma,H}$ generated by the two possible orientations of $\mathrm{Det}(D_{w}^{-})$, with their sum being set to zero. Our versions are isomorphic to these by using our choices of orientations of $\mathbf{D}^v(p)$ or of $\mathrm{Det}(D_{w}^-)$ as preferred bases for $o_{p,f}$ or $o_{\gamma,H}$, but the formulation as in \cite{SeBook},\cite{Ab} does not require one to choose such orientations; one would construct versions of  the various maps that we have considered by assigning, to each element of the appropriate zero-dimensional moduli spaces, not a sign $\pm 1$ but rather an isomorphism between the appropriate orientation modules $o_{p,f}$ or $o_{\gamma,H}$.
	
	In connection with the calculations in this subsection, one somewhat arbitrary pair of choices would remain after this, though they can be established as uniform conventions: since the $o_{p,f},o_{\gamma,H}$ do not have preferred generators, in order to define the pairings $L_{f,k}^{M},L_{H,k}^{F}$ one would need choices of isomorphisms $o_{p,f} \otimes o_{p,-f}\to \kappa$ and $o_{\gamma,H}\otimes o_{\bar{\gamma},\hat{H}}\to \Lambda_{\uparrow}$.  In our context these correspond respectively to the rules that $T_p\mathbf{D}^{-v}(p) \oplus T_p\mathbf{D}^v(p)=T_pM$ as oriented vector spaces, and that $\mathrm{Det}(D_{\bar{w}}^-)$ and $\mathrm{Det}(D_w^-)$ are related by identifying the former with $\mathrm{Det}(D_w^+)$ and then using the standard orientation on $\mathrm{Det}(D_w^+\#_SD_w^-)$ to relate the orientations on $\mathrm{Det}(D_w^+)$ and $\mathrm{Det}(D_w^-)$.  
	The facts that no sign appears in Corollary \ref{orpssadj}, and that the signs in Remark \ref{morsepairsign} and Corollary \ref{floerdualdiff} are identical, can be regarded as reflecting that these orientation conventions are compatible with each other.
\end{remark}


\begin{thebibliography}{99} 
\bibitem[Abo13]{Ab} M. Abouzaid. \emph{Symplectic cohomology and Viterbo's theorem}. In  \emph{Free loop spaces in geometry and topology}, IRMA Lectures in Mathematics and Theoretical Physics, \textbf{24}, EMS Press, Z\"urich, 2015, 271--485.
\bibitem[AD14]{AD} M. Audin and M. Damian. \emph{Morse theory and Floer homology}. Universitext, Springer, 2014.
\bibitem[BX22]{BX} S. Bai and G. Xu. \emph{Arnold conjecture over integers}. arXiv:2209.08599.
\bibitem[Ba94]{Bar} S. Barannikov. \emph{The framed Morse complex and its invariants}. Singularities and bifurcations, 93--115, Adv. Soviet Math., \textbf{21}, Amer. Math. Soc., Providence, RI, 1994.
\bibitem[BGO19]{BGO} N. Berkouk, G. Ginot, and S. Oudot. \emph{Level--sets persistence and sheaf theory}. arXiv:1907.09759.
\bibitem[Bo87]{Bou} N. Bourbaki. \emph{Topological vector spaces. Chapters 1–5}. Elements of Mathematics, Springer-Verlag, Berlin, 1987, Translated from the French by H. G. Eggleston and S. Madan.
\bibitem[Bu17]{B17} D. Burghelea. \emph{A refinement of Betti numbers and homology
in the presence of a continuous function, I}. Algebr. Geom. Topol. \textbf{17} (2017), no. 4, 2041--2080.
\bibitem[Bu18a]{Bubook} D. Burghelea. \emph{New topological invariants for real- and angle-valued maps. An alternative to Morse-Novikov theory}. World Scientific Publishing Co. Pte. Ltd., Hackensack, NJ, 2018.
\bibitem[Bu18b]{B18} D. Burghelea. \emph{A refinement of Betti numbers and homology in the presence of a continuous function II: The case of an angle-valued map}. Algebr. Geom. Topol. \textbf{18} (2018), no. 5, 3037--3087.
\bibitem[Bu20a]{B1} D. Burghelea. \emph{Alternative to Morse-Novikov Theory for closed 1-form. I.}. Eur. J. Math. \textbf{6} (2020), no. 3, 713--750.  Erratum in Eur. J. Math. \textbf{8} (2022), no. 1, 426.
\bibitem[Bu20b]{B20b} D. Burghelea. \emph{Alternative to Morse-Novikov Theory for closed 1-form. II.}. arXiv:2009.05858v5.
\bibitem[BD13]{BD} D. Burghelea and T. Dey. \emph{Topological persistence for circle valued maps}. Discrete Comput. Geom. \textbf{50} (2013), no. 1, 69--98.
\bibitem[BH01]{BH01} D. Burghelea and S. Haller. \textit{On the topology and analysis of a closed one form
(Novikov's theory revisited)}. Monogr. Enseign. Math. \textbf{38} (2001), 133--175.
\bibitem[BH17]{BH17} D. Burghelea and S. Haller. \emph{Topology of angle valued maps, bar codes and Jordan blocks}. J. Appl. Comput. Topol. \textbf{1} (2017), no. 1, 121--197.
\bibitem[CdSKM19]{CDKM} G. Carlsson, V. de Silva, S. Kali\v{s}nik, and D. Morozov. \emph{Parametrized homology via zigzag persistence}. Algebr. Geom. Topol. \textbf{19} (2019), no. 2, 657--700. 
\bibitem[CdSM09]{CDM} G. Carlsson, V. de Silva, and D. Morozov. \emph{Zigzag Persistent Homology and Real-valued Functions}, in  \emph{Proceedings 25th ACM Symposium on Computational Geometry (SoCG)}, 2009, pp. 247--256.
\bibitem[CdSGO16]{CDGO} F. Chazal, V. de Silva, M. Glisse, S. Oudot. \emph{The structure and stability of persistence modules},  SpringerBriefs in Mathematics. Springer, 2016.
\bibitem[CGG21]{CGG} E. Cineli, V. Ginzburg, and B. G\"urel. \emph{Topological entropy of Hamiltonian diffeomorphisms: a persistence homology and Floer theory perspective}. arXiv:2111.03983.
\bibitem[CEH09]{ext} D. Cohen-Steiner, H. Edelsbrunner, and J. Harer. \emph{Extending persistence using Poincar\'e and Lefschetz duality}. Found. Comput. Math. \textbf{9} (2009), no. 1, 79-–103. Erratum in Found. Comput. Math. \textbf{9} (2009), no. 1, 133--134.
\bibitem[CO20]{CO} J. Cochoy and S. Oudot. \emph{Decomposition of exact pfd persistence bimodules}. Discrete Comput. Geom. \textbf{63} (2020), no. 2, 255--293. 
\bibitem[Dol80]{Dold} A. Dold. \emph{Lectures on Algebraic Topology}. Second Edition. Springer-Verlag, Berlin, 1980.
\bibitem[EH10]{EH} H. Edelsbrunner and J. Harer.  \emph{Computational topology. An introduction}. American Mathematical Society, Providence, RI, 2010. 
\bibitem[EP03]{EP} M. Entov and L. Polterovich. \emph{Calabi quasimorphism and quantum homology}. Int. Math. Res. Not. 2003, no. 30, 1635--1676.
\bibitem[FH93]{FHor} A. Floer and H. Hofer. \emph{Coherent orientations for periodic orbit problems in symplectic geometry}. Math. Z. \textbf{212} (1993), no. 1, 13--38.
\bibitem[FH94]{FH} A. Floer and H. Hofer. \emph{Symplectic homology. I. Open sets in $\mathbb{C}^n$}. 
Math. Z. \textbf{215} (1994), no. 1, 37--88. 
\bibitem[Ha02]{Ha} A. Hatcher. \emph{Algebraic Topology}. Camb. Univ. Press, 2002.
\bibitem[HS95]{HS} H. Hofer and D. Salamon. \emph{Floer homology and Novikov rings}. In \emph{The Floer memorial volume}, 483--524, Progr. Math., \textbf{133}, Birkh\"auser, Basel, 1995.
\bibitem[HL99]{HL} M. Hutchings and Y.-J. Lee. \emph{Circle-valued Morse theory, Reidemeister torsion and Seiberg-Witten invariants of 3-manifolds}  Topology \textbf{38} (1999), 861–-888.
\bibitem[IH71]{IH} M. Ikeda and M. Haifawi. \emph{On the best approximation property in non-Archimedean normed spaces}. Indag. Math. \textbf{74} (1971), 49--52.
\bibitem[Kal75]{Kal} G. Kalmbach. \emph{On some results in Morse theory}.
Canadian J. Math. \textbf{27} (1975), 88--105.
\bibitem[KM09]{KM} J. Kati\'c and D. Milinkovi\'c. \emph{Coherent orientation of mixed moduli spaces in Morse-Floer theory}. Bull. Braz. Math. Soc. \textbf{40} (2009), no. 2, 253--300.
\bibitem[Lat94]{Lat} F. Latour. \emph{Existence de $1$-formes ferm\'ees non singuli\`eres dans une
classe de cohomologie de de Rham}. Publ. Math. de l’I.H.\'E.S., \textbf{80} (1994), p. 135--194.
\bibitem[Lau92]{Lau} F. Laudenbach. \emph{Appendix. On the Thom-Smale Complex}. Ast\'erisque \textbf{205} (1992), 219--233. 
\bibitem[LO95]{LO} H. V. L\^{e} and K. Ono. \emph{Symplectic fixed points, the Calabi invariant and Novikov homology}. Topology \textbf{34} (1995), no. 1, 155--176.
\bibitem[Lev77]{Lev} J. Levine. \emph{Knot Modules. I}. Trans. Amer. Math. Soc. \textbf{229} (1977), 1--50.
\bibitem[Mc10]{M10} D. McDuff. \emph{Monodromy in Hamiltonian Floer theory}.  Comment. Math. Helv. \textbf{85} (2010), no. 1, 95--133.
\bibitem[McSa12]{MS} D. McDuff and D. Salamon. {$J$-holomorphic curves and symplectic topology}. Second edition. American Mathematical Society Colloquium Publications, \textbf{52}. American Mathematical Society, Providence, RI, 2012.
\bibitem[Mil63]{Mil} J. Milnor. \emph{Morse theory}.  Annals of Mathematics Studies, \textbf{51}, Princeton University Press, Princeton, N.J., 1963. 
\bibitem[MoSp65]{MoSp} A. F. Monna and T. A. Springer. \emph{Sur La Structure Des Espaces De Banach Non-Archimediens}. Indag. Math. \textbf{68} (1965), 602--614.
\bibitem[Mor25]{Mor} M. Morse. \emph{Relations between the critical points of a real function of $n$ independent variables}. Trans. Amer. Math. Soc. \textbf{27} (1925), no. 3, 345--396.
\bibitem[Nov81]{Nov} S. P. Novikov. \emph{Multivalued functions and functionals. An analog of Morse theory}. Dokl. Akad. Nauk SSSR \textbf{260} (1981), no. 1, 31--355.
\bibitem[Oh97]{Oh} Y.-G. Oh. \emph{Symplectic topology as the geometry of action functional. I. Relative Floer theory on the cotangent bundle}. J. Differential Geom. \textbf{46} (1997), no. 3, 499--577.
\bibitem[Oh15]{ohbook} Y.-G. Oh. \emph{Symplectic topology and Floer homology}. New Mathematical Monographs, Cambridge University Press, 2 vols., 2015.
\bibitem[Paj95]{Pa95} A. V. Pajitnov. \emph{On the Novikov complex for rational Morse forms}.  Ann. Fac. Sci. Toulouse \textbf{4} (1995), 297-–338.
\bibitem[Paj06]{Pa} A. V. Pajitnov. \emph{Circle-valued Morse theory}. de Gruyter Studies in Mathematics \textbf{32}, Berlin, 2006.
\bibitem[Par16]{Par} J. Pardon. \emph{An algebraic approach to virtual fundamental cycles
on moduli spaces of pseudo-holomorphic curves}. Geom. Topol. \textbf{20} (2016), 779--1034.
\bibitem[PSS94]{PSS} S. Piunikhin, D. Salamon, and M. Schwarz.  \emph{Symplectic Floer-Donaldson theory and quantum cohomology}. In  \emph{Contact and symplectic geometry (Cambridge, 1994)}, Publ. Newton Inst., \textbf{8}, Cambridge Univ. Press, Cambridge, 1996, 171--200.
\bibitem[PS16]{PS} L. Polterovich and E. Shelukhin. \emph{Autonomous Hamiltonian flows, Hofer's geometry and persistence modules}. Selecta Math. \textbf{22} (2016), 227--296.
\bibitem[Qin10]{Qin} L. Qin. \emph{On moduli spaces and CW structures arising from Morse theory on Hilbert manifolds}. J. Topol. Anal. \textbf{2} (2010), no. 4, 469--526.
\bibitem[Ran02]{Ran} A. Ranicki. \emph{Algebraic and geometric surgery}. Oxford Mathematical Monographs, 2002.
\bibitem[Rot09]{Rot} J. Rotman. \emph{An introduction to homological algebra}. Second Edition. Universitext, Springer, 2009.
\bibitem[Sal97]{Sal} D. Salamon. \emph{Lectures on Floer homology}. Lecture Notes for the IAS/PCMI Graduate Summer School on Symplectic Geometry and Topology, 1997.
\bibitem[Sc93]{Sc} M. Schwarz. \emph{Morse homology}. Progr. Math. \textbf{111}, Birkh\"auser, 1993.
\bibitem[Sc95]{ScT} M. Schwarz. \emph{Cohomology operations from $S^1$-cobordisms in Floer homology}.  ETH dissertation, 1995.
\bibitem[Sc00]{Sc00} M. Schwarz. \emph{On the action spectrum for closed symplectically aspherical manifolds}. Pacific J. Math. \textbf{193} (2000), 419--461.
\bibitem[Se97]{Se} P. Seidel. \emph{$\pi_1$ of symplectic automorphism groups and invertibles in quantum homology rings}. Geom. Funct. Anal. \textbf{7} (1997), no. 6, 1046--1095.
\bibitem[Se08]{SeBook} P. Seidel. \emph{Fukaya Categories and Picard--Lefschetz Theory}. Zurich Lectures in Advanced Mathematics, EMS, Z\"urich, 2008.
 \bibitem[Sh22]{She} E.  Shelukhin. \emph{On the Hofer-Zehnder conjecture}. Ann. Math. (2) \textbf{195} (2022), no. 3, 775--839.
\bibitem[Us08]{U08} M. Usher. \emph{Spectral numbers in Floer theories}, Compositio Math. \textbf{144} (2008), 1581--1592.
\bibitem[Us10]{U10} M. Usher. \emph{Duality in filtered Floer-Novikov complexes}. J. Topol. Anal. \textbf{2} (2010), no. 2, 233--258.
\bibitem[Us11]{U11} M. Usher. \emph{Deformed Hamiltonian Floer theory, capacity estimates and Calabi quasimorphisms}. Geom. Topol. \textbf{15} (2011), no. 3, 1313--1417. 
\bibitem[UZ16]{UZ} M. Usher and J. Zhang. \emph{Persistent homology and Floer-Novikov theory}. Geom. Topol. \textbf{20} (2016), 3333--3430.
\bibitem[Vit92]{Vit} C. Viterbo. \emph{Symplectic topology as the geometry of generating functions}. Math. Ann. \textbf{292} (1992), no. 4, 685--710.
\bibitem[Zap15]{Zap} F. Zapolsky. \emph{The Lagrangian Floer-quantum-PSS package and canonical orientations in Floer theory}, arXiv:1507.02253.
\bibitem[ZC05]{ZC} A. Zomorodian and G. Carlsson. \emph{Computing persistent homology}. Discrete Comput. Geom. \textbf{33} (2005), 249--274.

\end{thebibliography}
\end{document}